\documentclass[11pt]{article}

\usepackage{amsmath}
\usepackage{authblk}
\usepackage{overpic}
\usepackage{soul}
\usepackage{lscape}

\usepackage{wrapfig}
\usepackage{multicol}

\usepackage{caption}

\usepackage{enumitem}
\setitemize{topsep=0pt,parsep=0pt,partopsep=0pt}
\setenumerate{noitemsep,topsep=0pt,parsep=0pt,partopsep=0pt}

\usepackage{units}
\usepackage{hyperref}
\usepackage{xurl}
\usepackage{framed}
\usepackage{latexsym}
\usepackage{amssymb}
\usepackage{color}
\usepackage{url}
\usepackage{fullpage}
\usepackage{comment}
\usepackage{float}

\usepackage{tikz}				%
\usetikzlibrary{cd}				%
\usetikzlibrary{math}			%
\usepackage{pgfplots}			%
\pgfplotsset{compat=1.13}
\usepgfplotslibrary{fillbetween}
\usetikzlibrary{patterns}

\usepackage{basic-commands}     %

\title{Curvature Sets Over Persistence Diagrams}

\author[1]{Mario G\'omez}
\author[2]{Facundo M\'emoli}

\affil[1]{Department of Mathematics,
		The Ohio State  University.\\
		\texttt{gomezflores.1@osu.edu}}

\affil[2]{Department of Mathematics and Department of Computer Science and Engineering, 
		The Ohio State University.\\
        \texttt{memoli@math.osu.edu}}

\date{\today}

\begin{document}
\maketitle

\begin{abstract}
	We study a family of invariants of compact metric spaces that combines the Curvature Sets defined by Gromov in the 1980s with Vietoris-Rips Persistent Homology. For given integers $k\geq 0$ and $n\geq 1$ we consider the dimension $k$ Vietoris-Rips persistence diagrams of \emph{all} subsets of a given metric space with cardinality at most $n$. We call these invariants \emph{persistence sets} and denote them as $\Dvr{n,k}$. We establish that (1) computing these invariants is often significantly more efficient than computing the usual Vietoris-Rips persistence diagrams, (2) these invariants have very good discriminating power and, in many cases, capture information that is imperceptible through standard Vietoris-Rips persistence diagrams, and (3)  they enjoy stability properties. We  precisely characterize some of them in the case of spheres and surfaces with constant curvature using a generalization of Ptolemy's inequality. We also identify a rich family of metric graphs for which $\Dvr{4,1}$ fully recovers their homotopy type by studying split-metric decompositions. Along the way we prove some useful properties of Vietoris-Rips persistence diagrams using Mayer-Vietoris sequences. These yield a geometric algorithm for computing the Vietoris-Rips persistence diagram of a space $X$ with cardinality $2k+2$ with quadratic time complexity as opposed to the much higher cost incurred by the usual algebraic algorithms relying on matrix reduction.
\end{abstract}

\tableofcontents
\setcounter{section}{0}

\section{Introduction}
The Gromov-Hausdorff (GH) distance, a notion of distance between compact metric spaces, was introduced by Gromov in the 1980s and was eventually adapted into data/shape analysis by the second author \cite{memoli-thesis,ms-sgp,ms}  as a tool for measuring the dissimilarity between shapes/datasets.

Despite its usefulness in providing a mathematical model for shape matching procedures, \cite{ms-sgp,ms,bbk-book}, the Gromov-Hausdorff distance leads to NP-hard problems: \cite{mem12} relates it to the well known Quadratic Assignment Problem, which is NP-hard, and Schmiedl in his PhD thesis \cite{gh_not_approx_polytime} (see also \cite{agarwal2018computing}) directly proves the NP-hardness of the computation of the Gromov-Hausdorff distance even for ultrametric spaces. Recent work has also identified certain Fixed Parameter Tractable algorithms for the GH distance between ultrametric spaces \cite{dgh-ult}.

These hardness results have motivated research in other directions:
\begin{itemize}
	\item[(I)] finding suitable \emph{relaxations} of the Gromov-Hausdorff distance which are more amenable to computations and \item[(II)] finding lower bounds for the Gromov-Hausdorff distance which are easier to compute, yet retain good discriminative power.
\end{itemize}

Related to the first thread, and based on ideas from optimal transport, the notion of Gromov-Wasserstein distance was proposed in \cite{memoli-dghlp,memoli-dghlp-long}. This notion of distance leads to continuous quadratic optimization problems (as oposed to the combinatorial nature of the problems induced by the Gromov-Hausdorff distance) and, as such, it has benefited from the wealth of continuous optimization computational techniques that are available in the literature \cite{peyre2016gromov,peyre2019computational} and has seen a number of applications in data analysis and machine learning \cite{vayer2020fused,demetci2020gromov,alvarez2018gromov,kawano2021classification,blumberg2020mrec}  in recent years.

The second thread mentioned above is that of obtaining  computationally tractable lower bounds for the usual Gromov-Hausdorff distance. Several such lower bounds were identified in \cite{mem12} by the second author, and then in \cite{clust-um-old,clust-um} and \cite{ccsg09a} it was proved that hierarchical clustering dendrograms and  \emph{persistence diagrams} or \emph{barcodes}, metric invariants which arose in the Applied Algebraic Topology community, provide a lower bound for the GH distance. These persistence diagrams will  eventually become central to the present paper, but before reviewing them, we will describe the notion of \emph{curvature sets} introduced by Gromov.

\paragraph{Gromov's curvature sets and curvature measures.}

Given  a compact metric space $(X,d_X)$, in the book \cite{gro99} Gromov identified a class of invariants of metric spaces indexed by the natural numbers that classifies compact metric spaces up to isometry. In more detail, Gromov defines for each $n\in \N$, the $n$-th \emph{curvature set} of $X$, denoted by $\Kn_n(X)$, as the collection of all $n\times n$ matrices that arise from restricting $d_X$ to all possible $n$-tuples of points chosen from $X$, possibly with repetitions. The terminology curvature sets is justified by the observation that these sets contain, in particular, metric information about configurations of closely clustered points in a given metric space. This information is enough to recover the curvature of a manifold; see Figure \ref{fig:cset-curve}.

\begin{figure}
	\centering
	\includegraphics[width=0.3\linewidth]{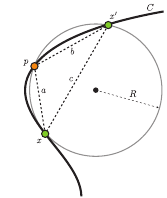}
	\caption{The curvature of a smooth curve $C$ can be estimated as the inverse of the radius $R$ of the circle passing through the points $x,x'$ and $p$. By  plane geometry results \cite[Theorem 2.3]{calabi1998differential}, this radius can be computed from the 3 interpoint distances $a$, $b$, and $c$,
		and hence from $\Kn_3(C)$, as $R=R(a,b,c) = \frac{a\,b\,c}{\left((a+b+c)(a+b-c)(a-b+c)(-a+b+c)\right)^{1/2}}$. In fact, \cite{calabi1998differential} proves that $R^{-1} = \kappa + \frac{1}{3}(b-a) \kappa_s + \cdots$ where $\kappa$ and $\kappa_s$ are the curvature and its arc length derivative at the point $p$.}
	\label{fig:cset-curve}
\end{figure}

These curvature sets have the property that $\Kn_n(X)=\Kn_n(Y)$ for all $n\in\N$ is equivalent to the statement that the compact metric spaces $X$ and $Y$ are isometric. Constructions similar to the curvature sets of Gromov were also identified by Peter Olver in \cite{olver-joint} in his study of invariants for curves and surfaces under different group actions (including the group of Euclidean isometries).

\cite{mem12} points out that the GH distance admits lower bounds based on these curvature sets: 
\begin{equation} \label{eq:cset} \dGH(X,Y)\geq \dgroM{X}{Y}:=\frac{1}{2}\sup_{n\in\N} \dH(\Kn_n(X),\Kn_n(Y))
\end{equation}
for all $X,Y$ compact metric spaces. Here, $d_{\mathcal{H}}$ denotes the Hausdorff distance on $\R^{n\times n}$ with $\ell^\infty$ distance. As we mentioned above, the computation of the Gromov-Hausdorff distance leads in general to NP-hard problems, whereas the lower bound in the equation above can be computed in polynomial time when restricted to definite values of $n$.  In \cite{mem12} it is argued that work of Peter Olver \cite{olver-joint} and Boutin and Kemper \cite{boutin} leads to identifying rich classes of shapes where these lower bounds permit full discrimination.

In the category of compact mm-spaces, that is triples $(X,d_X,\mu_X)$ where $(X,d_X)$ is a compact metric space and $\mu_X$  is a fully supported probability measure on $X$ (see Definition \ref{def:mm_space}), Gromov also discusses the following parallel construction: for an mm-space $(X,d_X,\mu_X)$ let $\Psi_X^{(n)}:X^{\times n}\longrightarrow \R^{n\times n}$ be the map that sends the $n$-tuple $(x_1,x_2,\ldots,x_n)$ to the matrix $M$ with elements $M_{ij}=d_X(x_i,x_j)$.   Then, the $n$-th \textbf{curvature measure} of $X$ is defined as 
\begin{equation}
	\label{eq:mun}
	\mu_{n}(X):=\Big(\Psi_X^{(n)}\Big)_{\#}\mu_X^{\otimes n},
\end{equation}
where $\mu_X^{\otimes n}$ is the product measure on $X^{\times n}$ and $(\Psi_X^{(n)})_\# \mu_X^{\otimes n}$ is the pushforward to $\R^{n \times n}$. Clearly, curvature measures and curvature sets are related by $\supp(\mu_n(X))=\Kn_n(X)$ for all $n\in\N$.  Gromov then proves in his mm-reconstruction theorem that the collection of all curvature measures  permit reconstructing any given mm-space up to isomorphism.
See Theorem \ref{thm:stability_dGW_hat} for a relationship, analogous to (\ref{eq:cset}), between the curvature measures and the Gromov-Wasserstein distance.

\paragraph{Persistent Homology.}
Ideas related to what is nowadays know as persistent homology appeared already in the late 1980s and early 1990s in the work of Patrizio Frosini \cite{frosini-thesis,frosini-metric-homotopies,frosini}, then in the work of Vanessa Robins \cite{robins1999towards}, in the work of Edelsbrunner and collaborators \cite{orig-topopers}, and then in the work of Carlsson and Zomorodian \cite{zomo-carlsson}. Some excellent  references for this topic are \cite{eh10,ghrist-barcodes,carlsson_2014,weinberger2011persistent}.

In a nutshell, persistent homology (PH) assigns to a given compact metric space $X$ and an integer $k\geq 0$, a multiset of points $\dgm_k^\vr(X)$ in the plane, known as the $k$-th (Vietoris-Rips) \emph{persistence diagram} of $X$. The standard PH pipeline is shown in Figure \ref{fig:PH}.

\begin{figure}
	\noindent
	\makebox[\textwidth]{
		\includegraphics[width=1.15\linewidth]{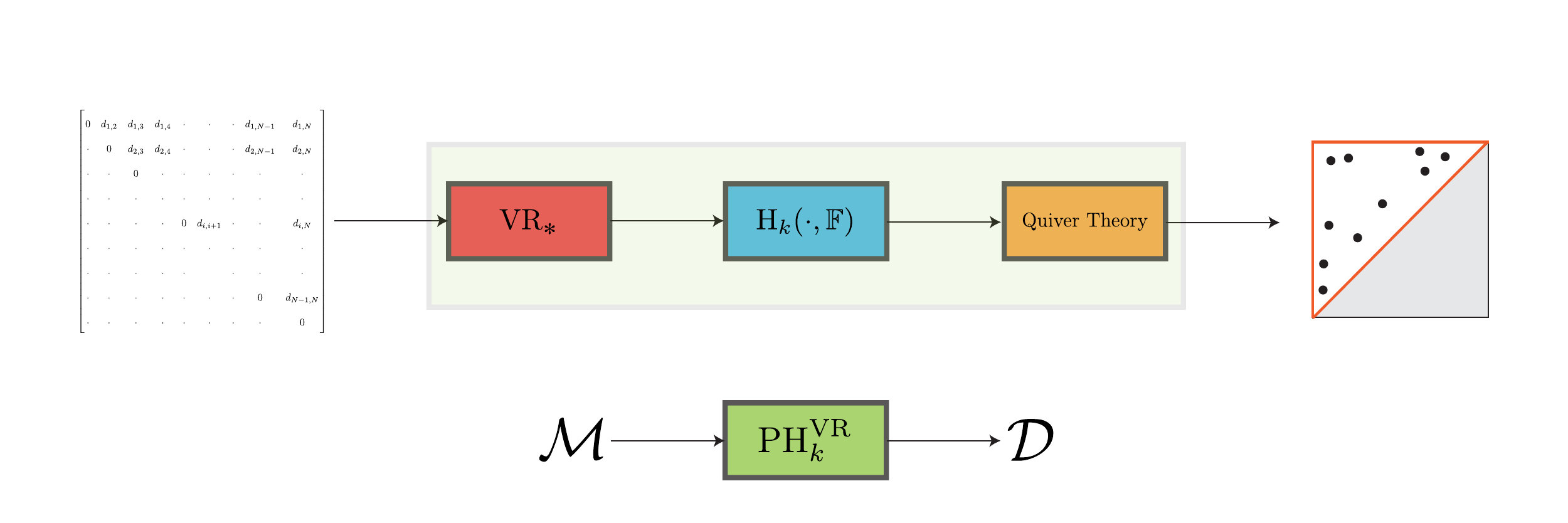}
	}
	\caption{The pipeline to compute a persistence diagram. Starting with a distance matrix, we compute the Vietoris-Rips complex and its reduced homology, and produce an interval decomposition. Together, we call these three steps $\PH_k^\vr$.}
	\label{fig:PH}
\end{figure}

These diagrams indicate the presence of $k$-dimensional multi-scale topological features in the space $X$, and can be compared via the  \emph{bottleneck distance} (which is closely related to but is stronger than the Hausdorff distance in $(\R^2,\ell^\infty)$).

Following work by Cohen-Steiner et al. \cite{cseh07}, in \cite{ccsg09a} it is proved that the maps $X\mapsto \dgm_k^\vr(X)$ sending a given compact metric space to its $k$-th Vietoris-Rips persistence diagrams is $2$-Lipschitz under the GH and bottleneck distances. 

Algorithmic work by Edelsbrunner and collaborators \cite{orig-topopers} and more recent developments \cite{ripser} guarantee that not only can $\dgm_k^\vr(X)$ be computed in polynomial time (in the cardinality of $X$) but also it is well known that the bottleneck distance can also be computed in polynomial time \cite{eh10}. This means that persistence diagrams provide another source of stable invariants which would permit estimating (lower bounding) the Gromov-Hausorff distance.

It is known that persistence diagrams are not full invariants of metric spaces. For instance, any two \emph{tree metric spaces}, that is metric graphs that are $\delta$-hyperbolic with $\delta=0$ \cite{gromov-hyperbolic}, have trivial persistence diagrams in all degrees $k\geq 1$. It is also not difficult to find two finite tree metric spaces with the same degree zero persistence diagrams. See \cite{osman-memoli} for more examples and \cite{memoli2019persistent} for results about stronger invariants (i.e. \emph{persistent homotopy groups}).

Despite the fact that persistence diagrams for a fixed degree $k$ can be computed with effort which depends polynomially on the size of the input metric space \cite{eh10,omega_matrix_mult}, the computations are actually quite onerous and, as of today, it is not realistic to compute the degree 1 Vietoris-Rips persistence diagram of a finite metric space with more than a few thousand points even with state of the art implementations such as Ripser \cite{ripser} and Ripser++ \cite{ripser-pp}.

\paragraph{Curvature sets over persistence diagrams.} In this paper, we consider a version  of the curvature set ideas which arises when combining their construction with Vietoris-Rips persistent homology. For a compact metric space $X$ and integers $n\geq 1$ and $k\geq 0$, the $(n,k)$-Vietoris-Rips persistence set of $X$ is (cf. Definition \ref{def:pset}) the collection $\Dvr{n,k}(X)$ of all persistence diagrams in degree $k$ of subsets of $X$ with cardinality at most $n$. In a manner similar to how the $n$-th curvature measure $\mu_n(X)$ arose above, we also study the probability measure $\Unk{n,k}{\mathrm{VR}}(X)$ defined as the pushforward of $\mu_n(X)$ under the degree $k$ Vietoris-Rips persistence diagram map (cf. Definition \ref{def:pmeas}). We  also study a more general version wherein  for any simplicial filtration functor $\mathfrak{F}$ (cf. Definition \ref{def:filtrations}), we consider both the persistence sets $\Dnk{n,k}{\mathfrak{F}}(X)$ and the the persistence measures $\Unk{n,k}{\mathfrak{F}}(X)$. 
Furthermore, as we discuss below, for certain choices of the parameters $k$ and $n$ curvature sets are not only more efficient to compute (in terms of memory requirements and/or in terms overall computational cost)  than standard persistence diagrams, but they also often  capture  information which is not directly visible through the lens of standard persistence diagrams.

	\begin{figure}[ht]
		\centering
		\includegraphics[width=\linewidth]{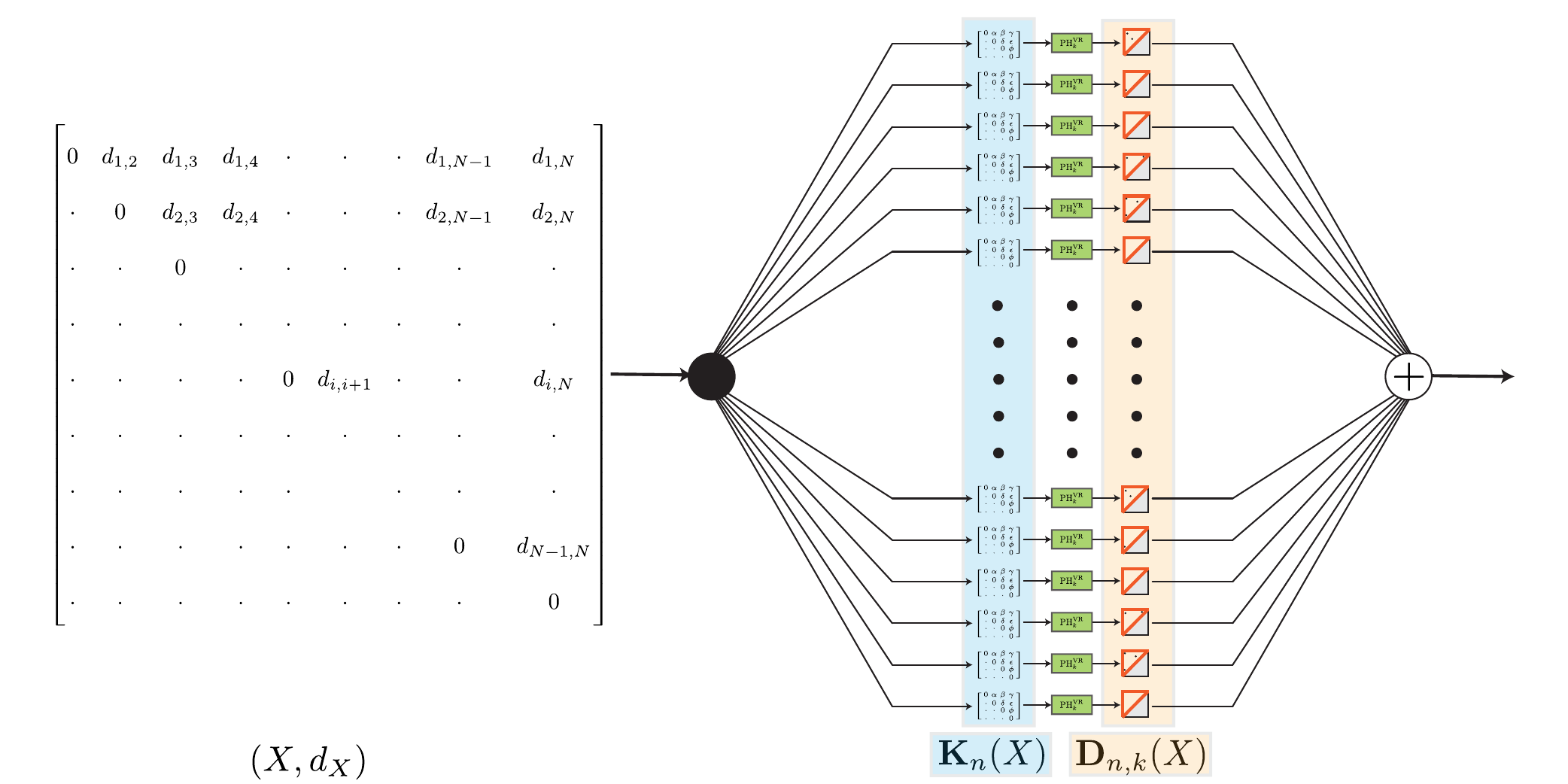}
		\caption{The pipeline to compute $\Dvr{n,k}$. Starting with a metric space $(X,d_X)$, we take samples of the distance matrix as elements of $\Kn_{n}(X)$, apply $\PH_k$ to each, and aggregate the resulting persistence diagrams. For example, Theorem \ref{thm:n=2k+2} guarantees that the VR-persistence diagram in dimension $k$ of a metric space with $n=2k+2$ points only has one point. The aggregation in this case means plotting the set $\Dvr{n,k}(X)$ by plotting all diagrams simultaneously in one set of axes. In general, the diagrams in $\Dvr{n,k}(X)$ have more than 1 point, so one possibility for aggregation is constructing a one-point summary or an average of a persistence diagram (for instance, a Chebyshev center or an $\ell_\infty$ mean) and then plotting all such points simultaneously. The figure aims to convey the eminently parallelizable nature of $\Dvr{n,k}(X)$.}
		\label{fig:multiplex}
	\end{figure}
\begin{figure}[ht]
	\centering
	
	\includegraphics[scale=0.65]{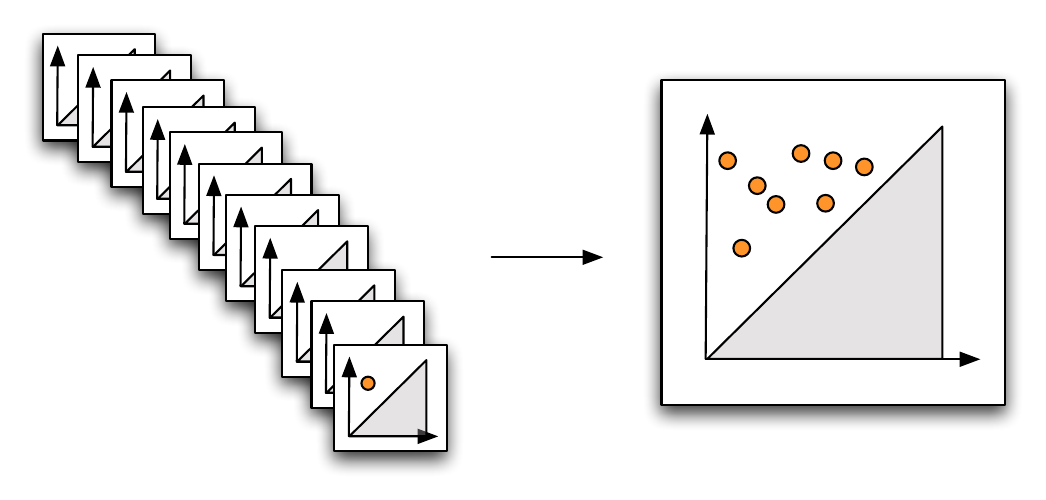}
	
	\caption{A graphical representation of how the \textbf{principal persistent set $\Dvr{2k+2,k}(X)$} is obtained by overlaying the persistence diagrams of all samples $Y \subset X$ (with $|Y| \leq 2k+2$) into a \emph{single} set of axes. This is made possible since by Theorem \ref{thm:n=2k+2} these diagrams have at most one off diagonal point.}
	\label{fig:stack-D-to1}
\end{figure}

\subsection{Contributions}
We believe that persistence sets are useful as an alternative  paradigm for the efficient computation of invariants/features based on persistent homology. We believe that persistence sets are useful as a general paradigm for the efficient computation of invariants/features based on persistent homology. Persistence sets are designed to generalize and complement --- not substitute the usual persistence diagrams. We provide a thorough study of persistence sets and, in particular, analyze the following points.

\paragraph{Persistence sets and measures generalize $\dgm_*^\vr$.}
The family $\{\Dvr{n,k}(X)\}_{n\geq 1,k\geq 0}$ of all persistence sets of $X$ generalizes the family $\{\dgm^\vr_k(X)\}_{k\geq 0}$ of all Vietoris-Rips persistence diagrams of $X$ in the sense that, when $n = |X| < \infty$, $\dgm^\vr_k(X)$ is an element of $\Dvr{n,k}(X)$ for each $k\geq 0$.

\paragraph{Some persistence sets and measures can discriminate spaces that $\dgm_*^\vr$ cannot.}
There are many cases in which Vietoris-Rips barcodes are unable to discriminate spaces, see discussion in Section 9.4 of \cite{osman-memoli}. For instance, the existence of a \emph{crushing} $X \to Y$ (in the sense of Hausmann) between metric spaces  such that $Y\subseteq X$  gives for each $r>0$ homotopy equivalences $\vr_r(X) \simeq \vr_r(Y)$ through Proposition 2.2 of \cite{hausmann-cohomology}. Furthermore, the VR-persistence diagrams of $X$ and $Y$ are equal; see  Figure \ref{fig:spiky-circle} for an example.

In contrast, it is interesting  that in many such scenarios \emph{some elements of the family of persistence sets can capture strictly more information than  VR persistent diagrams}.  In Example \ref{ex:D_n0} we show that the sets $\Dvr{n,0}(X)$ contain information about the distances in $X$, whereas $\dgm_0^\vr(X)$ is empty whenever $X$ is connected (recall that we use reduced homology). Additionally, in Example \ref{ex:circle_with_flares} we show a graph $G$ that consists of a cycle $C$ with 4 edges attached for which $\Dvr{4,1}(G)$ is different (more precisely, larger than) $\Dvr{4,1}(C)$; cf. Figure \ref{fig:circle_with_flares}. This observation generalizes to the $k$-sphere with $2k+2$ edges attached; see Proposition \ref{prop:sphere_with_flares} and Figure \ref{fig:spiky-sphere}.

\begin{figure}
	\centering
	\includegraphics{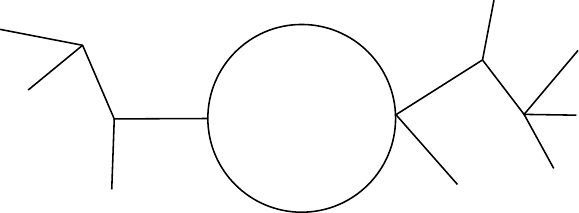}
	\caption{A graph $G$ formed by a circle $C$ with two trees attached. Since there is a crushing of $G$ to $C$ (in the sense of Hausmann \cite{hausmann-cohomology}), $\dgm_k^\vr(G)=\dgm_k^\vr(C)$ for all $k$.}
	\label{fig:spiky-circle}
\end{figure}

\begin{figure}[ht]
	\centering
	\includegraphics[height=150pt]{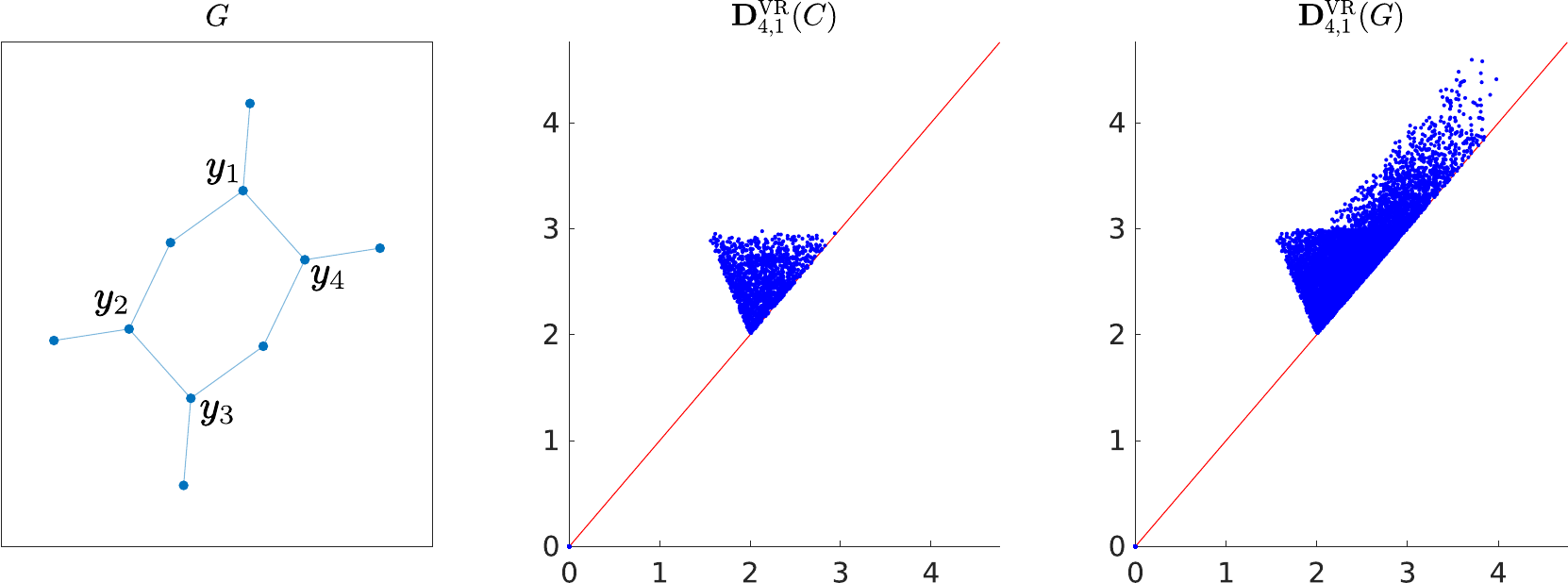}
	\caption{\textbf{Left:} A metric graph $G$ formed by a cycle $C$ with four edges attached. All edges have length 1. In the notation of Example \ref{ex:circle_with_flares}, the edges are attached at $y_1, y_2, y_3$, and $y_4$. \textbf{Middle:} The persistence set $\Dvr{4,1}(C)$. \textbf{Right:} Even though $\vr_*(G) \simeq \vr_*(C)$, and as a consequence the persistence diagrams are identical, the set $\Dvr{4,1}(G) \setminus \Dvr{4,1}(C)$ is non-empty (see Remark \ref{rmk:D41_vs_VR1}). The middle and right figures were obtained by sampling 100,000 configurations of 4 points uniformly from $G$. Of those, about 12.98\% were contained in $C$. The fraction of configurations in $G$ (resp. $C$) that produced a non-diagonal point in $\Dvr{4,1}(G)$ (resp. $\Dvr{4,1}(C)$) is 7.59\% (resp. 10.97\%).}
	\label{fig:circle_with_flares}
\end{figure}

\paragraph{Discriminating power on a classification task.}
In Section \ref{sec:classification}, we describe results on a shape classification experiment which indicate that persistent sets can be useful invariants for practical data classification applications. In order to carry out this test, we computed approximations of the persistence sets $\Dvr{2k+2,k}$ and the persistence measures $\Uvr{2k+2,k}$, for $k=0,1,2$, of 62 three-dimensional shapes in 6 different classes from the publicly available database \cite{sumner-paper}. We classified these shapes using the 1-nearest neighbor classifier induced by the Hausdorff and 1-Wasserstein distances between persistence sets and measures, respectively.

\paragraph{Computational cost, memory requirements,  paralellizability, and approximation.}
Besides its ability to often detect useful information that is not captured by standard VR persistence diagrams, another motivation for considering persistence sets $\Dvr{n,k}$ for small $n$ as features that can help in shape/data classification is that the cost incurred in their computation/approximation compares favourably against the  cost and memory requirements of computing $\dgm_{k}^\vr(X)$ as the size of $X$ increases. Furthermore, not only are the associated computational tasks  eminently parallelizable (cf. Figure \ref{fig:multiplex}) but also, when $n$ is small, the amount of memory needed for computing  persistent sets is also notably smaller than for computing persistence diagrams over the same data set. See Sections \ref{sec:computational_cost_Dnk_matrixtime} and \ref{sec:memory_comparison} for a detailed discussion.

\paragraph{Principal persistence sets, their characterization and an algorithm.}
Persistence sets are defined to be sets of persistence diagrams and, although a single persistence diagram is easy to visualize, large collections of them might not be so. However, our main result (Theorem \ref{thm:n=2k+2}) says that the degree $k$ persistence diagram of $X$ contains no points if $|X| < 2k+2$ and at most one point if $|X|=2k+2$. For that reason, we \emph{aggregate} all persistence diagrams in the \textbf{principal persistence set} $\Dvr{2k+2,k}(X)$ on the same axis; cf. Figure \ref{fig:stack-D-to1}.

Furthermore, Theorem \ref{thm:n=2k+2} gives a precise representation of the unique point in the degree $k$ persistence diagram of a metric space with at most $n_k:=2k+2$ points via a formula which induces an algorithm for computing the principal persistence sets. This algorithm is purely geometric in the sense that it does not rely on analyzing boundary matrices as the standard persistent homology algorithms but, in contrast, directly operates at the level of distance matrices. For any $k$, this geometric algorithm has cost $O(n_k^2)\approx O(k^2)$ as opposed to the much larger cost incurred by the algebraic algorithms; see Proposition \ref{prop:bd_times_algorithm}. This makes the practical approximation of principal persistence sets to be very efficient; see Corollary \ref{coro:comp-cost-pps}.

\begin{figure}[ht]
	\centering
	\includegraphics[scale=0.8]{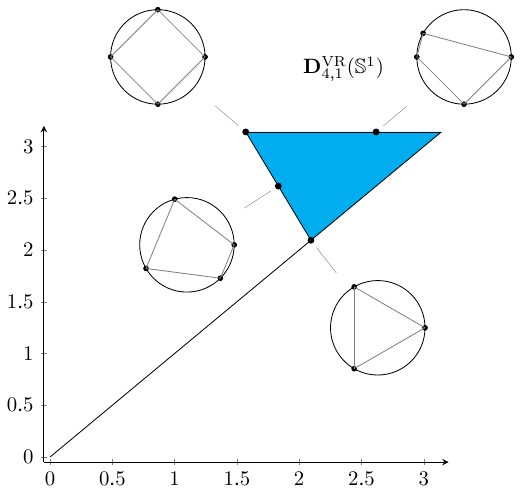}
	
	\caption{\textbf{Characterization of $\Dvr{4,1}(\Sphere{1})$:} The $(4,1)$-persistence set of $\Sphere{1}$ (with geodesic distance) is the shaded triangular area where the top left and top right points have coordinates $(\frac{\pi}{2},\pi)$ and $(\pi,\pi)$, respectively,  whereas the lowest diagonal point has coordinates $(\frac{2\pi}{3},\frac{2\pi}{3})$. This is the $k=1$ case of Theorem \ref{thm:Dnk_S1_k_odd}. The figure also shows exemplary  configurations $X \subset \Sphere{1}$ with $|X| \leq 4$ together with  their respective persistence diagrams inside of $\Dvr{4,1}(\Sphere{1})$.}
	\label{fig:example_S1_configurations}
\end{figure}

\paragraph{Characterization results.}
We fully characterize the principal persistence sets $\Dvr{{2k+2},k}(\Sphere{1})$ (Theorems \ref{thm:critical_tb_even} and \ref{thm:Dnk_S1_k_odd}). In particular,  these results prove that $\Dvr{4,1}(\Sphere{1})$ coincides with the triangle in $\R^2$ with vertices $(\frac{2\pi}{3},\frac{2\pi}{3})$, $(\frac{\pi}{2},\pi)$, and $(\pi,\pi)$; see Figure \ref{fig:example_S1_configurations}. We also characterize the persistence measure $\Unk{4,1}{\vr}(\Sphere{1})$, which is supported on $\Dvr{4,1}(\Sphere{1})$, in Proposition \ref{prop:U41_S1}. Furthermore, if $\Sphere{1}$ has the uniform probability measure, we show that $\Unk{4,1}{\vr}(\Sphere{1})$ has probability density function $f(t_b,t_d)=\frac{12}{\pi^3}(\pi-t_d)$, for any $(t_b,t_d)$ in the triangular region specified in Figure \ref{fig:example_S1_configurations}. Propositions \ref{prop:D41_R2_S2} and \ref{prop:Dk_Sn_stabilizes}, and Corollary \ref{cor:D41_Sm_stabilizes} provide additional information about higher dimensional spheres. In particular, we discuss the use of a MCMC random walk to effectively sample from $\Dvr{6,2}(\Sphere{2})$; see Conjecture \ref{conj:D62_S2_E}. Example \ref{ex:D41_spheres} has computational approximations of the persistence measure $\Uvr{4,1}$ of the 2-sphere and the torus. These characterization results are in the same spirit as those pioneered by Adamaszek and Adams on the Vietoris-Rips persistence diagrams of circles and spheres \cite{aa17}; see also \cite{osman-memoli}.

We also compute $\Dvr{4,1}(\R^n)$ using Ptolemy's inequality (Proposition \ref{prop:D41_R2_S2}). In fact, Ptolemy's inequality generalizes to non-Euclidean geometries, so we can also characterize the (4,1)-persistence sets of surfaces with constant curvature $M_\kappa$. For clarity, $M_0 = \R^2$, $M_\kappa$ is the sphere of radius $1/\sqrt{\kappa}$ if $\kappa>0$ or a rescaling of the hyperbolic plane if $\kappa<0$.
\begin{restatable*}{theorem}{DiagramMKappa}\label{thm:D41_Mk}
	Let $M_\kappa$ be the 2-dimensional model space with constant sectional curvature $\kappa$. Then:
	\begin{itemize}
		\item If $\kappa>0$, $\Dvr{4,1}(M_\kappa) = \left\{ (t_b,t_d)|\ \sin \left( \frac{\sqrt{ \kappa}}{2} t_d \right) \leq \sqrt{2}\sin \left( \frac{\sqrt{ \kappa}}{2} t_b \right) \text{ and } 0 < t_b < t_d \leq \frac{\pi}{\sqrt{\kappa}} \right\}$.
		\item If $\kappa=0$, $\Dvr{4,1}(M_0) = \left\{ (t_b,t_d)|\ 0 \leq t_b < t_d \leq \sqrt{2}t_b \right\}$.
		\item If $\kappa<0$, $\Dvr{4,1}(M_\kappa) = \left\{ (t_b,t_d)|\ \sinh\left( \frac{\sqrt{-\kappa}}{2} t_d \right) \leq \sqrt{2}\sinh\left( \frac{\sqrt{ -\kappa}}{2} t_b \right) \text{ and } 0 < t_b < t_d \right\}$.
	\end{itemize}
\end{restatable*}
\indent A similar result appears in \cite{ph-detects-curvature}, where the authors studied the \v{C}ech complex of triangles in the model spaces of constant curvature. Using the logarithmic persistence (that is, the fraction $t_d/t_b$), they detected the curvature of the ambient space both analytically and experimentally. See their paper for more details.

\paragraph{An application of $\Dvr{4,1}$ to detecting homotopy type of graphs.} In Section \ref{sec:special-class}, as an application of the characterization of $\Dvr{4,1}(\Sphere{1})$, we study a class of metric graphs for which $\Dvr{4,1}$, a rather coarse invariant which is fairly easy to estimate and compute in practice, is able to characterize the homotopy type of graphs in this class. %

\paragraph{Stability.}
In Theorems \ref{thm:stability_Dnk} and \ref{thm:stability_Unk} we establish the stability of persistence sets and measures under the modified Gromov-Hausdorff and Gromov-Wasserstein distances. Such results give lower bounds for these distances which are computable in polynomial time. In particular, see Section \ref{sec:dGH_spheres}.

\paragraph{Coordinates.}
As the objects $\Unk{n,k}{\vr}$ can be considerably complex, a system of coordinates $\{\zeta_\alpha:\D \to \R\}_{\alpha \in A}$ that exhausts the information contained in the persistence measures is desirable. Although we do not propose such a family (but \cite{acc13, coordinates-tropical, kernel-hiraoka} do), we do prove the stability of the 1-norm of the cumulative distribution associated to a Lipschitz coordinate. Specifically, let $\zeta:\D \to \R$. For an mm-space $\mm{X}$, consider the probability measure $\zeta_\#\Unk{n,k}{\FiltFunc}(X)$ on $\R$. Let $H_X(t;n,k,\FiltFunc,\zeta) := \Unk{n,k}{\FiltFunc}(X)\big( \zeta\inv(-\infty,t] \big)$ be its cumulative distribution function. We prove:

\begin{restatable*}{theorem}{StabilityCoordinates}\label{thm:stability-coordinates}
	Let $\zeta:\D \to \R$ be an $L(\zeta)$-Lipschitz coordinate function, and suppose $\FiltFunc$ is a stable filtration functor. Write $H_X(t) = H_X(t;n,k,\FiltFunc,\zeta)$ to simplify the notation. Then, for any two mm-spaces $X$ and $Y$,
	\begin{equation*}
		\int_\R |H_X(t) - H_Y(t)|dt \leq L(\zeta) L(\FiltFunc) \cdot \dGW{1}(X,Y).
	\end{equation*}
\end{restatable*}

\paragraph{Concentration results for $\Unk{n,k}{\mathrm{VR}}$.}
Another consequence of the stability of persistence measures is the concentration of $\Unk{n,k}{\FiltFunc}(X)$ as $n \to \infty$. Denote the expected value of a random variable $X$ distributed according to the probability measure $\mu$ with $E_\mu[X]$. Then:
\begin{restatable*}{theorem}{UnkConcentrates}\label{thm:Unk_concentrates}
	Let $\mm{X}$ be an mm-space and $\FiltFunc$ a stable filtration functor. For any $n,k \in \N$, consider the random variable $\mathbb{D}$ valued in $\Dnk{n,k}{\FiltFunc}(X)$ distributed according to $\Unk{n,k}{\FiltFunc}(X)$. Then:
	\begin{itemize}
		\item For any $\varepsilon>0$, $E_{\Unk{n,k}{\FiltFunc}(X)}\left[\dB\left(\mathbb{D}, \dgm_k^\FiltFunc(X) \right) \right] < \diam(X) \cdot C_X(n,\varepsilon) + \varepsilon$.
		\item As a consequence, the mm-space $\Dw{n,k}{\FiltFunc}(X) = \left(\Dnk{n,k}{\FiltFunc}(X), \dB, \Unk{n,k}{\FiltFunc}(X) \right)$ concentrates to a one-point mm-space as $n \to \infty$.
	\end{itemize}
\end{restatable*}
Similar results appear in \cite{blumberg2012robust} \cite{chazal2015subsampling}. The approach in \cite{chazal2015subsampling} is studying the expected value $E[\lambda_Z]$ of the persistence landscape $\lambda_Z$ of a sample $Z = \{x_1,\dots,x_m\} \subset X$. They show that this procedure is stable under the Gromov-Wasserstein distance, and provide a bound on the expected $\ell_\infty$ distance between the persistence landscape of $X$ and $\overline{\lambda_n^m}$. These two results are analogous, respectively, to our Theorem \ref{thm:stability_Unk} and to item 1 of \ref{thm:Unk_concentrates} above. As for \cite{blumberg2012robust}, the authors study the statistical robustness of persistent homology invariants. They have two results similar to ours. One is the stability of the measures $\Unk{n,k}{\FiltFunc}(X)$ (they write $\Phi_k^n(X)$ instead) under the Gromov-Prokhorov distance (instead of the Gromov-Wasserstein distance). The second is a central limit theorem, where the measures $\Unk{n,k}{\FiltFunc}(S_i)$ corresponding to an increasing sequence of finite samples $S_1 \subset S_2 \subset \cdots \subset X$ converge in probability to $\Unk{n,k}{\FiltFunc}(X)$.

\subsection{Related work}
The measures $\Unk{n,k}{\mathrm{VR}}$ first appeared in a preprint by Blumberg et al. \cite{blumberg2012robust} in 2012 and then in print in \cite{blumberg2014robust}. These measures were also exploited a few years later by Chazal et al. in the articles \cite{chazal-arxiv-2014,chazal2015subsampling} in order to devise bootstrapping methods for the estimation of persistence diagrams.

The connection to Gromov's curvature sets and measures was not recognized in either of these two papers. \cite{mem12} studied curvature sets and their role in shape comparison and, as a natural follow up, some results regarding the persistence sets $\Dvr{n,k}$ and the measures $\Unk{n,k}{\vr}$ (as well as the more general objects $\Dnk{n,k}{\mathfrak{F}}$ and $\Unk{n,k}{\mathfrak{F}}$) were first described in Banff in 2012 during a conference \cite{banff12} by the second author\footnote{Subsequent develoments were described in 2013 at ACAT 2013 in Bremen \cite{acat13} and  Bedlewo \cite{bedlewo13}, and then at IMA \cite{ima14} and at  SAMSI in 2014 \cite{samsi14}.}  as stable and computationally easier alternatives to the usual Vietoris-Rips persistence diagrams of metric spaces \cite{samsi14-slides}. 

In \cite{solomon2021geometry}  Bendich et al. discuss ideas related to our construction of $\Dnk{n,k}{\mathfrak{F}}.$  The authors pose questions about the discriminative power of a certain labeled version of the persistent sets $\Dvr{n,k}$ (even though they do not call them that).  \cite{memoli2018gromov} has recently explored the classificatory power of $\mu_2$ (see equation  (\ref{eq:mun})) as well as that of certain \emph{localizations} of $\mu_2$. In \cite{filtration-families} the authors identify novel classes of simplicial filtrations arising from curvature sets together with suitable notions of locality.

In terms of data centric applications, the neuroscience paper \cite{singh2008topological} made use of ideas related to $\Unk{n,k}{\vr}$ and $\Dvr{n,k}$ in the context of analysis of neuroscientific data.

\subsection{Acknowledgements} We thank Henry Adams for bringing his paper \cite{aa17} to our attention. The ideas contained therein were helpful in proving some of the results of Section \ref{sec:Dvr_S1_k_odd}. We also thank the anonymous reviewers for their transformative feedback.

\indent We acknowledge funding from these sources: NSF AF 1526513, NSF DMS 1723003, NSF CCF 1740761,  NSF CCF 1839358, and BSF 2020124.

\paragraph{Data availability statement}
The data analysed in this paper comes from the article \cite{sumner-paper} and is available in \url{https://people.csail.mit.edu/sumner/research/deftransfer/}.

\section{Background}
For us, $\M$ and $\M^\text{fin}$ will denote, respectively, the category of compact and finite metric spaces. The morphisms in both categories will be 1-Lipschitz maps, that is, functions $\varphi:X \to Y$ such that $d_Y(\varphi(x), \varphi(x')) \leq d_X(x,x')$ for all $(X,d_X),(Y,d_Y)$ in $\M$ or $\M^\text{fin}$. We say that two metric spaces are isometric if there exists a surjective isometry $\varphi:X \to Y$, \textit{i.e.} a surjective map such that $d_Y(\varphi(x), \varphi(x')) = d_X(x,x')$ for all $x,x' \in X$. We also say that a space is \define{geodesic} if for any $x,x' \in X$, there exists an isometry $\gamma:[0,d] \to X$ such that $d = d_X(x,x')$, $\gamma(0)=x$ and $\gamma(d)=x'$.

\subsection{Metric geometry}
In this section, we define the tools that we'll use to quantitatively compare metric spaces \cite{bbi01}.
\begin{defn}\label{def:rad-diam}
	For any subset $A$ of a metric space $X$, its \define{diameter} is $\diam_X(A) := \sup_{a,a' \in A} d_X(a,a')$, and its \define{radius} is $\rad_X(A) := \inf_{p \in X} \sup_{a \in A} d_X(p,a)$. Note that $\rad_X(A) \leq \diam_X(A)$. The \define{separation} of $X$ is $\sep(X) := \inf_{x \neq x'} d_X(x,x')$.
\end{defn}

\begin{defn}[Hausdorff distance]\label{def:hausdorff_distance}
	Let $A,B$ be subsets of a compact metric space $(X,d_X)$. The \define{Hausdorff distance} between $A$ and $B$ is defined as
	\begin{equation*}
		\dH^X(A,B) := \inf \left\{ \varepsilon>0 \ | \ A \subset B^\varepsilon \text{ and } B \subset A^\varepsilon \right\},
	\end{equation*}
	where $A^\varepsilon := \left\{ x \in X \ | \ \inf_{a \in A} d_X(x,a) < \varepsilon \right\}$ is the $\varepsilon$-thickening of $A$. It is known that $\dH^X(A,B)=0$ if, and only if their closures are equal: $\bar{A} = \bar{B}$.
\end{defn}

We will use an alternative definition that is useful for calculations, but is not standard in the literature.
\begin{defn}
	A \define{correspondence} between two sets $X$ and $Y$ is a set $R \subset X \times Y$ such that $\pi_1(R) = X$ and $\pi_2(R)=Y$, where $\pi_i$ is the projection to the $i$-th coordinate. We will denote the set of all correspondences between $X$ and $Y$ as $\mathcal{R}(X,Y)$.
\end{defn}

\begin{prop}[Proposition 2.1 of \cite{memoli-dghlp-long}]
	\label{prop:Hausdorff_distance}
	For any compact metric space $(X,d_X)$ and any $A,B \subset X$ closed,
	\begin{equation*}
		\dH^X(A,B) = \inf_{R \in {\mathcal R}(A,B)} \sup_{(a,b) \in R} d_X(a,b).
	\end{equation*}
\end{prop}

The standard method for comparing two metric spaces is a generalization of the Hausdorff distance.
\begin{defn}
	For any correspondence $R$ between $(X, d_X), (Y, d_Y) \in \M$, we define its \define{distortion} as
	\begin{equation*}
		\dis(R) := \max\left\{ |d_X(x,x') - d_Y(y,y')| : (x,y),(x',y') \in R \right\}.
	\end{equation*}
	Then the \define{Gromov-Hausdorff distance} between $X$ and $Y$ is defined as
	\begin{equation*}
		\dGH(X,Y) := \dfrac{1}{2} \inf_{R \in {\mathcal R}(X,Y)} \dis(R).
	\end{equation*}
\end{defn}

\subsection{Metric measure spaces}
To model the situation in which points are endowed with a notion of weight (signaling their trustworthiness), we will also consider finite metric spaces enriched with probability measures \cite{memoli-dghlp-long}. Recall that the \define{support} $\supp(\nu)$ of a Borel measure $\nu$ defined on a topological space $Z$ is defined as the minimal closed set $Z_0$ such that $\nu(Z \setminus Z_0)=0$. If $\varphi:Z \to X$ is a measurable map from a measure space $(Z,\Sigma_Z,\nu)$ into the measurable space $(X,\Sigma_X)$, then the \define{pushforward measure} of $\nu$ induced by $\varphi$ is the measure $\varphi_\#\nu$ on $X$ defined by $\varphi_\#\nu(A) := \nu(\varphi\inv(A))$ for all $A \in \Sigma_X$. %

\begin{defn}
	\label{def:mm_space}
	A \define{metric measure space} is a triple $\mm{X}$ where $(X,d_X)$ is a compact metric space and $\mu_X$ is a Borel probability measure on $X$ with full support, \textit{i.e.} $\supp(\mu)=X$. Two mm-spaces $\mm{X}$ and $\mm{Y}$ are isomorphic if there exists an isometry $\varphi:X \to Y$ such that $\varphi_\#\mu_X = \mu_Y$. We define the category of mm-spaces $\Mw$, where the objects are mm-spaces and the morphisms are 1-Lipschitz maps $\varphi:X \to Y$ such that $\varphi_\# \mu_X = \mu_Y$.
\end{defn}

The following definitions are used to compare mm-spaces.
\begin{defn}\label{def:couplings}
	Given two measure spaces $(X, \Sigma_X, \mu_X)$ and $(Y, \Sigma_Y, \mu_Y)$, a \define{coupling} between $\mu_X$ and $\mu_Y$ is a measure $\mu$ on $X \times Y$ such that $\mu(A \times Y) = \mu_X(A)$ and $\mu(X \times B) = \mu_Y(B)$ for all measurable $A \in \Sigma_X$ and $B \in \Sigma_Y$ (in other words, $(\pi_1)_\# \mu = \mu_X$ and $(\pi_2)_\# \mu = \mu_Y$). We denote the set of couplings between $\mu_X$ and $\mu_Y$ as $\coup(\mu_X, \mu_Y)$.
\end{defn}

\begin{remark}[The support of a coupling is a correspondence]
	Notice that, since $\mu_X$ is fully supported and $X$ is finite, then $\mu(\pi_1\inv(x)) = \mu_X(\{x\}) \neq 0$ for any fixed coupling $\mu \in \coup(\mu_X, \mu_Y)$. Thus, the set $\pi_1\inv(x) \cap \supp(\mu)$ is non-empty for every $x \in X$. The same argument on $Y$ shows that $\supp(\mu)$ is a correspondence between $X$ and $Y$.
\end{remark}

\begin{defn}
	Given a metric space $(Z, d_Z)$, let $\P(Z)$ be the set of Borel probability measures on $Z$. Given $\alpha,\beta \in \P(Z)$ and $p \geq 1$, the \define{Wasserstein distance} of order $p$ is defined as \cite{villani}:
	\begin{equation*}
	\dW{p}^Z(\alpha, \beta) :=
	\inf_{\mu \in \coup(\alpha, \beta)} \left( \iint_{Z \times Z} (d_Z(z,z'))^p \mu(dz \times dz') \right)^{1/p}.
	\end{equation*}
\end{defn}

\indent To compare two mm-spaces, we have the following distance.
\begin{defn}
	Given two mm-spaces $\mm{X}$ and $\mm{Y}$, $p \geq 1$, and $\mu \in \coup(\mu_X, \mu_Y)$, we define the \define{$p$-distortion} of $\mu$ as:
	\begin{equation*}
	\dis_p(\mu) := \left( \iint |d_X(x,x')-d_Y(y,y')|^p \mu(dx \times dy) \mu(dx' \times dy') \right)^{1/p}.
	\end{equation*}
	For $p=\infty$ we set $\dis_\infty(\mu) := \dis(\supp(\mu))$.\\
	\indent  The \define{Gromov-Wasserstein distance} of order $p \in [1,\infty]$ between $X$ and $Y$ is defined as \cite{memoli-dghlp-long}:
	\begin{equation*}
	\dGW{p}(X,Y) := \dfrac{1}{2} \inf_{\mu \in \coup(\mu_X, \mu_Y)} \dis_p(\mu).
	\end{equation*}
\end{defn}
\begin{remark}
	For each $p \in [1,\infty]$, $\dGW{p}$ defines a legitimate metric on the collection of isomorphism classes of mm-spaces in $\Mw$ \cite{memoli-dghlp-long}.
\end{remark}

\subsection{Simplicial complexes}
\begin{defn}
	Let $V$ be a set. An \define{abstract simplicial complex} $K$ with vertex set $V$ is a collection of finite subsets of $V$ such that if $\sigma \in K$, then every $\tau \subset \sigma$ is also in $K$. We also use $K$ to denote its geometric realization. A set $\sigma \in K$ is called a $k$-face if $|\sigma|=k+1$. A \define{simplicial map} $f:K_1 \to K_2$ is a set map $f:V_1 \to V_2$ between the vertex sets of $K_1$ and $K_2$ such that if $\sigma \in K_1$, then $f(\sigma) \in K_2$.
\end{defn}

We will focus on two particular complexes.
\begin{defn}
	Let $(X, d_X) \in \M$ and $r \geq 0$. The \define{Vietoris-Rips complex} of $X$ at scale $r$ is the simplicial complex
	\begin{equation*}
		\vrcomp_{r}(X) := \left\{ \sigma \subset X \text{ finite}: \diam_X(\sigma) \leq r \right\}.
	\end{equation*}
\end{defn}

\begin{defn}\label{def:cross_poly}
	Fix $n \geq 1$. Let $e_i := (0,\dots,1,\dots,0)$ be the $i$-th standard basis vector in $\R^n$ and $V:=\{\pm e_1, \dots, \pm e_n\}$. Let $\crosspoly{n}$ be the collection of subsets $\sigma \subset V$ that don't contain both $e_i$ and $-e_i$. This simplicial complex is called the \define{$n$-th cross-polytope}.
\end{defn}

\begin{figure}
	\centering
\begin{tabular}{ccc}

\begin{tikzpicture}[scale=1.25]
\tikzmath{
	\pt=0.025;
	\pos1 = 3.5;
	\pax1 = -1;	\pay1 = 0;
	\pax2 =  1;	\pay2 = 0;
	\pbx1 = \pos1-1;	\pby1 = 0;
	\pbx2 = \pos1;		\pby2 = 1;
	\pbx3 = \pos1+1;	\pby3 = 0;
	\pbx4 = \pos1;		\pby4 = -1;
}

\draw[fill] (\pax1,\pay1) circle [radius=\pt] node[below]{$-e_1$};
\draw[fill] (\pax2,\pay2) circle [radius=\pt] node[below]{$ e_1$};
\draw[fill] (\pbx1,\pby1) circle [radius=\pt] node[left ]{$-e_1$};
\draw[fill] (\pbx2,\pby2) circle [radius=\pt] node[above]{$ e_2$};
\draw[fill] (\pbx3,\pby3) circle [radius=\pt] node[right]{$ e_1$};
\draw[fill] (\pbx4,\pby4) circle [radius=\pt] node[below]{$-e_2$};

\draw[dotted] (\pax1,\pay1) -- (\pax2,\pay2);

\draw (\pbx1,\pby1) -- (\pbx2,\pby2);
\draw (\pbx2,\pby2) -- (\pbx3,\pby3);
\draw (\pbx3,\pby3) -- (\pbx4,\pby4);
\draw (\pbx4,\pby4) -- (\pbx1,\pby1);
\end{tikzpicture}

& ~ &
\begin{tikzpicture}[scale = 1.25,
	z={(-.3cm,-.2cm)}, %
	line join=round, line cap=round %
]
\draw (0,1,0) -- (-1,0,0) -- (0,-1,0) -- (1,0,0) -- (0,1,0) -- (0,0,1) -- (0,-1,0) (1,0,0) -- (0,0,1) -- (-1,0,0);
\draw (0,1,0) -- (0,0,-1) -- (0,-1,0) (1,0,0) -- (0,0,-1) -- (-1,0,0);

\draw[fill] ( 1, 0, 0) circle[radius=0.025] node[right]{$ e_2$};
\draw[fill] (-1, 0, 0) circle[radius=0.025] node[left ]{$-e_2$};
\draw[fill] ( 0, 1, 0) circle[radius=0.025] node[above]{$ e_3$};
\draw[fill] ( 0,-1, 0) circle[radius=0.025] node[below]{$-e_3$};
\draw[fill] ( 0, 0, 1) circle[radius=0.025];
\draw[dashed] (0,0, 1) -- (-0.75,-0.55,0) node[below left]{$e_1$};
\draw[fill] ( 0, 0,-1) circle[radius=0.025];
\draw[dashed] (0,0,-1) -- ( 0.75, 0.75,0) node[above]{$-e_1$};

\end{tikzpicture} 

\end{tabular} 	\caption{From left to right: $\crosspoly{1}, \crosspoly{2}, \crosspoly{3}$ (there is no edge between the vertices of $\crosspoly{1}$). See Definition \ref{def:cross_poly}.} \label{fig:crosspoly}
\end{figure}
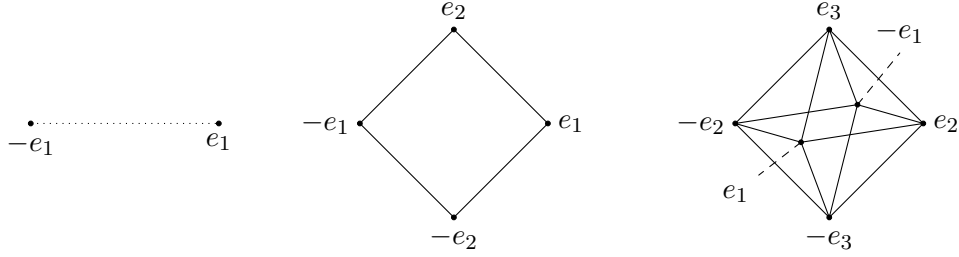

\subsection{Persistent homology}
We adopt definitions from \cite{mem17,filtration-families}.

\begin{defn}\label{def:filtrations}
	A \define{filtration} on a finite set $X$ is a function $F_X:\pow(X) \to \R$ such that $F_X(\sigma) \leq F_X(\tau)$ whenever $\sigma \subset \tau$, and we call the pair $(X,F_X)$ a \define{filtered set}. $\Filt$ will denote the category of finite filtered sets, where objects are pairs $(X,F_X)$ and the morphisms $\varphi:(X,F_X) \to (Y,F_Y)$ are set maps $\varphi:X \to Y$ such that $F_Y(\varphi(\sigma)) \leq F_X(\sigma)$. A \define{filtration functor} is any functor $\FiltFunc:\M^\mathrm{fin} \to \Filt$ where $(X, F_X) = \FiltFunc(X)$ and $F_X: \pow(X) \to \R$. Observe that filtration functors are equivariant under isometries.
\end{defn}

\begin{defn}
	Given $(X, d_X) \in \M^\text{fin}$, define the \define{Vietoris-Rips filtration} $F_X^\vr$ by setting $F_X^\vr(\sigma) := \diam(\sigma)$ for $\sigma \subset X$. It is straightforward to check that this construction is functorial, so we define the \define{Vietoris-Rips filtration functor} $\FiltFunc^\vr:\M^\text{fin} \to \Filt$ by $(X,d_X) \mapsto (X,F_X^\vr)$.
\end{defn}
More examples of filtration functors, such as the \v{C}ech filtration, can be found in \cite{filtration-families}.\\
\indent Given a filtration functor $\FiltFunc$, we assign a persistence diagram to $(X,d_X)$ as follows. Let $(X,F_X^\FiltFunc)=\FiltFunc(X,d_X)$. For every $r > 0$, we construct the simplicial complex $L_r := \left\{ \sigma \subset X: F_X^\FiltFunc(\sigma) \leq r \right\}$\footnote{Notice that if $\FiltFunc = \FiltFunc^\vr$, then $L_r = \vrcomp_{r}(X)$.}, giving a nested sequence of simplicial complexes $L_{r_0} \subset L_{r_1} \subset L_{r_2} \subset \cdots \subset L_{r_m}$. We apply reduced homology $\widetilde{H}_k(\cdot, \Fld)$ with field coefficients at each step, and we get a \emph{persistent vector space} $\PH_k^\FiltFunc(X)$ which decomposes as a sum of interval modules $\PH_k^\FiltFunc(X) \cong \bigoplus_{\alpha \in A} \Int{b_\alpha}{d_\alpha}$ where $A$ is a finite indexing set \cite{cds10}. We can also represent a persistent vector space by the multiset $\dgm_k^\FiltFunc(X) = \{(b_\alpha, d_\alpha) | \ 0 \leq b_\alpha < d_\alpha, \alpha \in A\}$, called a \emph{persistence diagram}. We denote the empty persistence diagram, which corresponds to the persistence module $\PH_k^\FiltFunc(X) = 0$, as $\emptyset$. Notice that using reduced homology implies that $\widetilde{H}_k(L_r)=0$ for $r \geq F_X^\FiltFunc(X)$, and so $d_\alpha < \infty$ for all $\alpha \in A$, regardless of the dimension $k$. In dimension 0, this removes the infinite interval. We denote by $\D$ the collection of all finite persistence diagrams. We say that a point $P = (b_\alpha, d_\alpha)$ in a persistence diagram $D \in \D$ has persistence $\pers(P) := d_\alpha-b_\alpha$ and define $\pers(D) := \max\{\pers(P) | \ P \in D\}$. Let $\dB$ be the \define{bottleneck distance}.

\begin{defn}\label{def:stable_filtration_functors}
	We say that a filtration functor $\FiltFunc:\M^\mathrm{fin} \to \Filt$ is \define{stable} if there exists a constant $L>0$ such that
	\begin{equation*}
		\dB(\dgm_k^\FiltFunc(X), \dgm_k^\FiltFunc(Y)) \leq L \cdot \dGH(X,Y)
	\end{equation*}
	for all $X,Y \in \M^\text{fin}$ and $k \in \N$. The infimal $L$ that satisfies the above is called the \define{Lipschitz constant} of $\FiltFunc$ and denoted by $L(\FiltFunc)$.
\end{defn}
The Vietoris-Rips and \v{C}ech filtrations are stable and, in fact, $L(\FiltFunc^\vr)=2$.
\section{Curvature sets, persistence diagrams and persistent sets}
Given a compact metric space $(X,d_X)$, Gromov identified a class of full invariants called \define{curvature sets} (see Section $1.19_+$ of \cite{gro99} for the definition, and Section $3\frac{1}{2}.4$ for the terminology ``curvature sets''). Intuitively, the $n$-th curvature set contains the metric information of all possible samples of $n$ points from $X$. In this section, we define persistence sets as an invariant that captures the persistent homology of all $n$-point samples of $X$. We start by recalling Gromov's definition, and defining an analogue of the Gromov-Hausdorff distance in terms of curvature sets. We then define persistence sets and study their stability with respect to this modified Gromov-Hausdorff distance. We also extend these constructions to mm-spaces.

\begin{defn}
	Let $(X,d_X)$ be a metric space. Given a positive integer $n$, let $\Psi_X^{(n)}:X^n \to \R^{n \times n}$ be the map that sends an $n$-tuple $(x_1,\dots,x_n)$ to the distance matrix $M$, where $M_{ij} = d_X(x_i,x_j)$. The \define{$n$-th curvature set} of $X$ is $\Kn_n(X) := \operatorname{im}(\Psi_X^{(n)})$, the collection of all distance matrices of $n$ points from $X$.
\end{defn}

\begin{remark}[Functoriality of curvature sets]\label{rmk:functorial_curvature_sets}
	Curvature sets are functorial in the sense that if $X$ is isometrically embedded in $Y$, then $\Kn_n(X) \subset \Kn_n(Y)$.
\end{remark}

\begin{example}
	$\Kn_2(X)$ is the set of distances of $X$. If $X$ is geodesic, $\Kn_2(X) = [0, \diam(X)]$.
\end{example}

\begin{example}\label{ex:Kn_two_points}
	Let $X = \{p,q\}$ be a two point metric space with $d_X(p,q) = \delta$. Then
	\begin{align*}
		\Kn_3(X) &= \left\{ \Psi_X^{(3)}(p,p,p), \Psi_X^{(3)}(p,p,q), \Psi_X^{(3)}(p,q,p), \Psi_X^{(3)}(q,p,p),\right. \\
		& \left.\phantom{= \{} \Psi_X^{(3)}(q,q,q), \Psi_X^{(3)}(q,q,p), \Psi_X^{(3)}(q,p,q), \Psi_X^{(3)}(p,q,q) \right\} \\
		&= \left\{\left(
		\begin{smallmatrix}
			0 & 0 & 0\\
			0 & 0 & 0\\
			0 & 0 & 0
		\end{smallmatrix} \right),
		\left(
		\begin{smallmatrix}
			0 & 0 & \delta\\
			0 & 0 & \delta\\
			\delta & \delta & 0
		\end{smallmatrix} \right),
		\left(\begin{smallmatrix}
			0 & \delta & 0\\
			\delta & 0 & \delta\\
			0 & \delta & 0
		\end{smallmatrix} \right),
		\left(\begin{smallmatrix}
			0 & \delta & \delta\\
			\delta & 0 & 0\\
			\delta & 0 & 0
		\end{smallmatrix} \right) \right\}.
	\end{align*}
	
	For $n \geq 2$ and $0 < k < n$, let $x_1 = \cdots = x_k=p$ and $x_{k+1} = \cdots = x_n = q$. Define
	\begin{equation*}
		M_k(\delta) := \Psi_X^{(n)}(x_1, \dots, x_n) = 
		\left(\begin{array}{c|c}
			\mathbf{0}_{k \times k}		& \delta \cdot \mathbf{1}_{k \times (n-k)} \\
			\hline
			\delta \cdot \mathbf{1}_{(n-k) \times k}		& \mathbf{0}_{(n-k) \times (n-k)}
		\end{array} \right),
	\end{equation*}
	where $\mathbf{0}_{r \times s}$ and $\mathbf{1}_{r \times s}$ are the $r \times s$ matrices with all entries equal to 0 and 1, respectively. If we make another choice of $x_1,\dots,x_n$, the resulting distance matrix will change only by a permutation of its rows and columns. Thus, if we define $M_k^\Pi(\delta) := \Pi^T \cdot M_k(\delta) \cdot \Pi$, for some permutation matrix $\Pi \in S_n$, then
	\begin{equation*}
		\Kn_n(X) = \left\{ \mathbf{0}_{n \times n} \right\} \cup \left\{ M_k^\Pi(\delta): 0 < k < n \text{ and } \Pi \in S_n \right\}.
	\end{equation*}
\end{example}

\noindent
\begin{minipage}{.7\textwidth}
	\begin{example}\label{ex:K_3(S1)}
		In this example we describe $\Kn_3(\Sphere{1})$, where $\Sphere{1} = [0,2\pi] / (0 \sim 2\pi)$ is equipped with the geodesic metric. Depending on the position of $x_1,x_2,x_3$, we need two cases. If the three points are not contained in the same semicircle, then $d_{12}+d_{23}+d_{31}=2\pi$. If they are, then there exists a point, say $x_2$, that lies on the shortest path joining the other two so that $d_{13} = d_{12}+d_{23} \leq \pi$. The other possibilities are $d_{12}=d_{13}+d_{32}$ and $d_{23}=d_{21}+d_{13}$.\\
		Let $M := \Psi_{\Sphere{1}}^{(3)}(x_1,x_2,x_3)$. Since $M$ is symmetric and its diagonal entries are 0, we only need 3 entries to characterize it. If we label $x=d_{12}, y=d_{23}$ and $z=d_{31}$, then $\Kn_3(\Sphere{1})$ is the boundary of the 3-simplex with vertices $(0,0,0)$, $(\pi, \pi, 0)$, $(\pi, 0, \pi)$, and $(0, \pi, \pi)$ in $\R^3$ (see Figure \ref{fig:K3_S1}). Each of the cases in the previous paragraph corresponds to a face of this simplex. See also Appendix A and Theorem 4.33 of \cite{curvature-sets-S1} for a more thorough calculation.
	\end{example}
\end{minipage}
\hfill
\begin{minipage}{0.28\textwidth}
	\centering
	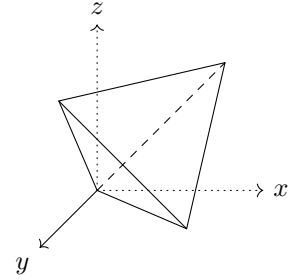
\begin{figure}[H]
		\begin{center}
\begin{tikzpicture}[scale=1.7,
	z={(-0.3cm,-0.3cm)}, %
	line join=round, line cap=round %
]
\draw[dashed] (0,0,0) -- (1,1,0);

\draw (0,0,0) -- (0,1,1);
\draw (0,0,0) -- (1,0,1);

\draw (0,1,1) -- (1,1,0) -- (1,0,1) -- cycle;

\draw[->] (0,0,0) -- (0,0,1.5) node[below left] {\small $y$};
\draw[dotted, ->] (0,0,0) -- (0,1.3,0) node[above] {\small $z$};
\draw[dotted, ->] (0,0,0) -- (1.3,0,0) node[right] {\small $x$};

\end{tikzpicture}  		\end{center}
		\caption{The curvature set $\Kn_3(\Sphere{1})$; cf. Example \ref{ex:K_3(S1)}}
		\label{fig:K3_S1}
	\end{figure}
\end{minipage}

\medskip
\indent Gromov proved that curvature sets are a full invariant of compact metric spaces, which means that the compact spaces $X$ and $Y$ are isometric if and only if $\Kn_n(X)=\Kn_n(Y)$ for all $n \geq 1$ \cite[Section 3.27]{gro99}. For this reason, the following definition from \cite{mem12} defines a bona-fide metric on compact metric spaces.

\begin{defn}[\cite{mem12}]\label{def:modifiedGH}
	The \define{modified Gromov-Hausdorff} distance between $X,Y \in \M$ is
	\begin{equation}\label{eq:modifiedGH}
		\widehat{d}_\mathcal{GH}(X,Y) := \dfrac{1}{2} \sup_{n \in \N} \dH(\Kn_n(X), \Kn_n(Y)).
	\end{equation}
	Here $\dH$ denotes the Hausdorff distance on $\R^{n \times n}$ with $\ell^\infty$ distance.
\end{defn}

\cite{mem12} proved that:
\begin{equation}\label{ineq:dGH_hat_vs_dGH}
    \widehat{d}_\mathcal{GH}(X,Y) \leq \dGH(X,Y).
\end{equation}

A benefit of $\widehat{d}_\mathcal{GH}$ when compared to the standard Gromov-Hausdorff distance is that the computation of the latter leads in general to NP-hard problems \cite{gh_not_approx_polytime}, whereas computing the lower bound in the equation above on certain values of $n$ leads to polynomial time problems. In \cite{mem12} it is argued that work of Peter Olver \cite{olver-joint} and Boutin and Kemper \cite{boutin} leads to identifying rich classes of shapes where these lower bounds permit full discrimination.\\

\indent The analogous definitions for mm-spaces are the following.
\begin{defn}\label{def:curvature_measures}
    Let $\mm{X}$ be an mm-space. The \define{$n$-th curvature measure} of $X$ is defined as
	\begin{equation*}
		\mu_n(X) := \left( \Psi_X^{(n)} \right)_\# \mu_X^{\otimes n},
	\end{equation*}
    where $\mu_X^{\otimes n}$ is the product measure on $X^n$. Observe that $\supp(\mu_n(X))=\Kn_n(X)$ for all $n \in \N$.\\
    We also define the \define{modified Gromov-Wasserstein distance} between $X,Y \in \Mw$ as
    \begin{equation*}
        \widehat{d}_{\mathcal{GW},p}(X,Y) := \dfrac{1}{2} \sup_{n \in \N} \dW{p}(\mu_n(X), \mu_n(Y)),
    \end{equation*}
    where $\dW{p}$ is the $p$-Wasserstein distance \cite{villani} on $\P(\R^{n \times n})$, and $\R^{n \times n}$ is equipped with the $\ell^\infty$ distance.
\end{defn}

\indent The modified $p$-Gromov-Wasserstein distance satisfies an inequality similar to (\ref{ineq:dGH_hat_vs_dGH}).
\begin{restatable}{theorem}{StabilitydGWHat}
    \label{thm:stability_dGW_hat}
    For any $X, Y \in \Mw$,
    \begin{equation*}
        \dW{p}(\mu_n(X),\mu_n(Y)) \leq 2 \binom{n}{2}^{\frac{1}{p}} \dGW{p}(X,Y)
    \end{equation*}
    for $1 \leq p < \infty$. If $p=\infty$,
    \begin{equation}
        \widehat{d}_{\mathcal{GW},\infty}(X,Y) \leq \dGW{\infty}(X,Y).
    \end{equation}
\end{restatable}
See Appendix \ref{sec:appendix} for the proof.

\begin{remark}[Interpretation as ``\textit{motifs}'']
	In network science \cite{mv20}, it is of interest to identify substructures of a dataset (network) $X$ which appear with high frequency. The interpretation of the definitions above is that the curvature sets $\Kn_n(X)$ for different $n \in \N$ capture the information of those substructures whose cardinality is at most $n$, whereas the curvature measures $\mu_n(X)$ capture their frequency of occurrence.
\end{remark}

\subsection{Persistence sets}
The idea behind curvature sets is to study a metric space by taking the distance matrix of a sample of $n$ points. This is the inspiration for the next definition: we want to study the persistence of a compact metric space $X$ by looking at the persistence diagrams of samples with $n$ points induced by a given filtration functor $\FiltFunc$.

\begin{defn}\label{def:pset}
	Fix $n \geq 1$ and $k \geq 0$. Let $(X,d_X) \in \M$ and $\FiltFunc:\M^\text{fin} \to \Filt$ be any filtration functor. The \define{(n,k)-$\FiltFunc$ persistence set} of $X$ is
	\begin{equation*}
		\Dnk{n,k}{\FiltFunc}(X) := \left\{\dgm_k^\FiltFunc(X'): X'\subset X \text{ such that } |X'| \leq n \right\}.
	\end{equation*}
	Even though the empty persistence diagram $\emptyset$ always belongs to the set $\Dnk{n,k}{\FiltFunc}(X)$, we establish the convention to omit writing it explicitly whenever convenient.
\end{defn}

\begin{remark}[\textbf{Persistence sets are functorial and isometry invariant}]
	\label{rmk:functorial_persistence_sets}
	Notice that, similarly to curvature sets (\textit{cf}. Remark \ref{rmk:functorial_curvature_sets}), persistence sets are functorial and isometry invariant. If $X \hookrightarrow Y$ isometrically, then $\Kn_n(X) \subset \Kn_n(Y)$, and consequently, $\Dnk{n,k}{\FiltFunc}(X) \subset \Dnk{n,k}{\FiltFunc}(Y)$ for all $n,k \in \N$. As such, they can be regarded, in principle, as \textit{signatures} that can be used to gain insight into datasets or to discriminate between different shapes.
\end{remark}

\begin{remark}\label{rmk:dgm_of_K_n}
	Recall from Definition \ref{def:filtrations} that filtration functors are equivariant under isometry. This implies that we can define the $\FiltFunc$-persistence diagram of a distance matrix as the diagram of the underlying pseudometric space. More explicitly, if a finite pseudometric space $X=\{x_1,\dots,x_n\}$ has distance matrix $\Psi_X^{(n)}(x_1,\dots,x_n) = M$, we define $\dgm_k^\FiltFunc(M) := \dgm_k^\FiltFunc(X)$. For that reason, we can view the persistence set $\Dnk{n,k}{\FiltFunc}(X)$ as the image of the map $\dgm_k^\FiltFunc:\Kn_n(X) \to \D$.
\end{remark}

Persistence sets inherit the stability of the filtration functor.
\begin{theorem}\label{thm:stability_Dnk}
	Let $\FiltFunc$ be a stable filtration functor with Lipschitz constant $L(\FiltFunc)$. Then for all $X,Y \in \M$ and integers $n \geq 1$ and $k \geq 0$, one has
	\begin{equation*}
		\dH^\D(\Dnk{n,k}{\FiltFunc}(X), \Dnk{n,k}{\FiltFunc}(Y)) \leq \frac{1}{2} L(\FiltFunc) \cdot \dH(\Kn_n(X), \Kn_n(Y)),
	\end{equation*}
	and thus
	\begin{equation*}
		\dH^\D(\Dnk{n,k}{\FiltFunc}(X), \Dnk{n,k}{\FiltFunc}(Y)) \leq L(\FiltFunc) \cdot \widehat{d}_\mathcal{GH}(X,Y),
	\end{equation*}
	where $\dH^\D$ denotes the Hausdorff distance between subsets of $\D$.
\end{theorem}

\begin{proof}%
	We will show that $\dH^\D(\Dnk{n,k}{\FiltFunc}(X), \Dnk{n,k}{\FiltFunc}(Y)) \leq \frac{1}{2} L(\FiltFunc) \cdot \dH(\Kn_n(X), \Kn_n(Y))$. Since $L(\FiltFunc) \cdot \widehat{d}_\mathcal{GH}(X,Y)$ is an upper bound for the right-hand side, the theorem will follow.\\
	\indent Assume $\dH(\Kn_n(X), \Kn_n(Y)) < \eta$. Pick any $D_1 \in \Dnk{n,k}{\FiltFunc}(X)$. Let $\mathbb{X} = (x_1, \dots, x_n) \in X^n$ such that $\Psi_X^{(n)}(\mathbb{X})=M_1$ and $D_1=\dgm_k^\FiltFunc(M_1)$. From the assumption on $\dH(\Kn_n(X), \Kn_n(Y))$, there exists $M_2 \in \Kn_n(Y)$ such that $\|M_1-M_2\|_{\infty} < \eta$. As before, let $\mathbb{Y} = (y_1,\dots,y_n)$  be such that $M_2=\Psi_Y^{(n)}(\mathbb{Y})$ and $D_2 = \dgm_k^\FiltFunc(M_2)$. By abuse of notation, consider $\mathbb{X}$ and $\mathbb{Y}$ as pseudometric spaces and observe that $D_1 = \dgm_k^\FiltFunc(\mathbb{X})$ and $D_2 = \dgm_k^\FiltFunc(\mathbb{Y})$ (see Remark \ref{rmk:dgm_of_K_n}). Then, by Definition \ref{def:stable_filtration_functors},
	\begin{equation*}
		\dB(D_1,D_2) \leq L(\FiltFunc) \cdot \dGH(\mathbb{X}, \mathbb{Y}).
	\end{equation*}
	With the correspondence $R = \left\{ (x_i,y_i) \in \mathbb{X} \times \mathbb{Y}: i=1, \dots, n\right\}$, we can bound the $\dGH(\mathbb{X}, \mathbb{Y})$ term by
	\begin{equation*}
		\dGH(\mathbb{X}, \mathbb{Y}) \leq \dfrac{1}{2} \dis(R) = \dfrac{1}{2} \max_{i,j=1,\dots,n} |d_X(x_i,x_j)-d_Y(y_i,y_j)| = \dfrac{1}{2} \|M_1-M_2\|_{\infty} < \dfrac{\eta}{2}.
	\end{equation*}
	In summary, for every $D_1 \in \Dnk{n,k}{\FiltFunc}(X)$, we can find $D_2 \in \Dnk{n,k}{\FiltFunc}(Y)$ such that $\dB(D_1,D_2) \leq L(\FiltFunc) \cdot \dGH(\mathbb{X}, \mathbb{Y}) < L(\FiltFunc) \cdot \eta/2$. Changing the roles of $X$ and $Y$ gives the same bound on the Hausdorff distance so, when we let $\eta \to \dH(\Kn_n(X), \Kn_n(Y))$, we obtain 
	\begin{equation*}
		\dH^\D(\Dnk{n,k}{\FiltFunc}(X), \Dnk{n,k}{\FiltFunc}(Y)) \leq \frac{1}{2} L(\FiltFunc) \cdot \dH(\Kn_n(X), \Kn_n(Y)).
	\end{equation*}
	as desired.
\end{proof}

\begin{remark}[Tightness of the bound]
	Recall that $L(\FiltFunc^\vr)=2$. Let $\delta_1 \neq \delta_2$ be positive real numbers. For $i=1,2$, let $X_i = \{x_1^{(i)},x_2^{(i)}\}$ be a two-point metric space with $d_{X_i}(x_1^{(i)}, x_2^{(i)}) = \delta_i > 0$. Observe that $\dgm_0^\vr(X_i) = \{(0,\delta_i) \}$, so
	\begin{equation*}
		\dH^\D(\Dvr{2,0}(X_1), \Dvr{2,0}(X_2)) = \dB(\dgm_{0}^\vr(X_1), \dgm_{0}^\vr(X_2)) = \|(0,\delta_1)-(0,\delta_2)\|_\infty = |\delta_1-\delta_2|.
	\end{equation*}
	On the other hand, $\widehat{d}_\mathcal{GH}(X_1, X_2) = \frac{1}{2}|\delta_1-\delta_2|$. To wit, since $\Kn_2(X_i) = \left\{ (\begin{smallmatrix}0 & 0\\ 0 & 0\end{smallmatrix}), (\begin{smallmatrix}0 & \delta_i\\ \delta_i & 0\end{smallmatrix}) \right\}$, we have the lower bound $\widehat{d}_\mathcal{GH}(X_1, X_2) \geq \frac{1}{2}\dH(\Kn_2(X_1), \Kn_2(X_2)) = \frac{1}{2}|\delta_1-\delta_2|$. The upper bound is given by $\dGH(X_1,X_2)=\frac{1}{2}|\delta_1-\delta_2|$. Thus, $\dH^\D(\Dvr{2,0}(X_1),\Dvr{2,0}(X_2)) = L(\FiltFunc^\vr) \cdot \widehat{d}_\mathcal{GH}(X_1,X_2)$.
\end{remark}

\subsubsection{VR-persistence sets of ultrametric spaces}
We now show that $\Dvr{n,0}$, the simplest of all persistence sets, can sometimes capture information that persistence diagrams cannot see.
\begin{defn}
	An \emph{ultrametric space} $(U, d_U)$ is a metric space such that every triple $x_1, x_2, x_3 \in U$ satisfies the \emph{ultrametric inequality}:
	\begin{equation*}
		d_{U}(x_1,x_3) \leq \max(d_{U}(x_1,x_2), d_{U}(x_2,x_3)).
	\end{equation*}
	Observe that applying the ultrametric inequality to $d_{U}(x_1,x_2)$ and $d_{U}(x_2,x_3)$ implies that the two largest distances among $d_{U}(x_1,x_3)$, $d_{U}(x_1,x_2)$, and $d_{U}(x_2,x_3)$ are equal. Ultrametric spaces are usually represented as dendrograms \cite{clust-um}, where $d_U(x_1, x_2)$ is the first value of $t$ such that $x_1$ and $x_2$ belong to the same cluster.
\end{defn}

\begin{example}[$\{\Dvr{n,0}\}_{n\geq 1}$ can distinguish spaces that $\dgm_0^\vr$ cannot.]\label{ex:D_n0}
	Let $X$ be a metric space with $N$ points. The collection of persistence sets $\{\Dvr{n,0}(X)\}_{n\geq 1}$ generally contains more information than $\dgm_0^\vr(X)$. Indeed, as we pointed out before, $\dgm_0^\vr(X)\in\Dvr{N,0}.$ The diagram $\dgm_0^\vr(X)$ contains $N-1$ (non-infinite) points (recall that we are using reduced homology) corresponding to the distances in a minimum spanning tree for $X$, while $\Dvr{2,0}(X)$ contains one point for every distinct distance in $X$, and $\Unk{2,0}{\vr}(X)$ counts the number of times each distance appears. Therefore, if all pairwise distances in $X$ are different, $\Dvr{2,0}$ will capture all ${N}\choose{2}$ pairwise distances whereas $\dgm_0^\vr(X)$ will be able to recover only $N-1$ of them. Now, if $X$ is instead assumed to be compact and connected, $\dgm_0^\vr(X)$ will be empty whereas $\Dvr{2,0}(X)$ will recover the set $\mathrm{im}(d_X) = [0,\diam(X)]$ of  all possible distances attained by pairs of points in $X$.\\
	\indent The difference between the invariants $\dgm_0^\vr(X)$ and $\Dvr{n,0}(X)$ becomes more apparent in the case of ultrametric spaces. Any ultrametric space $U$ is tree-like (see Definition \ref{def:tree-like-metric}), so by Lemma \ref{lemma:Dnk_tree}, both $\dgm_k^\vr(U) = \emptyset$ and $\Dvr{n,k}(U) = \{ \emptyset \}$ for $k \geq 1$. Thus, all the persistence information of ultrametric spaces is concentrated in dimension 0. With that in mind, Figure \ref{fig:Dn0_dendrograms} shows two ultrametric spaces $U_1$ and $U_2$ such that $\dgm_0^\vr(U_1) = \dgm_0^\vr(U_2) = \{ (0,1), (0,1), (0,2) \}$. Notice that $\Dvr{3,0}(U_1)$ consists of only  the diagram $D_1 := \{(0,1), (0,1)\}$, whereas $\Dvr{3,0}(U_2)$ consists of $D_1$ and $D_2 := \{(0,1), (0,2)\}$. Thus, $\Dvr{3,0}$ differentiates two (ultra)metric spaces that $\dgm_*^\vr(X)$ cannot tell apart.
	\begin{figure}[h]
		\centering
	\includegraphics[scale=0.65]{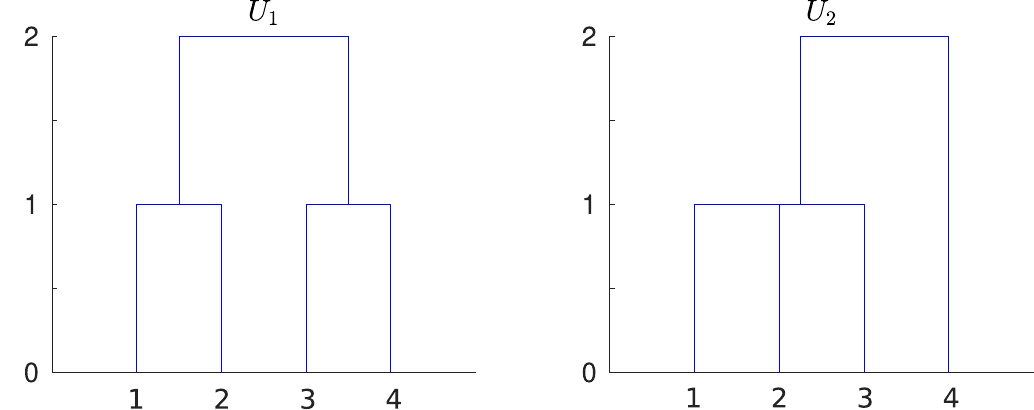}
		\caption{Two ultrametric spaces $U_1$, $U_2$ for which $\dgm_k^\vr(U_1)=\dgm_k^\vr(U_2)$ for all $k \geq 0$ but, in contrast, $\Dvr{n,0}(U_1) \neq \Dvr{n,0}(U_2)$ for $n=3$.}
\label{fig:Dn0_dendrograms}
	\end{figure}
\end{example}

\subsubsection{Computational cost and memory requirements}
\label{sec:computational_cost_Dnk_matrixtime}
One thing to keep in mind is that computing the single diagram $\dgm_{1}^\vr(X)$ when $n_X := |X| = 1000$ points is likely to be much more computationally expensive than computing 10,000 VR one-dimensional persistence diagrams obtained by randomly sampling points from $X$, i.e. approximating $\Dvr{n,1}(X)$ with small $n$. Let $c(n_X, k)$ denote the worst case time that it takes to compute $\dgm_k^\vr(X)$. Earlier algorithms, like the one in \cite{matrixtime}, are based on Gaussian elimination and their complexity is bounded in terms of the number of simplices in the filtration. In the worst case, computing $\dgm_k^\vr(X)$ requires knowledge of the $(k+1)$-simplices of $\vr_r(X)$, each of which is a subset of size $k+2$, so \cite{matrixtime} gives a worst-case bound of $c(n_X,k) \approx O\big(\nTuples{n_X}{k+2}^\omega\big)$. Here, we are assuming that multiplication of $m \times m$ matrices has cost\footnote{Currently, the best known constant is $\omega \approx 2.37286$ \cite{omega_matrix_mult}.} $O(m^\omega)$. In contrast, since there are $\nTuples{n_X}{n}$ possible $n$-tuples of points of $X$ (up to permutation), the complexity of computing $\Dvr{n,k}(X)$ is bounded by $O\big(c(n,k) \cdot \nTuples{n_X}{n}\big)$. For example, let $n=4$ and $k=1$. Since $k$ is a small constant, we approximate $\binom{n}{k} \approx n^k$. Then the worst-case bound for $\dgm_1^\vr(X)$ is $c(n_X,1) \approx O(n_X^{3\omega}) \approx O(n_X^{7.11})$, while $\Dvr{4,1}(X)$ only takes $O\left(c(4,1) \cdot \nTuples{n_X}{4}\right) \approx O(n_X^4)$. In general, $O\left(c(n,k)\cdot \nTuples{n_X}{n} \right)$ will be smaller than $c(n_X,k)$ as long as $n < \omega(k+2)$.\\
\indent Modern implementations of VR persistent homology \cite{ripser,ripser-pp} are much more efficient in practice, and their performance is linear in the number of simplices, that is, they have cost $c'(n_X,k) \approx O\big(\nTuples{n_X}{k+2}\big)$. Several sources give evidence for this claim. For example, the authors of \cite{ph-complexity-revised} argue that the practical linear bound is due to the sparsity of the boundary matrix. Similarly, the paper \cite{av-complexity-clique-filtrations} presents a rich family of examples where the expected runtime of the standard algorithm is better than the worst-case, at least for boundary matrices in degree 1. They also construct an example that realizes the worst-case runtime, although they argue that such an example is not typical in practice. In contrast, we will show that $\Dvr{n,k}(X)$ is non-empty only when $n \geq 2k+2$ in Theorem \ref{thm:n=2k+2}, so the cost of computing the full persistence set $\Dvr{n,k}(X)$ with modern algorithms is at least $O\big(c'(n,k) \cdot \nTuples{n_X}{2k+2}\big)$, which is larger than $c'(n_X,k)$.

\paragraph{Approximation.}
\label{pg:approx}
Another point which lends flexibility to the approximate computation of persistence sets is that one can actually easily cap the number of $n$-tuples to be considered by a parameter $N$, and this case the complexity associated to estimating $\Dvr{n,k}$ will be $O\big(\nTuples{n}{k+2} N\big)$. This is the pragmatic approach we have followed in the experiments reported in this paper and in the code provided in our Github repository \cite{github-repo}. In Section \ref{sec:approximation} we provide probabilistic convergence results as well as approximation bounds that provide a justification for this approach.

\paragraph{Parallelizability and memory requirements.}
Furthermore, these calculations are of course eminently pararelizable and, if $n\ll N$,  the memory requirements for computing an estimate to $\Dvr{n,k}(X)$ are substantially more modest than what computing $\dgm_k^\mathrm{VR}(X)$ would require since the boundary matrices that one needs to store in memory are several orders of magnitude smaller. We continue this discussion in Section \ref{sec:memory_comparison}, where we show datasets with increasing cardinality $n$ where the memory used to approximate principal persistence sets remains almost constant, whereas the memory required during the computation of persistence diagrams grows up to the point that the calculation cannot finish after a certain value of $n$.

\begin{prop}
	If $|X|=N$, the (worst case) computational cost of computing $\Dvr{n,k}(X)$ is $O\big(\binom{n}{k+2}^\omega \, N^n\big)$, where $\omega$ is the matrix multiplication exponent.
\end{prop}
	
Finally, if one is only interested in the principal persistence set, a much faster geometric algorithm is available, cf. Section \ref{sec:bd_times_algorithm}. See our Github repository \cite{github-repo} for a \texttt{parfor} based Matlab implementation.

\subsection{Persistence measures}
We now extend the constructions in the previous section to mm-spaces.

\begin{defn}\label{def:pmeas}
For each filtration functor $\FiltFunc$, integers $n \geq 1, k \geq 0$, and $X \in \Mw$, define the $(n,k)$-persistence measure of $X$ as (see Definition \ref{def:curvature_measures} and Remark \ref{rmk:dgm_of_K_n})
\begin{equation*}
	\Unk{n,k}{\FiltFunc}(X) := \left(\dgm_{k}^\FiltFunc \right)_\# \mu_n(X).
\end{equation*}
\end{defn}

We also have a stability result for these measures in terms of the Gromov-Wasserstein distance.
\begin{restatable}{theorem}{StabilityUnk}
    \label{thm:stability_Unk}
	Let $\FiltFunc$ be a stable filtration functor with Lipschitz constant $L(\FiltFunc)$. For all $X,Y \in \Mw$ and integers $n \geq 1$ and $k \geq 0$,
	\begin{equation*}
		\dW{p}^\D(\Unk{n,k}{\FiltFunc}(X), \Unk{n,k}{\FiltFunc}(Y)) \leq \frac{L(\FiltFunc)}{2} \cdot \dW{p}(\mu_n(X), \mu_n(Y))
	\end{equation*}
	and, as a consequence,
	\begin{equation*}
		\dW{p}^\D(\Unk{n,k}{\FiltFunc}(X), \Unk{n,k}{\FiltFunc}(Y)) \leq L(\FiltFunc) \cdot \widehat{d}_{\mathcal{GW},p}(X,Y).
	\end{equation*}
\end{restatable}
We prove this theorem in Appendix \ref{sec:appendix}.

\subsubsection{Probabilistic approximation of persistence sets}
\label{sec:approximation}

Regarding the idea of approximating $\Dvr{n,k}(X)$ using $N$ samples of $n$ points, consider the persistence sets shown in Figure \ref{fig:D41_spheres_torus}. These figures were obtained by sampling $N=10^6$ configurations of $n=4$ points uniformly at random from $\Sphere{1}$, $\Sphere{2}$, and from the torus $\mathbb{T}^2 = \Sphere{1} \times \Sphere{1}$. Observe that the analytical graph of $\Dvr{4,1}(\Sphere{1})$ in Figure \ref{fig:example_S1_configurations} and the approximation in the leftmost panel of Figure \ref{fig:D41_spheres_torus} are very similar. Their similarity indicates that using $N=10^6$ was more than enough to get a good approximation of $\Dvr{4,1}(\Sphere{1})$.\\
\indent More generally, consider an mm-space $\mm{X}$. Let $\mathbf{x}_1, \dots, \mathbf{x}_N \in X^n$ be i.i.d. random variables distributed according to the product measure $\mu_X^{\otimes n}$. Using the stochastic covering theorem from \cite[Theorem 34]{clust-um}, we find a lower bound for $N$ so that an approximation to $\Dvr{n,k}(X)$ via $\{ \dgm_k^\vr(\mathbf{x}_i) \}_{i=1}^N$ is $\epsilon$-close (with respect to the Hausdorff distance) to $\Dvr{n,k}(X)$ with probability at least $p$.

\smallskip
Now, define the function $f_X(\epsilon) := \displaystyle \min_{x \in X} \mu_X \left(B_\epsilon(x) \right)$. Recall that an mm-space space $X$ is (lower) Ahlfors regular (see Definition 3.18, page 252 of \cite{ahlfors-david}) if there exist constants $c,d>0$ such that $f_X(\epsilon)\geq \min(1,c\,\epsilon^d)$ for all $\epsilon>0.$ In the next theorem we assume that $X$ is Ahlfors regular.

\begin{restatable}[Approximation of $\Kn_n(X)$ and $\Dvr{n,k}(X)$]{theorem}{CoverageKn}
    \label{thm:coverage_Kn} 
     Let $n \geq 2$. Fix a confidence level $p \in [0,1]$ and $\epsilon > 0$. Let $\displaystyle N_0 =N_0(X;n,p,\epsilon):= \left\lceil \dfrac{-\ln\left[(1-p)f_X^n(\epsilon/2)\right]}{f_X^n(\epsilon/2)} \right\rceil$. Then, for all $N \geq N_0$,
    \begin{itemize}
        \item $\dH^{\Kn_n(X)}\left(\{ \Psi_X^{(n)}(\mathbf{x}_i) \}_{i=1}^N, \Kn_n(X) \right) \leq \epsilon$ with probability $\geq p$.
        \item $\dH^\D\left(\{ \dgm_k^\vr(\mathbf{x}_i) \}_{i=1}^N, \Dvr{n,k}(X) \right) \leq \epsilon$ with probability $\geq p$.
    \end{itemize}
    Furthermore, the estimators $\{ \Psi_X^{(n)}(\mathbf{x}_i) \}_{i=1}^N$ and $\{ \dgm_k^\vr(\mathbf{x}_i) \}_{i=1}^N$ converge to $\Kn_n(X)$ and $\Dvr{n,k}(X)$, respectively, almost surely as $N \to \infty$.
\end{restatable}
See Appendix \ref{sec:appendix_probability} for the proof. %

\subsection{Coordinates}
The objects $\Unk{n,k}{\FiltFunc}(X)$ can be complex, so it is important to find simple representations. Since these objects are probability measures on the space of persistence diagrams $\D$, we follow the statistical mechanics intuition and probe them via functions. In order to accomplish this, one should concentrate on families of functions $\zeta_\alpha: \D \to \R$, for $\alpha$ in some index set $A$. One example is the \textit{maximal persistence} of a persistence diagram: $\zeta(D) = \max_{(t_b,t_d) \in D}(t_d-t_b)$. In general, one desires to obtain a class of \textit{coordinates} \cite{acc13,coordinates-tropical} that is able to more or less canonically exhaust all the information contained in a given persistence diagram. A further desire is to design the class $\{ \zeta_\alpha \}_{\alpha \in A}$ in such a manner that it provides \textit{stable information} about a given measure $U \in \P(\D)$.\\
\indent In this section, we present a first result in that direction. To set up notation, let $\FiltFunc$ be a filtration functor. Let $n \geq 1$, $k \geq 0$ be integers, and take an mm-space $\mm{X}$. Consider a coordinate function $\zeta:\D \to \R$. The pushforward $\zeta_\#\Unk{n,k}{\FiltFunc}(X)$ is a probability measure on $\R$. We denote its distribution function by $H_X(t;n,k,\FiltFunc,\zeta) := \Unk{n,k}{\FiltFunc}(X)\big( \zeta\inv(-\infty,t] \big)$ defined for $t \in \R$.

\StabilityCoordinates
\begin{proof}
	According to \cite[Lemma 6.1]{memoli-dghlp-long},
	\begin{equation*}
		\int_\R \left| H_X(t)-H_Y(t) \right| dt \leq \inf_{\mu \in \coup_U} \int_{\Dnk{n,k}{\FiltFunc}(X) \times \Dnk{n,k}{\FiltFunc}(Y)} |\zeta(D)-\zeta(D')| \ \mu(dD \times dD'),
	\end{equation*}
	where $\coup_U$ is the set of couplings between $\Unk{n,k}{\FiltFunc}(X)$ and $\Unk{n,k}{\FiltFunc}(Y)$. Since $\zeta$ is Lipschitz, the right side is bounded above by
	\begin{align*}
		L(\zeta) &\cdot \inf_{\mu \in \coup_U} \int_{\Dnk{n,k}{\FiltFunc}(X) \times \Dnk{n,k}{\FiltFunc}(Y)} \dB(D,D') \ \mu(dD \times dD')\\
		&= L(\zeta) \cdot \inf_{\mu \in \coup_U} \diam_{\D,1}(\Dnk{n,k}{\FiltFunc}(X) \times \Dnk{n,k}{\FiltFunc}(Y)) \\
		&= L(\zeta) \cdot \dW{1}^{\D}(\Unk{n,k}{\FiltFunc}(X), \Unk{n,k}{\FiltFunc}(Y)) \\
		&\leq L(\zeta) L(\FiltFunc) \cdot \dGW{1}(X,Y).
	\end{align*}
	Theorem \ref{thm:stability_Unk} gives the last bound.
\end{proof}

Examples of the usefulness of a good system of coordinates in applications is given in \cite{kernel-hiraoka}. The authors codify a persistence diagram $D$ as a weighted sum of Dirac measures $\mu_D^w := \sum_{p \in D} w(p) \delta_p$ and embed it in a function space via
\begin{equation*}
	\mu_D^w \mapsto \sum_{p \in D} w(x)k(\cdot, x).
\end{equation*}
Here, $k:\R^2_{ad} \times \R^2_{ad} \to \R$ is a positive definite kernel, that is, a symmetric function such that $(k(x_i,x_j))_{i,j=1,\dots,n}$ is positive semi-definite for all $x_1,\dots,x_n \in \R^2_{ad} := \{ (b,d) \in \R^2 \ | \ b<d \}$. The weight function $w:\R^2_{ad} \to \R$ can be tuned to give more importance to points away or close to the diagonal, depending on which is more desirable. This kernel embedding has good theoretical properties, such as injectivity and stability with respect to the bottleneck distance, and good practical performance. Indeed, the authors considered several classification tasks with synthetic and real-world data. Their kernel methods performed well in classifications and sometimes outperformed other vectorization techniques, such as persistence landscapes and persistence images. What's more, the flexibility in the choice of kernel $k$ allowed detecting properties of points close to the diagonal when they were relevant to the experiment. For more details about the embeddings, other choices of kernels and embeddings, and parameter tuning, see their paper \cite{kernel-hiraoka}.

\section{Vietoris-Rips principal persistence sets}
From this point on, we focus on the Vietoris-Rips persistence sets $\Dvr{n,k}$ with $n=2k+2$. The reason to do so is Theorem \ref{thm:n=2k+2}, which states that the $k$-dimensional persistence diagram of $\vrcomp_{*}(X)$ is empty if $|X| < 2k+2$ and has at most one point if $|X|=2k+2$. What this means for persistence sets $\Dvr{n,k}(X)$ is that given a fixed $k$, the first interesting choice of $n$ is $n=2k+2$. We prove this fact in Section \ref{sec:properties_of_VR} and then use it to construct a graphical representation of $\Dvr{2k+2,k}(X)$ in Section \ref{sec:pps}. We also present the results of a classification experiment in Section \ref{sec:pps_experiments} and a comparison of the computational resources consumed by persistence sets and persistence diagrams in Section \ref{sec:memory_comparison}.

\subsection{Some properties of VR-filtrations and their persistence diagrams}
\label{sec:properties_of_VR}
Let $X$ be a finite metric space with $n$ points. The highest dimensional simplex of $\vrcomp_*(X)$ has dimension $n-1$, but even if $\vrcomp_*(X)$ contains $k$-dimensional simplices, it won't necessarily produce persistent homology in dimension $k$. The first definition of this section is inspired by the structure of the cross-polytope $\crosspoly{m}$; see Figure \ref{fig:crosspoly}. Recall that a set $\sigma \subset V  = \{ \pm e_1, \dots, \pm e_m \}$ is a face if it doesn't contain both $e_i$ and $-e_i$. In particular, there is an edge between $e_i$ and every other vertex except $-e_i$. The next definition tries to emulate this phenomenon in $\vrcomp_{*}(X)$.

\begin{defn}\label{def:tb_td}
	Let $(X, d_X)$ be a finite metric space, $A \subset X$, and fix $x_0 \in X$. Find two distinct points $x_1, x_2 \in A$ such that $d_X(x_0,x_1) \geq d_X(x_0,x_2) \geq d_X(x_0,a)$ for all $a \in A \setminus \{x_1,x_2\}$. Define 
	\begin{equation*}
		\tdA{x_0}{A} := d_X(x_0,x_1), \text{ and } \tbA{x_0}{A} := d_X(x_0,x_2).
	\end{equation*}
	\indent We set $\vdeathA{x_0}{A} := x_1$. When $A=X$ and there is no risk of confusion, we will denote $\tbA{x_0}{X}$, $\tdA{x_0}{X}$, and $\vdeathA{x_0}{X}$ simply as $\tb{x_0}, \td{x_0}$, and $\vdeath(x_0)$, respectively. Also define
	\begin{equation*}
		\tb{X} := \max_{x \in X} \tbA{x}{X} \,\,\,\mbox{and}\,\,\,
		\td{X} := \min_{x \in X} \tdA{x}{X}.
	\end{equation*}
\end{defn}

\indent In a few words, $\td{x} \geq \tb{x}$ are the two largest distances between $x$ and any other point of $X$. The motivation behind these choices is that if $r$ satisfies $\tb{x} \leq r < \td{x}$, then $\vr_r(X)$ contains all edges between $x$ and all other points of $X$, except for $\vdeath(x)$. If this holds for all $x \in X$, then $\vr_r(X)$ is isomorphic to a cross-polytope. Also, note that $\td{X}$ is the \emph{radius} $\mathbf{rad}(X)$ of $X$, cf. Definition \ref{def:rad-diam}. Also note that according to \cite[Proposition 9.6]{osman-memoli}, the death time of \emph{any} interval in $\dgm_*(X)$ is bounded by $\mathbf{rad}(X)$.\\
\indent Of course, $\vdeath(x)$ as defined above is not well defined. However, it is in the case that interests us.
\begin{lemma}\label{lemma:vd_is_unique}
	Let $(X,d_X)$ be a finite metric space and suppose that $\tb{X} < \td{X}$. Then $\vdeath:X\rightarrow X$ is well defined and $\vdeath\circ \vdeath = \mathrm{id}$.
\end{lemma}
\begin{proof}
	Given a point $x \in X$, suppose there exist $x_1 \neq x_2 \in X$ such that $d_X(x,x_1) = d_X(x,x_2) \geq d_X(x,x')$ for all $x' \in X$. Since $\tb{x}$ and $\td{x}$ are the two largest distances between $x$ and any $x' \in X$, we have $\tb{x} = \td{x}$. However, this implies $\td{X} \leq \td{x} = \tb{x} \leq \tb{X}$, which contradicts the hypothesis $\tb{X} < \td{X}$. Thus, we have a unique choice of $\vdeath(x)$ for every $x \in X$.\\
	\indent For the second claim, suppose that $\vdeath^2(x):=\vdeath(\vdeath(x)) \neq x$. Then $\td{\vdeath(x)} = d_X(\vdeath(x), \vdeath^2(x)) \geq d_X(\vdeath(x), x)$. Hence, the second largest distance $\tb{\vdeath(x)}$ is at least $d_X(\vdeath(x), x)$. However,
	\begin{equation*}
		\td{X} \leq \td{x} = d_X(x, \vdeath(x)) \leq \tb{\vdeath(x)} \leq \tb{X},
	\end{equation*}
	which is, again, a contradiction. Thus, $\vdeath^2(x) = x$.
\end{proof}

\indent Under these conditions, we can produce the claimed isomorphism between $\vr_r(X)$ and a cross-polytope.

\begin{prop}
    \label{prop:cross_polytopes}
    Let $(X,d_X)$ be a metric space with $|X|=n$, where $n \geq 2$ is even, and suppose that $\tb{X} < \td{X}$. Let $k=\frac{n}{2}-1$. Then $\vrcomp_{r}(X)$ is isomorphic, as a simplicial complex, to the cross-polytope $\crosspoly{k+1}$ for all $r \in [\tb{X},\td{X})$.
\end{prop}
\begin{proof}
	\indent Let $r \in [\tb{X}, \td{X})$. Lemma \ref{lemma:vd_is_unique} implies that we can partition $X$ into $k+1$ pairs $\{x_i^+, x_i^-\}$ such that $x_i^- = \vdeath(x_i^+)$, so define $f:\{\pm e_1, \dots, \pm e_{k+1}\} \to X$ as $f(\varepsilon \cdot e_i) = x_i^\varepsilon$, for $\varepsilon = \pm 1$. Both cross-polytopes and Vietoris-Rips complexes are flag complexes, so it's enough to verify that $f$ induces an isomorphism of their 1-skeleta. Indeed, for any $i=1,\dots,k+1$, $\varepsilon=\pm 1$, and $x \neq x_i^{-\varepsilon}$, we have $d_X(x_i^\varepsilon,x) \leq \tb{x_i^\varepsilon} \leq \tb{X} \leq r < \td{X} \leq \td{x_i^\varepsilon} = d_X(x_i^+, x_i^-)$. Thus, $\vr_r(X)$ contains the edges $[x_i^\varepsilon, x]$ for $x \neq x_i^{-\varepsilon}$, but not $[x_i^+, x_i^-]$. Since $f(\varepsilon \cdot e_i) = x_i^\varepsilon$, $f$ sends the simplices $[\varepsilon \cdot e_i, v]$ to the simplices $[x_i^\varepsilon, f(v)]$ and the non-simplex $[e_i, -e_i]$ to the non-simplex $[x_i^+, x_i^-]$.
\end{proof}

\indent A consequence of the previous proposition is that $H_k(\vrcomp_{r}(X)) \simeq H_k(\crosspoly{k+1}) = \Fld$ for $r \in [\tb{X}, \td{X})$. It turns out that $n=2k+2$ is the minimum number of points that $X$ needs to have in order to produce persistent homology in dimension $k$, which is what we prove next. The proof is inspired by the use of the Mayer-Vietoris sequence to find $H_k(\Sphere{k})$ by splitting $\Sphere{k}$ into two hemispheres that intersect in  an equator $\Sphere{k-1}$. Since the hemispheres are contractible, the Mayer-Vietoris sequence produces an isomorphism $H_k(\Sphere{k}) \simeq H_{k-1}(\Sphere{k-1})$. We emulate this by splitting $\vrcomp_r(X)$ into two halves which, under the right circumstances, are contractible and find the $k$-th persistent homology of $\vrcomp_*(X)$ in terms of the $(k-1)$-dimensional persistent homology of a subcomplex.\\
\indent Two related results appear in \cite{random-cliques-kahle, extremal-betti-numbers, viral-evolution}. The first two references prove that a flag complex with non-trivial $H_k$ has at least $2k+2$ vertices (Lemma 5.3 in \cite{random-cliques-kahle} and Proposition 5.4 in \cite{extremal-betti-numbers}); case (\ref{case:n<2k+2}) in our Theorem \ref{thm:n=2k+2} is a consequence of this fact. The decomposition $\vr_r(X) = \vr_r(B_0) \cup \vr_r(B_1)$ (see the proof for the definition of $B_0$ and $B_1$) already appears as Proposition 2.2 in the appendix of \cite{viral-evolution}. The novelty in  the next Theorem is the characterization of the persistence diagram $\dgm_k^\vr(X)$ in terms of $\tb{X}$ and $\td{X}$.

\begin{theorem}\label{thm:n=2k+2}
	Let $(X, d_X)$ be a metric space with $n$ points. Then:
	\begin{enumerate}[label=\Alph*.]
		\item\label{case:n<2k+2} For all integers $k > \frac{n}{2}-1$, $\dgm_k^\vr(X)=\emptyset$.
		\item\label{case:n=2k+2} If $n$ is even and $k = \frac{n}{2}-1$, then $\dgm_k^\vr(X)$ consists of a single point $(\tb{X}, \td{X} )$ if and only if $\tb{X} < \td{X}$, and is empty otherwise.

	\end{enumerate}
\end{theorem}

\begin{example}[The conclusion of Theorem 4.4 when $n=4$]
    \label{ex:D41_example}~\\
	\begin{minipage}{0.584\linewidth}
		Let us consider the case $k=1$ and $n=4$. Let $X = \{x_1, x_2, x_3, x_4\}$ as shown in Figure \ref{fig:D41_example}. In order for $\dgm_1^\vr(X)$ to be non-empty, $\vrcomp_r(X)$ has to contain all the ``outer edges'' and none of the ``diagonals''. That is, there exists $r>0$ such that
		\begin{equation*}
			d_{12}, d_{23}, d_{34}, d_{41} \leq r < d_{13}, d_{24}.
		\end{equation*}
		On the other hand, the calculations in Table \ref{table:tb_td_four_points_1} yield $\tb{X}=\max(d_{12}, d_{23}, d_{34}, d_{41})$, $\td{X}=\min(d_{13}, d_{24})$. We also have $\vdeath(x_1)=x_3$, $\vdeath(x_2)=x_4$, and $\vdeath \circ \vdeath = \mathrm{id}$. In either case, $\dgm_1^\vr(X) = \big(\max(d_{12}, d_{23}, d_{34}, d_{41}), \min(d_{13}, d_{24}) \big)$.
	\end{minipage}
	\hfill
	\begin{minipage}{0.38\linewidth}
			\centering
\begin{tikzpicture}[scale=1.55]
	\tikzmath{
		\pt=0.025;
		\px1 = 0;		\py1 = 0;
		\px2 = 1.55;	\py2 = 0.15;
		\px3 = 1.80;	\py3 = -1.2;
		\px4 = 0.2;		\py4 = -1.7;
	}

	\draw[fill] (\px1,\py1) circle [radius=\pt] node[above left ]{$x_1$};
	\draw[fill] (\px2,\py2) circle [radius=\pt] node[above right]{$x_2$};
	\draw[fill] (\px3,\py3) circle [radius=\pt] node[below right]{$x_3$};
	\draw[fill] (\px4,\py4) circle [radius=\pt] node[below left ]{$x_4$};

	\draw (\px1,\py1) -- (\px2,\py2) node[midway, above]{$d_{12}$};
	\draw (\px2,\py2) -- (\px3,\py3) node[midway, right]{$d_{23}$};
	\draw (\px3,\py3) -- (\px4,\py4) node[midway, below]{$d_{34}$};
	\draw (\px4,\py4) -- (\px1,\py1) node[midway, left ]{$d_{41}$};
	
	\draw[dashed] (\px1,\py1) -- (\px3,\py3) node[pos=0.25, above=4.5pt, right=4.5pt]{$d_{13}$};
	\draw[dashed] (\px2,\py2) -- (\px4,\py4) node[pos=0.60, above=3.0pt, left =3.0pt]{$d_{24}$};
\end{tikzpicture} 			
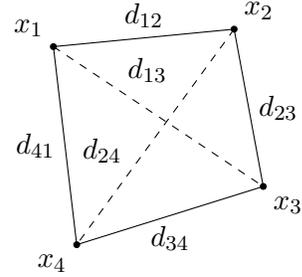
\captionof{figure}{A generic metric space with 4 points. In order for $\PH_1^\vr(X)$ to be non-zero, the two diagonals should be larger than the outer edges.}
			\label{fig:D41_example}
	\end{minipage}
	~\\~\\
	\indent However, if we had $d_{12}, d_{23}, d_{34} < d_{24} < d_{41} < d_{13}$ for example, then the 2-simplex $[x_2, x_3, x_4]$ appears before the would-be generator $[x_1,x_2]+[x_2,x_3]+[x_3,x_4]+[x_4,x_1]$, so $\dgm_1^\vr(X) = \emptyset$. According to Table \ref{table:tb_td_four_points_2}, $\tb{X}=d_{41} > d_{24}=\td{X}$, and $\vdeath(x_2)=x_4$ but $\vdeath(x_4)=x_1 \neq x_2$.\\
	\indent In general, we want to partition $X$ into pairs of ``opposite'' points, that is pairs $x,y$ such that $\vdeath(x)=y$ and $\vdeath(y)=x$. Intuitively, this says that the diagonals are larger than every other edge. If not, as in the second case, then no persistence is produced. As for $k=1$ and $n=4$, we will generally label the points as $x_1,x_2,x_3,x_4$ in such a way that
	\begin{equation*}
		\tb{X} =\max(d_{12}, d_{23}, d_{34}, d_{41}) \text{ and } \td{X} =\min(d_{13}, d_{24}).
	\end{equation*}
	\begin{center}
		\begin{minipage}{0.48\linewidth}
			\centering
			\begin{tabular}{c|cc}
				~ & $t_b$ & $t_d$ \\
				\hline
				$x_1$ & $\max(d_{41}, d_{12})$ & $d_{13}$ \\
				$x_2$ & $\max(d_{12}, d_{23})$ & $d_{24}$ \\
				$x_3$ & $\max(d_{23}, d_{34})$ & $d_{13}$ \\
				$x_4$ & $\max(d_{34}, d_{41})$ & $d_{24}$
			\end{tabular}
			\captionof{table}{$\tb{x_i}$ and $\td{x_i}$ when the sides of the quadrilateral $X$ are smaller than the diagonals.}
			\label{table:tb_td_four_points_1}
		\end{minipage}
		\hfill
		\begin{minipage}{0.48\linewidth}
			\centering
			\begin{tabular}{c|cc}
				~ & $t_b$ & $t_d$ \\
				\hline
				$x_1$ & $d_{41}$ & $d_{13}$ \\
				$x_2$ & $\max(d_{12}, d_{23})$ & $d_{24}$ \\
				$x_3$ & $\max(d_{23}, d_{34})$ & $d_{13}$ \\
				$x_4$ & $d_{24}$ & $d_{41}$
			\end{tabular}
			\captionof{table}{$\tb{x_i}$ and $\td{x_i}$ when the side $d_{41}$ of the quadrilateral $X$ is larger than the diagonal $d_{24}$.}
			\label{table:tb_td_four_points_2}
		\end{minipage}
	\end{center}
\end{example}

\begin{proof}[Proof of Theorem \ref{thm:n=2k+2}]
	The proof is by induction on $n$. Recall that $\PHvr_k(X)$ denotes the reduced homology of the VR-complex $\widetilde{H}_k(\vrcomp_*(X))$. If $n=1$, $\vrcomp_{r}(X)$ is contractible for all $r$, and so $\PHvr_k(X) = 0$ for all $k \geq 0 > \frac{n}{2}-1$. If $n=2$, let $X=\{x_0, x_1\}$. The space $\vrcomp_{r}(X)$ is two discrete points when $r \in [0,\diam(X))$ and an interval when $r \geq \diam(X)$. Then $\PHvr_k(X) = 0$ for all $k \geq 1 > \frac{n}{2}-1$, and $\PHvr_0(X) = \Int{0}{\diam(X)}$. Furthermore, this interval module equals $\Int{\tb{X}}{\td{X}}$ because $d_X(x_0,x_1) > d_X(x_0,x_0) = 0$, so $\tb{x_0}=0$ and $\td{x_0}=d_X(x_0,x_1)$. The same holds for $x_1$, so $\tb{X} = 0$ and $\td{X} = d_X(x_0, x_1) = \diam(X)$.\\
	\indent For the inductive step, assume that the proposition holds for every metric space with less than $n$ points. Fix $X$ with $|X|=n$ and an integer $k \geq \frac{n}{2}-1$. $\vrcomp_{r}(X)$ is contractible when $r \geq \diam(X)$, so let $r < \diam(X)$ and choose any pair $x_0, x_1 \in X$ such that $d_X(x_0,x_1) = \diam(X)$. Let $B_j = X \setminus \{x_j\}$ for $j=0,1$ and $A = X \setminus \{x_0,x_1\}$. Because of the restriction on $r$, $\vrcomp_{r}(X)$ contains no simplex $\sigma \supset [x_0,x_1]$, so $\vrcomp_{r}(X) = \vrcomp_{r}(B_0) \cup \vrcomp_{r}(B_1)$. At the same time, $\vrcomp_{r}(A) = \vrcomp_{r}(B_0) \cap \vrcomp_{r}(B_1)$, so we can use the Mayer-Vietoris sequence:

	\begin{equation}
	\label{eq:mayer_vietoris}
		\begin{tikzcd}
			[
			,row sep = 2ex
			,/tikz/column 1/.append style={anchor=base east}
			,/tikz/column 2/.append style={anchor=base west}
			]
			\widetilde{H}_k(\vrcomp_{r}(B_0)) \oplus \widetilde{H}_k(\vrcomp_{r}(B_1))
			\arrow{r}{}
			& \widetilde{H}_k(\vrcomp_{r}(X))
			\arrow[draw=none]{d}[name=Z, shape=coordinate]{}
			\arrow[parcurarrow=Z]{d}{} \\
			& \widetilde{H}_{k-1}(\vrcomp_{r}(A))
			\arrow{r}{(\iota_0,\iota_1)}
			& \widetilde{H}_{k-1}(\vrcomp_{r}(B_0)) \oplus \widetilde{H}_{k-1}(\vrcomp_{r}(B_1)).
		\end{tikzcd}
	\end{equation}

	~\\
	Since $|B_j| < n$, the induction hypothesis implies that $\PHvr_k(B_j) = 0$, and so $\partial_*$ is injective for any $r$. Now we verify the two claims in the statement.\\
	\noindent \textbf{Item \ref{case:n<2k+2}:} Suppose $k > \frac{n}{2}-1$.\\
	\indent Observe that $k-1 > \frac{n-2}{2}-1$ and $|A|=n-2$, so the induction hypothesis gives $\PHvr_{k-1}(A) = 0$. Then $\widetilde{H}_{k}(\vrcomp_{r}(X)) = 0$ for all $r \in [0,\diam(X)]$ and, since $\vrcomp_{r}(X)$ is contractible when $r \geq \diam(X)$, the homology of $\vr_r(X)$ is still 0 for $r \in [\diam(X), \infty)$.\\
	\noindent \textbf{Item \ref{case:n=2k+2}:} Suppose $k = \frac{n}{2}-1$.\\
	\indent By induction hypothesis, $\PHvr_{k-1}(A)$ is either a single interval $\Int{\tb{A}}{\td{A}}$ or 0 depending on whether $\tb{A} < \td{A}$ or not. Also define
	\begin{equation}\label{eq:recursive_tb}
		b := \max\left[\tb{A}, \max_{a \in A} d_X(x_0,a), \max_{a \in A} d_X(x_1,a)\right].
	\end{equation}
	We claim that $\PHvr_k(X) \cong \Int{b}{\td{A}}$ if and only if $b < \td{A}$.\\
	\noindent \textbf{Case 1:} If $r \in [0, \tb{A})$ or $r \in [\td{A}, \infty)$, then $\widetilde{H}_k(\vrcomp_{r}(X)) \cong 0$.\\
		\indent Since $\PHvr_{k-1}(A) \cong \Int{\tb{A}}{\td{A}}$, we have $\widetilde{H}_{k-1}(\vrcomp_{r}(A)) = 0$ for $r \notin [\tb{A}, \td{A})$. Now $\widetilde{H}_k(\vrcomp_{r}(X))=0$ follows from the Mayer-Vietoris sequence.\\
	\noindent \textbf{Case 2:} If $r \in [\tb{A},b)$, then $\widetilde{H}_k(\vrcomp_{r}(X)) \cong 0$.\\
	\indent Notice that we might have $b \geq \td{A}$. However, the conclusion for $r \in [\td{A}, b)$ follows from Case 1, so we can assume $r \in [\tb{A}, b) \cap [\tb{A}, \td{A})$. Additionally, if $b=\tb{A}$, then the interval $[\tb{A},b)$ is empty and there is nothing to prove. Suppose, then, $b = d_X(x_0,a_0) > \tb{A}$ for some $a_0 \in A$. In that case, $\vrcomp_r(B_1)$ doesn't contain the 1-simplex $[x_0,a_0]$, so $\vrcomp_r(A) \subset \vrcomp_r(B_1) \subset C(\vrcomp_r(A),x_0) \setminus [x_0,a_0]$. Additionally, since $r \in [\tb{A}, b) \cap [\tb{A}, \td{A})$, $\vr_r(A) \simeq \crosspoly{k}$ by Proposition \ref{prop:cross_polytopes}, that is, $\vr_r(A)$ has the homotopy type of $\Sphere{k-1}$. Then $C(\vr_r(A), x_0)$ has the homotopy type of a hemisphere of $\Sphere{k}$ whose equator is $\vr_r(A) \simeq \Sphere{k-1}$. Hence, $C(\vrcomp_r(A),x_0) \setminus [x_0,a_0]$ is homotopy equivalent to a punctured hemisphere of $\Sphere{k}$, which strong deformation retracts onto $\vr_r(A)$. Thus, the composition induced by inclusions
	\begin{equation*}
		\widetilde{H}_{k-1}(\vrcomp_r(A)) \to \widetilde{H}_{k-1}(\vrcomp_r(B_1)) \to \widetilde{H}_{k-1}(C(\vrcomp_r(A)) \setminus [x_0,a_0])
	\end{equation*}
	is an isomorphism. This implies that the first map $\widetilde{H}_{k-1}(\vrcomp_r(A)) \to \widetilde{H}_{k-1}(\vrcomp_r(B_1))$ is injective which, in turn, makes $\widetilde{H}_{k-1}(\vrcomp_{r}(A)) \to \widetilde{H}_{k-1}(\vrcomp_{r}(B_0)) \oplus \widetilde{H}_{k-1}(\vrcomp_{r}(B_1))$ injective. Since $\partial_*$ in (\ref{eq:mayer_vietoris}) is also an injection, $\widetilde{H}_k(\vrcomp_{r}(X))=0$ for $r \in [\tb{A},b)$.\\
	\noindent \textbf{Case 3:} If $r \in [b,\td{A})$, then $\widetilde{H}_k(\vrcomp_{r}(X)) \cong \Fld$.\\
	\indent The definition of $b$ implies that $\tb{A} \leq b$, so $\tb{A} < \td{A}$. Then, the induction hypothesis on $A$ implies that $\PHvr_{k-1}(A) = \Int{\tb{A}}{\td{A}}$ and, in particular, $ \widetilde{H}_{k-1}(\vr_r(A)) = \Fld$ for $r \in [b,\td{A})$. Now, since $\max_{a \in A} d_X(x_1,a) \leq b \leq r$, $\vrcomp_r(B_0)$ contains all simplices $[x_1,a_1,\dots,a_m]$, where $[a_1, \dots, a_m]$ is a simplex of $\vrcomp_{r}(A)$. In other words, $\vrcomp_{r}(B_0) = C(\vrcomp_{r}(A), x_1) \simeq *$. The same holds for $\vrcomp_r(B_1)$, so their homology is 0, and the Mayer-Vietoris sequence gives an isomorphism $\widetilde{H}_k(\vrcomp_{r}(X)) \xrightarrow{\sim} \widetilde{H}_{k-1}(\vrcomp_{r}(A)) \simeq \Fld$.\\
	\noindent \textbf{Case 4:} If $b \geq \td{A}$, then $\widetilde{H}_k(\vrcomp_{r}(X)) \cong 0$ for all $r \geq 0$.\\
	\indent The conclusion follows from Cases 1 and 2 because the interval $[b, \td{A})$ from Case 3 is empty and $[0, \infty) = [0, \tb{A}) \cup [\tb{A}, b) \cup [\td{A}, \infty)$.\\
	~\\
	\indent The last thing left to check is that $\vrcomp_{*}(X)$ produces persistent homology precisely when $\tb{X} < \td{X}$. So far we have $\PHvr_k(X) = \Int{b}{\td{A}}$ if and only if $b < \td{A}$, so now we show that $\tb{X} < \td{X}$ is equivalent to $b < \td{A}$.\\
	\noindent \textbf{Case 1:} $b < \td{A}$ implies $\tb{X} < \td{X}$.\\
	\indent Let $a \in A$. Since $\tb{A} \leq b < \td{A}$, $\vdeathA{a}{A}$ is well-defined by Lemma \ref{lemma:vd_is_unique}. Then for every $a' \neq \vdeathA{a}{A}$, $d_X(a,a') \leq \tbA{a}{A} < \tdA{a}{A}$. Also, for $j=0,1$, we have $d_X(a, x_j) \leq b < \td{A} \leq \tdA{a}{A}$ by definition of $b$. In other words, for every $x \in X \setminus \{\vdeathA{a}{A}\}$, $d_X(a,x) < \tdA{a}{A}$, which means that the point in $X$ furthest away from $a$ is still $\vdeathA{a}{A} \in A$. Thus, $\tdA{a}{X} = \tdA{a}{A}$ and $\tbA{a}{X} = \max\left[ \tbA{a}{A}, d_X(a,x_0), d_X(a,x_1) \right]$. Additionally, $d_X(x_0,x_1) = \diam(X)$ and $d_X(a,x_j) \leq b < \td{A} \leq \diam(X)$, so $\tdA{x_j}{X} = \diam(X)$, $\tbA{x_j}{X} = \max_{a \in A} d_X(x_j,a)$, and $\vdeathA{x_0}{X} = x_1$. Hence,
	\begin{equation*}
		\td{X}
		= \min \left\{ \tdA{x_0}{X}, \tdA{x_1}{X}, \min_{a \in A} \tdA{a}{X} \right\}
		= \min \left\{ \diam(X), \min_{a \in A} \tdA{a}{A} \right\} = \td{A},
	\end{equation*}
	
	and
	\begin{align*}
		b
		&= \max\left[\tb{A}, \max_{a \in A} d_X(x_0,a), \max_{a \in A} d_X(x_1,a)\right]\\
		&= \max\left[\max_{a \in A} \tbA{a}{A}, \max_{a \in A} d_X(x_0,a), \max_{a \in A} d_X(x_1,a)\right]\\
		&= \max\left[\max_{a \in A} \tbA{a}{X}, \tbA{x_0}{X}, \tbA{x_1}{X}\right] = \tb{X}.
	\end{align*}
	In conclusion, $\tb{X} = b < \td{A} = \td{X}$.\\
	\noindent \textbf{Case 2:} $b \geq \td{A}$ implies $\tb{X} \geq \td{X}$.\\
	\indent Let $a_0 \in A$ such that $\td{A} = \tdA{a_0}{A}$. Notice that $\tdA{a_0}{X}$ can differ from $\tdA{a_0}{A}$ if $d_X(a_0,x_j) \geq d_X(a_0,\vdeathA{a_0}{A})$ for some $j=0,1$. However, we have $b \geq d_X(a_0,x_j)$ by definition, so $b$ would still be greater than $\tdA{a_0}{X}$ even if $\tdA{a_0}{X} \neq \tdA{a_0}{A}$. With this in mind, we have two sub-cases.\\
	\noindent \textbf{Case 2.1:} $b=\tb{A}$.\\
	\indent Since $\tbA{a}{X}$ takes the maximum over a larger set than $\tbA{a}{A}$ does, $\tbA{a}{A} \leq \tbA{a}{X}$ for all $a \in A$. Then
	\begin{equation*}
		\tb{X} \geq \tb{A} = b \geq \tdA{a_0}{X} \geq \td{X}.
	\end{equation*}
	\noindent \textbf{Case 2.2:} $b > \tb{A}$.\\
	\indent Write $b = d_X(a_1,x_j)$, where $a_1 \in A$ and $j$ is either $0$ or $1$. Observe that $\tdA{x_j}{X} = \diam(X) = d_X(x_1,x_2) \geq d_X(a_1,x_j)$, so $\tbA{x_j}{X} \geq d_X(a_1,x_j)$. Then
	\begin{equation*}
		\tb{X} \geq \tbA{x_j}{X} \geq d_X(a_1,x_j) = b \geq \tdA{a_0}{X} \geq \td{X}.
	\end{equation*}
	This concludes the proof of Case 2.
\end{proof}

\subsubsection{A geometric algorithm for computing $\dgm_k^\vr(X)$ when $|X|=n$ and $k=\frac{n}{2}-1$.}
\label{sec:bd_times_algorithm}
Thanks to Theorem \ref{thm:n=2k+2}, we can compute $\dgm_k^\vr(X)$ in $O(n^2)$ time if $|X|=n=2k+2$. Indeed, both  $\tb{x}$ and $\td{x}$ can be found in at most $(n-1)+(n-2) = 2n-3$ steps because finding a maximum takes as many steps as the number of entries. We compute both quantities for each of the $n$ points in $X$ and then find $\tb{X} = \max_{x \in X} \tb{x}$ and $\td{X} = \min_{x \in X} \td{x}$ in $n$ steps each. After comparing $\tb{X}$ and $\td{X}$, we are able to determine whether $\dgm_k^\vr(X)$ is $\{({\tb{X}},{\td{X}})\}$ or empty in at most $n(2n-3)+2n+1 = O(n^2)$ steps. This is a significant improvement from the linear bound (in the number of simplices) $O\big(\nTuples{n}{k+2}\big) = O\big(\nTuples{n}{n/2+1}\big)$ discussed Section \ref{sec:computational_cost_Dnk_matrixtime}. We summarize this paragraph as follows:

\begin{prop}
    \label{prop:bd_times_algorithm}
    Let $X$ be a metric space with $n$ points and $k = \frac{n}{2}-1$. The cost of computing $\tb{X}$ and $\td{X}$ as in Definition \ref{def:tb_td} is $O(n^2)$.
\end{prop}
\noindent A \texttt{parfor} based Matlab implementation is provided in our Github repository \cite{github-repo}.

\subsection{The definition of VR-principal persistence sets}
\label{sec:pps}
Theorem \ref{thm:n=2k+2} has two consequences for $\vr$-persistence sets. The first is the following corollary.

\begin{corollary}\label{cor:Dnk=0}
	Let $X$ be any metric space. Given $k \geq 0$ fixed, $\Dvr{n,k}(X)$ is empty for all $n < 2k+2$.
\end{corollary}

This means that the first interesting choice of $n$ is $n=2k+2$, and in that case, any sample $Y \subset X$ with $|Y|=n$ will produce only one point in its persistence diagram. This case will be focus of the rest of the paper, so we give it a name.

\begin{defn}
    \label{def:PPS}
    $\Dvr{2k+2,k}(X)$ and $\Unk{2k+2,k}{\vr}(X)$ are called, respectively, the Vietoris-Rips \define{principal persistence set} and the \define{principal persistence measure} of $X$ in dimension $k$.
\end{defn}

Let $k \geq 0$ and $n=2k+2$. Notice that the results in Section \ref{sec:computational_cost_Dnk_matrixtime} imply the worst case bound $O(n^{k+2} \cdot N)$ for approximating principal persistence sets with $N$ samples (cf. page \pageref{pg:approx}). We improve this bound via the algorithm from Section \ref{sec:bd_times_algorithm}.

\begin{corollary}\label{coro:comp-cost-pps}
    Let $X$ be a metric space. Fix $k \geq 0$ and $n=2k+2$. The cost of approximating $\Dvr{2k+2,k}(X)$ with $N$ samples is $O(n^2 \cdot N)$.
\end{corollary}

The fact that the diagrams in $\Dvr{2k+2,k}(X)$ have at most one point allows us to visualize principal persistence sets as subsets of points in $\R^2$ (cf. Figure \ref{fig:stack-D-to1}), and also to recast their properties as properties of these subsets of $\R^2$.

\begin{defn}
    \label{def:PPS_aggregate}
    Let $\D_1 := \{D \in \D \mid |D| \leq 1\}$ and $\Dop := \{ (x,y) \in \R^2 \mid 0 < x < y \text{ or } x=y=0\}$. Define $\Phi:\D_1 \to \Dop$ by $\Phi(\emptyset) = (0,0)$ and $\Phi(\{(t_b,t_d)\}) = (t_b,t_d)$.
\end{defn}

The immediate use of $\Phi$ is to visualize principal persistence sets as subsets of $\R^2$ by plotting $\Phi\big(\Dvr{2k+2,k}(X)\big)$ (as we do in Figure \ref{fig:D41_spheres_torus}). Additionally, via the map $\Phi$ we can also import principal persistence measures, metrics, and the stability of principal persistence sets into easier concepts involving $\R^2$. For example, we can visualize the pushforward measure $\Phi_\# \Unk{2k+2,k}{\vr}(X)$ by coloring its support $\Phi\big(\Dvr{2k+2,k}(X)\big)$ according to density. See Figure \ref{fig:D41_spheres_torus}.\\
\indent If we define the metric $\dB$ on $\Dop$ by
\begin{equation}
    \label{eq:bottleneck_R2}
    \textstyle
    \dB\big( (x,y), (x',y') \big) := \min\left\{ \max(|x-x'|, |y-y'|), \frac{1}{2}\max(y-x, y'-x') \right\}
\end{equation}
then the map $\Phi$ is an isometry between $(\D_1,d_{\mathcal D})$ and $(\Dop,d_{\mathcal B})$. It follows that the stability theorems \ref{thm:stability_Dnk} and \ref{thm:stability_Unk} can be reprhased (in way that will be immediately useful in Section \ref{sec:classification}) as follows.

\begin{theorem}
    \label{thm:PPS_stability}
    Let $X,Y \in \M$. For any $k \geq 0$,
    \begin{equation*}
        \dH^{\Dop}(\Phi \circ \Dvr{2k+2,k}(X), \Phi \circ \Dvr{2k+2,k}(Y)) \leq \dH(\Kn_n(X), \Kn_n(Y)) \leq 2 \cdot \widehat{d}_\mathcal{GH}(X,Y),
    \end{equation*}
    and
    \begin{equation*}
        \dW{p}^{\Dop}(\Phi_\# \Unk{2k+2,k}{\vr}(X), \Phi_\# \Unk{2k+2,k}{\vr}(Y)) \leq \dW{p}(\mu_n(X), \mu_n(Y)) \leq 2 \cdot \widehat{d}_{\mathcal{GW},p}(X,Y),
    \end{equation*}
    where $\dH^{\Dop}$ and $\dW{p}^{\Dop}$ denote the Hausdorff and $p$-Wasserstein distances defined on $(\Dop, \dB)$.
\end{theorem}

\begin{remark}
    To reduce notational overload, we will simply write $\Dvr{2k+2,k}(X)$ and $\Unk{2k+2,k}{\vr}(X)$ instead of $\Phi \circ \Dvr{2k+2,k}(X)$ and $\Phi_\# \Unk{2k+2,k}{\vr}(X)$. Additionally, whenever referring to  distances between points in a given $\Dvr{2k+2,k}(X)$ or between two sets $\Dvr{2k+2,k}(\cdot)$ or between measures $\Uvr{2k+2,k}(\cdot)$, we will invoke the metrics $\dB$, $\dH^{\Dop}$, and $\dW{p}^{\Dop}$ described above.
\end{remark}

\begin{example}\label{ex:D41_spheres}
    Figure \ref{fig:D41_spheres_torus} shows computational approximations to the principal persistence measure $\Unk{4,1}{\vr}$ of $\Sphere{1}, \Sphere{2}$, and $\T^2 := \Sphere{1} \times \Sphere{1}$. The spheres are equipped with their usual Riemannian metrics $d_{\Sphere{1}}$ and $d_{\Sphere{2}}$ respectively. As for the torus, we used the $\ell^2$ product metric defined as
    \begin{equation*}
    	d_{\T^2}\left( (\theta_1,\theta_2), (\theta_1',\theta_2') \right) := \sqrt{ \left(d_{\Sphere{1}}(\theta_1,\theta_1') \right)^2 + \left(d_{\Sphere{1}}(\theta_2,\theta_2') \right)^2},
    \end{equation*}
    for all $(\theta_1,\theta_2), (\theta_1',\theta_2') \in \T^2$. The diagrams were computed with the algorithm in Section \ref{sec:bd_times_algorithm} implemented in MATLAB using $10^6$ 4-tuples of points sampled uniformly at random. The calculations took $12.11$ seconds for the circle, $20.08$ sec. for the sphere and $25.96$ sec. for the torus. The fraction of configurations that produced a non-diagonal point were 11.08 \% for the circle, 12.63 \% for the sphere and 14.80 \% for the torus.\\
    \indent In these graphs we observe the functioriality property $\Dvr{n,k}(X) \subset \Dvr{n,k}(Y)$ whenever $X \hookrightarrow Y$ (see Remark \ref{rmk:functorial_persistence_sets}). Notice that $\Sphere{1}$ embeds into $\Sphere{2}$ as the equator, and as slices $\Sphere{1} \times \{x_0\}$ and $\{x_0\} \times \Sphere{1}$ in $\mathbb{T}^2$. The effect on the persistence sets is that a copy of $\Dvr{4,1}(\Sphere{1})$ appears in both $\Dvr{4,1}(\Sphere{2})$ and $\Dvr{4,1}(\mathbb{T}^2)$.
\end{example}

\begin{figure}
\noindent\makebox[\textwidth]{
\begin{tabular}{ccc}
	\begin{overpic}[scale=0.32]{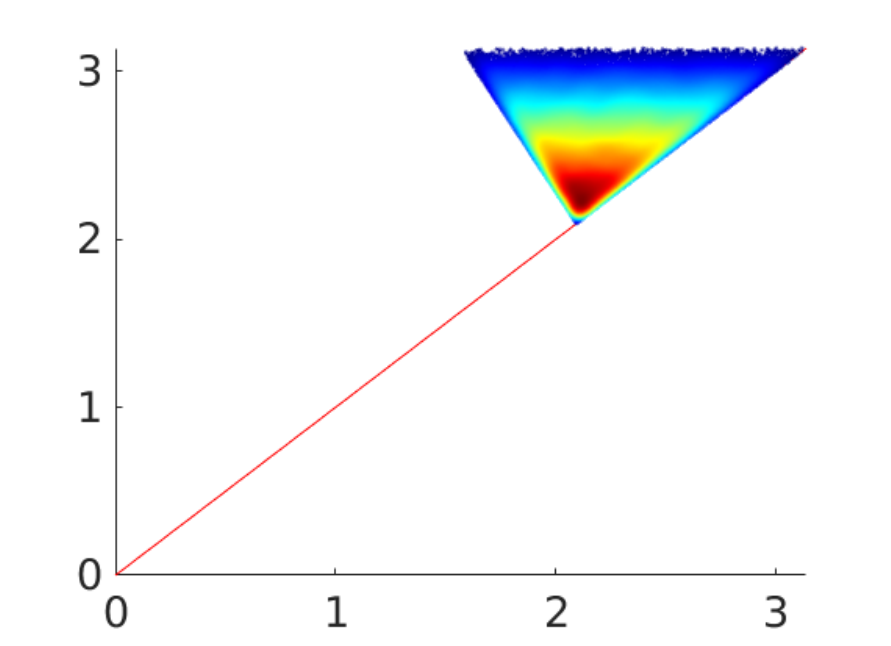}
		\put(50,77){  \makebox(0,0){
				$\Unk{4,1}{\vr}(\Sphere{1})$
			}
		}
	\end{overpic}
	
	&
	\begin{overpic}[scale=0.32]{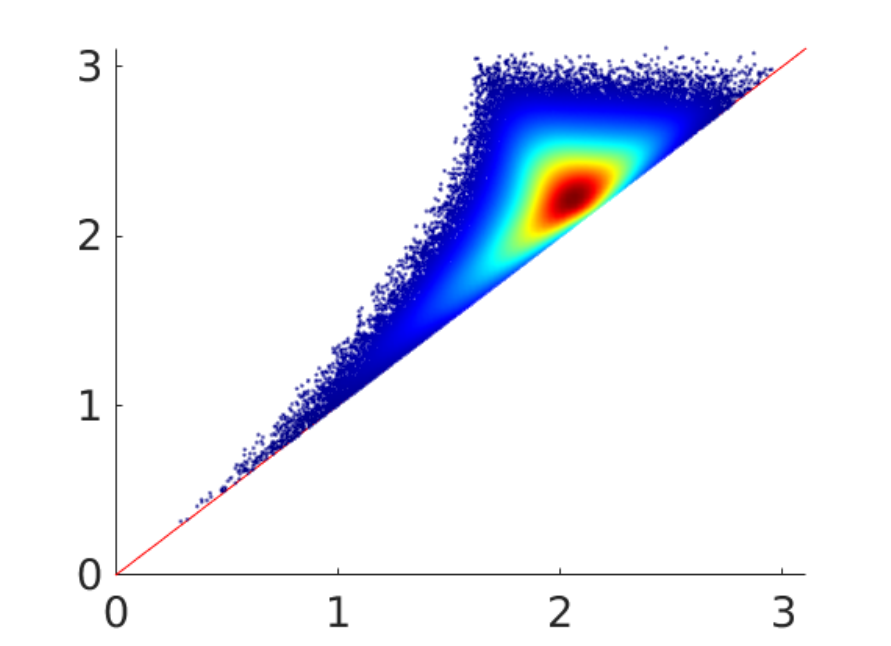}
		\put(50,77){  \makebox(0,0){
				$\Unk{4,1}{\vr}(\Sphere{2})$
			}
		}
	\end{overpic}
	
	&
	\begin{overpic}[scale=0.32]{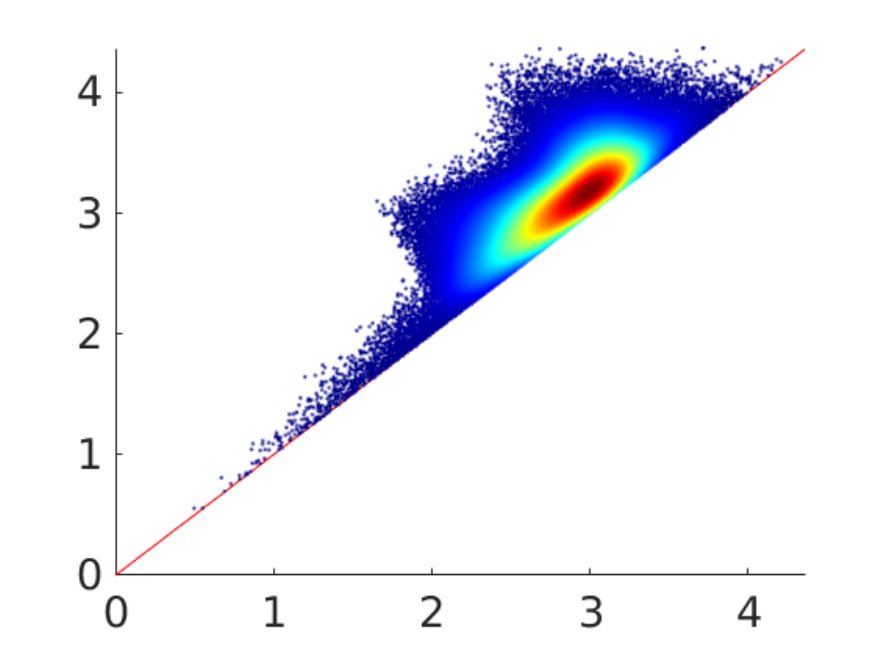}
		\put(50,77){  \makebox(0,0){
				$\Unk{4,1}{\vr}(\mathbb{T}^2)$
			}
		}
	\end{overpic}
\end{tabular}

} 	\caption{
		From left to right: computational approximations to the 1-dimensional persistence measures $\Unk{4,1}{\vr}(\Sphere{1}), \Unk{4,1}{\vr}(\Sphere{2})$, and $\Unk{4,1}{\vr}(\T^2)$. The colors represent the density of points in the diagram. The support of each measure (that is, the colored region) is the persistence set $\Dvr{4,1}$ of the corresponding metric space. Notice how these results agree with the functoriality property (cf. Remark \ref{rmk:functorial_persistence_sets}): namely, that the persistence set of $\Sphere{1}$ is a subset of the respective persistence sets of $\Sphere{2}$ and $\T^2$  (see Example \ref{ex:D41_spheres}).
	}
	\label{fig:D41_spheres_torus}
\end{figure}

\subsection{Discriminating power of VR-principal persistence sets}
\label{sec:pps_experiments}
In this section, we study the discriminating power of principal persistence sets in two synthetic examples and in one practical dataset. In the first example, we see that $\Dvr{4,1}(R)$ correlates with the ``size'' of the hole of a rectangle $R \subset \R^2$. The second example shows that $\Dvr{6,2}$ can tell apart a flat torus from a rectangle. Lastly, we show that various metrics induced by persistence sets (or persistence measures) can classify the 3D shapes from the paper \cite{sumner-paper} with classification error as low as 7.38 \%.
\begin{example}[$\Dvr{6,2}$ can distinguish the torus from a rectangle]
    \label{ex:D62_torus_rectangle}
    \indent Let $R>0$, and define $S_R := \frac{R}{\pi} \cdot \Sphere{1}$ be the circle with geodesic distance rescaled to have perimeter $2R$. Define the rectangle $Q_{R_1, R_2} := [0,2R_1] \times [0,2R_2] \subset \R^2$ and the torus $T_{R_1,R_2} := S_{R_1} \times S_{R_2}$, and equip both spaces with the $\ell^p$ product metric for some $p \geq 1$. Inspired by the observation that $H_2(T_{R_1,R_2}) \cong \Fld$ and $H_2(Q_{R_1,R_2}) \cong 0$, we ask if $\Dvr{6,2}$ can distinguish $T_{R_1, R_2}$ from $Q_{R_1, R_2}$. Table \ref{table:D62_torus_rectangle} shows experimental approximations to $\Dvr{6,2}(T_{R_1, R_2})$ and $\Dvr{6,2}(Q_{R_1, R_2})$ for several values of $R_1$ and $R_2$, and different $\ell^p$ metrics. The diagrams were obtained by uniformly sampling 1,000,000 6-point subsets from each space.\\
	\indent Regardless of the choice of parameters, the approximations of $\Dvr{6,2}(Q_{R_1, R_2})$ have almost no points, while those of $\Dvr{6,2}(T_{R_1, R_2})$ have a significant number of non-diagonal points. It is important to note that the diagrams $\Dvr{6,2}(Q_{R_1, R_2})$ with the $\ell^2$ metric have more points than are shown here. For instance, both $Q_{1,1}$ and $Q_{1,3}$ contain a circle of radius 1, so $\Dvr{6,2}(\Sphere{1}_E) \subset \Dvr{6,2}(Q_{1,1}) \subset \Dvr{6,2}(Q_{1,3})$ (cf. Theorem \ref{thm:critical_tb_even}). However, these examples show that the measures $\Unk{6,2}{\vr}(Q_{R_1,R_2})$ and $\Unk{6,2}{\vr}(T_{R_1,R_2})$ induced by the uniform measures on the respective spaces are different. Lastly, it is interesting to note that these computations require less points (6) than the number of vertices (7) in a minimal simplicial complex homeomorphic to the torus. See, for instance, Theorem 1 of \cite{triangulated-manifolds}.\\
    \indent $\Dvr{4,1}$ can also tell apart the torus and the rectangle. We will see in Proposition \ref{prop:D41_R2_S2} that any $(t_b, t_d) \in \Dvr{4,1}(\R^2)$ satisfies $t_d \leq \sqrt{2} t_b$. This holds, in particular, for any $(t_b, t_d) \in \Dvr{4,1}(Q_{R_1,R_2})$. In contrast, the set $X = \{(0,0), (R_1/2, 0), (R_1, 0), (3R_1/2, 0)\} \subset T_{R_1,R_2}$ satisfies $\tb{X} = R_1/2$ and $\td{X} = R_1$ but $\td{X} = 2\tb{X} > \sqrt{2} \tb{X}$.
    \begin{figure}[h]
    	\centering
    	\includegraphics[width=0.85\linewidth]{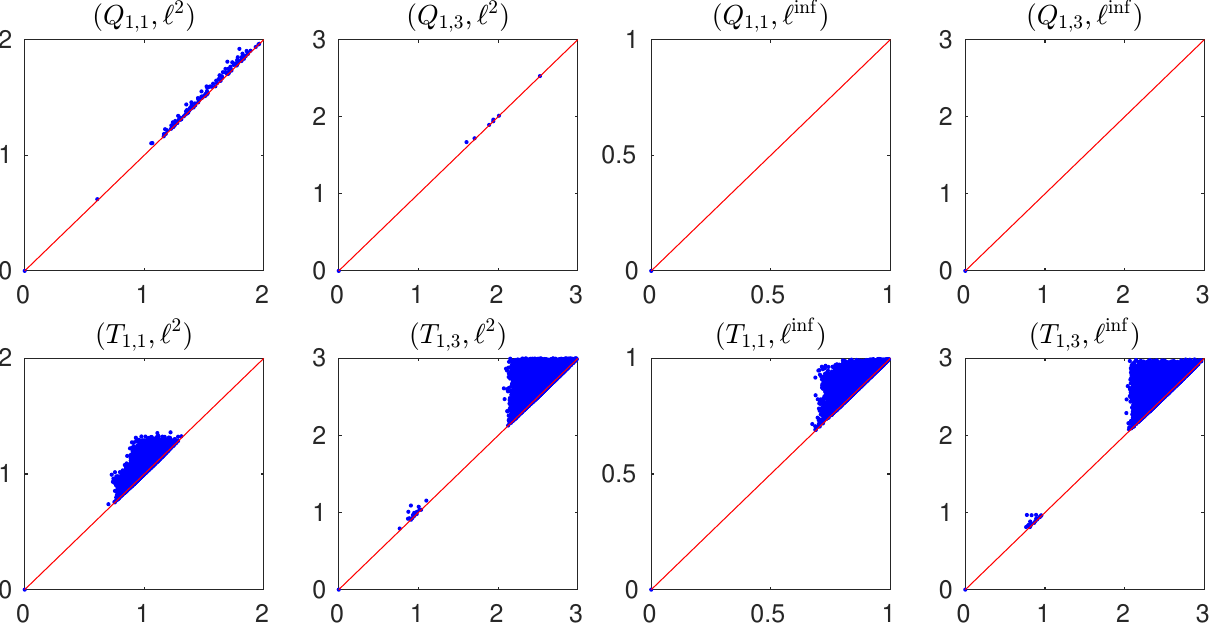}
    	\caption{Diagrams $\Dvr{6,2}$ for the torus $T_{R_1, R_2} = \frac{R_1}{\pi} \cdot \Sphere{1} \times \frac{R_2}{\pi} \cdot \Sphere{1}$ and the rectangle $Q_{R_1, R_2} = [0, 2R_1] \times [0, 2R_2]$ equipped with the $\ell^p$ product metric for $p=2, \infty$ (see Example \ref{ex:D62_torus_rectangle}). Out of the 1,000,000 configurations sampled from $Q_{R_1, R_2}$, the percentage that produced a non-diagonal point are, from left to right, $0.01\ \%$, $0.00\ \%$, $0.00\ \%$, and $0.00\ \%$. For $T_{R_1, R_2}$, the percentages are $2.20\ \%$, $2.00\ \%$, $2.86\ \%$, and $1.99\ \%$.}
    	\label{table:D62_torus_rectangle}
    \end{figure}
\end{example}

\begin{example}[Sampling effects]
	\label{ex:sampling_defects}
	The following two experiments illustrate how sampling  affects persistence sets and persistence diagrams. For both experiments we used $N=10^5$ tuples when estimating persistence sets.\\
	\smallskip\noindent\textbf{The first experiment.} In the first case, we consider the metric glueing $\Sp^1 \vee \Sp^2$, where each sphere is given its own geodesic metric. For a given parameter value $0 \leq p \leq 1$, we sample a set $X \subset \Sp^1 \vee \Sp^2$ with 1000 i.i.d. points as follows: with probability $p$ the point is sampled uniformly at random from $\Sp^1$ and with probability $(1-p)$ the point is sampled uniformly at random from $\Sp^2$. For $p=0,0.1,0.2,0.3,0.4$ and $0.5$ we calculated $\dgm_1^\vr(X)$, $\dgm_2^\vr(X)$, and $\Uvr{4,1}(X)$; the results are shown in Figure \ref{fig:S1vS2_results}. For $p>0$, the persistence diagrams clearly indicate that $X$ has one cycle in each dimension 1 and 2, whereas the measures $\Uvr{4,1}(X)$ are very similar to each other for $0 \leq p \leq 0.20$; compare with the central panel in Figure \ref{fig:D41_spheres_torus}. Note that the measures $\Uvr{4,1}(X)$ exhibit an ability to `detect' the $\Sp^1$ component for $p\geq 0.3$.
	\smallskip\\
	\noindent \textbf{The second experiment.} Let $\Sp^1_E$ and $D^2$ be the unit circle and the unit disk in $\R^2$, both equipped with the Euclidean metric. We sample a subset $X \subset D^2$ consisting of $1000$ points as follows: each point is sampled uniformly from the interior of $D^2$ with probability $p$ and from its boundary $\Sp^1_E$ with probability $1-p$. We endow $X$ with the Euclidean metric and calculate $\dgm_1^\vr(X)$ and $\Uvr{4,1}(X)$ for 16 different values of $p$; see Figure \ref{fig:S1_filled}. This time, $\dgm_1^\vr(X)$ does not have a unique significant cycle as early as $p=0.01$, and all the diagrams between $p=0.20$ and $p=1$ are virtually indistinguishable. In contrast, the measures $\Uvr{4,1}(X)$ assign a lot of weight to $\Dvr{4,1}(\Sp^1_E)$ for $0 \leq p \leq 0.30$ (cf. the third panel of Figure \ref{fig:D41_S1_S2}). This can be interpreted as signaling that these measures permit detecting topological features in the present of outliers/noise.
\end{example}

\begin{figure}[ht]
	\centering
	\includegraphics[width=\linewidth]{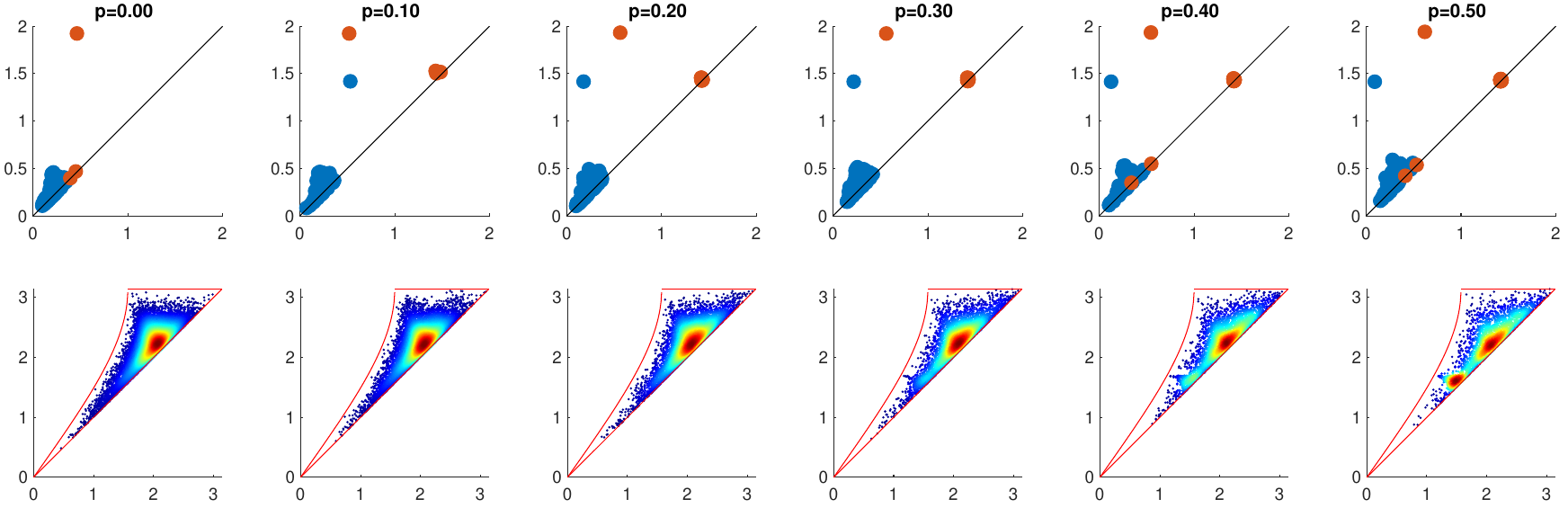}
	\caption{Given $0 \leq p \leq 1$, we sampled $X \subset \Sp^1 \vee \Sp^2$ with 1000 points so that each point is uniformly distributed in $\Sp^1$ with probability $p$ or in $\Sp^2$ with probability $1-p$. \textbf{Top:} $\dgm_1^\vr(X)$ (blue) and $\dgm_2^\vr(X)$ (orange). \textbf{Bottom:} $\Uvr{4,1}(X)$. See Example \ref{ex:sampling_defects}}
	\label{fig:S1vS2_results}
\end{figure}

\begin{example}[$\Dvr{4,1}(R)$ correlates with the size of the rectangle $R$.]
	\label{ex:D41_rectangles}
	Given $0 \leq a \leq b$ such that $a+b=1$, consider the boundary of the rectangle $R_{a,b} \subset \R^2$ with side lengths $a$ and $b$ and constant perimeter 2, and give $R_{a,b}$ the Euclidean metric. Figure \ref{fig:D41_rectangles} shows computational approximations of the persistence measures $\Unk{4,1}{\vr}(R_{a,b})$ for several values of $a$ and $b$. We sampled $10^5$ sets of 4 points uniformly at random from $R_{a,b}$. Observe that as $a$ increases, the minimal Euclidean distance from the origin to the support $\Dvr{4,1}(R_{a,b})$ of $\Unk{4,1}{\vr}(R_{a,b})$ increases. Also, note that the maximal persistence of points in $\Dvr{4,1}(R_{a,b})$ decreases rapidly with $a$. These two observations indicate that $\Dvr{4,1}(R_{a,b})$ is sensitive to the size of the ``hole" determined by the rectangle $R_{a,b}$.
	\begin{figure}[ht]
		\centering
		\includegraphics[width=\linewidth]{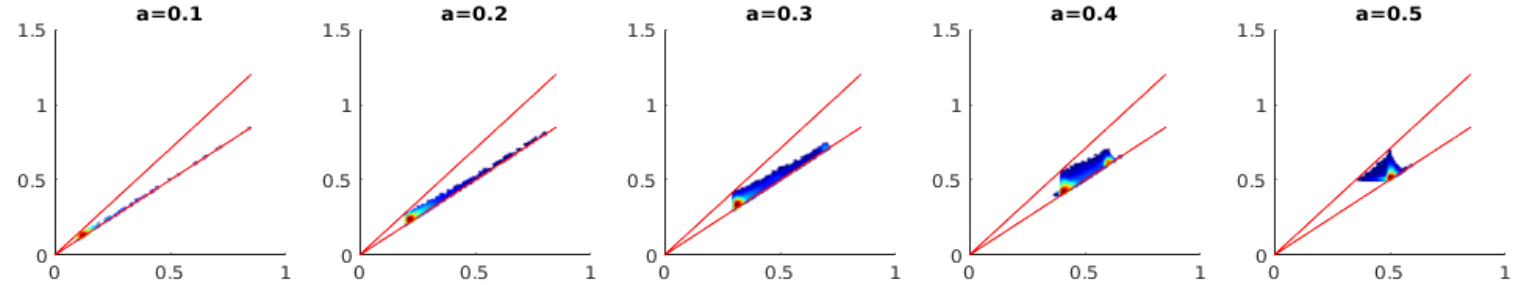}
		\caption{The persistence measures $\Unk{4,1}{\vr}(R_{a,b})$ of rectangles with side lengths $a, b$ such that $a+b=1$. The lines shown in red are the diagonal $t_d=t_b$ and the upper bound $t_d = \sqrt{2} t_b$ given by Proposition \ref{prop:D41_R2_S2}. These graphs were generated by sampling 4 points uniformly at random from each rectangle $10^5$ times. The percentage of samples that produced a non-diagonal point in each graph are, from left to right, $0.17\ \%$, $1.17\ \%$, $3.15\ \%$, $6.01\ \%$, and $9.12\ \%$.}
		\label{fig:D41_rectangles}
	\end{figure}
\end{example}

\begin{landscape}
	\begin{figure}
		\centering
		\includegraphics[width=0.85\linewidth]{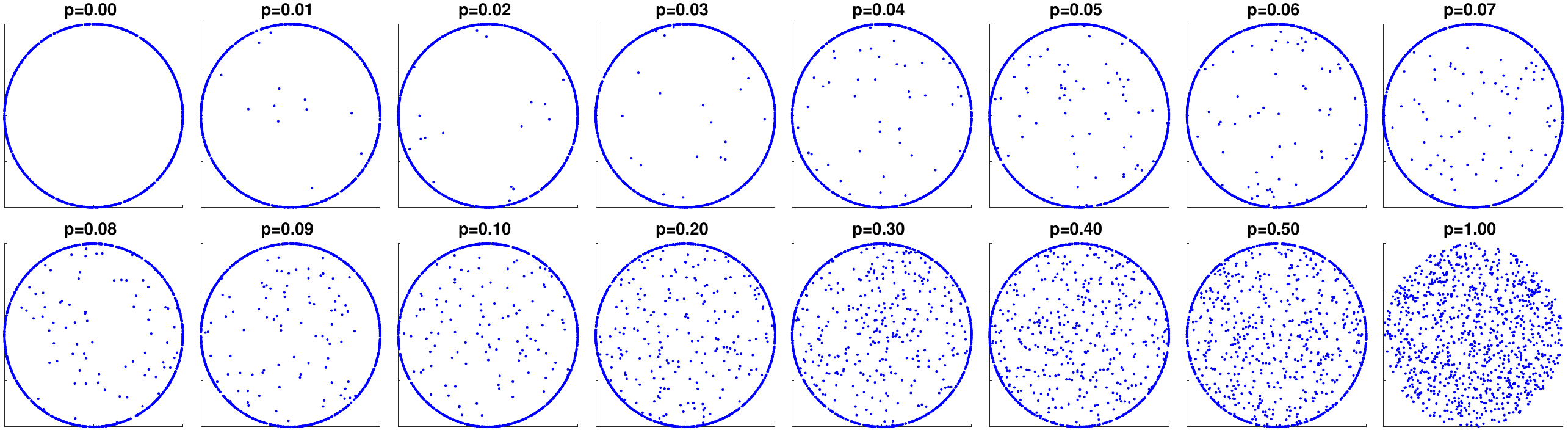}\\
		\includegraphics[width=0.85\linewidth]{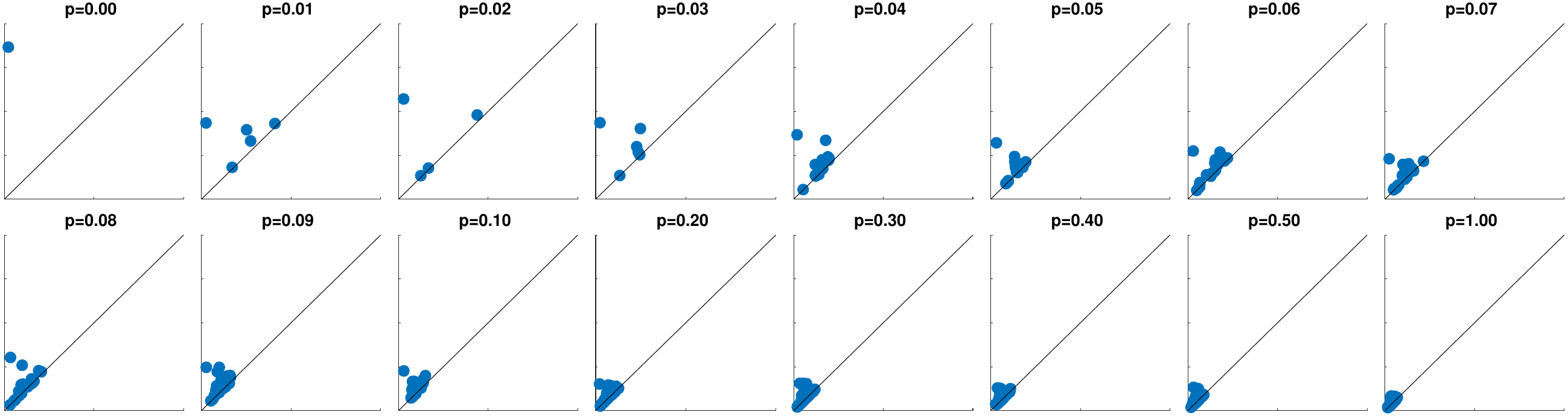}\\
		\includegraphics[width=0.85\linewidth]{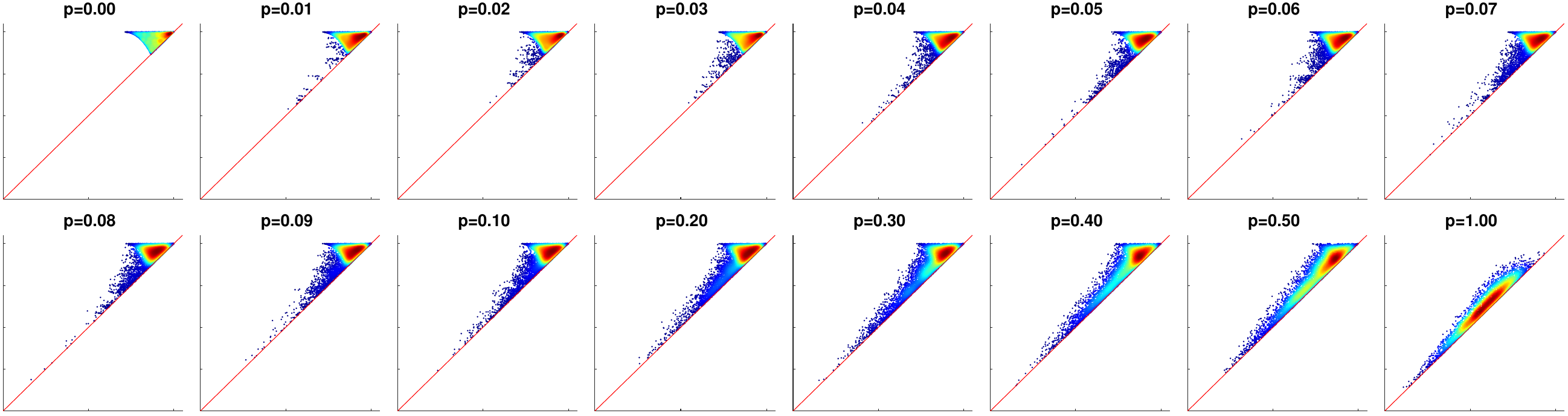}
		\caption{Given $0 \leq p \leq 1$, we sampled $X \subset D^2 \subset \R^2$ with 1000 points so that each point is uniformly distributed in the interior of $D^2$ with probability $p$ or on its boundary $\Sp^1_E$ with probability $1-p$. \textbf{Top:} Examples of $X$ for several values of $p$. \textbf{Middle:} $\dgm_1^\vr(X)$. \textbf{Bottom:} $\Uvr{4,1}(X)$.}
		\label{fig:S1_filled}
	\end{figure}
\end{landscape}

\subsubsection{Performance in a pose-invariant shape classification task}\label{sec:classification}

To test the discriminative power of persistence sets, we performed a classification experiment similar to the one outlined in \cite{ccsg09a}. In this experiment one has a database with multiple classes, each containing several \emph{poses} of the same shape, and the goal is to correctly classify the database in a way such that two different poses in the same class are clustered together whereas any two poses in two different classes are well discriminated; see Figure \ref{fig:collage}.\\
\indent We used a subset of the database from\footnote{As of writing this paper, the database is hosted at \url{https://people.csail.mit.edu/sumner/research/deftransfer/}.} \cite{sumner-paper} consisting of 62 shapes from six different classes: \texttt{camel}, \texttt{cat}, \texttt{elephant}, \texttt{face}, \texttt{head}, and \texttt{horse}. Each class has either 10 or 11 poses of the same shape. A pose is encoded with a mesh $(V_i,T_i)$ ($i=1, \dots, 62$) which consists of a set of vertices $V_i \subset \R^3$ and a set of triangles $T_i \subset V_i^3$. Let $G_i = (V_i, E_i)$ be the 1-skeleton of $(V_i, T_i)$ with an edge $\{p,q\} \in E_i$ weighted by the Euclidean distance $\|p-q\|$. Let $d_{G_i}$ be the shortest path distance on $G_i$. \emph{Note that this implies that any two poses of the same class will be (nearly) isometric}. We first select a subset $X_i \subset G_i$ of 4,000 points using farthest point sampling, that is, we start with a random initial point $p_1 \in V_i$, and at each step, we choose $p_{t+1} \in V_i$ to be any point that maximizes the Euclidean distance to $\{p_1, \dots, p_t\}$. The metric on $X_i \subset G_i$ is the restriction of $d_{G_i}$ and we denote it by $d_i$. We normalized each $X_i$ to have diameter 1 and endowed it with the uniform probability measure $\gamma_i$ to obtain an mm-space $(X_i, d_i, \gamma_i)$ representing the $i$th shape.\\
\indent For each shape, we computed the full $\Dvr{2,0}(X_i)$ (which consists of $\binom{4000}{2}+4000 \approx  10^7$ points counting repetitions), together with approximations to $\Dvr{4,1}(X_i)$ and $\Dvr{6,2}(X_i)$ with, respectively, $N=10^6$ and $N=10^7$ samples chosen uniformly at random (see the description on page \pageref{pg:approx}). These samples also induce an approximation of $\Uvr{2k+2,k}(X_i)$ given by the empirical measure $\gamma_{i,k}$ for $k=1,2$. For $k=0$, we have the exact measure $\Uvr{2,0}(X_i)$.\\

\begin{figure}[h]
\centering
\includegraphics[width=\linewidth]{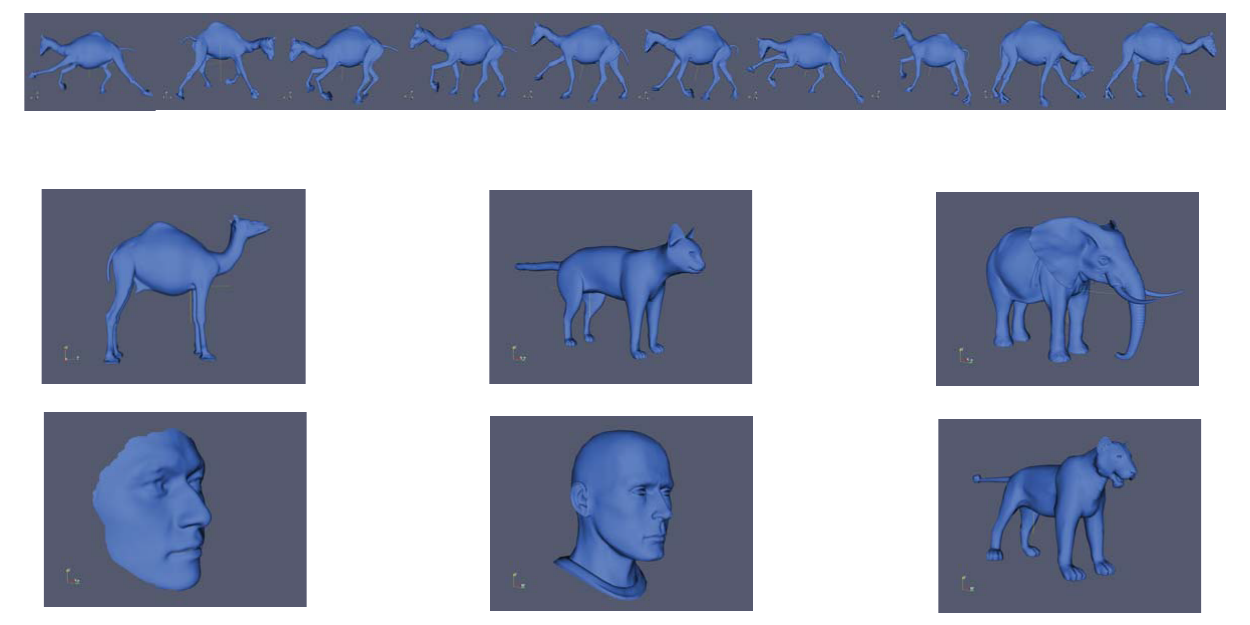}
\caption{Exemplar shapes from the the database of 3D shapes we used in our shape classification task. Note that different poses of the camel shape are nearly isometric when regarded as metric spaces (endowed with their geodesic distances). See Section \ref{sec:classification} for details.}
\label{fig:collage}
\end{figure}

\paragraph{Coarsening of $\Uvr{2k+2,k}(X_i)$.}
Ideally, we would like to compute the Wasserstein distance on $\Dop$ between the different $\gamma_{i,k}$s. However, this calculation would require finding the bottleneck distance between every pair of points in the product of the supports of $\gamma_{i,k}$ and $\gamma_{j,k}$. This cost matrix would be unmanageable (in the sense that its size would be at least $10^7 \times 10^7$ when $k=0$ and $10^6\times 10^6$ when $k=1,2$), so we replace $\gamma_{i,k}$ with a \emph{coarsened} measure $\gamma_{i,k}^c$ defined via a Voronoi partition on $\Dop$ as follows. Choose a set of landmark points $L := \{p_{1}, \dots, p_{\ell}\} \subset \Dop$. For every $t=1, \dots, \ell$, let $\gamma_{i,k}^c(p_t)$ be the sum of $\gamma_{i,k}(p)$ over all points $p \in \Dvr{2k+2,k}(X_i)$ that are closer to $p_t$ than to any other $p' \in L$. For $k=0$, $L$ consists of 850 points spaced uniformly on the line $\{0\} \times [0,1]$. For $k=1$ and $k=2$, we constructed a grid of uniformly spaced points in $[0,1] \times [0,1]$, and retained the origin and the points that were strictly above the diagonal; the final landmark set $L$ had 947 points.

\paragraph{The pairwise distance matrices arising from persistence sets and measures.}
We first  computed 8 distance matrices of size 62-by-62, where the $(i,j)$ entry of each matrix is the (Hausdorff or Wasserstein) distance between a certain invariant of $X_i$ and $X_j$ as we describe next:
\begin{itemize}
\item For each $k=0,1,2$, we computed the Hausdorff distance between the persistence sets $\Dvr{2k+2,k}(X_i)$ for $i=1, \dots, 62$ as subsets of $(\Dop, \dB)$ (recall Definition \ref{def:PPS_aggregate}). We denote these matrices as $\mathcal{H}_{0}$, $\mathcal{H}_1$ and $\mathcal{H}_2$.
\item The next three matrices are given by the 1-Wasserstein distance between the coarsened measures $\gamma_{i,k}^c$. We denote them as $\mathcal{W}_k$. We used the Matlab \texttt{mex} interface from \cite{matlab-emd} for this step.
\item The last two matrices are defined by the entry-wise maxima $$\mathcal{H}_{\max} := \max_{k} \mathcal{H}_k\,\,\mbox{and}\,\, \mathcal{W}_{\max} := \max_{k} \mathcal{W}_k.$$
\end{itemize}
Observe that, since we are operating on $(\Dop, \dB)$, we used equation (\ref{eq:bottleneck_R2}) to directly find $\dB(D_1, D_2)$ for $D_1 \in \Dvr{2k+2,k}(X_i)$ and $D_2 \in \Dvr{2k+2,k}(X_j)$ instead of optimizing between all partial matchings $\varphi:D_1 \to D_2$.

\paragraph{The pairwise distance matrices coming from VR-persistence diagrams.}
We also computed the VR-persistence diagrams of subsets $X_i' \subset X_i$ with  $|X_i'|=500$ obtained from the farthest point sampling induced by the metric of $X_i$.\footnote{The size $500$ was selected to be able to run the persistence calculations for $k=2$ in a reasonable amount of time without exceeding 2GB of memory usage which is the maximum our system could handle as Figure \ref{fig:Benchmark_break} shows.} We equip $X_i'$ with the metric inherited from $X_i$ and normalize it so that it has diameter 1. Then, we computed $\dgm_k^\vr(X_i')$ for $k=0,1,2$ with a modification of C. Trailie's wrapper for Ripser \cite{ripser}. Define the matrices $\mathcal{B}_k$ by setting the $(i,j)$-entry of $\mathcal{B}_k$ to be the bottleneck distance between $\dgm_k^\vr(X_i')$ and $\dgm_k^\vr(X_j')$. As before, we define $\mathcal{B}_{\max} := \max_k \mathcal{B}_k$. We used Hera to compute the bottleneck distances \cite{hera}.

\paragraph{Classification tasks and results.} 
Let $\mathbf{M}:=\{\mathcal{H}_k, \mathcal{W}_k, \mathcal{B}_k \mid k=0,1,2\} \cup \{\mathcal{H}_{\max}, \mathcal{W}_{\max}, \mathcal{B}_{\max}\}$. For each $M \in \mathbf{M}$, we performed a 1-nearest neighbor classification task. Our training set contains one random member $R_j$ from each class ($j=1, \dots, 6$), and we assign each $X_i$ to the class of the $R_j$ that is closest to $X_i$ as given by $M$. We repeated this experiment 2000 times and computed the average classification error $P_e(M)$. The results are shown in Table \ref{tab:sumner_results}, and the heatmaps of the matrices in $\mathbf{M}$ are shown in Figures \ref{fig:sumner_hausdorff}, \ref{fig:sumner_wass}, and \ref{fig:sumner_dB}.

We make the following remarks:
\begin{itemize}
    \item Regarding the $\mathcal{B}_k$ matrices, it is interesting to note that $\mathcal{B}_2$ performed much better than $\mathcal{B}_0$ and $\mathcal{B}_1$. $\mathcal{B}_0$ can apparently separate \texttt{head} from the other classes (see Figure \ref{fig:sumner_dB}).\footnote{We attribute this to the fact that the sampling density of the \texttt{head} shapes is much lower than that of other shapes: in fact we computed the ratio $\mathrm{area}/(\#\mathrm{vertices} \cdot \diam)$ for shapes of all classes to ascertain this. See Table \ref{tab:codensity} in Appendix \ref{sec:weighted-results} for more details.} In the same vein, we believe that $\mathcal{B}_1$ can separate instances of \texttt{face} and \texttt{head} from other classes because of the ``holes'' induced by the eyes and mouth.
    \item All the metrics induced by persistence sets ($\mathcal{H}_k$ and $\mathcal{W}_k$) perform better than $\mathcal{B}_k$ for $k=0 \,\mbox{or}\, 1$, and the best classification errors obtained by each metric are $P_e(\mathcal{H}_{\max}) = 9.17\ \%$, $P_e(\mathcal{W}_1) = 9.06\ \%$, and $P_e(\mathcal{B}_2) = 5.93\ \%$. That is, the performance of the best $\mathcal{H}_k$ and $\mathcal{W}_k$ is comparable to that of $\mathcal{B}_2$. This is promising especially since the computation of the latter is particularly costly.
    \item  An important observation from Table \ref{tab:sumner_results} is that $P_e(\mathcal{W}_{\max}) = 19.28\ \%$ despite the fact that $P_e(\mathcal{W}_1) = 9.06\ \%$. The reason is that $\mathcal{W}_0$ dominates the maximum in $\mathcal{W}_{\max}$, so the discriminating power of $\mathcal{W}_1$ is obfuscated.
\end{itemize}
Appendix \ref{sec:weighted-results} contains additional results regarding this classification task.

\begin{table}
    \centering
    \begin{tabular}{c|cccc}
        $k$ & $0$ & $1$ & $2$ & $\max{}$ \\
        \hline
        $P_e(\mathcal{B}_k)$ & 45.53 \% & 41.43 \% & 5.93 \% & 10.15 \% \\
        $P_e(\mathcal{H}_k)$ & 25.43 \%  & 12.42 \% & 19.66 \% & 9.17 \% \\
        $P_e(\mathcal{W}_k)$ & 19.50 \% & 9.06 \% & 20.27 \% & 19.28 \% 
    \end{tabular}
    \caption{Average classification error $P_e(M)$ over 2000 trials for all possible choices $M\in \mathbf{M}$. See the text for details.}
    \label{tab:sumner_results}
\end{table}

\begin{figure}
    \centering

    \includegraphics[width=\textwidth]{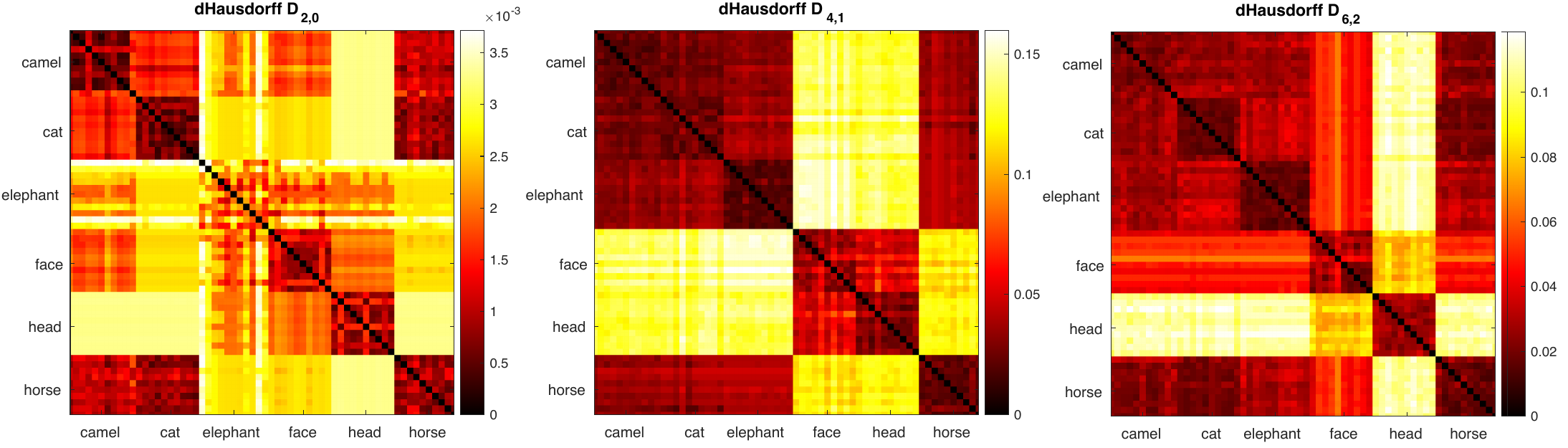}
    \includegraphics[width=\textwidth]{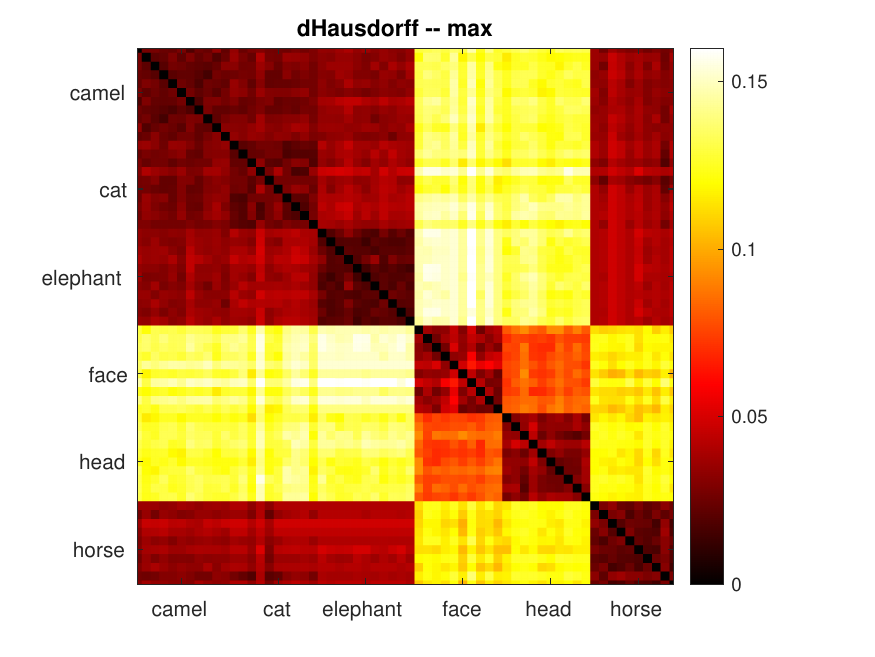}

    \caption{Heatmaps of the matrices $\mathcal{H}_0$, $\mathcal{H}_1$, $\mathcal{H}_2$, $\mathcal{H}_{\max}$. Notice that the scale of each matrix is different. Notice that $\Dvr{4,1}$ can tell apart the classes \texttt{face} and \texttt{head} from all the others. In addition, a head has a 2-dimensional cavity and a face doesn't, which suggests a reason why $\Dvr{6,2}$ can also tell those two classes apart.}
    \label{fig:sumner_hausdorff}
\end{figure}

\begin{figure}
    \centering

    \includegraphics[width=\textwidth]{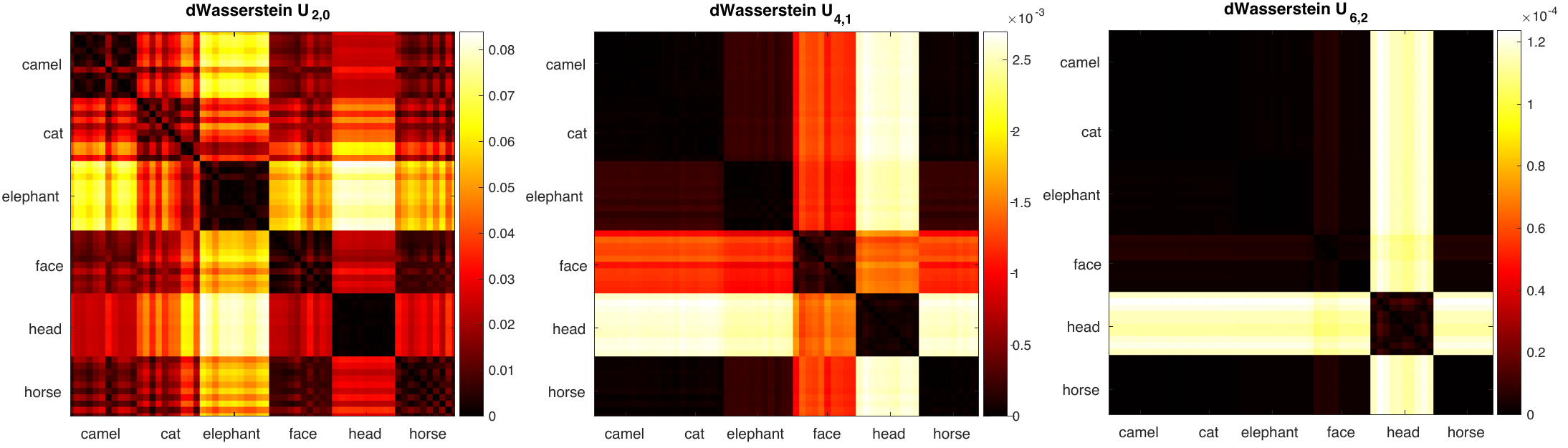}
	\includegraphics[width=\textwidth]{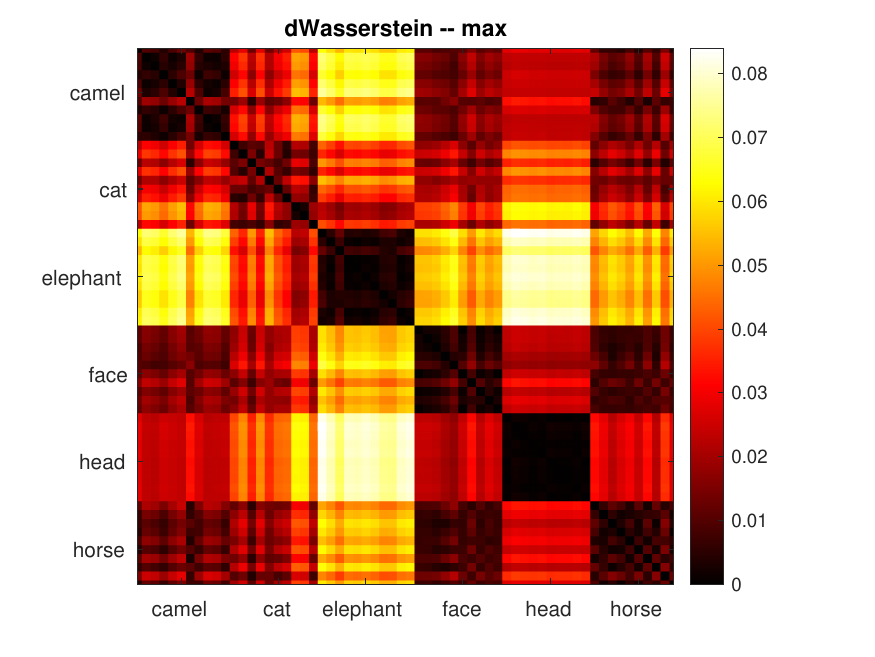}
	
    \caption{Heatmaps of the matrices $\mathcal{W}_0$, $\mathcal{W}_1$, $\mathcal{W}_2$, $\mathcal{W}_{\max}$. Notice that the scale of each matrix is different. Notice that $\Uvr{4,1}$ can tell apart the classes \texttt{face} and \texttt{head} from all the others. In addition, a head has a 2-dimensional cavity and a face doesn't, which suggests why  $\Uvr{6,2}$ can also tell those two classes apart.}
    \label{fig:sumner_wass}
\end{figure}

\begin{figure}
    \centering

    \includegraphics[width=\textwidth]{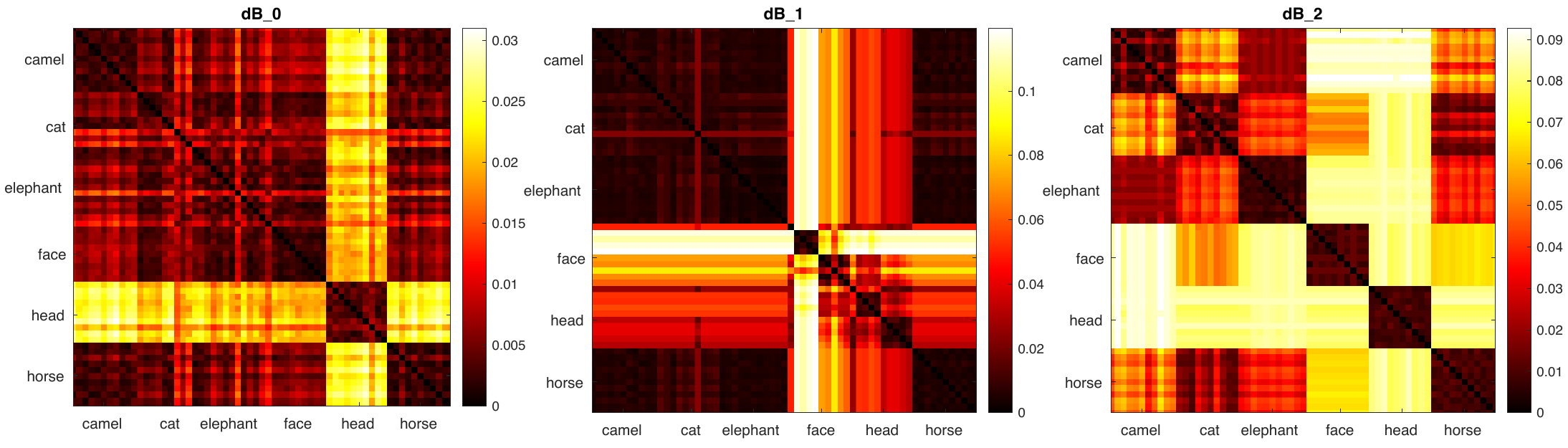}
    \includegraphics[width=0.95\textwidth]{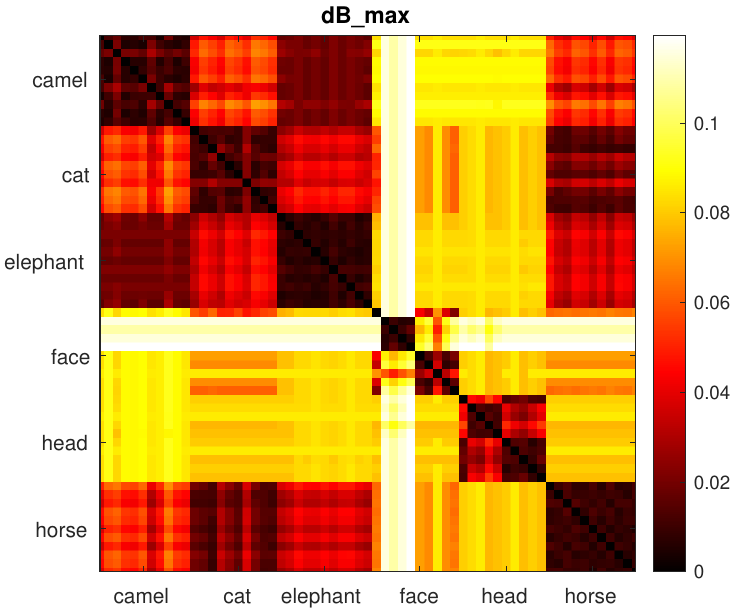}
    \caption{Heatmaps of the matrices $\mathcal{B}_0$, $\mathcal{B}_1$, $\mathcal{B}_2$, $\mathcal{B}_{\max}$. Notice that the scale of each matrix is different.}
    \label{fig:sumner_dB}
\end{figure}

\newpage
\subsection{Comparison of computational performance of VR-persistence sets and VR-persistent homology}
\label{sec:memory_comparison}

\paragraph{Time benchmarks of the geometric algorithm and VR-persistent homology.}
\indent We tested the algorithm from Section \ref{sec:bd_times_algorithm} by calculating $\dgm_k^\vr(X)$ for sets $X$ with $n_X := |X|=2k+2$ and $n_X$ ranging from $4$ to $50$. For each value of $n_X$, we used MATLAB to generate 250 sets $X \subset \R^2$ with a two-dimensional normal random variable. We then attempted to calculate $\dgm_k(X)$ in three ways: with the geometric algorithm (from Theorem \ref{thm:n=2k+2}) coded in MATLAB and in C++, and with Ripser (which is written in C++) using a MATLAB wrapper\footnote{The MATLAB wrapper was adapted from the one found in \url{https://github.com/ctralie/Math412S2017}.} for Ripser \cite{ripser} developed by C. Tralie. The results are given in the boxplot in Figure \ref{fig:timing_all}. Ripser was unable to compute $\dgm_k^\vr(X)$ for $n_X>28$, while neither version of the geometric algorithm had issues. C. Trailie's wrapper calls Ripser inside MATLAB with the \texttt{system} command, and we adopted that approach to run our C++ implementation of the geometric algorithm \cite{github-repo}. The time was measured with the \texttt{tic}, \texttt{toc} functions.\\
\indent It must be noted that both C++ executables required that we write the distance matrices to disk before running the programs. In contrast, MATLAB can run the geometric algorithm with the distance matrix loaded in memory, and this explains why MATLAB outperformed the other two programs. This observation has implications for the implementation of persistence sets. Principal persistence sets (i.e. when $n=2k+2$) can be calculated in any programming language without significant overhead after implementing the geometric algorithm. Similarly, the computation of non-principal persistence sets could be integrated into existing software for persistent homology in order to avoid the costly I/O operations described above. The tests in this section were performed in a Dell Precision 7540 Laptop with an Intel Core i7-9850H CPU and 8GB of RAM, running Fedora 35 and gcc version 11.3.1.
\begin{figure}
	\centering
	\includegraphics[width=\textwidth]{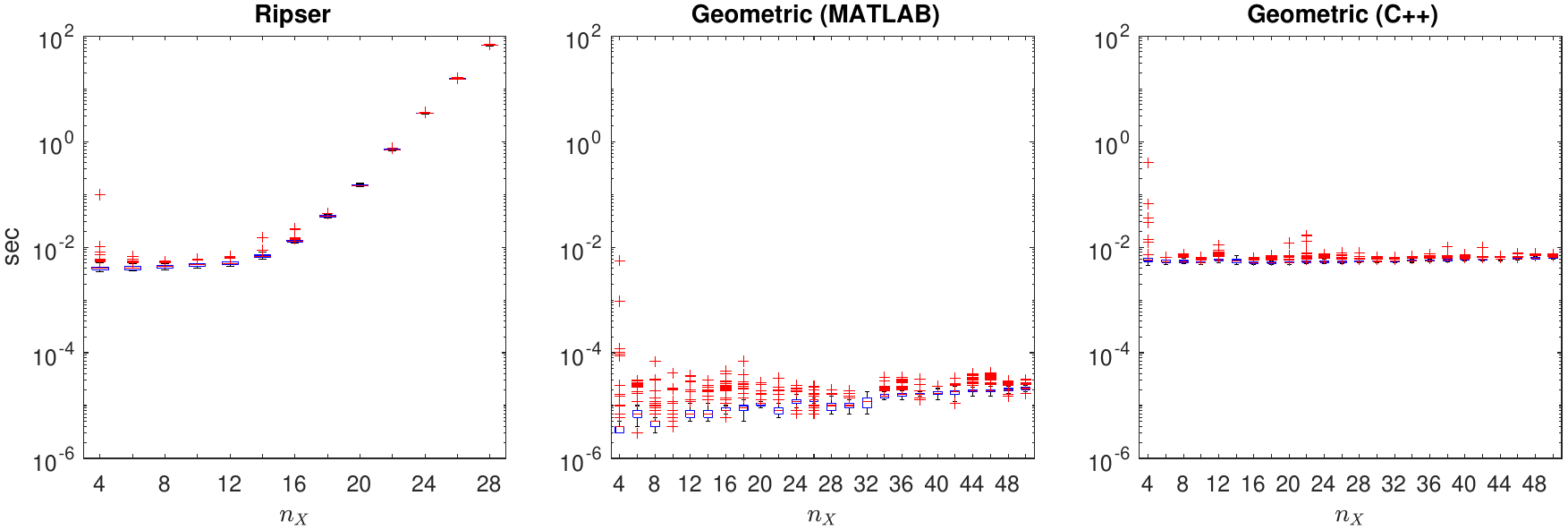}
	\caption{The time required by three algorithms to compute $\dgm_k^\vr(X)$ for a space $X$ with $n_X = 2k+2$ points. We did 250 repetitions for each value of $n_X$. Because of an unmanageable number of simplices to compute, Ripser could not finish the calculations past $n_X=28$.}
	\label{fig:timing_all}
\end{figure}
\paragraph{Benchmarks of VR-persistence sets and VR-persistent homology.}
\indent At this point, it is important to emphasize that we view persistence sets as a family of invariants that complements the standard persistent homology pipeline. Persistence sets are, in many cases, efficiently computable both in terms of their complexity and approximability and, importantly, in terms of memory requirements. These differ from the properties of standard persistence invariants, which require a lot of memory. To illustrate this point, we sampled a collection of sets $X$ uniformly at random from the sphere $\Sphere{3}$ with $n_X=|X|$ ranging from $100$ to $1000$ in increments of $100$. We attempted to calculate persistent homology in dimension $k$ with Ripser and an approximation to $\Dvr{2k+2,k}(X)$ with $N=10^6$. We used a C++ executable that implements the algorithm from Section \ref{sec:bd_times_algorithm} for the latter computation. We used $k=1,2,3$ in both experiments. The computation of $\dgm_k^\vr(X)$ failed to finish beyond $n_X = 500$ when $k=2$ and $n_X = 100$ when $k=3$. See Figure \ref{fig:Benchmark_break}. We measured ellapsed time and consumed memory with the \texttt{/usr/time -v} command. The tests in this section were performed in a Dell Precision 7540 Laptop with an Intel Core i7-9850H CPU and 8GB of RAM, running Fedora 35 and gcc version 11.3.1.
\begin{figure}
	\begin{minipage}{0.48\linewidth}
		\centering
		\includegraphics[scale=0.55]{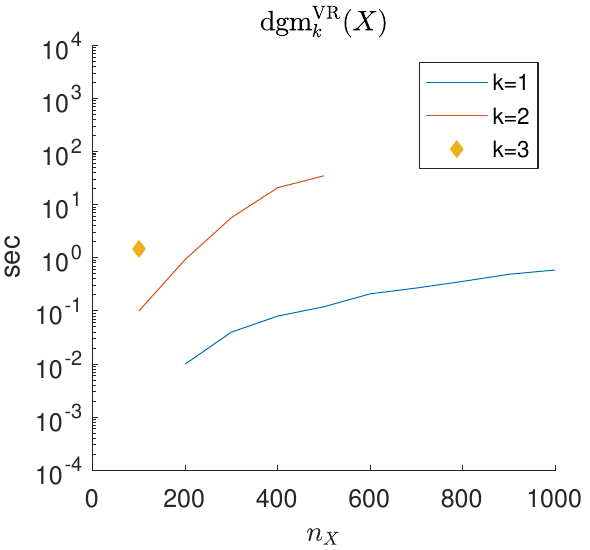}\\
		\includegraphics[scale=0.55]{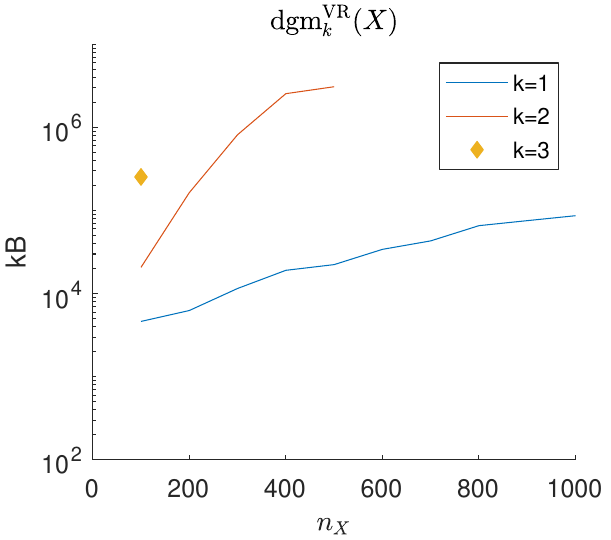}
	\end{minipage}
	\hfill
	\begin{minipage}{0.48\linewidth}
		\centering
		\includegraphics[scale=0.55]{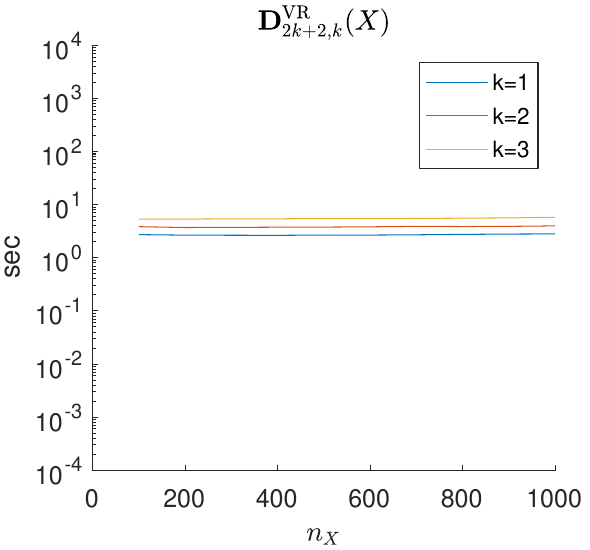} \\
		\includegraphics[scale=0.55]{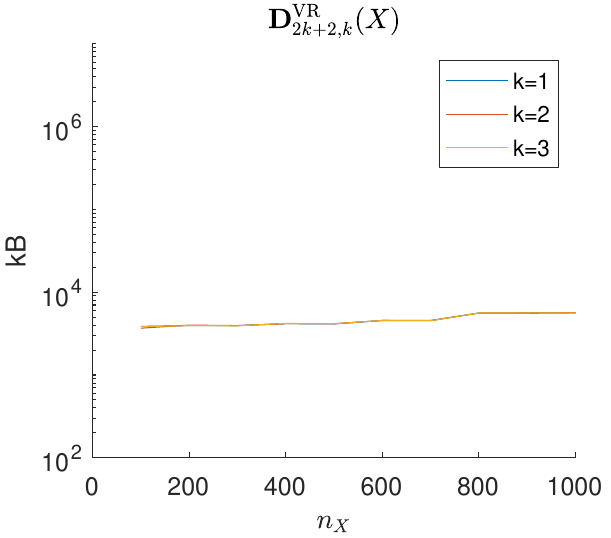}
	\end{minipage}
	\caption{\textbf{Left column:} The time (top) and memory (bottom) required to compute $\dgm_k^\vr(X)$ for a space with $n_X$ points. Ripser ran out of memory and could not complete the calculations for $n_X>500$ and $k=2$, and for $n_X>100$ and $k=3$. This is why the graph for $k=3$ has a single point rather than a line. \textbf{Right column:} The time (top) and memory (bottom) required to approximate $\Dvr{2k+2,k}(X)$ with $10^6$ samples for a space with $n_X$ points.}
	\label{fig:Benchmark_break}
\end{figure}

\paragraph{The impact of parallelization.}
We also compared the running time of persistent homology and principal persistence sets when the computation of the latter is parallelized. We selected 15 random shapes $X_{i_j}$ from the database in Section \ref{sec:classification} and computed approximations to $\Dvr{4,1}(X_{i_j})$ and $\Dvr{6,2}(X_{i_j})$ with $10^6$ and $10^7$ samples, respectively, and $\dgm_2^\vr(X_{i_j}')$ for a subset $X_{i_j}' \subset X_{i_j}$ with 500 points selected by farthest point sampling. The calculations were carried out in MATLAB running in a cluster computer (see below for the specs). We timed the computation of each $\Dvr{4,1}(X_{i_j})$, $\Dvr{6,2}(X_{i_j})$ and $\dgm_2^\vr(X_{i_j}')$ with the \texttt{tic}, \texttt{toc} functions, and we show the average running time in Figure \ref{fig:timing_sumner}. Although the computation of $\Dvr{6,2}(X_{i_j})$ with 1 core takes much longer than $\dgm_2^\vr(X_{i_j}')$, the parallelized computation of the former is in the same ballpark as the latter in terms of running time for the range of number of cores that we utilized. For example, $\Dvr{6,2}(X_{i_j})$ took less time than $\dgm_2^\vr(X_{i_j}')$ as soon as we had 4 cores, and the running time halved with 8 or more cores. In this test, we ran all calculations (both persistence sets and persistence diagrams) within the same node with the following specifications: it has a Broadwell architecture with AVX2 and runs Linux \texttt{3.10.0-1160.81.1.el7.x86\_64}. It has 22 cores available and 128 GB maximum memory. Our program used 12 cores and was allotted 44.50 GB memory, of which we used used 13.51 GB. The experiments ran with an average CPU frequency of 3.35 GHz.

\begin{figure}[h]
	\centering
	\includegraphics[scale=0.55]{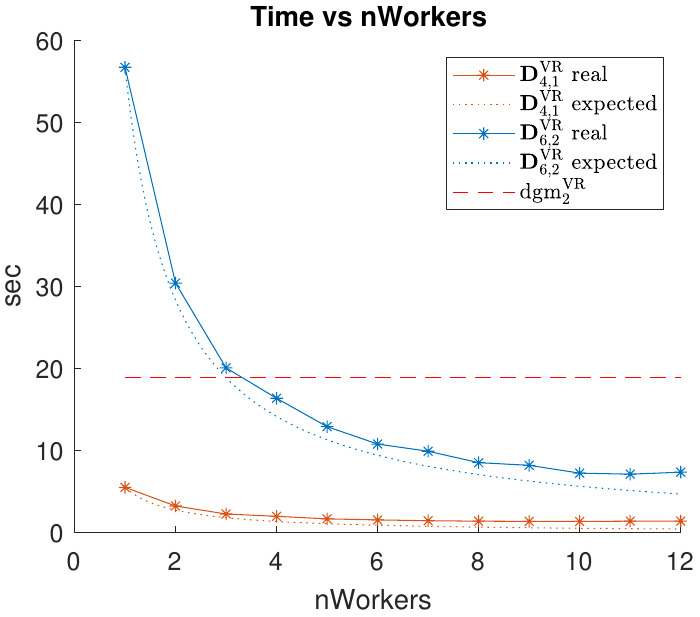}
	\caption{Running time of $\Dvr{2k+2,k}(X_i)$ for $k=1,2$ in parallel with a variable number of workers (nWorkers). The dashed line is the running time of $\dgm_2(X_i')$ sequentially. The dotted lines are the theoretical speedup guaranteed by Amdahl's law.}
	\label{fig:timing_sumner}
\end{figure}
\section{Vietoris-Rips Persistence sets of spheres}
In this section, we will describe the principal persistence sets $\Dvr{2k+2,k}(\Sphere{1})$ for all $k \geq 0$. After that, we will take advantage of functoriality to find some of the persistence sets of the higher dimensional spheres $\Sphere{m}$, $m \geq 2$, and describe the limitations (if any) to obtain higher principal persistence sets. We begin with a general technical lemma.

\begin{lemma}\label{lemma:persistence_bounds}
	Let $k \geq 1$ and $n=2k+2$. Let $(X, d_X)$ be a metric space with $n$ points. Then:
	\begin{enumerate}
	    \item\label{item:td_basic_bound} $\td{X} \leq 2\tb{X}$, and equality holds if and only if $\vdeath$ is well defined and, for every $i=1,\dots,n$, $d_X(x_i,\vdeath(x_i))=\td{X}$ and $d_X(x_i,x)=\tb{X}$ for every $x \neq \vdeath(x_i)$.
	    \item\label{item:separation_bound} $\pers(\dgm_k^\vr(X)) = \td{X}-\tb{X} \leq \sep(X)$.
	    \item\label{item:not_interval} If $X$ can be isometrically embedded on an interval, then $\tb{X} \geq \td{X}$.
	\end{enumerate}
\end{lemma}
\begin{proof}
	We prove the 3 claims in order.\\
	\indent \ref{item:td_basic_bound}. If $\tb{X} \geq \td{X}$, then $\pers(\dgm_k^\vr(X))=0$ and items \ref{item:td_basic_bound} and \ref{item:separation_bound} are trivially true. Suppose, then, $\tb{X} < \td{X}$. Choose any $x_0,x \in X$ such that $x \neq x_0,\vdeath(x_0)$. By definition of $\vdeath(x_0)$, we have $d_X(x_0, x) \leq \tb{x_0}$ and $d_X(x,\vdeath(x_0)) \leq \tb{\vdeath(x_0)}$. Then
	\begin{equation}
		\label{ineq:persistence_bounds}
		d_X(x_0,x)
		\geq d_X(x_0,\vdeath(x_0)) - d_X(x,\vdeath(x_0))
		\geq \td{x_0} - \tb{\vdeath(x_0)}
		\geq \td{X}-\tb{X}.
	\end{equation}
	Since $d_X(x_0,x) \leq \tb{X}$, we get the bound $\td{X} \leq 2\tb{X}$. If $\td{X} = 2\tb{X}$, then every intermediate inequality holds; in particular, we have $d_X(x_0,x) = \tb{X}$ and $d_X(x_0, \vdeath{x_0}) = \td{X}$.\\
	\indent \ref{item:separation_bound}. The finer bound $\sep(X) \geq \td{X}-\tb{X} = \pers(\dgm_k^\vr(X))$ follows by taking the minimum of $d_X(x_0,x)$ over $x_0$ and $x$ in inequality (\ref{ineq:persistence_bounds}).\\
	\indent \ref{item:not_interval}. Suppose, without loss of generality, that $X \subset \R$ and that $x_1 < x_2 < \cdots < x_n$. Notice that $\td{x_k} = \max(x_k-x_1, x_n-x_k)$ and, in particular, $\td{x_1} = \td{x_n} = x_n-x_1$. If $k \neq 1,n$, then $\tb{x_1} \geq x_k-x_1$ and $\tb{x_n} \geq x_n-x_k$. Then
	\begin{equation*}
		\tb{X} \geq \max(\tb{x_1}, \tb{x_n}) \geq \max(x_k-x_1, x_n-x_k) = \td{x_k} \geq \td{X}.
	\end{equation*}
\end{proof}

\subsection{Characterization of $\tb{X}$ and $\td{X}$ for $X \subset \Sphere{1}$}
Now we focus on subsets of the circle. We refer to a set $X=\{x_1, x_2, \dots, x_n\} \subset \Sphere{1}$ as a \emph{configuration} of $n$ points in $\Sphere{1}$.
\begin{defn}\label{def:circle_and_order}
	Let $\Sphere{1}$ be the quotient  $[0,2\pi]/0 \sim 2\pi$ equipped with the geodesic distance, \textit{i.e.}
	\begin{equation*}
		d_{\Sphere{1}}(x, y) := \min(|x-y|, 2\pi-|x-y|),
	\end{equation*}
	for $x,y \in \Sphere{1}$. Also, we adopt the \define{cyclic order} $\prec$ on $\Sphere{1}$ from \cite{aa17}. We refer to the increasing direction in $[0,2\pi]$ as counter-clockwise, and define $x \prec y \prec z$ to mean that the counter-clockwise path starting at $x$ meets $y$ before reaching $z$. We also use $\preceq$ to allow the points to be equal.
\end{defn}
\indent Throughout this section, $k \geq 1$ and $n=2k+2$ will be fixed. Addition of indices is done modulo $n$. Let $X=\{x_1, x_2, \dots, x_n\} \subset \Sphere{1}$ such that $x_i \prec x_{i+1} \prec x_{i+2}$ for all $i$. Write $d_{ij}=d_{\Sphere{1}}(x_i,x_j)$ for the distances, and assume $\tb{X} < \td{X}$.

\begin{figure}
	\centering
	\includegraphics[scale=0.75]{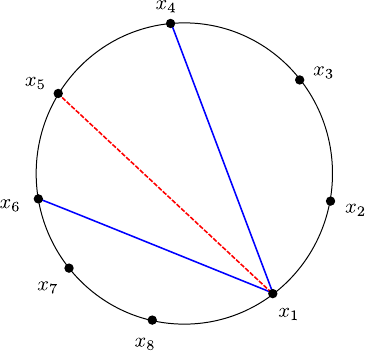}
	\caption{This configuration shows the edges that realize $\tb{x_1} = \max(d_{1,1+3}, d_{1,1-3})$ and $\td{x_1}=d_{1,1+3+1}$ when $k=3$ and $n=8$. The shortest path between $x_1$ and $x_5$ contains $x_8,x_7,x_6$, so when $r>d_{15}$, $\vrcomp_{r}(X)$ will contain a 4-simplex. These ideas were inspired by \cite{katz91}.}
	\label{fig:basic_S1_configuration}
\end{figure}

\begin{lemma}\label{lemma:technicals_S1}
   	Let $X=\{x_1, x_2, \dots, x_n\} \subset \Sphere{1}$ be such that $x_{i-1} \prec x_{i} \prec x_{i+1}$. Then:
    \begin{enumerate}[leftmargin=*]
        \item\label{item:bd_times_circle_points} For every $i$, $\tb{x_i} = \max(d_{i,i+k}, d_{i,i-k})$ and $\td{x_i}=d_{i,i+k+1}$.
        \item\label{item:bd_times_circle}
        $\tb{X} = \max_{i=1,...,n} d_{i,i+k}$ and $\td{X} = \min_{i=1,\dots,n} d_{i,i+k+1}.$
        \item\label{item:tb_splits} For every $i$, $d_{i,i+k}=d_{i,i+1}+d_{i+1,i+2}+\cdots+d_{i+k-1,i+k}.$
        \item\label{item:min_tb} $\tb{X} \geq \frac{k}{k+1} \pi$.
    \end{enumerate}
\end{lemma}
\begin{proof}
	\indent\ref{item:bd_times_circle_points} Let $r \in [\tb{X},\td{X})$. By Proposition \ref{prop:cross_polytopes}, $\vrcomp_{r}(X)$ is a cross-polytope with $n$ points. In particular, $\vrcomp_{r}(X)$ contains no simplices of dimension $k+1$. We claim that this forces $\td{x_i} = d_{i,i+k+1}$ for all $i$. Indeed, the shortest path between $x_i$ and $x_{i+k+1}$ contains either the set $\{x_{i+1},\dots,x_{i+k-1}\}$ or the set $\{x_{i+k+2},\dots,x_{i-1}\}$ (see Figure \ref{fig:basic_S1_configuration}). For any $x_j$ in that shortest path, $d_{i,j} \leq d_{i,i+k+1}$, so if we had $d_{i,i+k+1} \leq r$, $\vrcomp_{r}(X)$ would contain a $k+1$ simplex, either $[x_i,x_{i+1}, \dots, x_{i+k+1}]$ or $[x_{i+k+1},x_{i+k+2}, \dots, x_{i}]$. Thus, $r< d_{i,i+k+1}$ for all $i$.\\
	\indent In particular, $\vrcomp_r(X)$ doesn't contain the edge $[x_i,x_{i+k+1}]$. According to Definition \ref{def:cross_poly}, cross-polytopes contain all edges incident on a fixed point $x_i$ except one, so $[x_i,x_j] \in \vrcomp_{r}(X)$ for all $j \neq i+k+1$. As a consequence, $d_{i,j} \leq r < d_{i,i+k+1}$ for all $j \neq i+k+1$, so $\td{x_i} = d_{i,i+k+1}$ and $\tb{x_i} = \max_{j \neq i+k+1} d_{i,j}$. Additionally, the shortest path between $x_i$ and $x_{i+k}$ contains the set $\{x_{i+1},\dots,x_{i+k-1}\}$ rather than $\{x_{i+k+1}, \dots, x_{i-1}\}$, so $d_{i,i+j} \leq d_{i,i+k}$ for $j=1,\dots,k-1$ (otherwise, $\vrcomp_{r}(X)$ would contain the $k+2$ simplex $[x_{i+k}, x_{i+k+1}, \dots, x_i]$). The analogous statement $d_{i,i-j} \leq d_{i,i-k}$ holds for $j=1,2,\dots,k-1$. Thus, $\tb{x_i} = \max(d_{i,i+k}, d_{i,i-k})$.\\
	\indent \ref{item:bd_times_circle}. These equations follow by taking the maximum (resp. minimum) over all $i$ of the above expression for $\tb{x_i}$ (resp. $\td{x_i}$), as per Definition \ref{def:tb_td}.\\
	\indent \ref{item:tb_splits}. As we saw in the proof of item \ref{item:bd_times_circle_points}, the shortest path from $x_i$ to $x_{i+k}$ contains the set $\{x_{i+1},\dots,x_{i+k-1}\}$. The length of this path is $d_{i,i+k} = d_{i,i+1} + \cdots + d_{i+k-1,i+k}$.\\
	\indent \ref{item:min_tb}. By items \ref{item:bd_times_circle} and \ref{item:tb_splits}, $n \tb{X} \geq \sum_{i=1}^{n} d_{i,i+k} = \sum_{i=1}^{n} \sum_{j=1}^{k} d_{i+j-1,i+j} = \sum_{j=1}^{k} \sum_{i=1}^{n} d_{i+j-1,i+j} = k \cdot 2\pi$. Thus, $\tb{X} \geq \frac{2k}{n}\pi = \frac{k}{k+1}\pi$.\\
\end{proof}

\subsection{Characterization of $\Dvr{2k+2,k}(\Sphere{1})$ for $k$ even}
As a followup to Lemma \ref{lemma:technicals_S1} item \ref{item:min_tb}, we show that for every pair of values $t_b,t_d$ with $\frac{k}{k+1}\pi \leq t_b < t_d \leq \pi$, there exists $X \subset \Sphere{1}$ with $|X|=2k+2$ such that $\tb{X}=t_b$ and $\td{X}=t_d$.
\begin{theorem}\label{thm:critical_tb_even}
	For even $k$, $\Dvr{2k+2,k}(\Sphere{1}) = \left\{ (t_b,t_d): \dfrac{k}{k+1}\pi \leq t_b < t_d \leq \pi \right\}$.
\end{theorem}
\begin{proof}
	We will first construct what we call the critical configurations, those where $\tb{X} = \frac{k}{k+1}\pi$ and $\td{X}=t_d \in (\tb{X},\pi]$. Consider the points
	\begin{equation*}
		x_{i} =
		\begin{cases}
			\frac{\pi}{k+1} \cdot (i-1),			& i \text{ odd}\\
			\frac{\pi}{k+1} \cdot (i-1)-(\pi-t_d),	& i \text{ even},
		\end{cases}
	\end{equation*}
	for $i=1,\dots,n$. When $i$ is odd, $x_{i-1} < x_i$. If $i$ is even, by Lemma \ref{lemma:technicals_S1} item \ref{item:min_tb}, we have $x_i-x_{i-1} = -\frac{k\pi}{k+1}+t_d > -\frac{k\pi}{k+1}+t_b \geq 0$. Thus, $0 = x_1 < x_2 < \cdots < x_n$. Additionally, since $t_d \leq \diam(\Sphere{1})$, we have $x_{2k+2} = \frac{k\pi}{k+1} +t_d \leq \frac{(2k+1) \pi}{k+1} < 2\pi$, so we have $x_{i} \prec x_{i+1} \prec x_{i+2}$ for all $i$.\\
	\indent Since $k$ is even, $i$ and $i+k$ have the same parity, so if $1 \leq i \leq k+2$,
	\begin{equation}
		\label{eq:critical_tb_even}
		\textstyle
		x_{i+k}-x_i = \frac{\pi}{k+1}[(i+k-1)-(i-1)] = \frac{k}{k+1} \pi.
	\end{equation}
	If $k+3 \leq i \leq 2k+2$, $x_{i+k} = x_{i-k-2}$, and the last equation gives $|x_{i+k}-x_{i}| = x_{i} - x_{i-k-2} = \frac{k+2}{k+1} \pi$. Since $\frac{k}{k+1}\pi + \frac{k+2}{k+1}\pi = 2\pi$, for all $i$ we have $d_{i,i+k} = \min(|x_{i+k}-x_{i}|, 2\pi - |x_{i+k}-x_{i}|) = \min\left( \frac{k}{k+1}\pi, \frac{k+2}{k+1}\pi \right) = \frac{k}{k+1} \pi$. Thus, $\tb{X} = \max_i d_{i,i+k} = \frac{k}{k+1}\pi$. To find $\td{X} = \min_i d_{i,i+k+1}$, we have two cases depending on the parity of $i$. If $i \leq k+1$ is odd (and $i+k+1 \leq 2k+2$ even),
	\begin{equation}
		\label{eq:critical_td_even_1}
		\textstyle
		|x_{i+k+1}-x_{i}| = \frac{\pi}{k+1} [(i+k)-(i-1)] -(\pi-t_d) = t_d,
	\end{equation}
	and if $i \leq k+1$ is even,
	\begin{equation}
		\label{eq:critical_td_even_2}
		\textstyle
		|x_{i+k+1}-x_{i}| = \left|\frac{\pi}{k+1} [(i+k)-(i-1)] + (\pi-t_d)\right| = 2\pi - t_d.
	\end{equation}
	Since $d_{i,i+k+1} = \min(|x_{i+k+1}-x_{i}|, 2\pi - |x_{i+k+1}-x_{i}|)$, the above equations imply $d_{i,i+k+1} = t_d$ irrespective of the parity of $i$. If $i>k+1$, the index $i+k+1$ equals $i-k-1$ modulo $n$, and we have $1 \leq i-k-1 \leq k+1$. Hence, the paragraph above gives $d_{i,i+k+1} = d_{i-k-1,i} = t_d$. All in all, $\td{X} = \min d_{i,i+k+1} = t_d$.\\
	\indent Lastly, we can use these critical configurations to construct $X'$ such that $\tb{X'} = t_b > \frac{k}{k+1}\pi$. Let $\varepsilon:=t_b-\frac{k}{k+1}\pi > 0$. Define $x'_{1} := x_{1} + \varepsilon$, $x'_{k+2} := x_{k+2} + \varepsilon$, and $x_i' := x_i$ for $i \neq 1, k+2$. Write $d_{ij}' = d_{\Sphere{1}}(x_i', x_j')$. In order to use Lemma \ref{lemma:technicals_S1} item \ref{item:bd_times_circle} to find $\tb{X'}$ and $\td{X'}$, we have to check that $x_{i}' \prec x_{i+1}' \prec x_{i+2}'$ for all $1 \leq i \leq 2k+2$. This boils down to checking $x_{2k+2}' \prec 0 \prec x_{1}' \prec x_{2}'$ and $x_{k+1}' \prec x_{k+2}' \prec x_{k+3}'$ because $x_{i}' = x_{i}$ for all $i \neq 1, k+2$. Since the points are listed in counter-clockwise order, the desired cyclic orderings hold as long as $x_{1}' < x_{2}'$ and $x_{k+2}' < x_{k+3}'$. Furthermore, these inequalities are equivalent to $\varepsilon < x_{2}-x_{1}, x_{k+3}-x_{k+2}$. In fact, $\varepsilon = t_b - \frac{k}{k+1} \pi < t_d-\frac{k}{k+1}\pi = x_{2}-x_{1}$ and, since $t_d \leq \pi$, $x_{2} - x_{1} = t_d-\frac{k}{k+1}\pi \leq \frac{k+2}{k+1}\pi - t_d = x_{k+3}-x_{k+2}$. In conclusion, $x_{i}' \prec x_{i+1}' \prec x_{i+2}'$ for all $1 \leq i \leq 2k+2$, and by Lemma \ref{lemma:technicals_S1} item \ref{item:bd_times_circle}, $\tb{X'} = \max_i d_{i,i+k}'$ and $\td{X'} = \min_i d_{i,i+k+1}'$.\\
	\indent The only distances among $d_{i,i+k}'$ and $d_{i,i+k+1}'$ that might differ from the corresponding $d_{ij}$ are those involving $x_{1}'$ and $x_{k+2}'$, namely $d_{1,k+1}$, $d_{1-k,1} = d_{k+3,1}$, $d_{k+2,2k+2}$, $d_{2,k+2}$, and $d_{1,k+2}$. To compute the first pair of distances, the arguments following equation (\ref{eq:critical_tb_even}) give $d_{1,k+1} = x_{k+1} - x_{1}$ and $d_{k+3,1} = 2\pi - (x_{k+3} - x_{1})$. Then
	\begin{align*}
		x_{k+1}'-x_{1}' &=\textstyle x_{k+1}-x_{1}-\varepsilon = d_{1,k+1} - \varepsilon = \frac{k}{k+1}\pi - \varepsilon \text{, and }\\
		2\pi - (x_{k+3}'-x_{1}') &=\textstyle 2\pi - (x_{k+3}-x_{1}) + \varepsilon = d_{k+3,1} + \varepsilon = \frac{k}{k+1}\pi + \varepsilon = t_b.
	\end{align*}
	Both quantities are strictly less than $\pi$, so $d_{1,k+1}' = x_{k+1}'-x_{1}' = \frac{k}{k+1}\pi - \varepsilon$ and $d_{k+3,1}' = 2\pi - (x_{k+3}'-x_{1}') = \frac{k}{k+1}\pi + \varepsilon$. An analogous argument gives $d_{k+2,2k+2}' = \frac{k}{k+1}\pi - \varepsilon$ and $d_{2,k+2}' = \frac{k}{k+1}\pi + \varepsilon$. Lastly, since $x_{k+2}'-x_{1}' = x_{k+2}-x_{1}$, we have $d_{1,k+2}' = d_{1,k+2}$. Thus, $\td{X'} = \min d'_{i,i+k+1} = t_d$ and $\tb{X'} = \max d'_{i,i+k} = \max(\frac{k}{k+1}\pi - \varepsilon, \frac{k}{k+1}\pi, \frac{k}{k+1}\pi + \varepsilon) = \frac{k}{k+1}\pi + \varepsilon = t_b$.
\end{proof}

\begin{figure}[ht]
	\begin{minipage}{0.48\textwidth}
		\centering
\begin{tikzpicture}[scale=1.75]
	\tikzmath{
		\pt=0.025;
		\r = 1;		\pos=2.4;
		\cax=-\pos;	\cay=0;
		\cbx=0;		\cby=0;
		\ccx=+\pos;	\ccy=0;
		\ta1=  0; \pax1=\cax+cos(\ta1); \pay1=\cay+sin(\ta1);
		\ta2= 50; \pax2=\cax+cos(\ta2); \pay2=\cay+sin(\ta2);
		\ta3=120; \pax3=\cax+cos(\ta3); \pay3=\cay+sin(\ta3);
		\ta4=170; \pax4=\cax+cos(\ta4); \pay4=\cay+sin(\ta4);
		\ta5=240; \pax5=\cax+cos(\ta5); \pay5=\cay+sin(\ta5);
		\ta6=290; \pax6=\cax+cos(\ta6); \pay6=\cay+sin(\ta6);
	}

	\draw (\cax, \cay) circle [radius = \r];

	\foreach \p in {(\pax1,\pay1), (\pax2,\pay2), (\pax3,\pay3), (\pax4,\pay4),
		(\pax5,\pay5), (\pax6,\pay6)}
	{
		\draw[fill] \p circle [radius=\pt];
	}

	\node[right]		at (\pax1,\pay1) {$x_1$};
	\node[above right]	at (\pax2,\pay2) {$x_2$};
	\node[above]		at (\pax3,\pay3) {$x_3$};
	\node[left]			at (\pax4,\pay4) {$x_4$};
	\node[below]		at (\pax5,\pay5) {$x_5$};
	\node[below]		at (\pax6,\pay6) {$x_6$};

	\draw[blue] (\pax1,\pay1) -- (\pax3,\pay3);
	\draw[blue] (\pax3,\pay3) -- (\pax5,\pay5);
	\draw[blue] (\pax5,\pay5) -- (\pax1,\pay1);
	
	\draw[blue] (\pax2,\pay2) -- (\pax4,\pay4);
	\draw[blue] (\pax4,\pay4) -- (\pax6,\pay6);
	\draw[blue] (\pax6,\pay6) -- (\pax2,\pay2);

	\draw[red, dashed] (\pax1,\pay1) -- (\pax4,\pay4) node [midway, below] {$\td{x_1}$} ;
\end{tikzpicture} 	\end{minipage}
	\hfill
	\begin{minipage}{0.48\textwidth}
		\centering
\begin{tikzpicture}[scale=1.75]
	\tikzmath{
		\pt=0.025;
		\r = 1;		\pos=2.4;
		\cax=-\pos;	\cay=0;
		\cbx=0;		\cby=0;
		\ccx=+\pos;	\ccy=0;
		\ta1=  0; \pax1=\cax+cos(\ta1); \pay1=\cay+sin(\ta1);
		\ta2= 57; \pax2=\cax+cos(\ta2); \pay2=\cay+sin(\ta2);
		\ta3= 82; \pax3=\cax+cos(\ta3); \pay3=\cay+sin(\ta3);
		\ta4=139; \pax4=\cax+cos(\ta4); \pay4=\cay+sin(\ta4);
		\ta5=196; \pax5=\cax+cos(\ta5); \pay5=\cay+sin(\ta5);
		\ta6=221; \pax6=\cax+cos(\ta6); \pay6=\cay+sin(\ta6);
		\ta7=278; \pax7=\cax+cos(\ta7); \pay7=\cay+sin(\ta7);
		\ta8=303; \pax8=\cax+cos(\ta8); \pay8=\cay+sin(\ta8);
		\mpx1=(\pax1+\pax2)/2;	\mpy1=(\pay1+\pay2)/2;
		\mpx2=(\pax2+\pax3)/2;	\mpy2=(\pay2+\pay3)/2;
		\mpx3=(\pax3+\pax4)/2;	\mpy3=(\pay3+\pay4)/2;
		\mpx4=(\pax4+\pax5)/2;	\mpy4=(\pay4+\pay5)/2;
		\mpx5=(\pax5+\pax6)/2;	\mpy5=(\pay5+\pay6)/2;
		\mpx6=(\pax6+\pax7)/2;	\mpy6=(\pay6+\pay7)/2;
		\mpx7=(\pax7+\pax8)/2;	\mpy7=(\pay7+\pay8)/2;
		\mpx8=(\pax8+\pax1)/2;	\mpy8=(\pay8+\pay1)/2;
	}

	\draw (\cax, \cay) circle [radius = \r];

	\foreach \p in {(\pax1,\pay1), (\pax2,\pay2), (\pax3,\pay3), (\pax4,\pay4),
		(\pax5,\pay5), (\pax6,\pay6), (\pax7,\pay7), (\pax8,\pay8)}
	{
		\draw[fill] \p circle [radius=\pt];
	}

	\node[right]		at (\pax1,\pay1) {\small $x_1$};
	\node[above right]	at (\pax2,\pay2) {\small $x_2$};
	\node[above]		at (\pax3,\pay3) {\small $x_3$};
	\node[above left]	at (\pax4,\pay4) {\small $x_4$};
	\node[left]			at (\pax5,\pay5) {\small $x_5$};
	\node[below left]	at (\pax6,\pay6) {\small $x_6$};
	\node[below]		at (\pax7,\pay7) {\small $x_7$};
	\node[below right]	at (\pax8,\pay8) {\small $x_8$};

	\node[ xshift=10, yshift=10] at (\mpx1, \mpy1) {\large $L$};
	\node[ xshift= 4, yshift= 6] at (\mpx2, \mpy2) {\large $s$};
	\node[ xshift=-7, yshift=12] at (\mpx3, \mpy3) {\large $L$};
	\node[xshift=-12, yshift= 4] at (\mpx4, \mpy4) {\large $L$};
	\node[ xshift=-6, yshift=-4] at (\mpx5, \mpy5) {\large $s$};
	\node[xshift=-8, yshift=-12] at (\mpx6, \mpy6) {\large $L$};
	\node[ xshift= 5, yshift=-8] at (\mpx7, \mpy7) {\large $s$};
	\node[ xshift=10,yshift=-10] at (\mpx8, \mpy8) {\large $L$};

	\draw[blue] (\pax1,\pay1) -- (\pax4,\pay4) node [midway, above] {$\tb{x_1}$} ;
	\draw[blue] (\pax1,\pay1) -- (\pax6,\pay6) node [midway, below] {$\tb{x_1}$} ;
	\draw[red, dashed] (\pax1,\pay1) -- (\pax5,\pay5) node [midway, above] {$\td{x_1}$} ;
\end{tikzpicture} 	\end{minipage}
	\caption{\textbf{Left:} Example of a critical configuration for $k=2$ as in Theorem \ref{thm:critical_tb_even}. The solid blue lines have length $\tb{X} = 2\pi/3$, while the dotted red line has length $\td{X}$. \textbf{Right:} Example of a critical configuration for $k=3$ in Theorem \ref{thm:critical_tb_odd}. Here, $\tb{X}=2L+s$ and $\td{X}=2L+2s$. Both: The sequence $x_1, x_{1+k}, x_{1+2k}, \dots$ forms a regular $(k+1)$-gon in the left image and a $(2k+2)$-gon in the right.}
	\label{fig:Dnk_S1_critical}
\end{figure}
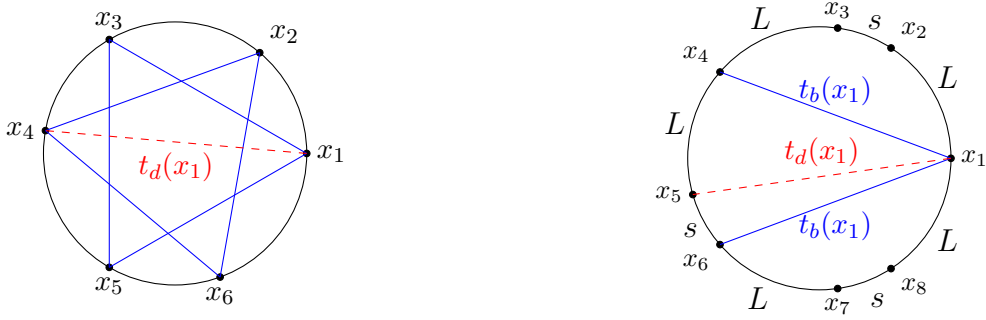

\subsection{Characterization of $\Dvr{2k+2,k}(\Sphere{1})$ for $k$ odd}
\label{sec:Dvr_S1_k_odd}
An important difference between even and odd $k$ is that only for even $k$ can we find configurations that have the minimal possible birth time $\tb{X}=\frac{k}{k+1}\pi$ given any $t_d \in (\tb{X},\pi]$. The difference is that sequences of the form $x_i,x_{i+k},x_{i+2k},\dots$ eventually reach all points when $k$ is odd, but only half of them when $k$ is even (see Figure \ref{fig:Dnk_S1_critical}). This allows us to separate $X \subset \Sphere{1}$ into two regular $(k+1)$-gons with fixed $\tb{X}$ and it still allows control on $\td{X}$, as shown in Proposition \ref{thm:critical_tb_even}. For odd $k$, we will instead use an idea from Proposition 5.4 of \cite{aa17}. We won't need the result in its full generality, so we only use part of its argument to provide a bound for $\tb{X}$ in terms of $\td{X}$.
\begin{theorem}\label{thm:critical_tb_odd}
	Let $k$ be an odd positive integer. Then $\td{X} \geq (k+1)(\pi - \tb{X})$, and this inequality is tight.
\end{theorem}
\begin{proof}
	Fix $i \in \{1,\dots,n\}$. Let $r \geq \frac{k}{k+1}\pi$ and $\delta = r - \frac{k-1}{k}\pi$. Notice that $k^2 = \frac{1}{2}(k-1) \cdot n + 1$, so the path that passes through the points $x_i, x_{i+k}, \dots, x_{i+k\cdot k}$ makes $\frac{1}{2}(k-1)$ revolutions around the circle and stops at $x_{i+k^2}=x_{i+1}$. At the same time, $d_{\ell,\ell+k} \leq \tb{X}$. These facts give:
	\begin{equation*}
		\frac{1}{2}(k-1) \cdot 2\pi + d_{i,i+1} = \sum_{j=1}^{k} d_{i+(j-1)k,i+jk} \leq k\tb{X}.
	\end{equation*}
	Thus, $(k-1)\pi + \max_{i=1,\dots,n} d_{i,i+1} \leq k\tb{X}$. However, by Lemma \ref{lemma:technicals_S1}, there exists an $\ell$ for which $d_{\ell,\ell+k+1}=\td{X}$. Let $\gamma$ be the path between $x_\ell$ and $x_{\ell+k+1}$ such that $d_{\ell,\ell+k+1} + |\gamma| = 2\pi$. Assume, without loss of generality, that $\gamma$ contains $x_{\ell+1}$. This means that $|\gamma| = d_{\ell,\ell+1} + d_{\ell+1,\ell+k+1}$, so
	\begin{equation*}
		d_{\ell,\ell+1} = |\gamma|-d_{\ell+1,\ell+k+1} = 2\pi-\td{X}-d_{\ell+1,\ell+k+1} \geq 2\pi - \td{X} - \tb{X}.
	\end{equation*}
	Thus, $k\tb{X} \geq (k-1)\pi+\max_{i=1,\dots,n} d_{i,i+1} \geq (k+1)\pi-\td{X}-\tb{X}$. Solving this inequality for $\td{X}$ gives the result.\\
    \indent In order to prove tightness, we describe the critical configurations in terms of the distances between consecutive points. Let $0 < t_b < t_d \leq \pi$ be such that $t_d = (k+1)(\pi-t_b)$. Replacing $t_d$ with the bounds $t_b$ and $\pi$ in the equation $t_d = (k+1)(\pi-t_b)$ implies $\frac{k}{k+1} \pi \leq t_b < \frac{k+1}{k+2}\pi$. Define $L := kt_b-(k-1)\pi$ and $s := -(k+2)t_b+(k+1)\pi$. Observe that the bounds $\frac{k}{k+1} \pi \leq t_b < \frac{k+1}{k+2}\pi$ imply that $0 < s \leq L$. Additionally, it can be checked that $(k+2)L+ks = 2\pi$. Let
    \begin{equation}
    	\label{eq:critical_tb_odd}
    	x_{i} :=
    	\begin{cases}
    		\left\lfloor \frac{i}{2} \right\rfloor L + \left\lfloor \frac{i-1}{2} \right\rfloor s	& 1 \leq i \leq k+1,\\
    		\left\lfloor \frac{i+1}{2} \right\rfloor L + \left\lfloor \frac{i-2}{2} \right\rfloor s		& k+2 \leq i \leq 2k+2.
    	\end{cases}
    \end{equation}
	\indent By Lemma \ref{lemma:technicals_S1} item \ref{item:bd_times_circle}, $\tb{X} = \min_i d_{i,i+k}$, so we compute the distances $d_{i,i+k} = \min(|x_{i+k}-x_{i}|, 2\pi - |x_{i+k}-x_{i}|)$. For $i=1$, since $k$ is odd, $\frac{k \pm 1}{2}$ is an integer and so
	\begin{equation*}
		\textstyle
		x_{i+k}-x_{i} = x_{k+1}-x_{1} = \left( \left\lfloor \frac{k+1}{2} \right\rfloor L + \left\lfloor \frac{k-1}{2} \right\rfloor s \right) - 0 = \frac{k+1}{2} L + \frac{k-1}{2} s.
	\end{equation*}
	If $2 \leq i \leq k+1$, we have $k+2 \leq i+k \leq 2k+1$. Also, observe that if $x-y \in \Z$, then $\lfloor x \rfloor - \lfloor y \rfloor = x-y$. Hence,
	\begin{align*}
		x_{i+k} - x_{i}
		&=\textstyle \left( \left\lfloor \frac{i+k+1}{2} \right\rfloor L + \left\lfloor \frac{i+k-2}{2} \right\rfloor s \right) - \left( \left\lfloor \frac{i}{2} \right\rfloor L + \left\lfloor \frac{i-1}{2} \right\rfloor s \right) \\
		&=\textstyle \left( \left\lfloor \frac{i+k+1}{2} \right\rfloor - \left\lfloor \frac{i}{2} \right\rfloor \right)L + \left( \left\lfloor \frac{i+k-2}{2} \right\rfloor - \left\lfloor \frac{i-1}{2} \right\rfloor \right) s
		=\textstyle \frac{k+1}{2} L + \frac{k-1}{2} s.
	\end{align*}
	For $i=k+2$,
	\begin{align*}
		x_{i+k} - x_{i} = x_{2k+2} - x_{k+2}
		&=\textstyle \left( \left\lfloor \frac{2k+3}{2} \right\rfloor L + \left\lfloor \frac{2k}{2} \right\rfloor s \right) - \left( \left\lfloor \frac{k+3}{2} \right\rfloor L + \left\lfloor \frac{k}{2} \right\rfloor s \right) \\
		&=\textstyle \left( \left\lfloor \frac{2k+3}{2} \right\rfloor - \left\lfloor \frac{k+3}{2} \right\rfloor \right)L + \left( \left\lfloor \frac{2k}{2} \right\rfloor - \left\lfloor \frac{k}{2} \right\rfloor \right) s
		=\textstyle \left( \frac{2k+2}{2} - \frac{k+3}{2} \right)L + \left( \frac{2k}{2} - \frac{k-1}{2} \right) s\\
		&=\textstyle \frac{k-1}{2} L + \frac{k+1}{2} s.
	\end{align*}
	If $k+3 \leq i \leq 2k+2$, then $i+k$ modulo $n$ is $i-k-2$. Since $|x_a-x_b| + |x_b-x_a| = 2\pi$ for any $a,b$, and $1 \leq i-k-2 \leq k$, the case above gives
	\begin{align*}
		|x_{i+k} - x_{i}|
		&= 2\pi - |x_{i} - x_{i+k}|
		= 2\pi - |x_{i} - x_{i-k-2}|\\
		&=\textstyle 2\pi - \left( \frac{k+1}{2} L + \frac{k-1}{2} s \right)
		= \frac{k+3}{2} L + \frac{k+1}{2} s.
	\end{align*}
	Also, since $(k+2)L + ks = 2\pi$, we have $\left(\frac{k+1}{2} L + \frac{k-1}{2} s \right) + \left( \frac{k+3}{2} L + \frac{k+1}{2} s \right) = 2\pi$. Thus, putting together the above calculations gives, for $i \neq k+2$,
	\begin{align}
		\label{eq:critical_tb_odd_distances}
		d_{i,i+k}
		&=\textstyle \min\left(|x_{i+k}-x_{i}|, 2\pi - |x_{i+k}-x_{i}| \right)
		=
		\begin{cases}
			|x_{i+k}-x_{i}|,		& 1 \leq i \leq k+1,\\
			2\pi - |x_{i+k}-x_{i}|	& k+3 \leq i \leq 2k+2.
		\end{cases}
	\end{align}
	In both cases we obtain $d_{i,i+k} = \frac{k+1}{2} L + \frac{k-1}{2} s$. For $i=k+2$, we have
	\begin{align*}
		d_{i,i+k}
		&=\textstyle \min\left(\frac{k-1}{2} L + \frac{k+1}{2} s, 2\pi - \frac{k-1}{2} L - \frac{k+1}{2} s\right) \\
		&=\textstyle \min\left( \frac{k-1}{2} L + \frac{k+1}{2} s, \frac{k+5}{2} L + \frac{k-1}{2} s \right) 
		= \frac{k-1}{2} L + \frac{k+1}{2} s.
	\end{align*}
	Hence, 
	\begin{equation}
		\label{eq:critical_tb_computed}
		\tb{X} = \max\left( \frac{k+1}{2} L + \frac{k-1}{2} s, \frac{k-1}{2} L + \frac{k+1}{2} s \right) = \frac{k+1}{2} L + \frac{k-1}{2} s = t_b.
	\end{equation}
	
    \indent To find $\td{X}$, we compute the distances $d_{i,i+k+1}$ (cf. Lemma \ref{lemma:technicals_S1} item \ref{item:bd_times_circle}). For $1 \leq i \leq k+1$,
    \begin{align*}
    	x_{i+k+1} - x_{i}
    	&=\textstyle \left( \left\lfloor \frac{i+k+2}{2} \right\rfloor L + \left\lfloor \frac{i+k-1}{2} \right\rfloor s \right) - \left( \left\lfloor \frac{i}{2} \right\rfloor L + \left\lfloor \frac{i-1}{2} \right\rfloor s \right) \\
    	&=\textstyle \left( \left\lfloor \frac{i+k+2}{2} \right\rfloor - \left\lfloor \frac{i}{2} \right\rfloor \right)L + \left( \left\lfloor \frac{i+k-1}{2} \right\rfloor - \left\lfloor \frac{i-1}{2} \right\rfloor \right) s.
    \end{align*}
    When $i$ is odd, the above simplifies to
    \begin{equation*}
    	\textstyle
    	x_{i+k+1} - x_i
		= \left( \frac{i+k+2}{2} - \frac{i-1}{2} \right)L + \left( \frac{i+k-2}{2} - \frac{i-1}{2} \right) s
		= \frac{k+3}{2} L + \frac{k-1}{2} s,
    \end{equation*}
    and when $i$ is even,
    \begin{equation*}
    	\textstyle
    	x_{i+k+1} - x_{i}
    	= \left( \frac{i+k+1}{2} - \frac{i}{2} \right)L + \left( \frac{i+k-1}{2} - \frac{i-2}{2} \right) s
    	= \frac{k+1}{2} L + \frac{k+1}{2} s.
    \end{equation*}
	Notice that $\left(\frac{k+3}{2} L + \frac{k-1}{2} s\right) + \left( \frac{k+1}{2} L + \frac{k+1}{2} s \right) = (k+2)L + ks = 2\pi$. When $k+2 \leq i \leq 2k+2$, we get $x_{i+k+1} = x_{i-k-1}$ and, since $1 \leq i-k-1 \leq k+1$, the above equations give
	\begin{align*}
		\textstyle
		|x_{i+k+1} - x_{i}| &= 2\pi - |x_{i}-x_{i+k+1}| = 2\pi - |x_{i} - x_{i-k-1}|
		=
		\begin{cases}
			\frac{k+1}{2} L + \frac{k+1}{2} s, & i \text{ odd},\\
			\frac{k+3}{2} L + \frac{k-1}{2} s, & i \text{ even}.
		\end{cases}
	\end{align*}
    Hence,
    \begin{align*}
    	d_{i,i+k+1}
    	&= \min\{ |x_{i+k+1} - x_{i}|, 2\pi - |x_{i+k+1} - x_{i}| \}
    	= \textstyle\min\{ \frac{k+1}{2} L + \frac{k+1}{2} s, \frac{k+3}{2} L + \frac{k-1}{2} s \}\\
    	&= \textstyle \frac{k+1}{2} L + \frac{k+1}{2} s
    	= (k+1)(\pi-t_b).
    \end{align*}
    Thus, $\td{X} = \min_{i} d_{i,i+k+1} = (k+1)(\pi-t_b) = (k+1)(\pi-\tb{X})$.
\end{proof}

\begin{theorem}\label{thm:Dnk_S1_k_odd}
	For odd $k$,
	\begin{equation}\label{eq:Dnk_S1_k_odd}
		\textstyle
		\Dvr{2k+2,k}(\Sphere{1}) = \left\{(t_b,t_d): (k+1)(\pi-t_b) \leq t_d \text{ and } \frac{k}{k+1}\pi \leq t_b < t_d \leq \pi \right\}.
	\end{equation}
\end{theorem}
\begin{proof}
    Theorem \ref{thm:critical_tb_odd} and Lemma \ref{lemma:technicals_S1} item \ref{item:min_tb} imply that $\Dvr{2k+2,k}(\Sphere{1})$ is contained in the right-hand side of (\ref{eq:Dnk_S1_k_odd}). To show the other inclusion, choose any pair $(t_b', t_d')$ in the right-hand side of (\ref{eq:Dnk_S1_k_odd}). We now exhibit a set $X' = \{x_1', \dots, x_n'\} \subset \Sphere{1}$ with $\tb{X'}=t_b'$ and $\td{X'}=t_d'$. Let $t_d := t_d'$ and $t_b = \pi-\frac{1}{k+1}t_d$. Notice that $t_d = (k+1)(\pi - t_b)$, so let $X = \{x_1, \dots, x_n\}$ be the set defined in (\ref{eq:critical_tb_odd}). Let $\varepsilon = t_b'-t_b$. Since $(k+1)(\pi-t_b) = t_d = t_d' \geq (k+1)(\pi-t_b')$, we must have $\varepsilon \geq 0$. Now define $x_1' := x_1 + \varepsilon$, $x_{k+2}' := x_{k+2}+\varepsilon$, and $x_i' := x_i$ for $i \neq 1, k+2$. Let $d_{ij} = d_{\Sphere{1}}(x_i,x_j)$ and $d_{ij}' = d_{\Sphere{1}}(x_i',x_j')$. We claim that $\tb{X'} = t_b+\varepsilon = t_b'$ and $\td{X'} = t_d = t_d'$.\\
    \indent Notice that $\varepsilon = t_b'-t_b < t_d'-t_b = t_d-t_b = \td{X}-\tb{X}$, which, by Lemma \ref{lemma:persistence_bounds} item \ref{item:separation_bound}, is bounded above by $\sep(X)$. Because of this, $x_1' = x_1+\varepsilon < x_1+\sep(X) \leq x_{2} = x_{2}'$, so $x_1 \prec x_1' \prec x_2'$. Analogously, $x_{k+2} \prec x_{k+2}' \prec x_{k+3}'$. Since $x_1'$ and $x_{k+2}'$ are the only points for which $x_i' \neq x_i$, the previous two inequalities combined with $x_i \prec x_{i+1} \prec x_{i+2}$ imply that $x_i' \prec x_{i+1}' \prec x_{i+2}'$. Hence, by Lemma \ref{lemma:technicals_S1}, $\tb{X'} = \max_i d_{i,i+k}'$ and $\td{X'} = \min_i d_{i,i+k+1}'$.\\
    \indent Now we find $d_{i,i+k}'$ in terms of $d_{i,i+k}$ and $\varepsilon$. Observe that $d_{i,i+k}' = d_{i,i+k}$ whenever $i \neq 1, 2, k+2, k+3$ because $x_i' \neq x_i$ only when $i=1, k+2$. In fact, $x_1 < x_{1+k} < x_{1-k}$ and $\varepsilon < \sep(X)$, so we can write the distances and absolute values in (\ref{eq:critical_tb_odd_distances}) as
	\begin{align*}
		d_{1,1+k} - \varepsilon &= |x_{1+k}-x_1|-\varepsilon = x_{1+k}-(x_1+\varepsilon) = |x_{1+k}'-x_{1}'|, \text{ and}\\
		d_{1-k,1} + \varepsilon &= 2\pi - |x_{1} - x_{1-k}| + \varepsilon = 2\pi - [x_{1-k} - (x_{1}+\varepsilon)] = 2\pi - |x_{1-k}' - x_{1}'|.
	\end{align*}
	In particular, the two quantities $|x_{1+k}'-x_{1}'|$ and $2\pi - |x_{1-k}' - x_{1}'|$ are bounded above by $t_b+\varepsilon = t_b' < \pi$ because both $d_{1,1+k} - \varepsilon$ and $d_{1,1-k} + \varepsilon$ are. Hence, $d_{1,1 \pm k}' = \min(|x_{1 \pm k}' - x_{1}'|, 2\pi - |x_{1 \pm k}' - x_{1}'|) = d_{1,1 \pm k} \mp \varepsilon$. An analogous argument gives $d_{k+2,(k+2) \pm k}' = d_{k+2,(k+2) \pm k} \mp \varepsilon$.\\
	\indent Now we compute $\tb{X'}$ and $\td{X'}$. Observe that (\ref{eq:critical_tb_computed}) gives $\tb{X} = d_{i,i+k}$ for all $i \neq k+2$ and, in particular, that $d_{1-k,1} \geq d_{i,i+k}$ for all $i$. By the above paragraph, the distances $d_{i,i+k}'$ are either equal to $d_{i,i+k}$ or differ by $\varepsilon$. For this reason, $d_{1-k,1}' = d_{1-k,1}+\varepsilon \geq d_{i,i+k} + \varepsilon \geq d_{i,i+k}'$. Thus, $\tb{X'} = \max_i d_{i,i+k}' = d_{1-k,1}' = d_{1-k,1} + \varepsilon = t_b+\varepsilon = t_b'$. To compute $\td{X'} = \min_i d_{i,i+k+1}'$, observe that the only values of $i$ for which the distance $d_{i,i+k+1}'$ might differ from $d_{i,i+k+1}$ are $i=1, k+2$. However, $x_1' = x_1+\varepsilon$ and $x_{k+2}' = x_{k+2}+\varepsilon$, so $|x_{k+2}'-x_{1}'| = |x_{k+2}-x_{1}|$ and, thus, $d_{1,k+2}' = d_{1,k+2}$. Hence, $\td{X'} = \min_i d_{i,i+k+1}' = \min_i d_{i,i+k+1} = t_d = t_d'$.
\end{proof}

\begin{remark}\label{rmk:D41_S1_scaled}
	The persistence sets of a circle $\frac{\lambda}{\pi} \cdot \Sphere{1}$ with diameter $\lambda$ are obtained by rescaling the results of this section. For example, $\Dvr{4,1}(\frac{\lambda}{\pi} \cdot \Sphere{1})$ is the set bounded by $2(\lambda - t_b) \leq t_d$ and $t_b < t_d \leq \lambda$.
\end{remark}

In general, there are multiple configurations with the same persistence diagram, even among those that minimize the death time. The exception is the configuration that has the minimal birth time, as the following lemma shows.
\begin{prop}\label{prop:persistence_regular_n_gon}
	For any $k \geq 0$, let $n = 2k+2$. If $X \subset \Sphere{1}$ has $n$ points and satisfies $\tb{X} = \frac{k}{k+1}\pi$ and $\td{X} = \pi$, then $X$ is a regular $n$-gon. As a consequence, the configuration $X$ with $n$ points such that $\dgm_k^\vr(X) = \{(\frac{k}{k+1}\pi, \pi)\}$ is unique up to rotations. 
\end{prop}
\begin{proof}
	An application of Lemma \ref{lemma:technicals_S1} item \ref{item:tb_splits} and the triangle inequality gives:
	\begin{align*}
		\frac{k}{k+1}\pi &= \tb{X}
		= \max(d_{i,i+k})
		\geq \frac{1}{2k+2} \sum_{i=1}^{2k+2} d_{i,i+k}
		= \frac{1}{2k+2} \sum_{i=1}^{2k+2} \sum_{j=1}^{k} d_{i+j-1,i+j} \\
		&= \frac{1}{2k+2} \sum_{j=1}^{k} \sum_{i=1}^{2k+2}  d_{i+j-1,i+j}
		= \frac{1}{2k+2} \sum_{j=1}^{k} \left[ \sum_{i=1}^{k+1}  d_{i+j-1,i+j} + \sum_{i=k+2}^{2k+2}  d_{i+j-1,i+j}\right] \\
		&\geq \frac{1}{2k+2} \sum_{j=1}^{k} \left[ d_{j,j+k+1} + d_{j+k+1,j} \right]
		\geq \frac{1}{2k+2} \sum_{j=1}^{k} \left[ 2\td{X} \right]
		= \frac{k}{k+1}\pi.
	\end{align*}
	Thus, all intermediate inequalities become equalities, most notably, $d_{i,i+k} = \frac{k}{k+1}\pi$ and $d_{j,j+k+1} = \sum_{i=1}^{k+1}  d_{i+j-1,i+j} = \pi$. Then
	$
		d_{i,i+1} = d_{i-k,i+1} - d_{i-k,i} = \frac{2\pi}{2k+2}.
	$
	That is, $X$ is a regular $n$-gon.
\end{proof}

\subsection{Characterization of $\Unk{4,1}{\vr}(\mathbb{S}^1)$}
In addition to the characterization of $\Dvr{4,1}(\Sphere{1})$ given in Theorem \ref{thm:Dnk_S1_k_odd}, we can also characterize the persistence measure $\Unk{4,1}{\vr}(\Sphere{1})$. To set up the context, consider the diagonal $\Do \subset \R^2$. Since any two points in $\Do$ are at bottleneck distance 0, we can view $\Dvr{4,1}(\Sphere{1})$ as a subset of $\R^2/\Do$. Let $\mathcal{L}$ be the pushforward of the Lebesgue measure under the quotient $\R^2 \to \R^2 / \Do$.
\begin{prop}\label{prop:U41_S1}
	Let $\mu_{\Sphere{1}}$ be the uniform measure on $\Sphere{1}$. With respect to $\mathcal{L}$, the persistence measure $\Unk{4,1}{\vr}(\Sphere{1})$ decomposes into a singular measure supported on $\Do$ and a measure supported on $\Dvr{4,1}(\Sphere{1}) \setminus \Do$ with Radon-Nikodym derivative
	\begin{equation*}
		f(t_b,t_d) = \frac{12}{\pi^3}\left(\pi-t_d \right),
	\end{equation*}
	for $(t_b,t_d) \in \Dvr{4,1}(\Sphere{1}) \setminus \Do$. In particular, the probability that the 1-dimensional persistence diagram of a 4-point subset of $\Sphere{1}$ is in $\Do$ is $\frac{8}{9}$.
\end{prop}
\begin{remark}
	\label{rmk:U41_S1_non_zero}
	Given a set $X = \{x_1,x_2,x_3,x_4\} \subset \Sphere{1}$ chosen uniformly at random, the probability that $\dgm_1(X)$ is a non-diagonal point is $\frac{1}{9} \approx 11 \%$. This is consistent with the 11.08 \% success rate obtained in the simulations; \textit{cf}. Example \ref{ex:D41_spheres}.
\end{remark}

Before proving Proposition \ref{prop:U41_S1}, we give an application where $\Uvr{4,1}$ can distinguish spaces that $\Dvr{4,1}$ cannot.

\begin{example}[Persistence sets of $\Sphere{1}$ without a segment.]
    \label{ex:circle_without_segment}
    Let $L \in (0,2\pi)$. Define $S_L := \Sphere{1} \setminus (2\pi-L, 2\pi)$ to be the circle with an open segment of length $L$ removed, and give $S_L$ the restriction of the geodesic metric induced from $\Sphere{1}$ (which, in particular, will not be geodesic). We will show in Proposition \ref{prop:circle_without_segment} that $\Dvr{2k+2,k}(S_{L})$ equals $\Dvr{2k+2,k}(\Sphere{1})$ for $0 < L \leq \frac{1}{k+1}\pi$ and is strictly contained in $\Dvr{2k+2,k}(\Sphere{1})$ otherwise.\\
    \indent As for the degree 1 persistence diagrams of $S_{L}$ and $\Sphere{1}$, they are different for any value of $L$. Indeed, if $r<L$, the balls of radius $r$ in $S_L$ are isometric to the corresponding balls in $[0, 2\pi-L]$ (with the absolute value metric). Hence, $\vr_r(S_{L}) \simeq \vr_r([0,2\pi-L]) \simeq *$ for $0<r<L$. This implies that, if $(b,d) \in \dgm^\vr_1(S_L)$, then $b\geq L>0$. However, it is known that $\dgm^\vr_1(\Sphere{1})=\{(0,\frac{2\pi}{3})\}$. Thus, $S_L$ and $\Sphere{1}$ are an example of a pair of spaces that can be distinguished by $\dgm_1^\vr$ but not by any $\Dvr{2k+2,k}$ for which $L \leq \frac{1}{k+1}\pi$.\\
    \indent The last invariant we consider is the measure $\Unk{4,1}{\vr}$, which can distinguish $\Sphere{1}$ and $S_{L}$ for every $L \in (0,2\pi)$. For instance, when $L=\frac{\pi}{2}$, there exists a circle worth of squares in $\Sphere{1}$ but, since all sides of a square have length $\frac{\pi}{2}$, only one square fits in $S_{\pi/2}$. Moreover, Proposition \ref{prop:U41_S1} says that the Radon-Nikodym derivative of $\Unk{4,1}{\vr}(\Sphere{1})$ away from the diagonal is independent of $t_b$ but, as Figure \ref{fig:U41_circle_truncated} shows, the derivative of $\Unk{4,1}{\vr}(S_{\pi/2})$ is not (compare with Figure \ref{fig:D41_spheres_torus}).\\
    \indent A related example appears in Section 9 of \cite{persistence-geodesic-spaces} which shows that $\dgm_1^\vr$ can  itself be insensitive to small holes. For a slightly deformed 2-dimensional torus $T$ and a small enough open disk $D \subset T$, the author shows that $\PH_1^\vr(T) \cong \PH_1^\vr(T \setminus D)$. It is interesting that in this case $\dgm^\vr_1$  cannot detect the absence of $D$, in contrast to the case of $\Sphere{1}$ and $S_L$.
    {
    \prop
        \label{prop:circle_without_segment}
        Let $L \in (0,2\pi)$. Define $S_L := \Sphere{1} \setminus (2\pi-L, 2\pi)$ to be the circle with an open segment of length $L$ removed, and equip $S_L$ with the metric induced by the inclusion $S_L \subset \Sp^1$. Then $\Dvr{2k+2,k}(S_L) \neq \Dvr{2k+2,k}(\Sphere{1})$ if and only if $L > \frac{\pi}{k+1}$.
    \endprop
    }
    \begin{proof}
        Suppose that $0 < L \leq \frac{\pi}{k+1}$. Let $X = \{x_1, \dots, x_{2k+2} \} \subset \Sphere{1}$ such that $x_{i-1} \prec x_{i} \prec x_{i+1}$ and  $\tb{X} < \td{X}$. By Lemma \ref{lemma:technicals_S1} items \ref{item:bd_times_circle} and \ref{item:min_tb}, $\tb{X} = d_{i,i+k}$ for some $i$ and $\tb{X} \geq \frac{k}{k+1}\pi$. In particular, by Lemma \ref{lemma:technicals_S1} item \ref{item:tb_splits}, one of the distances $d_{j,j+1}$ is at least $\frac{1}{k+1}\pi$ for some $i \leq j < i+k$. In other words, the gap between $x_j$ and $x_{j+1}$ is larger than or equal to $L$, so if we rotate $X$ anticlockwise by $2\pi-x_{j+1}$, we obtain a set $X' \subset S_L$ isometric to $X$. Hence, $\Dvr{2k+2,k}(\Sphere{1}) \subset \Dvr{2k+2,k}(S_{L})$. Since $S_{L} \hookrightarrow \Sphere{1}$, we also have the other inclusion.\\
        \indent Now suppose that $L > \frac{1}{k+1}\pi$. The point $\left(\frac{k}{k+1}\pi, \pi\right)$ is in $\Dvr{2k+2,k}(\Sphere{1})$ and, by Proposition \ref{prop:persistence_regular_n_gon}, it is generated by a regular $(2k+2)$-gon. The side length of that polygon is $\frac{1}{k+1}\pi < L$, so $S_L$ cannot contain any regular $(2k+2)$-gon and, thus, $\left(\frac{k}{k+1}\pi, \pi\right) \notin \Dvr{2k+2,k}(\Sphere{1})$. See, for example, $\Unk{4,1}{\vr}(S_{3\pi/4})$ in Figure \ref{fig:U41_circle_truncated}.
    \end{proof}
\end{example}

\begin{figure}[h]
	\centering
	\begin{minipage}{0.3\linewidth}
		\centering
\begin{tikzpicture}[scale=1.5]
	\draw (1,0) arc[start angle=0, end angle=270, radius=1];
	\node at (0,1) [above] {$S_{\pi/2} = \Sphere{1} \setminus (\frac{3\pi}{2}, 2\pi)$};
\end{tikzpicture} 	\end{minipage}
	\hfill
	\begin{minipage}{0.67\linewidth}
		\includegraphics[width=\linewidth]{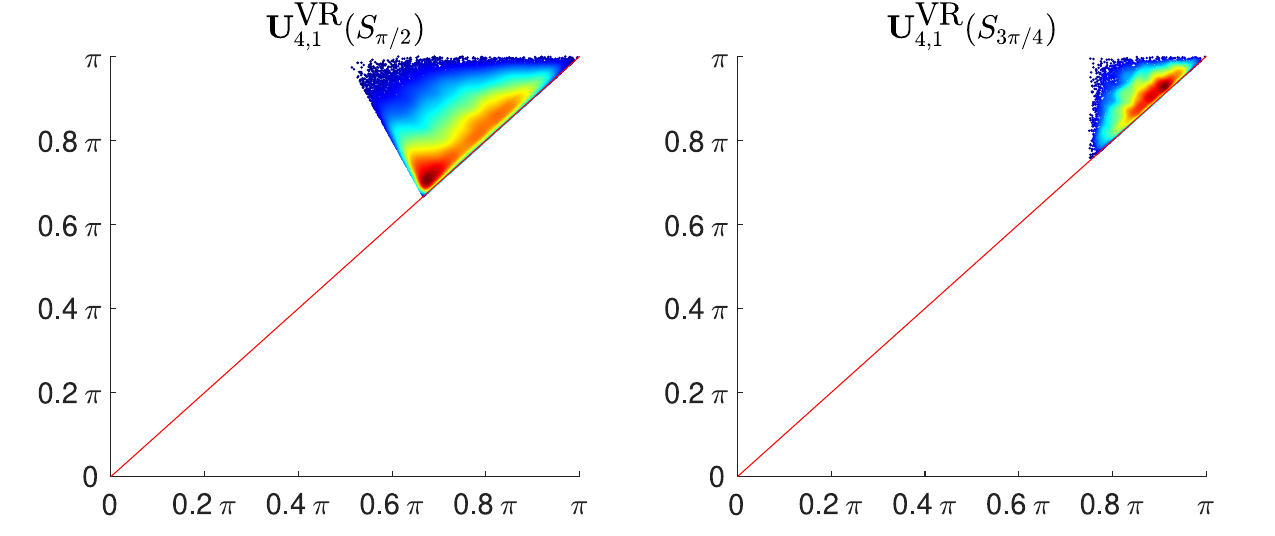}
	\end{minipage}
	\caption{\textbf{Left:} The space $S_{\pi/2}$ formed by removing a segment of length $\pi/2$ from $\Sphere{1}$. \textbf{Middle and right:} The persistence measures $\Unk{4,1}{\vr}(S_L)$ of the truncated circle $S_L$ for $L=\pi/2$ (middle) and $L=3\pi/4$ (right).}
	\label{fig:U41_circle_truncated}
\end{figure}

\begin{proof}[Proof of Proposition \ref{prop:U41_S1}]
	Since $\Unk{4,1}{\vr}(\Sphere{1})$ is a probability measure (and hence finite) and $\mathcal{L}$ is positive and $\sigma$-finite, the Lebesgue-Radon-Nikodym Theorem (\cite[Theorem 3.8]{folland}) says that $\Unk{4,1}{\vr}(\Sphere{1})$ decomposes as a sum of a singular measure and an absolutely continuous measure with respect to $\mathcal{L}$. We will show that $\Unk{4,1}{\vr}(\Sphere{1})$ is absolutely continuous in $\Dvr{4,1}(\Sphere{1}) \setminus \Do$ and that the persistence diagram of a 4-point subset of $\Sphere{1}$ is in $\Do$ with non-zero probability. These facts give the desired decomposition.\\
	\indent Recall that we model $\Sphere{1}$ as the quotient $[0, 2\pi]/0 \sim 2\pi$. Let $X = \{x_1 , x_2 , x_3 , x_4\} \subset [0, 2\pi]$ be 4 points chosen uniformly at random. Since $\tb{X}$ and $\td{X}$ only depend on the distances between the $x_i$, we may assume $x_1=0$. Notice that the tuple $(x_2,x_3,x_4)$ is still distributed uniformly in $[0,2\pi]^3$. Relabel $x_i$ as $x^{(j)} \in [0, 2\pi]$ so that $0 = x^{(1)} < x^{(2)} < x^{(3)} < x^{(4)}$ and set $y_i := x^{(i+1)} - x^{(i)}$ for $i=1, 2, 3$ and $y_4 := 2\pi - x^{(4)}$. Let $D := \{ (x^{(2)}, x^{(3)}, x^{(4)}) \in [0,2\pi]^3 : x^{(2)} < x^{(3)} < x^{(4)} \}$, and $\Delta_3(2\pi) := \{(y_1 , y_2 , y_3) \in [0, 2\pi]^3: y_1 + y_2 + y_3 \leq 2\pi\}$. Since the only difference between $(x_2, x_3, x_4)$ and $(x^{(2)}, x^{(3)}, x^{(4)})$ is the order of the coordinates, the latter is uniformly distributed in $D$. Furthermore, the pushforward of the uniform measure on $D$ onto $\Delta_3(2 \pi)$ under the map $\Psi(x^{(2)}, x^{(3)}, x^{(4)}) = (x^{(2)}, x^{(3)}-x^{(2)}, x^{(4)} - x^{(3)})$ is the uniform measure because the Jacobian of $\Psi$ has determinant 1. Hence, we will model a configuration of four points in $\Sphere{1}$ as the set of distances $y_1, y_2, y_3, y_4$ instead, where $(y_1,y_2,y_3) \in \Delta_3(2\pi)$ is uniformly distributed and $y_4 = 2\pi - (y_1+y_2+y_3)$.\\
	\indent Now we characterize the measure on non-diagonal points of $\Dvr{4,1}(\Sphere{1})$. Fix $(t_b,t_d) \in \Dvr{4,1}(\Sphere{1})$ with $t_b < t_d$. By Lemma \ref{lemma:technicals_S1}, $\tb{X} = \max_i y_i$ and $\td{X} = \min_i(y_i+y_{i+1})$. Since $\Delta_3(2\pi)$ has the uniform measure, the probability that $t_b \leq \tb{X} < \td{X} \leq t_d$ is the volume of the set
	\begin{equation*}
	    R(t_b,t_d) := \{(y_1,y_2,y_3) \in \Delta_3(2\pi): t_b \leq \tb{X} < \td{X} \leq t_d\}
	\end{equation*}
	divided by $\operatorname{Vol}(\Delta_3(2\pi)) = \frac{(2\pi)^3}{3!}$. We will find $\operatorname{Vol}(R(t_b,t_d))$ using an integral with a suitable parametrization of $y_1,y_2,y_3$.\\
	\indent Assume that $\tb{X} = y_1$. There are four choices for $\td{X}$, but to start, let $\td{X}=y_1+y_2$. Since $y_3 \leq y_1$ by definition of $\tb{X}$, we have $y_3+y_2 \leq y_1+y_2$ and, by definition of $\td{X}$, $y_1+y_2 = y_3+y_2$. Thus, the case $\td{X}=y_1+y_2$ is a subset of the case when $\td{X} = y_2+y_3$. Similarly, $\td{X} = y_1+y_4$ implies $\td{X} = y_3+y_4$, so we only have to consider two choices for $\td{X}$.\\
	\indent Let $R'(t_b,t_d)$ be the subset of $R(t_b,t_d)$ where $\tb{X}=y_1$ and $\td{X}=y_2+y_3$. Observe that the inequalities $y_2+y_3 \leq y_3+y_4$ if and only if $y_2 \leq y_4$, so the conditions $\tb{X} = y_1 \geq t_b$ and $\td{X} = y_2+y_3 \leq t_d$ are equivalent to the system of inequalities $t_b \leq y_1 < y_2+y_3 \leq t_d$, $y_2 \leq y_4 \leq y_1$, and $y_3 \leq y_1$. Consider the substitution $s=y_2+y_3$ and rewrite $y_4 = 2\pi-y_1-s$. These changes give, for example, that $y_4 \leq y_1$ is equivalent to $2\pi-2y_1 = y_4-y_1+s \leq s$. In a similar fashion, substituting $s$ and $y_4$ into the rest of the inequalities yields the following characterization of $R'(t_b,t_d)$ in terms of $y_1,y_2$ and $s$:
	\begin{align*}\label{eq:ineq_y3_y_4_final}
		t_b \leq y_1 &< t_d\\
		\max(2\pi-2y_1,y_1) \leq s &\leq t_d\\
		s-y_1 \leq y_2 &\leq 2\pi-s-y_1.
	\end{align*}
	Notice that the Jacobian $\left| \frac{\partial(y_1,y_2,y_3)}{\partial(y_1,y_2,s)} \right|$ of the transformation $(y_1,y_2,y_3) \mapsto (y_1,y_2,s) = (y_1,y_2,y_2+y_3)$ is 1. Also, when defining $R'(t_b,t_d)$, we had four choices for $\tb{X}$ (all four $y_i$) and for each, two choices for $\td{X}$. Then
	\begin{equation*}
		\operatorname{Vol}(R(t_b,t_d)) = 8\operatorname{Vol}(R'(t_b,t_d)) = \int_{t_b}^{t_d}\int_{\max(2\pi-2y_1,y_1)}^{t_d}\int_{s-y_1}^{2\pi-s-y_1} 8\, \mathrm{d}y_2\, \mathrm{d}s\, \mathrm{d}y_1.
	\end{equation*}
	If we define $f(t_b,t_d) := \frac{1}{\operatorname{Vol}(\Delta_3(2\pi))} \int_{t_d-t_b}^{2\pi-t_d-t_b} 8\, \mathrm{d}y_2$ and $F(t_b, t_d) := \mathbb{P}(t_b \leq \tb{X} < \td{X} \leq t_d)$, we obtain
	\begin{equation}\label{eq:VolR}
		F(t_b,t_d) = \frac{\operatorname{Vol}(R(t_b,t_d))}{\operatorname{Vol}(\Delta_3(2\pi))} = \int_{t_b}^{t_d} \int_{\max(2(\pi-\tau_b),\tau_b)}^{t_d} f(\tau_b, \tau_d) \, \mathrm{d}\tau_d \, \mathrm{d}\tau_b.
	\end{equation}
	Notice that the lower bound on $\tau_d$ equals the bound $t_d \geq 2(\pi-t_b)$ given by Theorem \ref{thm:Dnk_S1_k_odd} when $k=1$. In other words, $F(t_b,t_d)$ is the integral of $f(\tau_b,\tau_d)$ over the subset of $\Dvr{4,1}(\Sphere{1}) \setminus \Do$ where $t_b \leq \tau_b < \tau_d \leq t_d$. In particular, $F(t_b, t_d)$ is absolutely continuous with respect to $\mathcal{L}$ and its Radon-Nikodym derivative is
	\begin{equation*}
		f(t_b,t_d) = \frac{1}{\operatorname{Vol}(\Delta_3(2\pi))} \int_{t_d-t_b}^{2\pi-t_d-t_b} 8\, \mathrm{d}y_2 = \dfrac{16(\pi-t_d)}{(2\pi)^3/3!} = \frac{12}{\pi^3}(\pi-t_d).
	\end{equation*}
	Furthermore, the probability that $\tb{X}<\tb{X}$ equals $F(\pi/2, \pi) = \frac{\operatorname{Vol}(R(\pi/2, \pi))}{\operatorname{Vol}(\Delta_3(2\pi))} = \frac{4\pi^3/27}{(2\pi)^3/3!} = \frac{1}{9}$. Hence, the probability that $\dgm_1(X)$ is in $\Do$ is $\frac{8}{9}$.
\end{proof}

\subsection{Persistence sets of Ptolemaic spaces}
Example \ref{ex:D41_example} showed that in a metric space with four points, the birth time of its one-dimensional persistent homology is given by the length of the largest side and the death time, by that of the smaller diagonal. In this section, we use Ptolemy's inequality, which relates the lengths of the diagonals and sides of Euclidean quadrilaterals, to bound the first persistence set $\Dvr{4,1}$ of several spaces and show examples where the bound is attained.
\begin{defn}
	A metric space $(X, d_X)$ is called \define{Ptolemaic} if for any $x_1,x_2,x_3,x_4 \in X$,
	\begin{equation}\label{ineq:ptolemy}
		d_X(x_1,x_3) \cdot d_X(x_2,x_4) \leq d_X(x_1,x_2) \cdot d_X(x_3,x_4) + d_X(x_1,x_4) \cdot d_X(x_2,x_3).
	\end{equation}
\end{defn}

\indent It should be noted that the inequality holds for any permutation of $x_1,x_2,x_3,x_4$ and, in $\R^m$, equality holds if and only if the points $x_1,x_2,x_3,x_4$ lie on a circle or a line. Examples of Ptolemaic metric spaces include the Euclidean spaces $\R^m$ and $\text{CAT}(0)$ spaces; see \cite{ptolemy-and-cat0} for a more complete list of references. The basic result of this section is the following.

\begin{prop}\label{prop:D41_ptolemy}
	Let $(X,d_X)$ be Ptolemaic. Then $t_d \leq \sqrt{2}t_b$ for any $(t_b, t_d) \in \Dvr{4,1}(X)$.
\end{prop}
\begin{proof}
	Let $X' = \{x_1,x_2,x_3,x_4\} \subset X$ be such that $\tb{X'} < \td{X'}$. As per Example \ref{ex:D41_example}, relabel the points so that $\tb{X'} = \max(d_{12}, d_{23}, d_{34}, d_{41})$ and $\td{X'} = \min(d_{13}, d_{24})$. Then, Ptolemy's inequality gives
	\begin{align*}
		\left(\td{X'} \right)^2
		\leq d_{13} d_{24}
		\leq d_{12} d_{34} + d_{23} d_{14}
		\leq 2 \left(\tb{X'} \right)^2.
	\end{align*}
	Taking square root gives the result.
\end{proof}

\begin{remark}\label{rmk:D41-ptolemy-equality}
	If $\td{X'}=\sqrt{2}\tb{X'}$ in the proof of Proposition \ref{prop:D41_ptolemy}, we have $d_{13} \cdot d_{24} = d_{12} \cdot d_{34} + d_{23} \cdot d_{41}$.	In particular, if $X' \subset \R^m$, then $X'$ must lie on a circle. In other words, any point in the boundary $t_d = \sqrt{2}t_b$ of $\Dvr{4,1}(\R^m)$ is the persistence diagram of a concyclic 4-point set $X'$.
\end{remark}

Another way to phrase the above proposition is to say that $\Dvr{4,1}(X)$ is contained in the set
\begin{equation}\label{eq:Ptolemaic_region}
	P_R := \left\{ (t_b,t_d)| 0 \leq t_b < t_d \leq \min(\sqrt{2}t_b, R) \right\},
\end{equation}
with $R=\diam(X)$. A key example where the containment is strict is the following.

\begin{prop}\label{prop:D41_S1_E}
	Let $\Sphere{1}_E$ denote the unit circle in $\R^2$ equipped with the Euclidean metric. Then
	\begin{equation*}
		\Dvr{4,1}(\Sphere{1}_E) = \left\{ (t_b,t_d)\ \big|\ 2t_b \sqrt{1-\dfrac{t_b^2}{4}} \leq t_d, \text{ and } \sqrt{2} \leq t_b < t_d \leq 2 \right\}.
	\end{equation*}
\end{prop}
\begin{proof}
	\indent Observe that the Euclidean distance $d_E$ between two points in $\Sphere{1}$ is related to their geodesic distance $d$ by $d_E = f_E(d) := 2\sin(d/2)$. Since $f_E$ is increasing on $[-\pi,\pi]$, an interval that contains all possible distances between points in $\Sphere{1}$, a configuration $X = \{x_1,x_2,x_3,x_4\} \subset \Sphere{1}$ produces non-zero persistence if and only if its Euclidean counterpart $X_E \subset \Sphere{1}_E$ does. For this reason, $\Dvr{4,1}(\Sphere{1}_E) = f_E\left(\Dvr{4,1}(\Sphere{1})\right)$.\\
	\indent From Theorem \ref{thm:Dnk_S1_k_odd},
	\begin{equation*}
		\Dvr{4,1}(\Sphere{1}) = \{ (t_b,t_d) \mid 2(\pi-t_b) \leq t_d \text{ and } \pi/2 \leq t_b < t_d \leq \pi\}.
	\end{equation*}
	Applying $f_E$ to the bound $t_d \geq 2(\pi-t_b)$ gives
	\begin{align*}
		t_{d,E} = 2\sin(t_d/2) 
		&\geq 2\sin(\pi-t_b) = 2\sin(t_b)
		= 2\sin(2\arcsin(t_{b,E}/2)) \\
		&= 4 \sin(\arcsin(t_{b,E}/2)) \cos(\arcsin(t_{b,E}/2))
		= 2 t_{b,E} \sqrt{1-t_{b,E}^2/4},
	\end{align*}
	while the image of the bound $\pi/2 \leq t_b < t_d \leq \pi$ under $f_E$ is $\sqrt{2} \leq t_{b,E} < t_{d,E} \leq 2$.
\end{proof}

\indent Even though $\Dvr{4,1}(\Sphere{1}_E)$ doesn't attain equality in the bound given by Proposition \ref{prop:D41_ptolemy}, it can be used to show that other spaces do. Two examples are $\Sphere{2}$ and $\R^2$.

\begin{prop}\label{prop:D41_R2_S2}
	For $n \geq 2$,
	\begin{equation*}
	\Dvr{4,1}(\R^n) = \left\{ (t_b,t_d) \ \big|\ 0 < t_b < t_d \leq \sqrt{2}t_b \right\}
	\ \text{and} \
	\Dvr{4,1}(\Sphere{n}_E) = \left\{ (t_b,t_d) \ \big|\ 0 < t_b < t_d \leq \min(\sqrt{2}t_b, 2) \right\}.
	\end{equation*}
	In particular, both sets are convex.
\end{prop}
\begin{proof}
	Since both $\R^n$ and $\Sphere{n}_E \subset \R^{n+1}$ are Ptolemaic spaces, Proposition \ref{prop:D41_ptolemy} gives $\Dvr{4,1}(\R^n) \subset P_\infty$ and $\Dvr{4,1}(\Sphere{n}_E) \subset P_2$ (see equation \ref{eq:Ptolemaic_region}). To show the other direction, notice that $\R^n$ contains circles $R \cdot \Sphere{1}_E$ of any radius $R>0$. By functoriality of persistence sets (Remark \ref{rmk:functorial_persistence_sets}), $\Dvr{4,1}(R \cdot \Sphere{1}_E) \subset \Dvr{4,1}(\R^n)$ so, in particular, $\Dvr{4,1}(\R^n)$ contains the line $[\sqrt{2} R, 2R) \times 2R$ that bounds $\Dvr{4,1}(R \cdot \Sphere{1}_E)$ from above (see Proposition \ref{prop:D41_S1_E} and Figure \ref{fig:D41_S1_S2}). The inequality $t_b < t_d \leq \sqrt{2} t_b$ can be rearranged to $\frac{\sqrt{2}}{2} t_d \leq t_b < t_d$, so given any point $(t_b,t_d) \in P_\infty$, taking $R = t_d/2$ gives $(t_b,t_d) \in [\sqrt{2} R, 2R) \times 2R \subset \Dvr{4,1}(\R^2)$. Thus, $P_\infty \subset \Dvr{4,1}(\R^n)$. The same argument with the added restriction of $R \leq 1$ shows that $P_2 \subset \Dvr{4,1}(\Sphere{n}_E)$. Lastly, $\Dvr{4,1}(\R^n)$ (resp. $\Dvr{4,1}(\Sphere{2}_n)$) is convex because it is the intersection of two (resp. three) half-spaces.
\end{proof}

\begin{figure}[ht]
	\centering
	\makebox[\textwidth][c]{\includegraphics[scale=0.57]{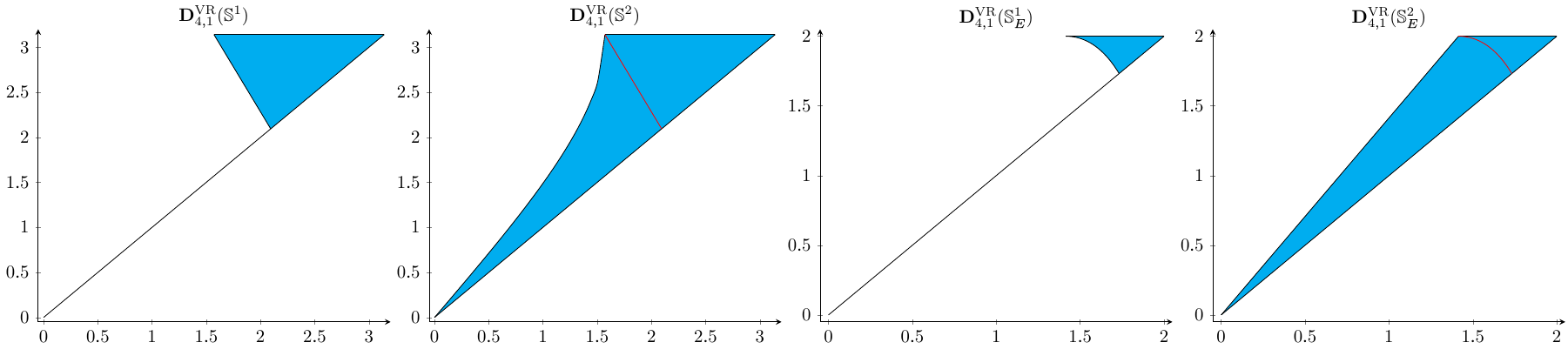}}
	
	\caption{
		\textbf{From left to right:} $\Dvr{4,1}(\Sphere{1})$, $\Dvr{4,1}(\Sphere{2})$, $\Dvr{4,1}(\Sphere{1}_E)$, and $\Dvr{4,1}(\Sphere{2}_E)$. Notice that $\Dvr{4,1}(\Sphere{1}) \subset \Dvr{4,1}(\Sphere{2})$, as indicated by the red line in the second diagram from the left. The analogous statement holds for $\Sphere{1}_E \subset \Sphere{2}_E$. Cf. the two rightmost figures with Proposition \ref{prop:D41_R2_S2}.
	}
	\label{fig:D41_S1_S2}
\end{figure}

\indent Two observations summarize the proof of Proposition \ref{prop:D41_R2_S2}: Ptolemy's inequality gives a region $P_\infty$ that contains $\Dvr{4,1}(\R^2)$, while the circles in $\R^2$ produce enough points to fill $P_\infty$. It turns out that this technique can be generalized to other spaces, provided that we have a suitable analogue of Ptolemy's inequality. This is explored in the next section.

\subsection{Persistence sets of surfaces with constant curvature}
\label{sec:D41_Mk}
Consider the surface $M_\kappa$ with constant sectional curvature $\kappa$. In this section, we will characterize $\Dvr{4,1}(M_\kappa)$. Proposition \ref{prop:D41_R2_S2} already has the case $\kappa=0$, so now we deal with $\kappa \neq 0$. To fix notation, let $x,y \in \R^3$. Define $\langle x,y \rangle := x_1y_1+x_2y_2+x_3y_3$ and $\langle x|y \rangle := -x_1y_1+x_2y_2+x_3y_3$. We model $M_\kappa$ as $\displaystyle M_\kappa := \left\{x \in \R^3 \ | \ \langle x,x \rangle = \frac{1}{\kappa} \right\} \text{ if } \kappa>0$, and $\displaystyle M_\kappa := \left\{x \in \R^3 \ | \ \langle x|x \rangle = \frac{1}{\kappa} \text{ and } x_1>0 \right\} \text{ if } \kappa<0$.
In other words, $M_\kappa$ is the sphere of radius $1/\sqrt{\kappa}$ if $\kappa>0$, or a rescaling of the hyperbolic plane if $\kappa<0$. The geodesic distance in $M_\kappa$ is given by
\begin{equation}\label{eq:hyperbolic_distance}
	d_{M_\kappa}(x,y) :=
	\begin{cases}
		\frac{1}{\sqrt{+\kappa}} \arccos(\kappa \langle x,y \rangle), & \text{if } \kappa>0,\\
		\frac{1}{\sqrt{-\kappa}} \arccosh(\kappa \langle x|y \rangle), & \text{if } \kappa<0.
	\end{cases}
\end{equation}

\indent To use the same technique as in Proposition \ref{prop:D41_R2_S2}, we use a version of Ptolemy's inequality for spaces of non-zero curvature.
\begin{theorem}[Spherical and Hyperbolic Ptolemy's inequality, \cite{ptolemy-hyperbolic, ptolemy-spherical}]~\\
	\label{thm:ptolemy-curvature}
	Let $x_1,x_2,x_3,x_4 \in M_\kappa$, and $d_{ij}=d_{M_\kappa}(x_i,x_j)$.	Then
	\begin{equation}\label{ineq:ptolemy-spherical-basic}
		\textstyle
		s_\kappa\left(d_{13}/2\right) s_\kappa\left(d_{24}/2\right) \leq s_\kappa\left(d_{12}/2\right) s_\kappa\left(d_{34}/2\right) + s_\kappa\left(d_{14}/2\right) s_\kappa\left(d_{23}/2\right),
	\end{equation}
	where $s_\kappa(t)$ is defined as $\sin(\sqrt{\kappa}t)$ if $\kappa>0$, and $\sinh(\sqrt{-\kappa}t)$ if $\kappa<0$.
\end{theorem}

\indent With these tools, we are ready to prove the main theorem of this section.

\DiagramMKappa
\begin{proof}
	The case $\kappa=0$ was already done in Proposition \ref{prop:D41_R2_S2}. For $\kappa>0$, let
	\begin{equation*}
		\textstyle P := \left\{ (t_b,t_d)|\ \sin \left( \frac{\sqrt{ \kappa}}{2} t_d \right) \leq \sqrt{2}\sin \left( \frac{\sqrt{ \kappa}}{2} t_b \right) \text{ and } 0 < t_b < t_d \leq \frac{\pi}{\sqrt{\kappa}} \right\}.
	\end{equation*}
	Let $X = \{x_1,x_2,x_3,x_4\} \subset M_\kappa$ and $d_{ij} = d_{M_\kappa}(x_i,x_j)$. Suppose that $\tb{X} < \td{X}$ and label the $x_i$ so that $\tb{X} = \max(d_{12}, d_{23}, d_{34}, d_{41})$ and $\td{X} = \min(d_{13}, d_{24})$. Let $s_{ij} := \sin\left(\frac{\sqrt{\kappa}}{2} d_{ij} \right)$. By Theorem \ref{thm:ptolemy-curvature}, $s_{13}s_{24} \leq s_{12}s_{34}+s_{14}s_{23}$,
	and, since the function $t \mapsto \sin\left(\frac{\sqrt{\kappa}}{2} t\right)$ is increasing when $\frac{\sqrt{\kappa}}{2} t \in \left[0,\frac{\sqrt{\kappa}}{2} \diam(M_\kappa)\right] = \left[0,\frac{\pi}{2}\right]$, we get
	\begin{align*}
		\textstyle \sin^2\left(\frac{\sqrt{\kappa}}{2} \td{X}\right)
		= (\min(s_{13},s_{24}))^2
		\leq s_{13}s_{24}
		\leq s_{12}s_{34}+s_{14}s_{23}
		\leq \textstyle 2\sin^2\left(\frac{\sqrt{\kappa}}{2} \tb{X}\right).
	\end{align*}
	Thus,
	\begin{equation}
		\label{ineq:D41_S2_boundary}
		\sin \left( \frac{\sqrt{ \kappa}}{2} t_d \right) \leq \sqrt{2}\sin \left( \frac{\sqrt{ \kappa}}{2} t_b \right).
	\end{equation}
	\indent This shows that $\Dvr{4,1}(M_\kappa) \subset P$. For the other direction, let $0 < t \leq 1$ and $s \in (0,\pi/2]$, and consider $X = \{x_1,x_2,x_3,x_4\}$ where
	\begin{multicols}{2}
	\begin{itemize}
		\item $x_1 =\textstyle  \left(\frac{1}{\sqrt{\kappa}} \sqrt{1-t^2}, \frac{t}{\sqrt{\kappa}}, 0\right)$
		\item $x_2 =\textstyle  \left(\frac{1}{\sqrt{\kappa}} \sqrt{1-t^2}, \frac{t}{\sqrt{\kappa}} \sin(s), \frac{t}{\sqrt{\kappa}} \cos(s)\right)$
		\item $x_3 =\textstyle  \left(\frac{1}{\sqrt{\kappa}} \sqrt{1-t^2}, -\frac{t}{\sqrt{\kappa}}, 0\right)$
		\item $x_4 =\textstyle  \left(\frac{1}{\sqrt{\kappa}} \sqrt{1-t^2}, -\frac{t}{\sqrt{\kappa}} \sin(s), -\frac{t}{\sqrt{\kappa}} \cos(s)\right)$.
	\end{itemize}
	\end{multicols}
	\noindent Notice that the set $\left\{(p_1, p_2, p_3) \in M_\kappa: p_1 = \frac{1}{\sqrt{\kappa}} \sqrt{1-t^2} \right\}$ is a circle with radius $t/\sqrt{|\kappa|}$. Inside of this circle, the configuration $\{x_1, x_2, x_3, x_3\}$ is a parallelogram where $x_1$ and $x_2$ are antipodal to $x_3$ and $x_4$, respectively. Indeed, it can be checked that:
	\begin{multicols}{2}
	\begin{itemize}
		\item $x_i \in M_\kappa$,
		\item $\langle x_1, x_3 \rangle = \langle x_2, x_4 \rangle = \frac{1}{\kappa}(1-2t^2)$,
		\item $\langle x_1, x_2 \rangle = \langle x_3, x_4 \rangle = \frac{1}{\kappa}(1-t^2(1-\sin(s)))$,
		\item $\langle x_1, x_4 \rangle = \langle x_2, x_3 \rangle = \frac{1}{\kappa}(1-t^2(1+\sin(s)))$,
	\end{itemize}
	\end{multicols}
	\noindent and (since $s \in (0, \pi/2]$) $\langle x_1,x_3 \rangle < \langle x_1,x_4 \rangle \leq \langle x_1,x_2 \rangle$. Since $\arccos(t)$ is decreasing, we have
	\begin{align*}
		\tb{X} &= \frac{1}{\sqrt{\kappa}} \arccos\left(\kappa \langle x_1, x_4 \rangle\right) = \frac{1}{\sqrt{\kappa}} \arccos(1-t^2(1+\sin(s))), \text{ and }\\
		\td{X} &= \frac{1}{\sqrt{\kappa}} \arccos\left(\kappa \langle x_1, x_3 \rangle\right) = \frac{1}{\sqrt{\kappa}} \arccos(1-2t^2).
	\end{align*}
	Notice that for a fixed $t$, $\tb{X}$ is minimized at $s=0$ and the equality in (\ref{ineq:D41_S2_boundary}) is achieved. Also, $\td{X}$ is maximized at $t=1$, at which point $\td{X} = \frac{\pi}{\sqrt{\kappa}}$. Now, let $(t_b,t_d) \in P$ be arbitrary. If we set $\tb{X}=t_b$ and $\td{X}=t_d$, we can solve the equations above to get
	\begin{equation*}
		t = \sqrt{\frac{1-\cos(\sqrt{\kappa}t_d)}{2}}, \text{ and } \sin(s) = 2 \cdot \frac{1-\cos(\sqrt{\kappa}t_b)}{1-\cos(\sqrt{\kappa}t_d)}-1.
	\end{equation*}
	Such a $t$ exists because $\cos(\sqrt{\kappa}t_d) \leq 1$. As for $s$, the half-angle identity $1-\cos(x) = 2\sin^2(x/2)$ gives the equivalent expression $\sin(s) =  2 \cdot \frac{\sin^2(\sqrt{\kappa}\,t_b/2)}{\sin^2(\sqrt{\kappa}\,t_d/2)}-1$. Since $(t_b,t_d)$ satisfies inequality (\ref{ineq:D41_S2_boundary}), the right side is bounded below by 0 and, since $t_b < t_d \leq \frac{\pi}{\sqrt{\kappa}}$, it is also bounded above by 1. Thus, there exists an $s \in [0,\pi/2]$ that satisfies the equality. This finishes the proof of $P \subset \Dvr{4,1}(M_\kappa)$.\\
	\indent The proof for $\kappa<0$ proceeds in much the same way. The only major change is in the definition of the points $x_i$ when showing $P \subset \Dvr{4,1}(M_\kappa)$:
	\begin{multicols}{2}
	\begin{itemize}
		\item $x_1 =\left(\frac{1}{\sqrt{-\kappa}} \sqrt{1+t^2}, \frac{t}{\sqrt{-\kappa}}, 0\right)$
		\item $x_2 =\left(\frac{1}{\sqrt{-\kappa}} \sqrt{1+t^2}, \frac{t}{\sqrt{-\kappa}} \sin(s), \frac{t}{\sqrt{-\kappa}} \cos(s)\right)$
		\item $x_3 =\left(\frac{1}{\sqrt{-\kappa}} \sqrt{1+t^2}, -\frac{t}{\sqrt{-\kappa}}, 0\right)$
		\item $x_4 =\left(\frac{1}{\sqrt{-\kappa}} \sqrt{1+t^2}, -\frac{t}{\sqrt{-\kappa}} \sin(s), -\frac{t}{\sqrt{-\kappa}} \cos(s)\right)$.
	\end{itemize}
	\end{multicols}
	Other than that, and the fact that $M_\kappa$ is unbounded when $\kappa<0$, the proof is completely analogous.
\end{proof}

\begin{remark}\label{rmk:ph-detects-curvature}
	A related result appears in \cite{ph-detects-curvature}. The authors explore the question of whether persistent homology can detect the curvature of the ambient $M_\kappa$. On the theoretical side, they found a geometric formula to compute the \v{C}ech persistence diagram $\dgm_1^\text{\v{C}ech}(T)$ of a sample $T \subset M_\kappa$ with \emph{three} points, much in the same vein as our Theorem \ref{thm:n=2k+2}. They used it to find the logarithmic persistence $P_a(\kappa) := t_d(T_{\kappa,a})/t_b(T_{\kappa,a})$ for an equilateral triangle $T_{\kappa,a}$ of fixed side length $a>0$, and proved that $P_a$, when viewed as a function of $\kappa$, is invertible. On the experimental side, they sampled 1000 points from a unit disk in $M_\kappa$ and were able to approximate $\kappa$ using, among other things, average persistence landscapes in dimension 1 of 100 such samples. For example, one method consisted in finding a collection of landscapes $L_\kappa$ labeled with a known curvature $\kappa$, and estimating $\kappa_*$ for an unlabeled $L_*$ with the average curvature of the three nearest neighbors of $L_*$. They were also able to approximate $\kappa_*$ without labeled examples by using PCA. See their paper \cite{ph-detects-curvature} for more details. Compare with Figure \ref{fig:D41_Mk}.\\
	\indent Our Theorem \ref{thm:D41_Mk} is in the same spirit. The curvature value $\kappa$ determines the boundary of $\Dvr{4,1}(M_\kappa)$, and instead of triangles, we use squares with a given $t_d$ and minimal $t_b$ to find $\kappa$. Additionally, we can qualitatively detect the sign of the curvature by looking at the boundary of $\Dvr{4,1}(M_\kappa)$: it is concave up when $\kappa>0$, a straight line when $\kappa=0$, and concave down when $\kappa<0$. See Figure \ref{fig:D41_Mk_boundary}.
\end{remark}

\begin{figure}[ht]
	\begin{center}
		\includegraphics[scale=0.75]{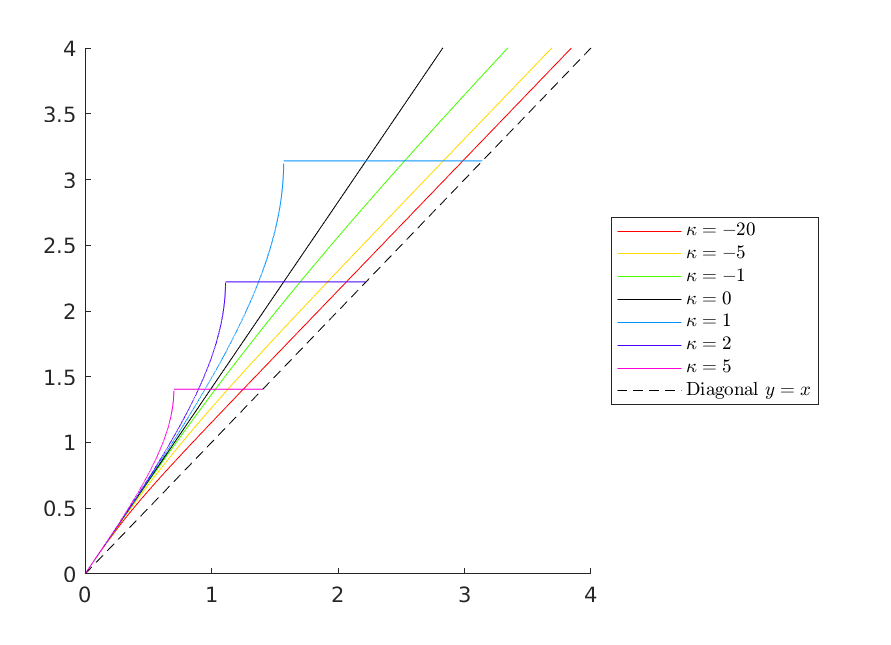}
	\end{center}
	\caption{The boundary of $\Dvr{4,1}(M_\kappa)$ for multiple $\kappa$ (see Theorem \ref{thm:D41_Mk}). Observe this set is bounded only when $\kappa>0$, and that the left boundary of these persistence sets is concave up when $\kappa>0$, a straight line when $\kappa=0$, and concave down when $\kappa>0$.}
	\label{fig:D41_Mk_boundary}
\end{figure}

\begin{figure}
	\begin{center}
		\noindent\makebox[\textwidth]{
			\includegraphics[scale=0.6]{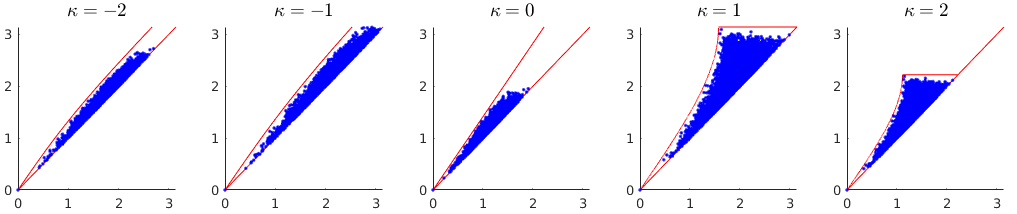}
		}
	\end{center}
	\caption{The diagrams $\Dvr{4,1}(D_\kappa)$ for disks $D_\kappa \subset M_\kappa$ of radius $R=\pi/\sqrt{|\kappa|}$ for various $\kappa \neq 0$ (compare with Theorem \ref{thm:D41_Mk}). Also shown is $\Dvr{4,1}(D_0)$ for $D_0 \subset M_0$, a disk of radius 1.}
	\label{fig:D41_Mk}
\end{figure}

\subsection{Persistence sets of spheres}
After surfaces, the next case we study is higher dimensional Euclidean spheres. Observe that if $n \leq m$, an $n$-point subset of $\Sphere{m}_E$ is contained in a sphere with smaller dimension. Hence, the computation of the persistence sets of spheres can be reduced to a specific dimension that depends on $n$. After proving this result and giving an example, we comment on the first unknown case.

\begin{prop}\label{prop:Dk_Sn_stabilizes}
	For all $m \geq n-1$ and all $k \geq 0$,
	\begin{equation*}
		\Dvr{n,k}(\Sphere{m}_E) = \Dvr{n,k}(\Sphere{n-1}_E) = \bigcup_{\lambda \in [0,1]} \lambda \cdot \Dvr{n,k}(\Sphere{n-2}_E).
	\end{equation*}
\end{prop}
\begin{proof}
	$\Sphere{m}_E$ contains copies of $\lambda \cdot \Sphere{n-2}_E$ for $\lambda \in [0,1]$, so $\bigcup_{\lambda \in [0,1]} \lambda \cdot \Dvr{n,k}(\Sphere{n-2}_E) \subset \Dvr{n,k}(\Sphere{m}_E)$. For the other direction, notice that a set $X \subset \Sphere{m}_E \subset \R^{m+1}$ with $n$ points generates an $(n-1)$-hyperplane which intersects $\Sphere{m}_E$ on a $(n-2)$-dimensional sphere of radius $\lambda \leq 1$. Thus, $X \subset \lambda \cdot \Sphere{n-2}_E$, so $\Dvr{n,k}(\Sphere{m}_E) \subset \bigcup_{\lambda \in [0,1]} \lambda \cdot \Dvr{n,k}(\Sphere{n-2}_E)$.
\end{proof}

\indent If $n=4$, the above proposition reduces the computation of $\Dvr{4,1}(\Sphere{m}_E)$ to the union of rescalings of $\Dvr{4,1}(\Sphere{2}_E)$. However, as seen in the proof of Proposition \ref{prop:D41_R2_S2}, $\Dvr{4,1}(\Sphere{2})$ is itself $\bigcup_{\lambda \in [0,1]} \lambda \cdot \Dvr{4,1}(\Sphere{1}_E)$. This observation extends the above result to $\Sphere{n-2}_E$ instead of $\Sphere{n-1}_E$.
\begin{corollary}\label{cor:D41_Sm_stabilizes}
	For all $m \geq 2$,
	\begin{align*}
		\Dvr{4,1}(\Sphere{m}_E) & = \Dvr{4,1}(\Sphere{2}_E) 
		= \left\{ (t_b,t_d)| 0 \leq t_b < t_d \leq \min(\sqrt{2}t_b, \pi) \right\}.
	\end{align*}
\end{corollary}
\begin{proof}
	By Proposition \ref{prop:Dk_Sn_stabilizes}, for every $m \geq 3$, $\Dvr{4,1}(\Sphere{m}_E) = \bigcup_{\lambda \in [0,1]} \lambda \cdot \Dvr{4,1}(\Sphere{2}_E)$, and by the proof of Proposition \ref{prop:D41_R2_S2}, $\Dvr{4,1}(\Sphere{2}_E)$ is convex and equals $\bigcup_{\lambda \in [0,1]} \lambda \cdot \Dvr{4,1}(\Sphere{1}_E)$. Hence, $\bigcup_{\lambda \in [0,1]} \lambda \cdot \Dvr{4,1}(\Sphere{2}_E) = \Dvr{4,1}(\Sphere{2}_E)$.
\end{proof}

The reason why the dimension of the spheres was reduced between Proposition \ref{prop:Dk_Sn_stabilizes} and Corollary \ref{cor:D41_Sm_stabilizes} comes from two facts. First, when $X \subset \R^m$ is a 4-point set, the quotient $\td{X}/\tb{X}$ is maximized when $X$ is concyclic (see Remark \ref{rmk:D41-ptolemy-equality}), and second, $\Sphere{2}_E$ contains all circles of radius $0 \leq \lambda \leq 1$. Thus, to bound $\td{X}/\tb{X}$ for $X \subset \Sphere{m}$, we take a concyclic $X'$ such that $\td{X}/\tb{X} \leq \sqrt{2} = \td{X'}/\tb{X'}$. Since $\Sphere{2}$ contains enough circles, the information in $\Dvr{4,1}(\Sphere{2}_E)$ is sufficient to determine $\Dvr{4,1}(\Sphere{m})$. We are curious to see how far can this technique be pushed. More specifically, we are interested in finding inequalities in higher dimensions that will play the same role as Ptolemy's has, namely, provide bounds for $\Dvr{n,k}(\Sphere{m})$ and whose equality condition happens in a sphere of dimension lower than $n-2$. If such an inequality existed, we could improve the equality $\Dvr{n,k}(\Sphere{m}) = \Dvr{n,k}(\Sphere{n-1})$ in Proposition \ref{prop:Dk_Sn_stabilizes} to a lower dimensional sphere in the same way as we did in Corollary \ref{cor:D41_Sm_stabilizes}.\\
\indent At this point, we have characterized $\Dvr{2k+2,k}(\Sphere{m}_E)$ for $k=1$ and any $m$. For $k=2$, the first case $\Dvr{6,2}(\Sphere{1}_E)$ can be obtained from Theorem \ref{thm:critical_tb_even}. This is the extent of our knowledge of the sets $\Dvr{6,2}(\Sphere{m})$. We now discuss our partial results for $m=2$.

\begin{lemma}\label{lemma:polynomial_rho_0}
	Fix $\sqrt{2/3} \leq r \leq 1$ and define $f(\rho) := 2 - r\rho + 2\sqrt{(1-r^2)(1-\rho^2)} - 3r^2$. The equation $f(\rho)=0$ has a unique solution $\rho_0 := r \cdot \frac{8-9r^2}{4-3r^2}$ that satisfies $-r < \rho_0 \leq r$.
\end{lemma}
\begin{proof}
	Isolating the square root and squaring the resulting equation gives $-4(1-r^2)\rho^2 + 4(1-r^2) = r^2 \rho^2 +2(3r^2-2) r\rho + (3r^2-2)^2$. After reordering terms and simplifying, we get
	\begin{equation*}
		g(\rho) := (4-3r^2) \rho^2 + (6r^3-4r) \rho + (9r^4 - 8r^2) = 0.
	\end{equation*}
	Observe that $g(-r)=0$, and that $g(\rho)/(\rho+r) = (4-3r^2)\rho + (9r^2-8)r$. Hence, $g(\rho)$ has the solutions $\rho=-r$ and $\rho = \rho_0$, where $\rho_0 := r \cdot \frac{8-9r^2}{4-3r^2}$. However, $\rho=-r$ is not a solution of $f(\rho)=0$ because $f(-r) = 4+6r^2 > 0$. Still, $f(r)=4-6r^2 \leq 0$ because $r \geq \sqrt{2/3}$, so since $f$ is continuous, $f(\rho)=0$ must have a solution $-r < \rho \leq r$. This solution must be $\rho=\rho_0$.
\end{proof}
\begin{prop}\label{prop:D62_S2_E}
	Let $P_{6,2}$ be
	\begin{equation*}
	    P_{6,2} := \left\{ (t_b,t_d) \mid 0 < t_b < t_d \leq 2 \text{ and either } t_d \leq \frac{2}{\sqrt{3}} t_b \text{ or } 4t_b^2 \cdot \dfrac{3-t_b^2}{4-t_b^2} \leq t_d^2 \right\}.
	\end{equation*}
	Then $P_{6,2} \subset \Dvr{6,2}(\Sphere{2}_E)$.
\end{prop}
\begin{proof}
	We begin by noting the lines $t_d = \frac{2}{\sqrt{3}}t_b$ and $t_d=2$ intersect at $t_b=\sqrt{3}$. Then, $P_{6,2}$ splits as the union of the sets
	\begin{align*}
		A &:=\textstyle \left\{ (t_b,t_d) \mid 0 < t_b \leq \sqrt{3} \text{ and } t_d \leq \frac{2}{\sqrt{3}} t_b \right\},\\
		B &:= \left\{ (t_b,t_d) \mid \sqrt{3} \leq t_b < t_d \leq 2 \right\} \text{ and }\\
		C &:= \left\{ (t_b,t_d) \mid \sqrt{2} \leq t_b \leq \sqrt{3}, t_b < t_d \leq 2 \text{ and } 4t_b^2 \cdot \dfrac{3-t_b^2}{4-t_b^2} \leq t_d^2 \right\}.
	\end{align*}
	We will show that the configurations $X \subset \Sphere{2}_E$ that generate $A \cup B$ are the sets $X$ inscribed in a circle of radius $r \in (0,1] \subset \Sphere{2}_E$. $C$ will take a bit more work, but we will show that it is generated by equilateral triangles inscribed at parallel circles of controlled radii.\\
	\indent By Theorem \ref{thm:critical_tb_even}, $\Dvr{6,2}(\Sphere{1}) = \{(t_b,t_d) \mid \frac{2\pi}{3} \leq t_b < t_d \leq \pi\}$ and, analogously to Proposition \ref{prop:D41_S1_E} (recall that $f_E(t) = 2\sin(t/2)$),
	\begin{equation*}
		\Dvr{6,2}(\Sphere{1}_E) = f_E\left( \Dvr{6,2}(\Sphere{1}) \right) = \{(t_b,t_d) \mid \sqrt{3} \leq t_b < t_d \leq 2\} = B.
	\end{equation*}
	Since $\Sphere{1}_E \hookrightarrow \Sphere{2}_E$, $B \subset \Dvr{6,2}(\Sphere{2}_E)$. More generally, we have $\bigcup_{0 < \lambda \leq 1} \lambda \cdot \Dvr{6,2}(\Sphere{1}_E) \subset \Dvr{6,2}(\Sphere{2}_E)$ because $\lambda \cdot \Sphere{1}_E \hookrightarrow \Sphere{2}_E$ for $0 < \lambda \leq 1$. Since $\bigcup_{0 < \lambda \leq 1} \lambda \cdot \Dvr{6,2}(\Sphere{1}_E) = A \cup B$, we have $A \cup B \subset \Dvr{6,2}(\Sphere{2}_E)$.\\
	\indent Now we show $C \subset \Dvr{6,2}(\Sphere{2}_E)$. Given $\sqrt{2/3} \leq r \leq 1$, let $\rho_0 := r \cdot \frac{8-9r^2}{4-3r^2}$, and $\max(0,\rho_0) \leq \rho \leq r$. Let $X = \{x_1, ..., x_6\} \subset \Sphere{2}_E$, where each $x_i$ has azimuthal angle $\frac{2\pi}{6} (i-1)$, the points $x_1, x_2, x_3$ are at height $\sqrt{1-r^2}$, and $x_4, x_5, x_6$ are at height $-\sqrt{1-\rho^2}$. More explicitly,
	\begin{multicols}{2}
		\begin{itemize}
			\item $x_1 =\left(r, 0, \sqrt{1-r^2}\right)$,
			\item $x_2 =\left(-\frac{1}{2}r, \phantom{-}\frac{\sqrt{3}}{2} r, \sqrt{1-r^2}\right)$,
			\item $x_3 =\left(-\frac{1}{2}r, -\frac{\sqrt{3}}{2} r, \sqrt{1-r^2}\right)$,
			\item $x_4 =\left(-\rho, 0, -\sqrt{1-\rho^2}\right)$,
			\item $x_5 =\left( \frac{1}{2}\rho, -\frac{\sqrt{3}}{2} \rho, -\sqrt{1-\rho^2}\right)$,
			\item $x_6 =\left( \frac{1}{2}\rho, \phantom{-}\frac{\sqrt{3}}{2} \rho, -\sqrt{1-\rho^2}\right)$.
		\end{itemize}
	\end{multicols}
	\noindent We can verify that
	\begin{equation*}
		d_{ij}^2 =
		\begin{cases}
			3r^2 & \text{ if } i \neq j \in \{1,2,3\}, \\
			3\rho^2 & \text{ if } i \neq j \in \{4,5,6\}, \\
			2 - r\rho + 2\sqrt{(1-r^2)(1-\rho^2)} & \text{ if } i \in \{1,2,3\}, j \in \{4,5,6\}, j \neq i+3, \text{ and } \\
			2 + 2r\rho + 2\sqrt{(1-r^2)(1-\rho^2)} & \text{ if } j=i+3.
		\end{cases}
	\end{equation*}
	Given $(t_b, t_d) \in C$, we claim there exist $\sqrt{2/3} \leq r \leq 1$ and $\rho_0 \leq \rho \leq r$ such that $t_b^2 = 3r^2$, $t_d^2 = 2 + 2r\rho + 2\sqrt{(1-r^2)(1-\rho^2)}$, $\tb{X}=t_b$, and $\td{X}=t_d$.\\
	\indent Finding $r$ is immediate. Since $\sqrt{2} \leq t_b \leq \sqrt{3}$, $r := t_b/\sqrt{3}$ satisfies $\sqrt{2/3} \leq r \leq 1$. To find $\rho$, define $g(\rho) := 2 + 2r\rho + 2\sqrt{(1-r^2)(1-\rho^2)} - t_d^2$. If we show that $g(\rho_0) \leq 0 \leq g(r)$, there will exist $\rho_0 \leq \rho \leq r$ such that $g(\rho)=0$. To wit, since $t_d \leq 2$,
	\begin{equation*}
		g(r) = 2+2r^2+2(1-r^2)-t_d^2 = 4-t_d^2 \geq 0.
	\end{equation*}
	For the other inequality, recall that $t_d^2 \geq 4t_b^2 \cdot \frac{3-t_b^2}{4-t_b^2}$, and that $f(\rho_0)=0$ by Lemma \ref{lemma:polynomial_rho_0}. Then
	\begin{align*}
		g(\rho_0)
		&= 2 + 2r\rho_0 + 2\sqrt{(1-r^2)(1-\rho_0^2)} - t_d^2 
		= f(\rho_0) +3r^2+3r\rho_0 - t_d^2 
		= 3r \left( r+ r \cdot \frac{8-9r^2}{4-3r^2} \right) - t_d^2 \\
		&= 3r^2 \cdot \frac{12-12r^2}{4-3r^2} - t_d^2
		= t_b^2 \cdot \frac{12-4t_b^2}{4-t_b^2} - t_d^2
		= 4t_b^2 \cdot \frac{3-t_b^2}{4-t_b^2} - t_d^2 \leq 0,
	\end{align*}
	as desired.\\
	\indent The remaining facts to verify are $\tb{X}=t_b$ and $\td{X}=t_d$. By definition of $C$ and the previous paragraph, we have $3r^2 = t_b < t_d = 2 +2r\rho + 2\sqrt{(1-r^2)(1-\rho^2)}$. Thus, if we can show
	\begin{equation*}
		\max\left(3\rho^2, 2 - r\rho + 2\sqrt{(1-r^2)(1-\rho^2)} \right) \leq 3r^2,
	\end{equation*}
	we will have $\tb{x_i} = 3r^2$ and $\td{x_i} = 2 +2r\rho + 2\sqrt{(1-r^2)(1-\rho^2)}$ for all $i=1, \dots, 6$. The inequality $3\rho^2 \leq 3r^2$ is immediate from the assumption $\rho \leq r$. For the second inequality, recall that the function $f(\rho)$ from Lemma \ref{lemma:polynomial_rho_0} has a unique zero at $\rho=\rho_0$ and that $f(r) \leq 0$. Hence, since $f$ is continuous, we have $f(\rho) \leq 0$ for any $\rho_0 \leq \rho \leq r$. Thus, $2 - r\rho + 2\sqrt{(1-r^2)(1-\rho^2)} \leq 3r^2$.\\
	\indent In conclusion, for every $(t_b, t_d) \in C$, we found $X \subset \Sphere{2}_E$ with $|X|=6$ such that $\tb{X}=t_b$ and $\td{X}=t_d$. Hence, $C \subset \Dvr{6,2}(\Sphere{2}_E)$. Together with the case of $A \cup B$, we have $P_{6,2} = A \cup B \cup C \subset \Dvr{6,2}(\Sphere{2}_E)$.
\end{proof}
\noindent
\begin{minipage}{0.5\linewidth}
$P_{6,2}$ is shown in blue in Figure \ref{fig:D62_S2}. It is generated by two parallel equilateral triangles inscribed in $\Sphere{2}_E$. We haven't been able to prove that $\Dvr{6,2}(\Sphere{2}_E) = P_{6,2}$, but we have strong experimental evidence. We first sampled 4.5 million configurations uniformly at random from $\Sphere{2}_E$ and retained only the 88,708 configurations that produced non-trivial persistence (1.9713 \% of the total samples). This produces a set $\mathbf{D}_\text{unif}$ of persistence diagrams, which are shown in green in Figure \ref{fig:D62_S2}. The second step was a biased MCMC random walk. The Metropolis-Hasting MCMC starts with a choice of parameter $\sigma^2$, an initial configuration $X_0$, and a set $\mathbf{D}_0=\mathbf{D}_\text{unif}$ \cite{mcmc-conceptual}. We additionally fix a radius $\varepsilon$. At each step $t$, we obtain $X_{t-1}^{\sigma^2}$ by perturbing the previous configuration $X_{t-1}$ with Gaussian noise of variance $\sigma^2$. Let $D_{t-1}$ and $D_{t-1}^{\sigma^2}$ be the persistence diagrams of $X_{t-1}$ and $X_{t-1}^{\sigma^2}$, respectively. We then compute the cardinalities $N_{\text{pre}} = |B_\epsilon(D_{t-1}) \cap \mathbf{D}_{t-1}|$ and $N_{\text{post}} = |B_\epsilon(D_{t-1}^{\sigma^2}) \cap \mathbf{D}_{t-1}|$.
\end{minipage}
\hfill
\begin{minipage}{0.45\linewidth}
	\centering
	\includegraphics[scale=0.48]{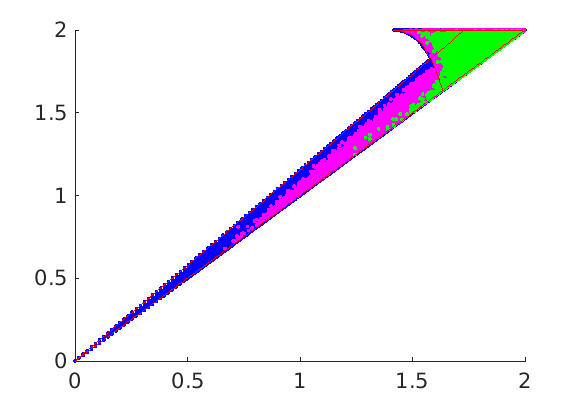}
	\captionof{figure}{The set $P_{6,2} \subset \Dvr{6,2}(\Sphere{2}_E)$ described in Proposition \ref{prop:D62_S2_E}. The green points were generated with a uniform sample of sets with 6 points. Under that, the magenta points were generated with an MCMC random walk. The blue points were generated by the vertices of two parallel equilateral triangles inscribed in $\Sphere{2}_E$.}
	\label{fig:D62_S2}
\end{minipage}

The balls of radius $\varepsilon$ are defined with the bottleneck distance. In the next step is where we diverge from the usual algorithm. Normally, we would accept the new configuration $X_{t-1}^{\sigma^2}$ with probability\footnote{Formally, the probability of acceptance is $Q(D_{t-1}^{\sigma^2})/Q(D_{t-1})$, where $Q$ is a function on $\Dvr{6,2}(\Sphere{2})$ proportional to the measure $\Unk{6,2}{\vr}(\Sphere{2})$. We are taking $Q(D)$ to be the number of diagrams in an $\varepsilon$-ball around $D$.} $\min(1,N_\text{post}/N_\text{pre})$. Eventually, the distribution of persistence diagrams in $\mathbf{D}_t$ would approximate the distribution that we are sampling from, that is, $\Unk{6,2}{\vr}(\Sphere{2}_E)$. However, sampling uniformly from $\Sphere{2}_E$ also produces diagrams with that distribution and this method did not produce points close to the boundary of $\Dvr{6,2}(\Sphere{2})$. Instead, we accept $X_{t-1}^{\sigma^2}$ with probability $\min(1,N_\text{pre}/N_\text{post})$, and set $X_t := X_{t-1}^{\sigma^2}$ and $\mathbf{D}_t := \mathbf{D}_{t-1} \cup \{D_{t-1}^{\sigma^2}\}$. This causes the random walk to diverge from the diagrams that already are in $\mathbf{D}_\text{unif}$ and produces configurations closer to the boundary of $\Dvr{6,2}(\Sphere{2}_E)$. The diagrams produced by the random walk are colored in magenta in Figure \ref{fig:D62_S2}. This figure suggests that there are no points outside of $P_{6,2}$.

\begin{conjecture}\label{conj:D62_S2_E}
	$\Dvr{6,2}(\Sphere{2}_E) = P_{6,2}$.
\end{conjecture}

\subsection{Principal persistence sets can differentiate spheres}
Any non-diagonal point $(t_b, t_d) \in \Dvr{2k+2,k}(X)$ corresponds to a subset $A \subset X$ coinciding with the vertex set of a cross-polytope inscribed in $X$. If, in addition $t_d = 2t_b$, then the cross-polytope must be regular, as Lemma \ref{lemma:technicals_S1} item \ref{item:td_basic_bound} shows. For example, $\Sphere{m}$ admits a particular inscribed regular cross-polytope depending on the dimension $m$. It turns out that principal persistence sets can pick up this difference, and that is enough to tell apart spheres of different dimensions.

\begin{prop}
	\label{prop:cross_polytope_in_sphere}
	$(\pi/2, \pi) \in \Dvr{2k+2,k}(\Sphere{m})$ if and only if $1 \leq k \leq m$.
\end{prop}
\begin{proof}
	Let $k \geq 1$. Suppose that a set $X = \{x_1, \dots, x_{2k+2}\} \subset \Sphere{m}$ satisfies $\tb{X} = \pi/2$ and $\td{X} = \pi$, and label the points so that $\vdeath(x_i) = x_{i+k+1}$. By Lemma \ref{lemma:persistence_bounds} item \ref{item:td_basic_bound}, we must have $d_{\Sphere{m}}(x_i, x_j) = \tb{X} = \pi/2$ for all $j \neq i+k+1$ and $d_{\Sphere{m}}(x_i, x_{i+k+2}) = \td{X} = \pi$ for all $i$. The fact that $d_{\Sphere{m}}(x_i, x_j) = \pi/2 = \arccos \langle x_i, x_j \rangle$ means that $x_1, \dots, x_{k+1}$ are mutually orthogonal and, hence, linearly independent. This forces $k \leq m$. Conversely, for any $1 \leq k \leq m$, we can construct a set of mutually orthogonal vectors $x_1, \dots, x_{k+1} \in \Sphere{m}$ by setting $x_i$ as, for instance, the $i$-th standard basis vector of $\R^{m+1}$. In that case, $X := \{\pm x_1, \dots, \pm x_{k+1}\}$ has $2k+2$ points and satisfies $\tb{X} = \pi/2$ and $\td{X} = \pi$.
\end{proof}

\begin{remark}[Principal persistence sets and fundamental classes of spheres]
	\label{rmk:cross_polytope_in_sphere}
	The point $(\pi/2, \pi) \in \Dvr{2m+2,m}(\Sphere{m})$ is generated by a regular cross-polytope $X \in \Sphere{m}$ with $2m+2$ vertices. It is interesting to note that the $m$-simplices of $\vr_r(X)$, when $\pi/2 \leq r < \pi$, determine an $m$-chain that represents the fundamental class $[\Sphere{m}]$.
\end{remark}

\begin{remark}[Distances between persistence sets can distinguish spheres]
	\label{rmk:dendrogram_sphers}
	For $m=1, \dots, 5$ and $k=1,\dots,5$, we computed an approximation $D_{k}(\Sphere{m})$ of the principal persistence set $\Dvr{2k+2,k}(\Sphere{m})$ by sampling $10^5$ configurations of $2k+2$ points uniformly at random from $\Sphere{m}$. Then, for each $k$, we computed the Hausdorff distance induced by the bottleneck distance for all $1 \leq i,j \leq 5$ which we denote  by $d_{k}(\Sphere{i}, \Sphere{j}) := \dH^{\D}(D_{k}(\Sphere{i}), D_{k}(\Sphere{j}))$. Analogously to Definition \ref{def:modifiedGH}, we set $d(\Sphere{i}, \Sphere{j}) := \max_{k} d_{k}(\Sphere{i}, \Sphere{j})$. Lastly, we computed the single-linkage hierarchical clustering; the resulting dendrogram is shown in Figure \ref{fig:dendrogram_spheres} and it indicates that principal persistence sets can discriminate these 5 spheres.
\end{remark}

\begin{figure}
	\centering
	\includegraphics[scale=0.60]{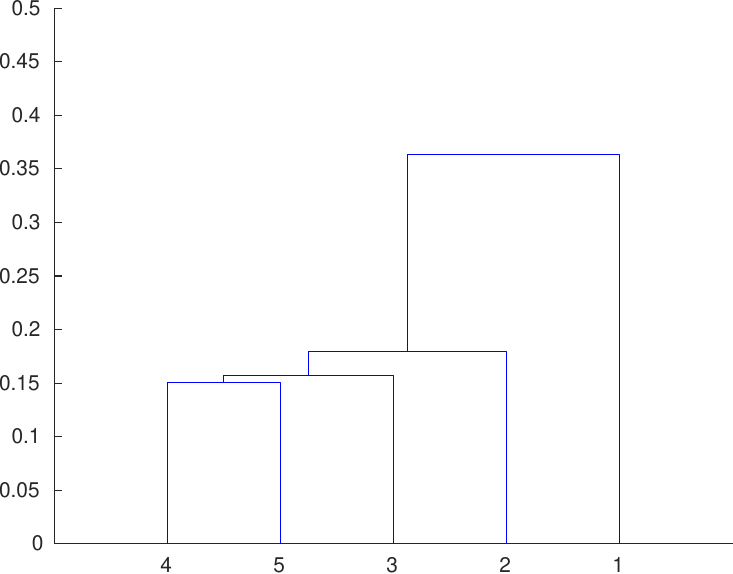}
	\caption{The dendrogram induced by the distances $d(\Sphere{i}, \Sphere{j})$ in Remark \ref{rmk:dendrogram_sphers}, for $1 \leq i,j \leq 5$.}
	\label{fig:dendrogram_spheres}
\end{figure}

\subsubsection{Lower bounds for $\dGH(\Sphere{1},\Sphere{m})$}
\label{sec:dGH_spheres}
In this section, we use information about the persistence sets of the spheres, together with the stability in Theorem \ref{thm:stability_Dnk}, to find lower bounds for the Gromov-Hausdorff distance between the circle and other spheres.
\begin{example}\label{ex:D_GH_S1_S2}
	Since $\Dvr{4,1}(\Sphere{1}) \subset \Dvr{4,1}(\Sphere{2})$,
	\begin{equation*}
		\dH^\D(\Dvr{4,1}(\Sphere{1}), \Dvr{4,1}(\Sphere{2})) = \sup_{D_2 \in \Dvr{4,1}(\Sphere{2})} \inf_{D_1 \in \Dvr{4,1}(\Sphere{1})} \dB(D_1,D_2).
	\end{equation*}
	Fix a diagram $D_2 = (x_2,y_2) \in \Dvr{4,1}(\Sphere{2}) \setminus \Dvr{4,1}(\Sphere{1})$ and take $D_1 = (x_1,y_1) \in \Dvr{4,1}(\Sphere{1})$ arbitrary. The distance $\dB(D_1,D_2)$ can be realized by either the $\ell^\infty$ distance between $D_1$ and $D_2$ or by half the persistence of either diagram,
	so in order to minimize $\dB(D_1,D_2)$, let's start by finding the minimum of $\|D_1-D_2\|_\infty = \max(|x_1-x_2|, |y_1-y_2|)$.\\
	\indent Clearly, this distance is smallest when $D_1$ is on the line $\ell$ with equation $y=2(\pi-x)$ (case $k=1$ in Theorem \ref{thm:critical_tb_odd}). Additionally, the maximum is minimized when $|x_1-x_2| = |y_1-y_2|$. If both conditions can be achieved, we will have minimized the $\ell^\infty$ distance. The only possibility, though, is $x_2 \leq x_1$ and $y_2 \leq y_1$ (if either inequality is reversed, the $\ell^\infty$ distance would be larger because $\ell$ has negative slope). In that case, the solutions to the system of equations $x_1-x_2 = y_1-y_2$ and $y_1=2(\pi-x_1)$ are $x_1 = \frac{1}{3}(2\pi+x_2-y_2)$ and $y_1 = \frac{2}{3}(\pi-x_2+y_2)$. Thus,
	\begin{equation*}
		d_{\ell^\infty}(D_2,\ell) = \frac{1}{3}(2\pi-2x_2-y_2).
	\end{equation*}
	This quantity is positive because $x_2,y_2$ is below $\ell$, that is, $y_2 \leq 2\pi-2x$.\\
	\indent Now fix $D_1$ as the solution described in the previous paragraph and let $D_2$ vary. The distance $\dB(D_1,D_2)$ can be equal to $\frac{1}{2}\pers(D_i)$ if that quantity is larger than $d_{\ell^\infty}(D_2,\ell)$ for either $i=1,2$. Notice, also, that $\pers(D_1) = \pers(D_2)$ because $x_1-x_2 = y_1-y_2$. If we can find $D_2$ such that
	\begin{equation}\label{eq:pers=d_infty}
		\frac{1}{2}\pers(D_2) = d_{\ell^\infty}(D_2,\ell),
	\end{equation}
	then the maximum will have been achieved. Equation (\ref{eq:pers=d_infty}) can be simplified to $y_2 = -\frac{1}{5}x_2 + \frac{4\pi}{5}$. The point $D_2 = (x_2,y_2)$ that realizes the Hausdorff distance will be in the intersection of this line and $\Dvr{4,1}(\Sphere{2})$ and have maximal persistence. That is achieved in the intersection with the left boundary, the curve $x = 2\arcsin\left( \frac{1}{\sqrt{2}} \sin\left(\frac{y}{2} \right) \right)$ (use $\kappa=1$ in Theorem \ref{thm:D41_Mk}). That point is $x_2 \approx 1.3788, y_2 = 2.2375$ (see Figure \ref{fig:dH_D41_S1_S2}) and will give $\dH^\D(\Dvr{4,1}(\Sphere{1}), \Dvr{4,1}(\Sphere{2}) \approx 0.4293$. Thus,
	\begin{equation*}
		\dGH(\Sphere{1}, \Sphere{2}) \geq \frac{1}{2} \dH^\D(\Dvr{4,1}(\Sphere{1}), \Dvr{4,1}(\Sphere{2})) \approx 0.2147 \approx \frac{\pi}{14.6344}.
	\end{equation*}
\end{example}

\begin{figure}
	\centering
	\includegraphics[scale=0.5]{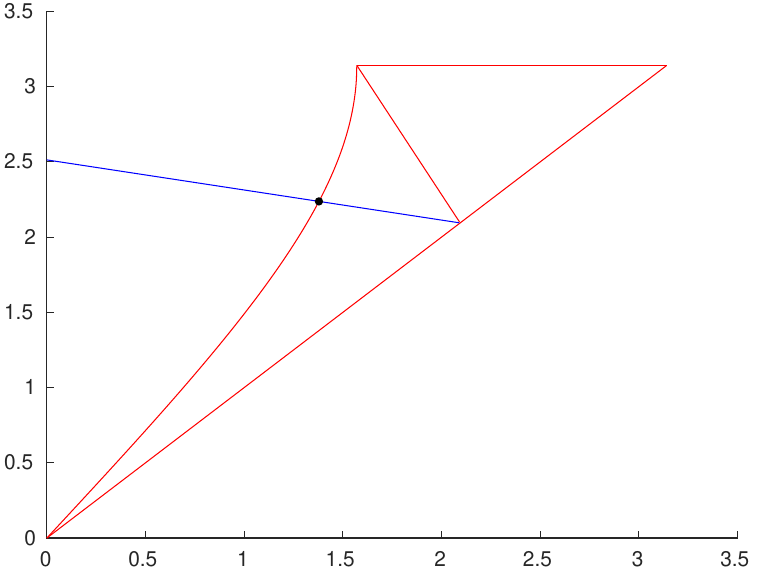}
	
	\caption{
		The point $D_2$ that realizes the Hausdorff distance between $\Dvr{4,1}(\Sphere{1})$ and $\Dvr{4,1}(\Sphere{2})$ with respect to the bottleneck distance (see Example \ref{ex:D_GH_S1_S2}). The shaded region is $\Dvr{4,1}(\Sphere{1})$ and the black lines outline $\Dvr{4,1}(\Sphere{2})$. The blue line is $y_2 = -\frac{1}{5}x_2 + \frac{4\pi}{5}$, the region where $\frac{1}{2}\pers(D_2) = d_{\ell^\infty}(D_2,\ell)$, and $\ell$ is the line $y=2(\pi-x) \subset \partial(\Dvr{4,1}(\Sphere{1}))$.
	}
	\label{fig:dH_D41_S1_S2}
\end{figure}

We can obtain a better bound when $k \geq 3$.
\begin{example}\label{ex:D_GH_S1_Sm}
	Let $n=2k+2$; we a seek lower bound for $\dGH(\Sphere{1}, \Sphere{k})$ for $k \geq 3$. First, similarly to Example \ref{ex:D_GH_S1_S2}, we have
	\begin{equation*}
		\dH^\D(\Dvr{n,k}(\Sphere{1}), \Dvr{n,k}(\Sphere{k})) = \sup_{D_2 \in \Dvr{n,k}(\Sphere{k})} \inf_{D_1 \in \Dvr{n,k}(\Sphere{1})} \dB(D_1,D_2).
	\end{equation*}
	We now exhibit a configuration, more specifically, a cross-polytope $X \subset \Sphere{k}$, in order to fix a specific diagram $D_k$. Let $X = \{\pm e_1, \dots, \pm e_{k+1}\} \subset \R^{k+1}$, where $e_i$ is the $i$-th standard basis vector. Notice that $d_{\Sphere{k}}(\pm e_i, \pm e_j) = \frac{\pi}{2}$ if $j \neq i$, and $d_{\Sphere{k}}(e_i, -e_i) = \pi$. Then $\tb{e_i} = \tb{-e_i} = \frac{\pi}{2}$ and  $\td{e_i} = \td{-e_i} = \pi$, so $\tb{X} = \frac{\pi}{2}$ and $\td{X} = \pi$. Since $X$ has $2k+2$ points, we just proved that $D_k = (\frac{\pi}{2}, \pi) \in \Dvr{n,k}(\Sphere{k})$. Then 
	\begin{equation*}
		\dH^\D(\Dvr{n,k}(\Sphere{1}), \Dvr{n,k}(\Sphere{k})) \geq \inf_{D_1 \in \Dvr{n,k}(\Sphere{1})} \dB(D_1,D_k).
	\end{equation*}
	For concreteness, write $D_1 = \{(x,y)\}$. Let $\varphi:D_1 \to D_k$ be the unique bijection. By Lemma \ref{lemma:technicals_S1}, $x \geq \frac{k}{k+1} \pi$, so
	\begin{equation*}
		J(\varphi) = \left\| \left(\frac{\pi}{2}, \pi\right) - (x,y) \right\|_{\infty} \geq x-\frac{\pi}{2} \geq \frac{k-1}{2(k+1)} \pi.
	\end{equation*}
	On the other hand, since $y \leq \pi$, $\pers(D_1) = y-x \leq \frac{\pi}{k+1}$. Thus, for the empty matching $\emptyset:\emptyset \to \emptyset$, we have
	\begin{equation*}
		J(\emptyset) = \max\left( \frac{1}{2}\pers(D_1), \frac{1}{2} \pers(D_k) \right) = \frac{1}{2} \pers(D_k) = \frac{\pi}{4}.
	\end{equation*}
	Since $\frac{\pi}{4} \leq \frac{k-1}{2(k+1)} \pi$ whenever $k \geq 3$, we have $\dB(D_1,D_k) = \min_{\varphi} J(\varphi) = \frac{\pi}{4}$ for all $D_1 \in \Dvr{n,k}(\Sphere{1})$. Thus, by Theorem \ref{thm:stability_Dnk},
	\begin{equation*}
		\dGH(\Sphere{1}, \Sphere{k}) \geq \frac{1}{2} \dH^\D(\Dvr{n,k}(\Sphere{1}), \Dvr{n,k}(\Sphere{k})) \geq \frac{1}{2} \inf_{D_1 \in \Dvr{n,k}(\Sphere{1})} \dB(D_1,D_k) = \frac{\pi}{8}.
	\end{equation*}
\end{example}
\section{Concentration of persistence measures}
\label{sec:statistical_metric_motifs}
By paring $\Dnk{n,k}{\FiltFunc}(X)$ with the persistence measure $\Unk{n,k}{\FiltFunc}(X)$, we can view persistence sets as an mm-space
\begin{equation*}
	\Dw{n,k}{\FiltFunc}(X) := \left( \Dnk{n,k}{\FiltFunc}(X), \dB, \Unk{n,k}{\FiltFunc}(X) \right) \in \Mw,
\end{equation*}
where $\dB$ is restricted to pairs in $\Dnk{n,k}{\FiltFunc}(X) \times \Dnk{n,k}{\FiltFunc}(X)$.\\
\indent The main result in this section is that $\Dw{n,k}{\FiltFunc}(X)$ \textit{concentrates} to a one-point mm-space $*$ as $n \to \infty$. Since $*$ is generic, we also prove that the expected bottleneck distance between a random diagram $\mathbb{D} \in \Dnk{n,k}{\FiltFunc}(X)$ and $\dgm_k^\FiltFunc(X)$, the degree-$k$ persistence diagram of the space $X$, goes to 0 as $n \to \infty$, effectively showing that $\Dw{n,k}{\FiltFunc}(X)$ concentrates to $\dgm_k^\FiltFunc(X)$ when the latter is viewed as a one-point mm-space equipped with the trivial choices of metric and probability measure.

\begin{example}[The case of a space with two points]
	Fix $\varepsilon>0$ and $\alpha \in (0,1)$. Consider the metric space $X=\{p,q\}$ with two points at distance $\varepsilon$ and mass $\mu_X(p) = \alpha, \mu_X(q)=1-\alpha$. Let us first describe the elements of $K_n(X)$ for a fixed $n \in \N$. Let $x_1=\dots=x_k=p$ and $x_{k+1}=\dots=x_n=q$. The distance matrix of this set of points is
	\begin{equation*}
		M_k := \Psi_n(x_1, \dots, x_n) = 
		\left(\begin{array}{c|c}
			\mathbf{0}_{k \times k}		& \varepsilon \cdot \mathbf{1}_{k \times (n-k)} \\
			\hline
			\varepsilon \cdot \mathbf{1}_{(n-k) \times k}		& \mathbf{0}_{(n-k) \times (n-k)}
		\end{array} \right),
	\end{equation*}
	where $\mathbf{1}_{r \times s}$ is the $r \times s$ matrix with all entries equal to 1. Then, the non-zero matrices in $K_n(X)$ have the form $M_k^\Pi := \Pi^T \cdot M_k(\delta) \cdot \Pi$, where $\Pi \in S_n$ is a permutation matrix and $1 \leq k < n$. Also, let $M_0 := \mathbf{0}_{n \times n}$ be the zero matrix. Let $\mu_n$ be the curvature measure on $K_n(X)$, the measure that we get a particular distance matrix $M \in K_n(X)$ when randomly choosing $n$ points from $p$ and $q$ according to $\mu_X$. Observe that $w_n := \mu_n(M_0) = \alpha^n + (1-\alpha)^n$ since $M_0 = \Psi_n(p,\dots,p)$ and $M_0 = \Psi_n(q,\dots,q)$, while the rest of the mass $1-w_n$ is distributed among the non-zero matrices. Notice that $w_n \to 0$ as $n \to \infty$.\\
	\indent Now we describe the persistence set $\Dw{n,0}{\vr}(X)$. The measure $\Unk{n,0}{\vr}$ is supported on the two point set $\Dnk{n,0}{\vr}(X) = \{\mathbf{0}_\D, (0,\varepsilon)\}$, where $\mathbf{0}_\D$ is the empty persistence diagram. From the computations above, $\Unk{n,0}{\vr}(\mathbf{0}_\D)=w_n$ and $\Unk{n,0}{\vr}\big((0,\varepsilon)\big) = 1-w_n$.
	The fact that $w_n \to 0$ as $n \to \infty$ means that the measure $\Unk{n,0}{\vr}$ concentrates at the point $(0,\varepsilon)$ so, as an mm-space, $\Dw{n,0}{\vr}(X)$ converges to the 1-point mm-space $\{(0,\varepsilon)\} \subset \D$ equipped with the Dirac delta measure $\delta_{(0,\varepsilon)}$. This is the persistence diagram $\dgm_0^\vr(X)$ viewed as a 1-point mm-space.
\end{example}
We now generalize this result.

\subsection{A concentration theorem}
Let $\mm{X}$ be an mm-space. Using terminology from \cite[Section 5.3]{clust-um}, we define the function $f_X:\R^+ \to \R^+$ given by $\varepsilon \mapsto \inf_{x \in X} \mu_X(B_\varepsilon(x))$. Suppose that $f_X(\varepsilon)>0$ for every $\varepsilon>0$. Define also  $$C_X:\N \times \R_+ \to \R_+$$ given by
\begin{equation*}
	(n,\varepsilon) \mapsto \dfrac{e^{-nf_X(\varepsilon/4)}}{f_X(\varepsilon/4)}.
\end{equation*}

The authors used $C_X$ to formalize the intuition that, as $n$ increases, $n$-point subsets of $X$ should be close (in Hausdorff distance) to $X$ with high probability. The following theorem uses $C_X$ to bound the measure of the set $Q_X(n,\varepsilon)$ of samples that fail this condition.
\begin{theorem}[Covering theorem {\cite[Theorem 34]{clust-um}}]
	Let $\mm{X}$ be an mm-space. For a given $n \in \N$ and $\varepsilon > 0$, consider the set
	\begin{equation*}
		Q_X(n,\varepsilon) := \{ (x_1,\dots, x_n) \in X^{n}| \dH^X(\{x_i\}_{i=1}^n, X) > \varepsilon \}.
	\end{equation*}
	Then
	\begin{equation*}
		\mu_X^{\otimes n}(Q_X(n,\varepsilon)) \leq C_X(n,\varepsilon).
	\end{equation*}
\end{theorem}

We now prove our concentration result. Denote the expected value of a random variable $X$ distributed according to the probability measure $\mu$ with $E_\mu[X]$. Then:

\UnkConcentrates
\begin{proof}
	Fix $\varepsilon > 0$. Let $\mathbb{X}=(x_1,\dots,x_n) \in X^n$ be a random variable distributed according to $\mu_X^{\otimes n}$ and let $\mathbb{D} = \dgm_k^\FiltFunc\left(\Psi_X^{(n)}(\mathbb{X}) \right)$ be its persistence diagram. Since $\Unk{n,k}{\FiltFunc}(X)$ is the pushforward of the product measure $\mu_X^{\otimes n}$ under the map $\dgm_k^\FiltFunc \circ \Psi_X^{(n)}: X^n \to K_n(X) \to \D$, we can make a change of variables to rewrite the expected value of $\dB(\mathbb{D}, \dgm_k^\FiltFunc(X))$ as follows:
	\begin{align*}
		E_{\Unk{n,k}{\FiltFunc}(X)}\left[\dB\left(\mathbb{D}, \dgm_k^\FiltFunc(X) \right) \right]
		&= E_{\mu_X^{\otimes n}}\left[ \dB\left( \dgm_k^\FiltFunc\left[\Psi_X^{(n)}(\mathbb{X})\right], \dgm_k^\FiltFunc(X) \right) \right]\\
		&= \int_{X^n} \dB\left(\dgm_k^\FiltFunc\left[\Psi_X^{(n)}(\mathbb{X}) \right], \dgm_k^\FiltFunc(X) \right) \ \mu_X^{\otimes n}(d\mathbb{X}).
	\end{align*}
	
	By stability of $\FiltFunc$, the last integral is bounded above by
	\begin{align*}
		L(\FiltFunc) \int_{X^n} \dGH(\mathbb{X}, X) \ \mu_X^{\otimes n}(d\mathbb{X}) \leq
		L(\FiltFunc) \int_{X^n} \dH(\mathbb{X}, X) \ \mu_X^{\otimes n}(d\mathbb{X}),
	\end{align*}
	where, by abuse of notation, we see $\mathbb{X}$ as a subspace of $X$. In that case, $\dH(\mathbb{X}, X) = \rad_X(\mathbb{X}) \leq \diam(X)$, so we split the above integral into the sets $Q_X(n,\varepsilon)$ and $X^n \setminus Q_X(n,\varepsilon)$:
	\begin{align*}
		\int_{X^n} \dH(\mathbb{X}, X) \ \mu_X^{\otimes n}(d\mathbb{X})
		&= \int_{X^n} \rad_X(\mathbb{X}) \ \mu_X^{\otimes n}(d\mathbb{X})\\
		&= \int_{Q_X(n,\varepsilon)} \rad_X(\mathbb{X}) \ \mu_X^{\otimes n}(d\mathbb{X}) + \int_{X^n \setminus Q_X(n,\varepsilon)} \rad_X(\mathbb{X}) \ \mu_X^{\otimes n}(d\mathbb{X})\\
		&\leq \int_{Q_X(n,\varepsilon)} \diam(X) \ \mu_X^{\otimes n}(d\mathbb{X}) + \int_{X^n} \varepsilon \ \mu_X^{\otimes n}(d\mathbb{X})\\
		&= \diam(X) \cdot \mu_X^{\otimes n}(Q_X(n,\varepsilon)) + \varepsilon\\
		&< \diam(X) \cdot C_X(n,\varepsilon) + \varepsilon.
	\end{align*}
	This proves the first claim.\\
	\indent To show that $\Dw{n,k}{\FiltFunc}(X)$ concentrates to a point, we will show that $\dGW{1}\left(\Dw{n,k}{\FiltFunc}(X), *\right) \to 0$. For any mm-space $\mm{Z}$,
	\begin{equation*}
		\dGW{1}(Z,*) = \dfrac{1}{2} \iint_{Z \times Z} d_X(z,z')\, \mu_Z(dz) \,\mu_Z(dz').
	\end{equation*}
	Then, using the triangle inequality,
	\begin{align*}
		\dGW{1}\left(\Dw{n,k}{\FiltFunc}(X), *\right)
		&= \frac{1}{2} \iint_{\Dnk{n,k}{\FiltFunc}(X) \times \Dnk{n,k}{\FiltFunc}(X)} \dB(D,D') \ \Unk{n,k}{\FiltFunc}(dD)\ \Unk{n,k}{\FiltFunc}(dD')\\
		&\leq \frac{1}{2} \iint_{\Dnk{n,k}{\FiltFunc}(X) \times \Dnk{n,k}{\FiltFunc}(X)} \left[ \dB(D,\dgm_k^\FiltFunc(X)) + \dB(\dgm_k^\FiltFunc(X),D') \right] \ \Unk{n,k}{\FiltFunc}(dD)\ \Unk{n,k}{\FiltFunc}(dD')\\
		&= \int_{\Dnk{n,k}{\FiltFunc}(X)} \dB(D,\dgm_k^\FiltFunc(X)) \ \Unk{n,k}{\FiltFunc}(dD)\\
		&= E_{\Unk{n,k}{\FiltFunc}(X)}\left[\dB\left(\mathbb{D}, \dgm_k^\FiltFunc(X) \right) \right] \\
		&< \diam(X) \cdot C_X(n,\varepsilon) + \varepsilon.
	\end{align*}
	However, for any fixed $\varepsilon$, $C_X(n,\varepsilon) \to 0$ as $n \to \infty$. Thus, $E_{\Unk{n,k}{\FiltFunc}(X)}\left[\dB\left(\mathbb{D}, \dgm_k^\FiltFunc(X) \right) \right] \to 0$ and, with that, $\dGW{1}\left(\Dw{n,k}{\FiltFunc}(X), * \right) \to 0$.
\end{proof}

\begin{remark}
	We can give an explicit upper bound for $E_{\Unk{n,k}{\FiltFunc}(X)}\left[\dB\left(\mathbb{D}, \dgm_k^\FiltFunc(X) \right) \right]$ in the case that $\mu_X$ is Ahlfors regular (see Definition 3.18, page 252 of \cite{ahlfors-david}). Given $d \geq 0$, $\mu_X$ is Ahlfors $d$-regular if there exists a constant $C \geq 1$ such that
	\begin{equation*}
		\frac{r^d}{C} \leq \mu_X(B_r(x)) \leq Cr^d
	\end{equation*}
	for all $x \in X$.\\
	\indent To obtain the promised upper bound, set $\varepsilon = 4C^{1/d} \left(\frac{\ln n}{n} \right)^{1/d}$. If $\mu_X$ is Ahlfors $d$-regular,
	\begin{equation*}
		f_X(\varepsilon/4) = \inf_{x \in X} \mu_X(B_{\varepsilon/4}(x)) \geq \frac{(\varepsilon/4)^d}{C} = \frac{\ln(n)}{n},
	\end{equation*}
	and
	\begin{equation*}
		C_X(n,\varepsilon) = \dfrac{e^{-nf_X(\varepsilon/4)}}{f_X(\varepsilon/4)} \leq \frac{e^{-\ln(n)}}{\ln(n)/n} = \frac{1}{\ln(n)}.
	\end{equation*}
	Then,
	\begin{align*}
		E_{\Unk{n,k}{\FiltFunc}(X)}\left[\dB\left(\mathbb{D}, \dgm_k^\FiltFunc(X) \right) \right]
		&< \diam(X) \cdot C_X(n,\varepsilon) + \varepsilon\\
		&\leq \frac{\diam(X)}{\ln(n)} + 4C^{1/d} \left(\frac{\ln n}{n} \right)^{1/d},
	\end{align*}
	which goes to 0 as $n \to \infty$.
\end{remark}
\section{Persistence sets of metric graphs}\label{sec:special-class}
Let $G$ be a metric graph, that is, the geometric realization of a finite one-dimensional simplicial complex equipped with the shortest path distance induced by a collection of weights $\ell_e$ on the edges $e \in E(G)$ (see \cite[Section 3.2.2]{bbi01} or \cite{metric_graphs, graph_approx_geodesic_space} for other definitions). The central question in this section is what features of $G$ are detected by $\Dvr{2k+2,k}(G)$. Our first setting is when $G$ is a metric tree.
\begin{defn}\label{def:tree-like-metric}
	We say that a metric space $X$ is \define{tree-like} if there exists a metric tree $T$ such that $X$ is isometrically embedded in $T$. See Figure \ref{fig:tree-metric-space}.
\end{defn}

\begin{lemma}\label{lemma:Dnk_tree}
	Let $k \geq 1$ and $n \geq 1$ be fixed. For any metric tree $T$ and $X \subset T$ with $|X|=n$, $\PH_k(X)=0$ and, thus, $\Dvr{n,k}(T)$ is empty. In particular, if $n=2k+2$, then $\tb{X} \geq \td{X}$.
\end{lemma}
\begin{proof}
	Observe that $X$ is tree-like, so by Theorem 2.1 of the appendix of \cite{viral-evolution}, the persistence module $\PH_k(X)$ is $0$ for any $k \geq 1$. In particular, if $n=2k+2$, Theorem \ref{thm:n=2k+2} implies that $\tb{X} \geq \td{X}$.
\end{proof}
~\\
\begin{minipage}{0.57\linewidth}
	As a consequence, a metric graph $G$ must have a cycle if $\Dvr{n,k}(G)$ is to be non-empty and, even if it does, not all configurations $X \subset G$ with $|X|=n$ have $\tb{X} < \td{X}$. In fact, $X$ can be tree-like even if there is no metric tree $T$ such that $X \hookrightarrow T \hookrightarrow G$. We will see an example in the proof of Proposition \ref{prop:wedge_of_circles}. Hence, it would be useful to have a notion of a minimal metric graph $\Gamma_X$ containing $X$ so that, if $\Gamma_X$ is a tree, then $\mathrm{PH}_k^\vr(X)=0$. For now, we deal with the case of $n=4$, where split metric decompositions provide one possible construction for $\Gamma_X$.
\end{minipage}
\hfill
\begin{minipage}{0.40\linewidth}
	\centering
	\includegraphics[scale=1.0]{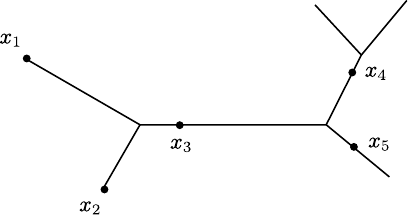}
	\captionof{figure}{A tree-like metric space $X = \{x_1,x_2,x_3,x_4,x_5\}$ and a metric tree $T$ such that $X \hookrightarrow T$.}
	\label{fig:tree-metric-space}
\end{minipage}

\subsection{Split metric decompositions}
We follow the exposition in \cite{metric-decomposition}. Let $(X, d_X)$ be a finite pseudo-metric space. Given a partition $X = A \cup B$, let
\begin{multline*}
	\beta_{\{a,a'\},\{b,b'\}} := \frac{1}{2}\big( \max\left[d_X(a,b)+d_X(a',b'),\ d_X(a,b')+d_X(a',b),\ d_X(a,a')+d_X(b,b')\right]\\
		- d_X(a,a') - d_X(b,b') \big),
\end{multline*}
and define the \define{isolation index} $\alpha_{A,B} := \min\left\{\beta_{\{a,a'\}, \{b,b'\}} \ \big| \ a,a' \in A \text{ and } b,b' \in B \right\}$.\\
\indent Notice that both $\alpha_{A,B}$ and $\beta_{\{a,a'\}, \{b,b'\}}$ are non-negative. Also, if $A=\{a,a'\}$ and $B=\{b,b'\}$, $\alpha_{A,B} = \beta_{A,B}$. If the isolation index $\alpha_{A,B}$ is non-zero, then the unordered partition $A,B$ is called a \define{$d_X$-split}.
The main theorem regarding isolation indices and split metrics is the following.
\begin{theorem}[\cite{metric-decomposition}]\label{thm:metric-decomposition}
	Any (pseudo-)metric $d_X$ on a finite set $X$ can be written uniquely as
	\begin{equation*}
		d_X = d_0 + \sum \alpha_{A,B} \delta_{A,B}, \text{ where }
		\delta_{A,B}(x,y) :=
    	\begin{cases}
    		0, & \text{if } x,y \in A \text{ or } x,y \in B,\\
    		1, & \text{otherwise}.
    	\end{cases}
	\end{equation*}
	where the sum runs over all $d_X$-splits $A,B$. $\delta_{A,B}$ is called a split-metric, and the term $d_0$ is a (pseudo-)metric that has no $d_0$-splits (also called split-prime metric). 
\end{theorem}

\indent  The importance of split metric decompositions is motivated by the following example. If $X = \{x_1,x_2,x_3,x_4\}$, then the metric graph $\Gamma_X$ shown in Figure \ref{fig:tight_span_4pts} contains an isometric copy of $X$ \cite{simoes-pereira-tree-realization, dress-tight-span} and the length of the edges of $\Gamma_X$ is given by isolation indices \cite{metric-decomposition}. Furthermore, any metric on 4 points does not contain a split-prime component \cite{metric-decomposition}. Another related construction is the \emph{tight span} of a metric space, which is an extension of $X$ that is universal in the sense that it is the smallest injective space in which $X$ embeds \cite{dress-tight-span}. In Figure \ref{fig:tight_span_4pts}, for instance, the tight span of $X$ can be obtained from $\Gamma_X$ by filling in the rectangle with a 2-cell equipped with the $L^1$ metric. See \cite{metric-decomposition} for more connections between the tight span and metric decompositions.\\
\indent Regarding persistent homology, the tight span has several properties that make it suitable for studying Vietoris-Rips complexes \cite{osman-memoli}. The key fact is the following. Let $M$ be a metric space and let $T_M \supset M$ be its tight span. Let $B_{r}^{T_M}(m) \subset T_M$ be the open ball of radius $r$ around $m \in M$. Then, there exists a filtered homotopy equivalence $f_r:\vr_{2r}(M) \to \bigcup_{m \in M} B_{r}^{T_M}(m)$. This theorem gives the type of construction that we want: an extension of a metric space $X$ where we can study the Vietoris-Rips complex of $X$. Given this property and the similarity of the tight span and $\Gamma_X$, it is reasonable to expect that split metric decompositions are also a good tool to study the Vietoris-Rips complex of $X$. Split metric decompositions do have an important advantage in our setting. They produce a graph $\Gamma_X$ such that $X \hookrightarrow \Gamma_X$ with edges of lengths that are computable with isolation indices. For these reasons, we now study the persistence diagram of $X \hookrightarrow \Gamma_X$.

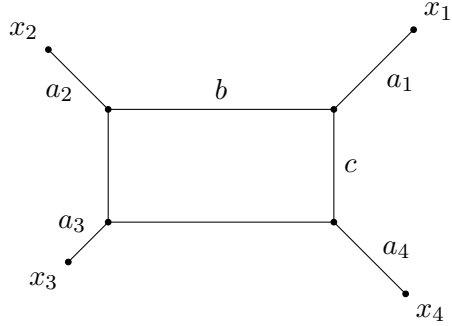
\begin{figure}[ht]
	\centering
\begin{tikzpicture}[scale=1.5]
	\tikzmath{
		\pt=0.025;
		\a = 1;
		\b = 2;
		\d1 = 1;
		\d2 = 0.75;
		\d3 = 0.5;
		\d4 = 0.9;
		\t1 = 45;
		\t2 = 45;
		\t3 = 45;
		\t4 = 45;
		\pax1 =  0;	\pay1 =  0;
		\pax2 =-\b;	\pay2 =  0;
		\pax3 =-\b;	\pay3 =-\a;
		\pax4 =  0;	\pay4 =-\a;
		\pbx1 =  \pax1+cos(\t1)*\d1;		\pby1 =  \pay1+sin(\t1)*\d1;
		\pbx2 =  \pax2-cos(\t2)*\d2;		\pby2 =  \pay2+sin(\t2)*\d2;
		\pbx3 =  \pax3-cos(\t3)*\d3;		\pby3 =  \pay3-sin(\t3)*\d3;
		\pbx4 =  \pax4+cos(\t4)*\d4;		\pby4 =  \pay4-sin(\t4)*\d4;
	}

	\draw[fill] (\pax1,\pay1) circle [radius=\pt];
	\draw[fill] (\pax2,\pay2) circle [radius=\pt];
	\draw[fill] (\pax3,\pay3) circle [radius=\pt];
	\draw[fill] (\pax4,\pay4) circle [radius=\pt];
	
	\draw[fill] (\pbx1,\pby1) circle [radius=\pt] node[above right]{$x_1$};
	\draw[fill] (\pbx2,\pby2) circle [radius=\pt] node[above left ]{$x_2$};
	\draw[fill] (\pbx3,\pby3) circle [radius=\pt] node[below left ]{$x_3$};
	\draw[fill] (\pbx4,\pby4) circle [radius=\pt] node[below right]{$x_4$};

	\draw (\pax1,\pay1) -- (\pax2,\pay2) node[midway, above]{$b$};
	\draw (\pax2,\pay2) -- (\pax3,\pay3);
	\draw (\pax3,\pay3) -- (\pax4,\pay4);
	\draw (\pax4,\pay4) -- (\pax1,\pay1) node[midway, right]{$c$};
	\draw (\pax1,\pay1) -- (\pbx1,\pby1) node[midway, below=5pt, right=1pt]{$a_1$};
	\draw (\pax2,\pay2) -- (\pbx2,\pby2) node[midway, below=5pt, left=-2pt]{$a_2$};
	\draw (\pax3,\pay3) -- (\pbx3,\pby3) node[midway, above=7pt, left=-3pt]{$a_3$};
	\draw (\pax4,\pay4) -- (\pbx4,\pby4) node[midway, above=3pt, right=1pt]{$a_4$};
\end{tikzpicture} 	\caption{The metric graph $\Gamma_X$ resulting from the split-metric decomposition of a metric space $(X,d_X)$ with 4 points (Theorem \ref{thm:metric-decomposition}). In this case, $a_i = \alpha_{\{x_i\}, X \setminus \{x_i\}}$, $b = \alpha_{\{x_2,x_3\}, \{x_1,x_4\}}$, $c = \alpha_{\{x_1,x_2\}, \{x_3,x_4\}}$, and $\alpha_{\{x_1,x_3\}, \{x_2,x_4\}}=0$. Notice that $d_X = \sum_{i=1}^4 a_{i} \cdot \delta_{x_i} + b \cdot \delta_{\{x_2,x_3\}, \{x_1,x_4\}} + c \cdot \delta_{\{x_1,x_2\}, \{x_3,x_4\}}$.}
	\label{fig:tight_span_4pts}
\end{figure}

\begin{prop}\label{prop:persistence_tight_span}
	Let $\Gamma_X$ be the metric graph shown in Figure \ref{fig:tight_span_4pts}, and $X = \{x_1,x_2,x_3,x_4\} \subset \Gamma_X$. Let $a_i = \alpha_{\{x_i\}, X \setminus \{x_i\}}$, $b = \alpha_{\{x_2,x_3\}, \{x_1,x_4\}}$, and $c = \alpha_{\{x_1,x_2\}, \{x_3,x_4\}}$.
	\begin{enumerate}[nosep]
		\item\label{item:diagonals_kill} If $\tb{X} < \td{X}$, then $\tb{X} = \max(d_{12}, d_{23}, d_{34}, d_{41})$ and $\td{X} = \min(d_{13}, d_{24})$.
		\item\label{item:persistence_tight_span} $\tb{X} < \td{X}$ if and only if
		\begin{equation}\label{eq:persistence_tight_span}
			\begin{split}
				|a_{2}-a_{1}|, |a_{4}-a_{3}|<b,\\
				|a_{3}-a_{2}|, |a_{1}-a_{4}|<c.
			\end{split}
		\end{equation}
		\item\label{item:persistence_leq_separation} $\td{X}-\tb{X} \leq \min(b,c)$, regardless of whether $\tb{X} < \td{X}$ or not.
	\end{enumerate}
	
\end{prop}
\begin{proof}
	\ref{item:diagonals_kill}. If $\tb{X} < \td{X}$, the desired formulas for $\tb{X}$ and $\td{X}$ hold if and only if $\vdeath(x_1)=x_3$ and $\vdeath(x_2)=x_4$. To see why, recall that $\vdeath$ is well defined by Lemma \ref{lemma:vd_is_unique}, and suppose $\vdeath(x_1) = x_4$ and $\vdeath(x_2) = x_3$. In particular, this means that $d_{13} < d_{14}$ and $d_{24} < d_{23}$. Since $X$ is isometrically embedded in $\Gamma_X$, $d_{ij}$ equals the length of the shortest path in $\Gamma_X$ between $x_i$ and $x_j$. Then, the inequalities $d_{13} < d_{14}$ and $d_{24} < d_{23}$ are equivalent to
	\begin{align*}
			a_{1} + (b+c) + a_{3} < a_{1} + c + a_{4} \text{ and }
			a_{2} + (b+c) + a_{4} < a_{2} + c + a_{3}.
	\end{align*}
	After rearranging terms, we get $b < a_{4} - a_{3} < -b$, a contradiction. The case $\vdeath(x_1)=x_4$ and $\vdeath(x_2)=x_3$ follows analogously, so $\vdeath(x_1)=x_3$ and $\vdeath(x_2)=x_4$.\\
	\ref{item:persistence_tight_span}. Notice that the inequalities $d_{23}<d_{13}$ and $d_{14}<d_{24}$ are equivalent to
	\begin{align*}
		a_{2}+c+a_{3} &< a_{1} + (b+c) + a_{3}\\
		a_{1}+c+a_{4} &< a_{2} + (b+c) + a_{4},
	\end{align*}
	which, after rearranging terms, result in $-b < a_{2}-a_{1} < b$. Using similar combinations, we find that $\max(d_{12},d_{23},d_{34},d_{41}) < \min(d_{13},d_{24})$ is equivalent to the system of inequalities in (\ref{eq:persistence_tight_span}).\\
	\indent If (\ref{eq:persistence_tight_span}) holds, then for all $1 \leq i \leq 4$, $d_{i,i+2} \geq \min(d_{13}, d_{24}) > \max(d_{12},d_{23},d_{34},d_{41}) \geq \max(d_{i-1,i}, d_{i,i+1})$. As a consequence, $\td{x_i} = d_{i,i+2}$ and $\tb{x_i} = \max(d_{i-1,i}, d_{i,i+1})$. Hence, $\tb{X}=\max_i \tb{x_i} = \max(d_{12},d_{23},d_{34},d_{41})$ and $\td{X} = \min_i \td{x_i} = \min(d_{13},d_{24})$, and thus, $\tb{X} < \td{X}$. Conversely, if $\tb{X}<\td{X}$, then item \ref{item:diagonals_kill} and the paragraph above imply (\ref{eq:persistence_tight_span}).\\
	\ref{item:persistence_leq_separation}. If $\tb{X} \geq \td{X}$, the bound is trivially satisfied. Suppose then, without loss of generality, that $\tb{X}=d_{12}$. Since $a_{3}+b+a_{4} = d_{34} \leq d_{12} = a_{1} + b + a_{2}$, we have
	\begin{align*}
		\td{X} - \tb{X}
		&= \min(d_{13},d_{24}) - d_{12}
		\leq \frac{1}{2}[d_{13}+d_{24}] - d_{12} \\
		&= \frac{1}{2}[a_{1}+a_{2}+a_{3}+a_{4} + 2(b+c)] - (a_{1}+b+a_{2})\\
		&\leq \frac{1}{2}[a_{1}+a_{2}+(a_{1}+a_{2}) + 2(b+c)] - (a_{1}+a_{2})-b = c.
	\end{align*}
	On the other hand, $d_{14} \leq d_{12}$ and $d_{23} \leq d_{12}$ give $a_{4}+c \leq a_{2}+b$ and $a_{3}+c \leq a_{1}+b$. Then
	\begin{align*}
		\td{X} - \tb{X}
		&\leq \frac{1}{2}[a_{1}+a_{2}+a_{3}+a_{4} + 2(b+c)] - (a_{1}+b+a_{2})\\
		&\leq \frac{1}{2}[a_{1}+a_{2}+(a_{1}+a_{2}) + 4b] - (a_{1}+a_{2})-b = b.
	\end{align*}
	In summary, $\td{X}-\tb{X} \leq \min(b,c)$.
\end{proof}

The following examples illustrate different uses of Proposition \ref{prop:persistence_tight_span}.

\begin{prop}\label{prop:wedge_of_circles}
	Let $\lambda_1,\dots,\lambda_n$ be positive numbers, and consider the wedge $\bigvee_{k=1}^n \frac{\lambda_k}{\pi} \cdot \Sphere{1}$ of $n$ circles at a common point $p_0 = \bigcap_{k=1}^n \frac{\lambda_k}{\pi} \cdot \Sphere{1}$. Then
	\begin{equation*}
		\Dvr{4,1}\left( \bigvee_{k=1}^n \frac{\lambda_k}{\pi} \cdot \Sphere{1} \right) = \bigcup_{k=1}^{n} \frac{\lambda_k}{\pi} \cdot \Dvr{4,1}\left(\Sphere{1} \right).
	\end{equation*}
	\begin{figure}
		\centering
		\makebox[\textwidth][c]{
			\includegraphics[scale=0.8]{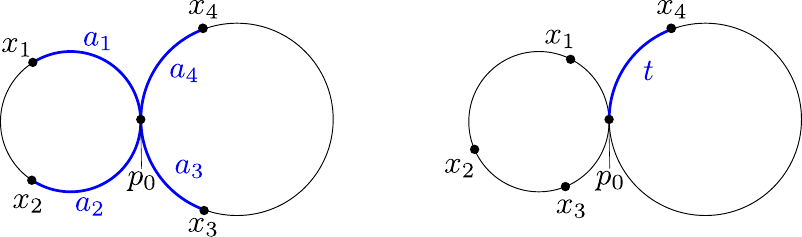}
		}
		\caption{A metric graph formed by the wedge of two circles at 0 as in Proposition \ref{prop:wedge_of_circles}. \textbf{Left:} $\Gamma_X$ is a metric tree. Center: One circle contains three points, while the other only has one. \textbf{Right:} Both circles have two out of four points.}
		\label{fig:wedge_of_circles}
	\end{figure}
\end{prop}
\begin{proof}
	Let $S_k = \frac{\lambda_k}{\pi} \Sphere{1}$ and $G = \bigvee_{k=1}^n \frac{\lambda_k}{\pi} \cdot \Sphere{1}$. Observe that the set $\Dvr{4,1}(S_k)$ is the triangle in $\R^2$ bounded by
	\begin{equation}\label{eq:D41_ineq_Sk}
		2(\lambda_k - t_b) \leq t_d \text{ and } t_b < t_d \leq \lambda_k
		\tag{$\star_k$}
	\end{equation}
	with vertices $(\frac{1}{2}\lambda_k,\lambda_k), (\frac{2}{3}\lambda_k, \frac{2}{3}\lambda_k)$, and $(\lambda_k, \lambda_k)$ (see Remark \ref{rmk:D41_S1_scaled}). By functoriality of persistence sets, $\bigcup_{i=1}^{n} \Dvr{4,1}(S_k) \subset \Dvr{4,1}(G)$. We now show the other inclusion.\\
	\indent Let $X = \{x_1,x_2,x_3,x_4\} \subset G$, and set $d_{ij}=d_G(x_i,x_j)$. Define $X_k = X \cap (S_k \setminus \{p_0\})$. The proof will go case by case depending on the cardinality of the sets $X_k$.\\
	\noindent \textbf{Case 1:} $|X_{a}| = 4$ for some $a$.\\
	\indent Observe that $X$ is contained in $S_{a}$, so $\dgm_1^\vr(X) \in \Dvr{4,1}(S_{a})$.\\
	\noindent \textbf{Case $\mathbf{1'}$:} $|X_{k_0}|=3$ for some $a$ and $|X_j|=0$ for all $j \neq a$\\
	\indent In this case, $X = X_{a} \cup \{p_0\} \subset S_{a}$, so, similarly to Case 1, we have $\dgm_1^\vr(X) \in \Dvr{4,1}(S_{a})$.\\
	\noindent \textbf{Case 2:} $|X_a| = 3$ and $|X_b|=1$ for some $a \neq b$.\\
	\indent For concreteness, write $X_a = \{x_1,x_2,x_3\}$ and $X_b = \{x_4\}$, and assume $x_2$ is in the connected component of $S_a \setminus \{x_1,x_3\}$ that doesn't contain $p_0$; see Figure \ref{fig:wedge_of_circles}. Then $d_{24} > d_{21},d_{23}$, so $v_d(x_2) = x_4$ (see Definition \ref{def:tb_td}). If $\tb{X} \geq \td{X}$, then $\dgm_1^\vr(X)$ is the empty diagram and it belongs to $\Dvr{4,1}(S_k)$ for all $k$. Assume, then, that $\tb{X} < \td{X}$. In particular, we have $v_d(x_1)=x_3$. Let $X' = \{p_0, x_1, x_2,x_3\}$, and $t = d_G(p_0,x_4)$. Let $d_{0i}=d_G(p_0,x_i)$ for $i \neq 4$, and notice that $d_{i4} = d_{i0}+t$. This implies that $\tb{X'} \leq \tb{X}$ and $\td{X'} \leq \td{X}$. Now we have two cases, depending on whether $\tb{X'} < \td{X'}$ or not. If the inequality holds, and since $X' \subset S_a$, then $\tb{X'}$ and $\td{X'}$ satisfy ($\star_a$). This allows us to verify $(\star_a)$ for $\tb{X}$ and $\td{X}$. Indeed, we have
	\begin{equation*}
		2\lambda_a \leq 2\tb{X'}+\td{X'} \leq 2\tb{X}+\td{X}.
	\end{equation*}
	Also, $\td{X} = \min_i \td{x_i} \leq \td{x_1} = d_{13} \leq \lambda_a$, regardless of the position of $x_4$. Thus, $\tb{X}$ and $\td{X}$ satisfy $(\star_a)$, so $(\tb{X}, \td{X}) \in \Dvr{4,1}(S_a)$.\\
	\indent For the second case, it is possible for $\tb{X}$ to be smaller than $\td{X}$ even if $\tb{X'} \geq \td{X'}$, However, several conditions must be met. First, recall that any 4-point metric space has a split metric decomposition as in Figure \ref{fig:tight_span_4pts}. By Proposition \ref{prop:persistence_tight_span} item \ref{item:persistence_leq_separation}, $b,c > 0$. Moreover,
	\begin{equation*}
		\beta_{\{x_4\},\{x_i,x_j\}} = \frac{1}{2}(d_{i4}+d_{j4}-d_{ij}) = \frac{1}{2}(d_{i0}+d_{j0}-d_{ij}) + t = \beta_{\{x_0\},\{x_i,x_j\}}+t.
	\end{equation*}
	Thus, $\alpha_{\{x_4\}, X \setminus \{x_4\}} = \alpha_{\{x_0\}, X' \setminus \{x_0\}}+t$. By Theorem 2 of \cite{metric-decomposition}\footnote{Formally, Theorem 2 of \cite{metric-decomposition} gives the conclusion for two metrics $d$ and $d'$ defined on the same set $X$. However, the result depends only on the values of the metrics, not on the specific underlying sets $X$ and $X'$ as long as there is a bijection $X \to X'$.}, all other isolation indices satisfy $\alpha_{A,B} = \alpha_{A',B'}$, where $X = A \cup B$, $X'=A' \cup B'$ and $A'$ is the set $A$ with $x_4$ replaced by $x_0$. $B'$ is defined analogously. In other words, the only isolation indices that are different between $X$ and $X'$ are $\alpha_{\{x_4\}, X \setminus \{x_4\}}$ and $\alpha_{\{x_0\}, X' \setminus \{x_0\}}$. For this reason, $X'$ has the split metric decomposition shown in Figure \ref{fig:tight_span_4pts} except that $x_4$ is changed to $x_0$ and $a_4$ is changed to $a_4-t \geq 0$. In particular, since $b,c>0$, $X'$ is not tree-like. In other words, $X'$ is not contained in any semicircle of $S_a$, so
	\begin{equation}\label{eq:perimeter_subset}
		2\lambda_a = d_{01}+d_{12}+d_{23}+d_{30}.
	\end{equation}
	\indent The second set of conditions comes from comparing $\tb{X}$ and $\td{X}$ with $\tb{X'}$ and $\td{X'}$. First, observe that $\td{X} = \min(d_{13}, d_{24}) = \min(d_{13}, d_{20}+t)$. Since $\td{X'} = \min(d_{13},d_{20})$ is smaller than $\tb{X'}$ and $\tb{X'} \leq \tb{X} < \td{X}$, we need $d_{20} < d_{13}$. Second, $\tb{X'}$ cannot be $d_{i0}$ for $i=1,3$. Otherwise, $\tb{X} = \max(d_{12},d_{23},d_{30}+t,d_{01}+t)$ would be $d_{i4}=d_{i0}+t$ for either $i=1,3$. This, however, induces a contradiction:
	\begin{equation*}
		\tb{X} = d_{i0} + t = \tb{X'}+t \geq \td{X'}+t \geq \td{X}.
	\end{equation*}
	For concreteness, let $\tb{X'}= \max(d_{12},d_{23}) = d_{12}$. Also, since $d_{02} = \td{X'} \leq \tb{X'} = d_{12}$, we must have $d_{02} = \min(d_{01}+d_{12}, d_{23}+d_{30}) = d_{23}+d_{30} \leq d_{12}$.\\
	\indent Now we are ready to prove that $(\tb{X}, \td{X}) \in \Dvr{4,1}(S_a)$. By Equation (\ref{eq:perimeter_subset}) and the conditions in the preceding paragraph,
	\begin{equation*}
		2\lambda_a = d_{01}+d_{12}+(d_{23}+d_{30}) \leq 3 d_{12} \leq 3\tb{X}.
	\end{equation*}
	Hence, $\tb{X} \geq \frac{2}{3}\lambda_a$. Then
	\begin{equation*}
		2\tb{X}+\td{X} > 3\tb{X} \geq 2\lambda_a.
	\end{equation*}
	Lastly, $\td{X} = \min(d_{13}, d_{24}) \leq d_{13} \leq \lambda_a$. Thus, $\tb{X}$ and $\td{X}$ satisfy $(\star_a)$.\\
	\noindent\textbf{Case $\mathbf{2'}$:} $|X_a| = 2$, $|X_b|=1$ for some $a \neq b$, and $|X_c|=0$ for all $c \neq a,b$.\\
	\indent $X = X_a \cup \{p_0\} \cup X_b$ and the proof in Case 2 is still valid if we replace $X_a$ with $X_a \cup \{p_0\}$.\\
	\noindent\textbf{Case 3:} $|X_a|=2$ and either $|X_b|=2$ or $|X_b|=|X_c|=1$ for $a \neq b \neq c$.\\
	\indent Let $X_a = \{x_1,x_2\}$ and $X_a' = X \setminus X_a$. Let $a_i = d_G(x_i,p_0)$. Notice that $d_{ij}=a_i+a_j$ for $i \in \{1,2\}$ and $j \in \{3,4\}$. Then:
	\begin{equation*}
		d_{13}+d_{24}
		= d_{14}+d_{23}
		= a_1+a_2+a_3+a_4
\,\,\,\,\mbox{and}\,\,\,\,	d_{12}+d_{34}
		\leq a_1+a_2+a_3+a_4.
	\end{equation*}
	As a consequence,
	\begin{equation*}
		\alpha_{\{x_1,x_3\}, \{x_2,x_4\}} = \beta_{\{x_1,x_3\}, \{x_2,x_4\}}
		= \frac{1}{2} [\max(d_{13}+d_{24}, d_{14}+d_{23}, d_{12}+d_{34} ) - d_{13}-d_{24} ]
		= 0.
	\end{equation*}
	Analogously, $\alpha_{\{x_1,x_4\}, \{x_2,x_3\}} = 0 \leq \alpha_{\{x_1,x_2\}, \{x_3,x_4\}}$. Then $b=0$ in Proposition \ref{prop:persistence_tight_span} and item \ref{item:persistence_leq_separation} gives that $\dgm_1^\vr(X)$ is the empty diagram. Note, in particular, that $\Gamma_X$ is a metric tree.\\
	\noindent \textbf{Case 4:} $|X_{a}| \leq 1$ for all $a$.\\
	\indent Observe that $X$ is isometrically embedded in the tree $T \subset G$ formed by the four shortest paths joining each $x_i$ to $p_0$. $\dgm_1^\vr(X)$ is empty by Lemma \ref{lemma:Dnk_tree}.
\end{proof}

\indent The proof of Proposition \ref{prop:wedge_of_circles} shows that a configuration $X \subset G$ produces persistence only if it is close to a cycle in the sense that either $X$ is contained in a circle $\frac{\lambda_i}{\pi} \cdot \Sphere{1}$, or only one point of $X$ is outside of $\frac{\lambda_i}{\pi} \cdot \Sphere{1}$. In both cases, the metric graph $\Gamma_X$ contains a cycle since both $b$ and $c$ in Figure \ref{fig:tight_span_4pts} are non-zero. In any other scenario, $\Gamma_X$ is a metric tree. This might lead to the conjecture that $\Dvr{4,1}(G) = \bigcup_{C \subset G} \Dvr{4,1}(C)$ where the union runs over all cycles $C \subset G$. However, the following examples show that this is false.

\begin{example}\label{ex:circle_with_flares}
	Recall the cyclic order $\prec$ from Definition \ref{def:circle_and_order}. Let $G$ be a metric graph formed by attaching edges of length $L$ to a cycle $C$ at the points $y_1 \prec y_2 \prec y_3 \prec y_4$; see Figure \ref{fig:circle_with_flares}. Let $X = \{x_1,x_2,x_3,x_4\} \subset G$. If $X \subset C$, then no new persistence is produced, so the points in $X$ have to be in the attached edges. Also, if $\tb{X}$ is to be smaller than $\td{X}$, then each $x_i$ must be on a different edge. For example, if $x_1$ and $x_2$ are on the edge attached to $y_1$, and $x_3$ and $x_4$ are on the edges adjacent to $y_3$ and $y_4$, respectively, let $X' = \{x_1,x_2,y_3,y_4\}$. This $X'$ consists of two points inside of a cycle and two points outside, so as we saw in Proposition \ref{prop:wedge_of_circles} when $|X_1|=|X_2|=2$, $X'$ is tree-like, and attaching edges at $y_3$ and $y_4$ doesn't change that. Thus, $X$ is also a tree-like metric space.\\
	\indent Suppose, then, that each $x_i$ is on the edge attached to $y_i$. Let $Y = \{y_1, y_2, y_3, y_4\}$. Since the decomposition in Theorem \ref{thm:metric-decomposition} is unique, the isolation indices of the metrics of $X$ and $Y$ satisfy $\alpha_{\{x_i\}, X \setminus \{x_i\}} = \alpha_{\{y_i\}, Y \setminus \{y_i\}} + d_G(x_i,y_i)$, and $\alpha_{\{x_i,x_j\},\{x_h,x_k\}} = \alpha_{\{y_i,y_j\},\{y_h,y_k\}}$, where $\{i,j,h,k\} = \{1,2,3,4\}$. Suppose that $\alpha_{\{y_1,y_3\}, \{y_2, y_4\}}=0$, and let $m := \min(\alpha_{\{y_1,y_2\}, \{y_3, y_4\}}, \alpha_{\{y_1,y_4\}, \{y_2, y_3\}})$. By Proposition \ref{prop:persistence_tight_span}, $\td{X}-\tb{X} \leq m$, so
	\begin{equation*}
		\Dvr{4,1}(G) \subset \Dvr{4,1}(\Sphere{1}) \cup \{ (t_b,t_d) \ | \ \tb{Y} \leq t_b < t_d \leq t_b+m, \text{ and } t_b \leq \tb{Y}+2L \}.
	\end{equation*}
	\indent Observe that $\Dvr{4,1}(G)$ can contain points outside of $\Dvr{4,1}(C)$. For example, if $\tb{Y} < \td{Y}$, then the point $(\tb{Y}+2L, \td{Y}+2L) \in \Dvr{4,1}(G)$.
\end{example}

\begin{remark}[$\Dvr{4,1}$ captures  information that is invisible to $\dgm_*^\vr$]
\label{rmk:D41_vs_VR1}
	Note that, in the last example, the simplicial complex $\vr_r(G)$ is homotopy equivalent to $\vr_r(C)$ at every scale $r$. The reason is that the VR complex of a wedge sum $X \vee Y$ decomposes as $\vr_r(X \vee Y) \simeq \vr_r(X) \vee \vr_r(Y)$ (see Proposition 3.7 of \cite{vr_metric_gluings} or Theorem 4.1 in \cite{osman-memoli} for a reformulation in terms of persistence modules). Since $G$ is the wedge sum of $C$ with 4 edges $E_i$, Lemma \ref{lemma:Dnk_tree} gives that each $\vr(E_i)$ is contractible and, hence, $\vr_r(G) \simeq \vr_r(C)$ which implies that $\dgm_*(G)=\dgm(C)$. In contrast, $\Dvr{4,1}(G) \neq \Dvr{4,1}(C)$. In other words, $\Dvr{4,1}$ is able to detect features of $G$ which the Vietoris-Rips persistence diagram does not. See Figure \ref{fig:circle_with_flares}.
\end{remark}

Let $F_k$ be a geodesic space formed by attaching $2k+2$ edges of length $L$ to $\Sphere{k}$ at the vertices of the regular cross-polytope. We can generalize Example \ref{ex:circle_with_flares} to the following proposition (cf. Figure \ref{fig:spiky-sphere}).
\begin{prop}
    \label{prop:sphere_with_flares}
    $\Sphere{k}$ and $F_k$ have the same persistence diagrams, but $\Dvr{2k+2,k}(\Sphere{k}) \subsetneq \Dvr{2k+2,k}(F_{k})$.
\end{prop}
\begin{proof}
    By the explanation in the previous remark, the persistence diagrams of $\Sphere{k}$ and $F_k$ are equal. Also, $\Sphere{k} \hookrightarrow F_k$, so $\Dvr{2k+2,k}(\Sphere{k}) \subset \Dvr{2k+2,k}(F_{k})$. To see that the containment is strict, suppose that the $i$-th edge was attached to $y_i \in \Sphere{k}$ for $i=1, \dots, 2k+2$ and choose the labels so that $y_i$ and $y_{i+k+1}$ are antipodal (addition of indices is done modulo $2k+2$). Thus, $d_{\Sphere{k}}(y_i,y_j)$ equals $\pi/2$ if $j \neq i, i+k+1$ and $\pi$ if $j=i+k+1$. If $x_i$ is the point on the $i$-th edge at distance $L$ from $y_i$, then $d_{F_{k}}(x_i,x_j)$ is $\pi/2+2L$ when $j \neq i, i+k+1$ and $\pi+2L$ when $j=i+k+1$. Hence, $\tb{X}=\pi/2+2L$ and $\td{X}=\pi+2L$. Since every point $(t_b, t_d) \in \Dvr{2k+2,k}(\Sphere{k})$ satisfies $t_b < t_d \leq \diam(\Sphere{k}) = \pi$, $(\tb{X}, \td{X}) \in \Dvr{2k+2,k}(F_k) \setminus \Dvr{2k+2,k}(\Sphere{k})$.
\end{proof}
\begin{figure}
	\centering
	\includegraphics[width=0.5\linewidth]{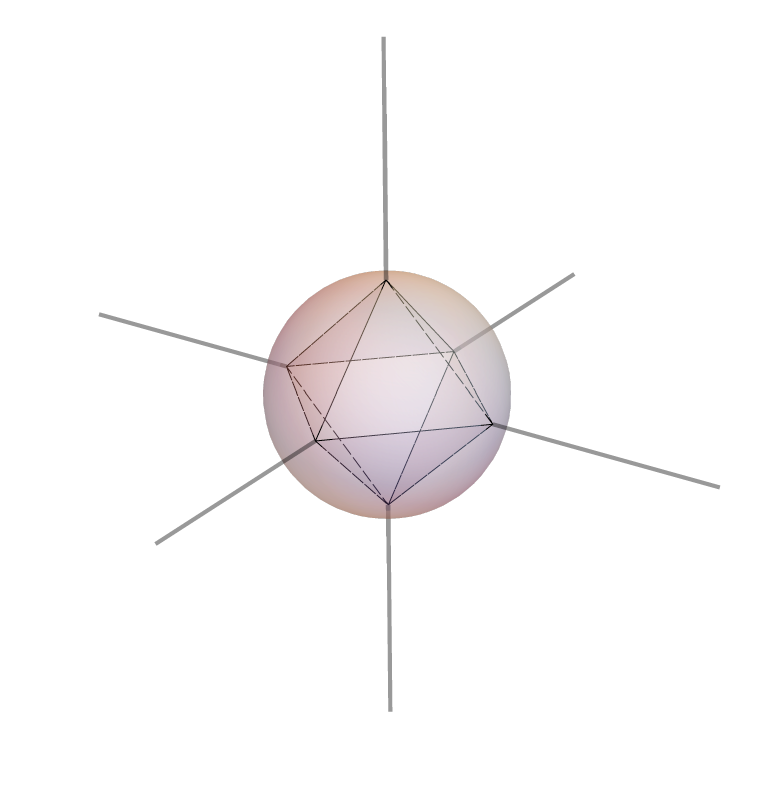}
	\caption{A spiky sphere. See Proposition \ref{prop:sphere_with_flares}.
	\label{fig:spiky-sphere}}
\end{figure}

\begin{example}\label{ex:D41_cube}
	Not all cycles $C \subset G$ with the induced subspace metric produce the persistence sets of a cycle graph. For instance, let $G$ be the metric graph with edges of length 1 shown in Figure \ref{fig:non_isometric_cycle}. Let $C$ be the cycle that passes through the vertices $1, 2, 6, 5, 8, 7, 3, 4$. $C$ has length 8, but there is no point $(2,4)$ in $\Dvr{4,1}(G)$. The reason is that the shortest path between points in $C$ is often not contained in $C$, and so $C$ is not isometric to a circle. For example, the edge connecting 1 and 5 is not contained in $C$ despite it being the shortest path between its endpoints. We will explain this phenomenon in the next section.
	
	\begin{figure}
		\centering
		\includegraphics[width=275pt, height=150pt]{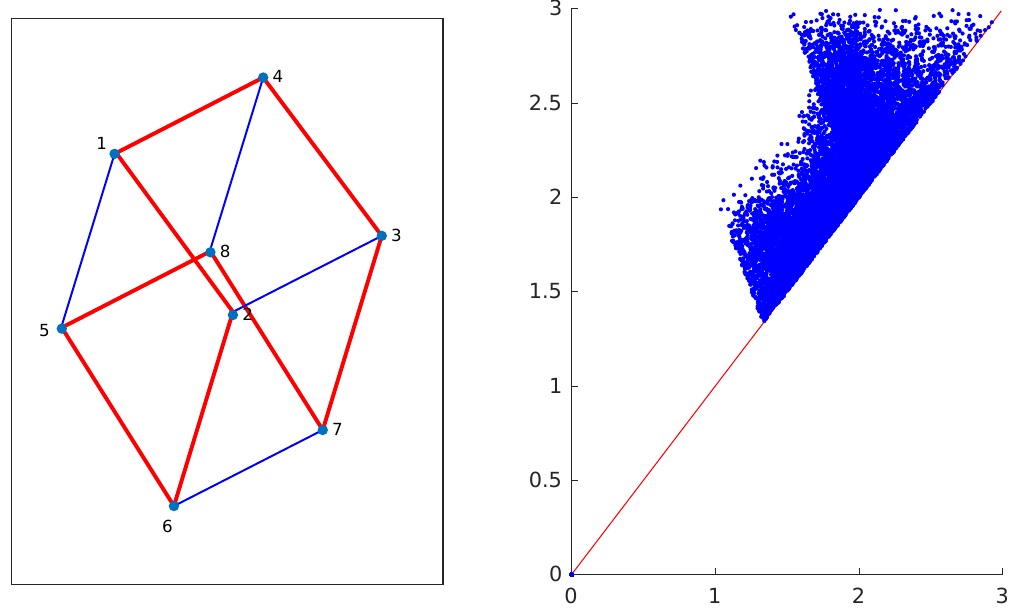}
		\caption{\textbf{Left:} The cube metric graph $G$ contains a cycle that is not isometric to a circle. \textbf{Right:} Its persistence set $\Dvr{4,1}(G)$ does not contain the point $(2,4)$. See Example \ref{ex:D41_cube}. This figure was obtained by sampling 100,000 configurations of 4 points from $G$. About 13 \% of those configurations produced a non-diagonal point.}
		\label{fig:non_isometric_cycle}
	\end{figure}
\end{example}

\subsection{A family of metric graphs whose homotopy type is  characterized via  $\Dvr{4,1}$.}
\indent Recall that the persistence set $\Dvr{4,1}(\frac{\lambda}{\pi} \cdot \Sphere{1})$ is a triangle with vertices $(\lambda/2, \lambda), (\frac{2}{3}\lambda, \frac{2}{3}\lambda), (\lambda, \lambda)$. Observe that the only point in $\Dvr{4,1}(\frac{\lambda}{\pi} \cdot \Sphere{1})$ that satisfies $t_d = 2t_b$ is $(\lambda/2, \lambda)$. A similar observation holds in Examples \ref{ex:circle_with_flares} and \ref{ex:D41_cube}. In both cases, the metric graph in question contains an isometrically embedded cycle, and by functoriality, the persistence set of the metric graph contains a triangle generated by such a cycle. However, not all cycles produce such a triangle as Example \ref{ex:D41_cube} shows. Proposition \ref{prop:good_families} gives conditions under which $\Dvr{4,1}(G)$ is capable of detecting all cycles in $G$, and examples of admissible graphs are shown in Figure \ref{fig:good_families}.

\begin{prop}\label{prop:good_families}
	Let $T_1,\dots,T_m$ be a set of metric trees and, for each $k=1,\dots,n$, let $C_k$ be a cycle. Suppose that all cycles have different length. Let $G$ be a metric graph formed by iteratively attaching either a metric tree $T_i$ or a cycle $C_k$ along a vertex or an edge $e$ that satisfies the following property. For any cycle $C \subset G$ that intersects $e$, their lengths satisfy $|e| < \frac{1}{3} |C|$. Then the first Betti number of $G$ equals the number of points $(\lambda/2, \lambda) \in \Dvr{4,1}(G)$.
\end{prop}

We prove this statement at the end of the section. For now, we begin the road to the proof by recalling Lemma \ref{lemma:persistence_bounds} item \ref{item:td_basic_bound}. For a 4-point set $X$, Lemma \ref{lemma:persistence_bounds} says that if $\td{X} = 2\tb{X}$, then $X$ has to be a square, that is, $d_X(x_i,x_{i+1})=\tb{X}$ and $d_X(x_i,x_{i+2})=\td{X}$ for $i=1,\dots,4$. If $X$ is a subset of a metric graph $G$, it is tempting to suggest that $X$ must be contained in a cycle $C \subset G$ isometric to $\frac{\lambda}{\pi} \cdot \Sphere{1}$. However, as Figure \ref{fig:not_isometric_to_S1} shows, this is not always the case. Still, if $G$ satisfies the hypothesis of Theorem \ref{thm:square_in_graph_gluing}, then at least we can ensure that $X$ lies in a specific metric subgraph. Before that, we need one more preparatory result which was inspired by Theorem 3.15 in \cite{vr_metric_gluings}.
\begin{figure}[h]
\begin{minipage}{0.48\linewidth}
	\centering
	\includegraphics[scale=1]{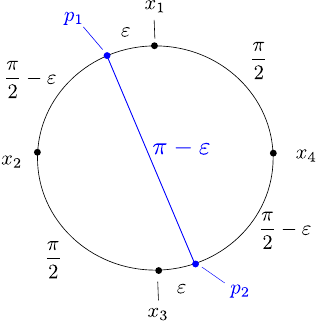}
	\captionof{figure}{A metric graph $G$ and a set $X \subset G$ such that $\tb{X}=\pi/2$ and $\td{X}=\pi$. Notice that the outer black cycle $C$ contains $X$ but is not isometric to a circle. If it were, the shortest path in $G$ between $p_1$ and $p_2$ would be contained in $C$, but that path is the blue edge of length $\pi-\varepsilon$.}
	\label{fig:not_isometric_to_S1}
\end{minipage}
\hfill
\begin{minipage}{0.48\linewidth}
	\centering
	\includegraphics[scale=1]{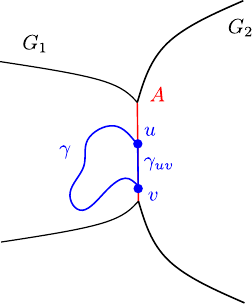}
	\captionof{figure}{In Lemma \ref{lemma:gluing_metric_graphs}, any path in $G_1$ between $u$ and $v$ has length greater than $\alpha$.}
	\label{fig:gluing_metric_graphs}
\end{minipage}
\end{figure}
\begin{lemma}\label{lemma:gluing_metric_graphs}
	Let $G = G_1 \cup_A G_2$ be a metric gluing of the metric graphs $G_1$ and $G_2$ such that $A = G_1 \cap G_2$ is a closed path of length $\alpha$. Let $\ell_j$ be the length of the shortest cycle contained in $G_j$ that intersects $A$, and set $\ell = \min(\ell_1, \ell_2)$. Assume that $\alpha < \frac{\ell}{2}$. Then the shortest path $\gamma_{uv}$ between any two points $u,v \in A$ is contained in $A$. As a consequence, if $\frac{\lambda}{\pi} \cdot \Sphere{1} \hookrightarrow G$ is an isometric embedding, then $\frac{\lambda}{\pi} \cdot \Sphere{1}$ is contained in either $G_1$ or $G_2$.
\end{lemma}
\begin{proof}
	Let $\gamma_{uv}'$ be the shortest path contained in $A$ that connects $u$ and $v$. We will show that $\gamma_{uv} = \gamma_{uv}'$. Let $\gamma$ be any path that joins $u$ and $v$, and is contained in either $G_1$ or $G_2$ but not in $A$; see Figure \ref{fig:gluing_metric_graphs}. Since $\gamma$ is not contained in $A$, $\gamma \cup \gamma_{uv}'$ contains a non-trivial cycle $C$ that intersects $A$. Since $\gamma_{uv}' \subset A$, its length is smaller than $\alpha$. Then
	\begin{equation*}
		2\alpha < \ell \leq |\gamma|+|\gamma_{uv}'| = |\gamma|+\alpha.
	\end{equation*}
	Thus, $|\gamma| > \alpha \geq |\gamma_{uv}'| = d_G(u,v)$. More generally, any path $\gamma$ between $u$ and $v$ can be split into subpaths $\gamma_1, \dots, \gamma_k$ such that either $\gamma_j \subset A$, or  $\gamma_j \subset G_i$ for some $i=1,2$ and $\gamma_j \cap A = \{u', v'\}$, where $u'$ and $v'$ are the endpoints of $\gamma_j$. Applying the reasoning above to each $\gamma_j$ that is not contained in $A$ shows that $|\gamma| \geq |\gamma_{uv}'|$. In particular, we must have $\gamma_{uv} = \gamma_{uv}'$.\\
	\indent Now, a cycle $C \subset G$ is isometric to $\frac{\lambda}{\pi} \cdot \Sphere{1}$ if there is a shortest path between any $x,x' \in C$ contained in $C$. If $C \cap A$ has several connected components, then $C$ can be decomposed as the union of paths in $A$ and paths contained in $G_1$ or $G_2$. If we pick two points $u$ and $v$ that lie in different connected components of $G \cap A$, then the shortest sub-path of $C$ between them will contain a sub-path that lies either in $G_1$ or $G_2$. By the previous paragraph, the sub-path contained in $G_1$ or $G_2$ has length larger than $\alpha \geq d_G(u,v)$. Thus, the shortest path between $u$ and $v$ lies outside of $C$, so $C$ is not isometric to $\frac{\lambda}{\pi} \cdot \Sphere{1}$. Instead, the only possibility for $C$ to be isometric to $\frac{\lambda}{\pi} \cdot \Sphere{1}$ is that $C \cap A$ is either empty or connected. This implies $C \subset G_1$ or $C \subset G_2$.
\end{proof}

\indent The next theorem is the main result of this section and similar in spirit to Proposition \ref{prop:wedge_of_circles}. The proof of Proposition \ref{prop:wedge_of_circles} relied on the observation that if $X \subset G$ has $\tb{X} < \td{X}$ then either $X$ lies inside a cycle $\frac{\lambda_i}{\pi} \Sphere{1}$ or, at worse, only one point lies outside. In a more general metric gluing $G_1 \cup_A G_2$, however, the condition $\tb{X} < \td{X}$ is not enough to guarantee that most of $X$ lies inside one component. Instead, we give hypotheses on $G_1 \cup_A G_2$ under which the stronger condition $\td{X} = 2\tb{X}$ (as opposed to just $\tb{X} < \td{X}$) implies that $X$ is contained in either $G_1$ or $G_2$.

\begin{theorem}\label{thm:square_in_graph_gluing}
	Let $G = G_1 \cup_A G_2$ be a metric gluing of the metric graphs $G_1$ and $G_2$ such that $A = G_1 \cap G_2$ is a path of length $\alpha$. Let $\ell_j$ be the length of the shortest cycle contained in $G_j$ that intersects $A$, and set $\ell = \min(\ell_1, \ell_2)$. Assume that $\alpha < \frac{\ell}{3}$. If $X = \{x_1,x_2,x_3,x_4\} \subset G$ satisfies $\tb{X} = \lambda/2$ and $\td{X} = \lambda$, then either $X \subset G_1$ or $X \subset G_2$.
\end{theorem}
\begin{proof}
	Let $\gamma_{ij}$ be a shortest path in $G$ from $x_i$ to $x_j$. Since $\td{X} = 2\tb{X}$, Lemma \ref{lemma:persistence_bounds} item \ref{item:td_basic_bound} gives that $d_G(x_i, \vdeath(x_i)) = \lambda$ and $d_G(x_i, x) = \lambda/2$ for every $x \neq \vdeath(x_i)$. For this reason, we relabel the points $x_i$ so that $\lambda = |\gamma_{13}| = |\gamma_{24}|$ and $\lambda/2 = |\gamma_{12}| = |\gamma_{23}| = |\gamma_{34}| = |\gamma_{41}|$.\\
	\indent During this proof, if a path $\gamma$ has one endpoint in $G_1$ and one in $G_2$, we decompose it as $\gamma^{(1)} \cup \gamma^{(A)} \cup \gamma^{(2)}$, where $\gamma^{(i)} \subset G_i$, $\gamma^{(A)} \subset A$ and each intersection $\gamma^{(i)} \cap \gamma^{(A)}$ is a single point. Let $X_1 := X \cap G_1$ and $X_2 := X \cap G_2$. We will break down the proof depending on the size of $X_1$ and $X_2$.
	\begin{figure}
		\centering
		\begin{minipage}{0.32\linewidth}
			\centering
			\includegraphics[scale=0.90]{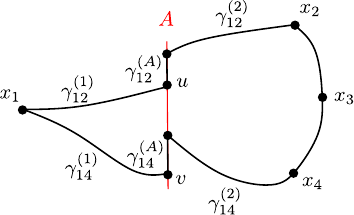}
		\end{minipage}
		\hfill
		\begin{minipage}{0.32\linewidth}
			\centering
			\includegraphics[scale=0.90]{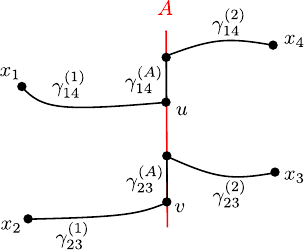}
		\end{minipage}
		\hfill
		\begin{minipage}{0.32\linewidth}
			\centering
			\includegraphics[scale=0.90]{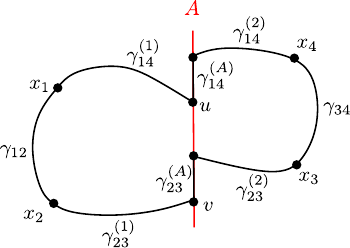}
		\end{minipage}
	\caption{Possible arrangements of paths between 4 points in Theorem \ref{thm:square_in_graph_gluing}. \textbf{Left:} $x_1 \in G_1$ and $x_2,x_3,x_4 \in G_2$ (Case 1). Middle: $X_1 = \{x_1,x_2\}$ and $X_2=\{x_3,x_4\}$ (Case 3). \textbf{Right:} The paths between points of $X$ form a cycle in $G$ (Case 3.2).}
	\label{fig:square_in_graph_gluing}
	\end{figure}
	~\\
	\noindent \textbf{Case 0:} If either $X_1$ or $X_2$ is empty, the theorem holds immediately.\\
	\noindent \textbf{Case 1:} $X_1$ or $X_2$ is a singleton.\\
	\indent Suppose that $X_1 = \{x_1\}$ (see Figure \ref{fig:square_in_graph_gluing}). Let $u := \gamma_{21}^{(1)} \cap A$ and $v := \gamma_{14}^{(1)} \cap A$. By Lemma \ref{lemma:gluing_metric_graphs}, $d_G(u,v) < |\gamma_{21}^{(1)}| + |\gamma_{14}^{(1)}|$. However, if $\gamma_{uv}$ is a shortest path between $u$ and $v$, then $\gamma_{24}' := \gamma_{21}^{(2)} \cup \gamma_{21}^{(A)} \cup \gamma_{uv} \cup \gamma_{14}^{(A)} \cup \gamma_{14}^{(2)}$ is a path between $x_2$ and $x_4$ such that
	\begin{align*}
		|\gamma_{24}'| &\leq |\gamma_{21}^{(2)}| + |\gamma_{21}^{(A)}| + |\gamma_{uv}| + |\gamma_{14}^{(A)}| + |\gamma_{14}^{(2)}|\\
		&< |\gamma_{21}^{(2)}|+ |\gamma_{21}^{(A)}| + |\gamma_{21}^{(1)}| + |\gamma_{14}^{(1)}| + |\gamma_{14}^{(A)}| + |\gamma_{14}^{(2)}|\\
		&= |\gamma_{21}|+|\gamma_{14}| = \lambda/2 + \lambda/2 = \lambda.
	\end{align*}
	This contradicts the assumption that $d_G(x_2,x_4) = \lambda$.\\
	\noindent \textbf{Case 2:} $|X_1|=|X_2|=2$ and $\diam(X_1)=\diam(X_2)=\lambda$.\\
	\indent Without loss of generality, write $X_1 = \{x_1, x_3\}$ and $X_2 = \{x_2,x_4\}$. The path $\gamma_{12} \cup \gamma_{23} \cup \gamma_{31}$ is a cycle in $G$ that intersects both $G_1$ and $G_2$. Let $u = \gamma_{12}^{(1)} \cap A$ and $v = \gamma_{23}^{(1)} \cap A$, and let $\gamma_{uv} \subset A$ be a path between them. By Lemma \ref{lemma:gluing_metric_graphs}, $d_G(u,v) < |\gamma_{12}^{(2)}|+|\gamma_{12}^{(A)}|+|\gamma_{23}^{(2)}|+|\gamma_{23}^{(A)}|$, so following the reasoning of Case 1, $\gamma_{12}^{(1)} \cup \gamma_{uv} \cup \gamma_{23}^{(1)}$ is a path between $x_1$ and $x_3$ with length less than $|\gamma_{12}|+|\gamma_{23}|=\lambda$. This is again a contradiction.\\
	\noindent \textbf{Case 3:} $|X_1|=|X_2|=2$ and $\diam(X_1)=\diam(X_2)=\lambda/2$.\\
	\indent Now we can assume $X_1 = \{x_1,x_2\}$ and $X_2=\{x_3,x_4\}$ (See Figure \ref{fig:square_in_graph_gluing}). Let $u = \gamma_{14}^{(1)} \cap A$, and $v = \gamma_{23}^{(1)} \cap A$. By the triangle inequality,
	\begin{equation}\label{ineq:d1_d4}
		\lambda = d_G(x_1,x_3) \leq d_G(x_1,u)+d_G(u,v)+d_G(v,x_3).
	\end{equation}
	Analogously,
	\begin{equation}\label{ineq:d2_d3}
		\lambda \leq d_G(x_2,v)+d_G(v,u)+d_G(u,x_4).
	\end{equation}
	On the other hand, since $\gamma_{23}$ is the shortest path between $x_2$ and $x_3$ and it passes through $v$, $\lambda/2 = d_G(x_2,x_3) = d_G(x_2,v)+d_G(v,x_3)$. If there existed a path between $v$ and $x_3$ of length smaller than $d_G(v,x_3)$, then the concatenation of that path and $\gamma_{23}^{(1)}$ would give a path between $x_2$ and $x_3$ shorter than $\gamma_{23}$. The same reasoning applies to $x_2$ and $v$, so the above equality holds. By a similar argument, we get $\lambda/2 = d_G(x_1,u)+d_G(u,x_4)$. Adding these two equations gives
	\begin{equation*}
		d_G(x_1,u) + d_G(x_2,v) + d_X(v,x_3) + d_G(u,x_4) = \lambda,
	\end{equation*}
	and combining this last equation with inequalities (\ref{ineq:d1_d4}) and (\ref{ineq:d2_d3}) produces, respectively,
	\begin{align}
		d_G(x_2,v) + d_G(u,x_4) &\leq d_G(u,v) \label{ineq:d2_d4_version2}\\
		d_G(x_1,u) + d_G(v,x_3) &\leq d_G(v,u) \label{ineq:d1_d3_version2}.
	\end{align}
	Then, using inequalities \ref{ineq:d1_d3_version2} and \ref{ineq:d1_d4}, we obtain $\lambda \leq 2 d_G(u,v)$. Furthermore, since $u,v \in A$, we get $\lambda/2 \leq d_G(u,v) \leq \alpha$. Now we break down case 3 depending on whether $\gamma_{12}$ and $\gamma_{34}$ intersect $A$ or not.\\
	\noindent \textbf{Case 3.1:} Suppose that $\gamma_{12}$ intersects $A$.\\
	\indent Write $\gamma_{12} = \gamma_{12}^{(1)} \cup \gamma_{12}^{(A)} \cup \gamma_{12}^{(2)}$, and let $w_i = \gamma_{12}^{(i)} \cap \gamma_{12}^{(A)}$. Let $\gamma_{w_1}$ be a shortest path between $u$ and $w_1$. By the triangle inequality,
	\begin{equation*}
		|\gamma_{w_1}| = d_G(u,w_1) \leq d_G(u,x_1) + d_G(x_1,w_1) \leq d_G(x_1,x_4)+d_G(x_1,x_2) = \lambda.
	\end{equation*}
	If $u \neq w_1$, then $\gamma_{14}^{(1)} \cup \gamma_{w_1} \cup \gamma_{12}^{(1)}$ is a cycle that intersects $A$ of length at most $2\lambda \leq 2\alpha$. Then, $\ell$ is smaller than $2\alpha$ by definition. However, this is a contradiction because $3\alpha < \ell$ by hypothesis. Thus, $w_1=u$, and an analogous argument shows that $w_2=v$. Since $\gamma_{12}$ is a shortest path between $x_1$ and $x_2$,
	\begin{align*}
		\lambda/2 = d_G(x_1,x_2)
		&= d_G(x_1,w_1)+d_G(w_1,w_2)+d_G(w_2,x_2)\\
		&= d_G(x_1,u) + d_G(u,v) + d_G(v,x_2)
		\geq d_G(u,v) \geq \lambda/2.
	\end{align*}
	Thus, $x_1=u$ and $x_2=v$. In other words, $X_1 \subset A \subset G_2$, so $X = X_1 \cup X_2 \subset G_2$. Naturally, if $\gamma_{34}$ intersected $A$ instead of $\gamma_{12}$, then an analogous argument would give $X \subset G_1$.\\
	\noindent \textbf{Case 3.2:} Neither $\gamma_{34}$ nor $\gamma_{12}$ intersect $A$ (see Figure \ref{fig:square_in_graph_gluing}).\\
	\indent Once more, let $u = \gamma_{14}^{(1)} \cap A$, $v = \gamma_{23}^{(1)} \cap A$, and $\nu = d_G(u,v)$. Define the cycles $C = \gamma_{12} \cup \gamma_{23} \cup \gamma_{34} \cup \gamma_{41}$, $C_1 = \gamma_{12} \cup \gamma_{23}^{(1)} \cup \gamma_{uv} \cup \gamma_{41}^{(1)}$ and $C_2 = \gamma_{34} \cup \gamma_{41}^{(2)} \cup \gamma_{41}^{(A)} \cup \gamma_{uv} \cup \gamma_{23}^{(A)} \cup \gamma_{23}^{(2)}$. Set $L=|C|$ and $L_j = |C_j|$ for $j=1,2$. Clearly, $L=2\lambda$ and $L_1+L_2-2\nu = L = 2\lambda$. For this reason, write $\lambda = \frac{L_1+L_2}{2}-\nu$.\\
	\indent For brevity, let $\delta_1 = d_G(x_1,u), \delta_2 = d_G(x_2,v), \delta_3 = d_G(x_3,v)$, and $\delta_4 = d_G(x_4,u)$. By definition of $u$ and $v$, we have
	\begin{equation}\label{eq:system_1}
		\lambda/2 = d_G(x_1,x_4) = d_G(x_1,u)+d_G(u,x_4) = \delta_1+\delta_4,
	\end{equation}
	and
	\begin{equation}\label{eq:system_2}
		\lambda/2 = d_G(x_2,x_3) = \delta_2+\delta_3.
	\end{equation}
	Additionally,
	\begin{align}\label{eq:system_3}
		L_1 &= |\gamma_{12}| + |\gamma_{23}^{(1)}| + |\gamma_{uv}| + |\gamma_{14}^{(1)}| \nonumber\\
		&= d_G(x_1,x_2) + d_G(x_2,v) + d_G(u,v) + d_G(u,x_1) \nonumber\\
		&= \lambda/2+\delta_2+\nu+\delta_1,
	\end{align}
	and
	\begin{align}\label{eq:system_4}
		L_2 &= |\gamma_{34}| + |\gamma_{41}^{(2)} \cup \gamma_{41}^{(A)}| + |\gamma_{uv}| + |\gamma_{23}^{(A)} \cup \gamma_{23}^{(2)}| \nonumber\\
		&= d_G(x_3,x_4) + d_G(x_4,u) + d_G(u,v) + d_G(v,x_3) \nonumber\\
		&= \lambda/2+\delta_4+\nu+\delta_3.
	\end{align}
	\indent If we interpret the $\delta_i$ as variables and $L_1,L_2,\nu$, and $\lambda$ as constants, equations (\ref{eq:system_1}) - (\ref{eq:system_4}) form a system of 4 equations with 4 variables. It can be seen that the matrix of coefficients has rank 3, so the solution has one parameter. Thus, choosing $\delta_4 = t$ gives the general solution
	\begin{align}\label{eq:system_complete}
		\delta_1 = \lambda/2-t,\,\,
		\delta_2 = L_1-\lambda-\nu+t,\,\, 
		\delta_3 = L_2-\lambda/2-\nu-t,\,\,
		\delta_4 = t. 
	\end{align}
	This means that there exists a particular number $0 \leq t \leq \lambda/2$ such that the distances between points of $X$ and $u$ and $v$ are given by the equations above. With this tool at hand, we now claim that at least one of the paths $\gamma_1 := \gamma_{14}^{(1)} \cup \gamma_{uv} \cup \gamma_{23}^{(A)} \cup \gamma_{23}^{(2)}$ or $\gamma_2 := \gamma_{14}^{(2)} \cup \gamma_{14}^{(A)} \cup \gamma_{uv} \cup \gamma_{23}^{(1)}$ has length less than $\lambda$. This would imply that either $d_G(x_1,x_3)$ or $d_G(x_2,x_4)$ is less than $\lambda$, violating the assumption that $\td{X} = \lambda$.\\
	\indent An equivalent formulation of the claim is
	\begin{equation}\label{eq:optimization_objective}
		\max_{t}\left( \min(|\gamma_1|,|\gamma_2|) \right) < \lambda.
	\end{equation}
	\indent If this inequality holds, then either $|\gamma_1|$ or $|\gamma_2|$ is smaller than $\lambda$, regardless of the value of $t$. Notice, though, that $|\gamma_1| = \delta_1+\nu+\delta_3$ and $|\gamma_2|=\delta_4+\nu+\delta_2$. Using the equations in (\ref{eq:system_complete}), we see that $|\gamma_1|+|\gamma_2| = L_1+L_2-\lambda$ is a quantity independent of $t$. Thus, the maximum in equation (\ref{eq:optimization_objective}) is achieved when $|\gamma_1|=|\gamma_2|$. This happens when $t=\frac{1}{4}(L_2-L_1+\lambda)$, and gives
	$
		|\gamma_1| = \frac{L_1+L_2}{2}-\nu-\frac{\lambda}{2} = \frac{L_1+L_2}{4}+\frac{\nu}{2}.
	$
	The claim is that this quantity is less than $\lambda = \frac{L_1+L_2}{2}-\nu$. Solving for $\nu$ gives the equivalent $\nu < \frac{L_1+L_2}{6}$. Recall that $\gamma_{uv} \subset A$, that $A$ is a path of length $\alpha < \frac{\ell}{3}$, and that $\ell$ is the length of the smallest cycle contained in either $G_1$ or $G_2$ that intersects $A$. Since $C_i \subset G_i$, we have $\nu \leq \alpha < \frac{\ell}{3} \leq \frac{L_1+L_2}{6}$, as desired. This forces $d_G(x_1,x_3) \leq |\gamma_1| < \lambda$, violating the assumption that $\td{X} = \lambda$. This concludes the proof of Case 3.2, and gives the Theorem.
\end{proof}

\indent To close up this section, we explore a consequence of Theorem \ref{thm:square_in_graph_gluing}. Once more, this application is inspired by \cite{vr_metric_gluings}, specifically Proposition 4.1.

\begin{theorem}\label{thm:num_cycles_equals_corners_D41}
	Let $T_1,\dots,T_m$ be a set of metric trees. For each $k=1,\dots,n$, let $\lambda_k>0$ and let $C_k$ be a cycle of length $L_k = 2\lambda_k$. Suppose that all $\lambda_k$ are distinct. Let $G$ be a metric graph formed by iteratively attaching either a metric tree $T_i$ or a cycle $C_k$ along a vertex or an edge $e$ that satisfies the following property. For any cycle $C \subset G$ that intersects $e$, their lengths satisfy $|e| < \frac{1}{3} |C|$. Then, the number of points $(\lambda/2, \lambda) \in \Dvr{4,1}(G)$ is equal to the number of cycles $C_k$ that were attached. Furthermore, if $X \subset G$ is a set of 4 points such that $\tb{X} = \lambda/2$ and $\td{X} = \lambda$, then $X$ is contained in a cycle $C_k$ and $L_k = 2\lambda$.
\end{theorem}
\begin{proof}
	First, label the metric trees and the cycles as $G_1, G_2, \dots, G_N$ depending on the order that they were attached. Consider a cycle $C_k$ and denote it as $G_m$. Suppose that there is a path $\gamma$ between $x,x' \in C_k$ that intersects $C_k$ only at $x$ and $x'$. We claim that the edge $[x,x']$ is in $C_k$. Otherwise, since we are only attaching metric graphs at an edge or a vertex, there are two different metric graphs attached to $C_k$, one at $x$ and one at $x'$. However, if we follow $\gamma$, we will find a metric graph that was attached to the previous metric graphs at two disconnected segments. This contradicts the construction of $G$, so $[x,x']$ is an edge of $C_k$. Thus, $d_G(x,x') < |\gamma|$. Moreover, the only paths between non-adjacent points $x,x' \in C_k$ lie in $C_k$. Thus, $C_k$ is isometric to a circle which, as a metric space, has $\diam_G(C_k)=\lambda_k$. Then $(\lambda_k/2, \lambda_k) \in \Dvr{4,1}(C_k) \subset \Dvr{4,1}(G)$.\\
	\indent Now, suppose that there is a point $(\lambda/2, \lambda) \in \Dvr{4,1}(G)$ generated by a set $X = \{x_1, x_2, x_3, x_4 \}\subset G$, with the labels chosen so that $\td{X} = \min \{d_G(x_{1}, x_{3}), d_G(x_{2}, x_{4}) \}$. By Lemma \ref{lemma:persistence_bounds} item \ref{item:td_basic_bound}, $\tb{X} = \lambda/2 = d_G(x_{i}, x_{i+1})$ and $\td{X} = \lambda = d_G(x_{i}, x_{i+2})$ for all $1 \leq i \leq 4$. Find the largest $m$ such that $X \cap G_m \neq \emptyset$. By Theorem \ref{thm:square_in_graph_gluing}, either $X \subset G_1 \cup \cdots \cup G_{m-1}$, or $X \subset G_m$. If $X$ is not contained in $G_m$, we can keep using Theorem \ref{thm:square_in_graph_gluing} to remove metric graphs until we find one that contains $X$. Notice that $X$ cannot be contained in a metric tree $T_i$ because of Lemma \ref{lemma:Dnk_tree}, so $X \subset C_k$ for some $k$. Let $\gamma_{i}$ be the shortest path between $x_i$ and $x_{i+1}$. Then the sum $d_G(x_1,x_2)+d_G(x_2,x_3)+d_G(x_3,x_4)+d_G(x_4,x_1) = 2\lambda$ equals $L_k$ because the path $\gamma_1 \cup \gamma_2 \cup \gamma_3 \cup \gamma_4$ is a cycle contained in $C_k$. Since $L_k = 2\lambda_k$, $\lambda = \lambda_k$.
\end{proof}

\indent Now we prove Proposition \ref{prop:good_families}, which was stated at the start of the section. Since the metric graphs in Theorem \ref{thm:num_cycles_equals_corners_D41} are pasted along a contractible space, we can detect the homotopy type of the metric graph.

\begin{figure}[h]
    \centering
    \includegraphics[width=275pt, height=250pt]{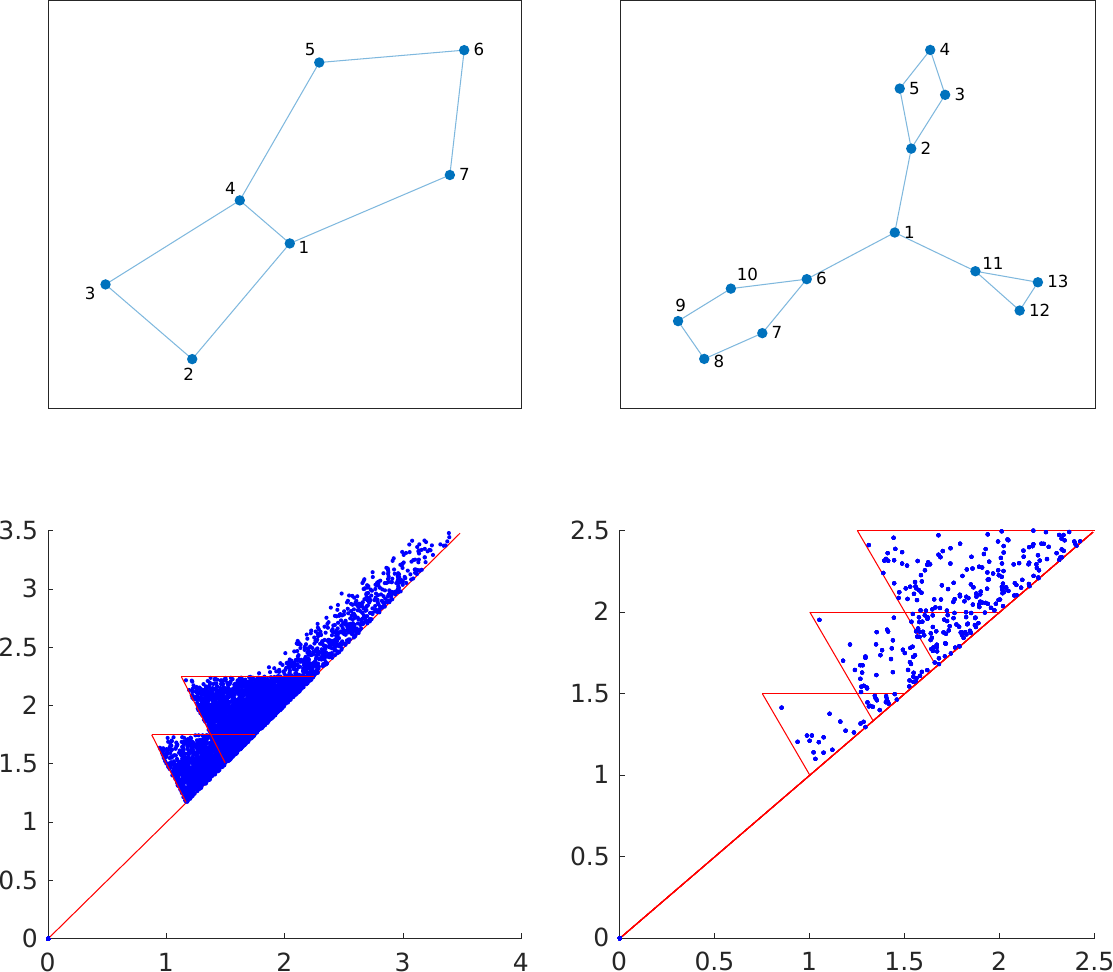}
    \caption{Two examples of admissible metric graphs $G$ as in Proposition \ref{prop:good_families} and their persistence set $\Dvr{4,1}(G)$. The red triangles are the boundaries of the sets $\Dvr{4,1}(C)$ for every cycle $C \subset G$. \textbf{Left:} Two cycles of lengths $\ell_1=3.5$ and $\ell_2=4.5$ pasted over an edge of length $\alpha = 0.5 < \frac{1}{3} \min(\ell_1, \ell_2)$. \textbf{Right:} A tree of cycles. Each persistence set was found by sampling 100,000 uniform configurations from $G$.}
    \label{fig:good_families}
\end{figure}

\begin{proof}[Proof of Proposition \ref{prop:good_families}]
	Attaching a metric tree to a metric graph doesn't change its homotopy type, while attaching a cycle $C_k$ to $G_1 \cup \dots \cup G_m$ along a contractible subspace induces a homotopy equivalence $(G_1 \cup \cdots \cup G_m) \cup C_k \simeq (G_1 \cup \cdots \cup G_m) \vee C_k$. Thus, by induction, $G \simeq C_1 \vee \cdots \vee C_n$, and $\beta_1(G)=n$.
\end{proof}
\section{Discussion and Questions}
 Here we mention other results that can be obtained: 

\begin{itemize}
\item As an application of the stability theorem and of our characterization results, one can show that the Gromov-Hausdorff distance between $\Sphere{1}$ and $\Sphere{m}$ is bounded below by $\frac{\pi}{14.6344}$ when $m=2$ and by $\frac{\pi}{8}$ when $m \geq 3$. See \cite{curvature-sets-pds} for details.

\item As the objects $\Unk{n,k}{\vr}$ can be considerably complex, a system of coordinates $\{\zeta_\alpha:\D \to \R\}_{\alpha \in A}$ that exhausts the information contained in the persistence measures is desirable. See the preprint version \cite{curvature-sets-pds} for results in this direction. 

\item Another consequence of the stability of persistence measures is the concentration of $\Unk{n,k}{\FiltFunc}(X)$ as $n \to \infty$, which can also be found in \cite{curvature-sets-pds}.
\end{itemize}

Now we outline some open questions and conjectures.
\begin{itemize}
	\item \textbf{Are there rich classes of compact metric spaces that can be distinguished with persistence sets?}\\
	This question is a generalization of Theorem \ref{thm:num_cycles_equals_corners_D41} and Proposition \ref{prop:good_families}. The persistence set $\Dvr{4,1}(G)$ captures the number and length of cycles in a metric graph $G$ that was constructed according to the instructions in Theorem \ref{thm:num_cycles_equals_corners_D41}. Are there other families of compact metric spaces where higher order diagrams $\Dvr{n,k}(G)$ can detect relevant features? In other words, are there families $\mathcal{C}$ of compact metric spaces such that
	\begin{equation*}
		\sup_{n,k} \dH^\D(\Dvr{n,k}(X), \Dvr{n,k}(Y))
	\end{equation*}
    is a metric when $X, Y \in \mathcal{C}$?

	\item \textbf{Description $\Dvr{2k+2,k}(\Sphere{m}_E)$ for all $k$ and $m$}:  Propositions \ref{prop:D41_R2_S2} and \ref{prop:Dk_Sn_stabilizes} are a step in that direction. In fact, the latter implies that we only need to find $\Dvr{2k+2,k}(\Sphere{2k}_E)$ to determine $\Dvr{2k+2,k}(\Sphere{m}_E)$ for all spheres with $m \geq 2k+1$. In particular for $\Dvr{6,2}(\Sphere{2}_E)$, does Conjecture \ref{conj:D62_S2_E} hold?
	
	\item \textbf{Description of $\Dvr{2k+2,k}(\Sphere{n}_E):$} When $k=1$, Corollary \ref{cor:D41_Sm_stabilizes} shows that $\Dvr{4,1}(\Sphere{m})$ stabilizes at $m=2$ instead of $m=3$, as given by Proposition \ref{prop:Dk_Sn_stabilizes}. The key to the reduction was the use of Ptolemy's inequality as in Theorem \ref{thm:D41_Mk}. A natural follow up question, even if it is subsumed by the previous one, is when does $\Dvr{2k+2,k}(\Sphere{m}_E)$ really stabilize for general $k$.
\end{itemize}

\newcommand{\etalchar}[1]{$^{#1}$}

\appendix
\section{Relegated proofs}
\label{sec:appendix}

\StabilitydGWHat*
\begin{proof}
	Let $p \in [1,\infty)$. We will first construct a bound for $\dW{p}(\mu_n(X), \mu_n(Y))$. Taking the limit as $p \to \infty$ and the supremum over $n$ will give the desired bound for $\widehat{d}_{\mathcal{GW},\infty}(X,Y)$. Let $\eta > 0$ such that $\frac{\eta}{2} > \dGW{p}(X,Y)$ and let $\mu$ be a coupling of $\mu_X$ and $\mu_Y$ such that
	\begin{align*}
		\big(2\dGW{p}(X,Y) \big)^p
		&\leq \big(\dis_p(\mu) \big)^p
		= \iint_{(X \times Y)^2} |d_X(x,x')-d_Y(y,y')|^p \mu(dx \times dy) \mu(dx' \times dy') < \eta^p.
	\end{align*}
	Recall that the curvature sets $K_n(X)$ are equipped with the curvature measure $\mu_n(X) = \left(\Psi_X^{(n)}\right)_\# \mu_X^{\otimes n}$. Observe that the product $\mu^{\otimes n}$ is a coupling of the product measures $\mu_X^{\otimes n}$ and $\mu_Y^{\otimes n}$, so the pushforward $\nu = (\Psi_X^{(n)} \times \Psi_Y^{(n)})_\# \mu^{\otimes n}$ is a coupling between the curvature measures $\mu_n(X)$ and $\mu_n(Y)$. Let $\mathbb{X} = (x_1,\dots,x_n) \in X^n$ and $\mathbb{Y} = (y_1,\dots,y_n) \in Y^n$. Then, by definition of Wasserstein distance, and by a change of variables, we obtain
	\begin{align*}
		\big(\dW{p}(\mu_n(X), \mu_n(Y)) \big)^p
		&\leq \iint_{K_n(X) \times K_n(Y)} \|M_X-M_Y\|_\infty^p \nu(dM_X \times dM_Y)\\
		&= \int_{X^n \times Y^n} \|\Psi_X^{(n)}(x_1,\dots,x_n) - \Psi_Y^{(n)}(y_1,\dots,y_n)\|_\infty^p \ \mu^{\otimes n}(d\mathbb{X} \times d\mathbb{Y}).
	\end{align*}
	Denote the previous integral by $I$. Define
	$\Delta_{ij}(\mathbb{X}, \mathbb{Y}) := |d_X(x_i,x_j)-d_Y(y_i,y_j)|,
	$ for $1 \leq i,j \leq n$. Observe that $\|\Psi_X^{(n)}(x_1,\dots,x_n) - \Psi_Y^{(n)}(y_1,\dots,y_n)\|_\infty = \max_{1 \leq i, j \leq n} \Delta_{ij}(\mathbb{X}, \mathbb{Y})$, so bounding the maximum with a sum gives
	\begin{align*}
		I &= \int_{X^n \times Y^n} \max_{1 \leq i, j \leq n} \Delta_{ij}^p(\mathbb{X}, \mathbb{Y}) \ \mu^{\otimes n}(d\mathbb{X} \times d\mathbb{Y})\\
		&\leq \sum_{1 \leq i,j \leq n} \int_{(X \times Y)^n} \Delta_{ij}^p(\mathbb{X}, \mathbb{Y}) \ \mu^{\otimes n}(d\mathbb{X} \times d\mathbb{Y})\\
		&= \binom{n}{2} \int_{(X \times Y)^n} \Delta_{12}^p(\mathbb{X}, \mathbb{Y}) \ \mu(dx_1 \times dy_1) \mu(dx_2 \times dy_2) \cdots \mu(dx_n \times dy_n)\\
		&= \binom{n}{2} \int_{(X \times Y)^2} |d_X(x_1,x_2)-d_Y(y_1,y_2)|^p \ \mu(dx_1 \times dy_1) \mu(dx_2 \times dy_2)\\
		&< \binom{n}{2} \eta^p.
	\end{align*}
	Then, taking the $p$-th root and letting $\eta \searrow 2\dGW{p}(X,Y)$ gives
	$
	\dW{p}(\mu_n(X), \mu_n(Y)) \leq 2\binom{n}{2}^{\frac{1}{p}} \dGW{p}(X,Y).
	$
	Lastly, $\binom{n}{2}^{\frac{1}{p}}$ approaches 1 as $p \to \infty$, so the limit of the above inequality is
	\begin{equation*}
		\frac{1}{2} \dW{\infty}(\mu_n(X), \mu_n(Y)) \leq \dGW{\infty}(X,Y).
	\end{equation*}
	Taking the supremum over $n$ in the left side gives the desired bound for $\widehat{d}_{\mathcal{GW},\infty}(X,Y)$.
\end{proof}

\StabilityUnk*
\begin{proof}
	This proof follows roughly the same outline as that of Theorem \ref{thm:stability_Dnk}. Let $\eta > \dW{p}(\mu_n(X), \mu_n(Y))$. Choose a coupling $\mu \in \coup(\mu_n(X), \mu_n(Y))$ such that
	\begin{equation*}
		\Big(\dW{p}(\mu_n(X), \mu_n(Y)) \Big)^p \leq \iint_{\Kn_n(X) \times \Kn_n(Y)} \|M-M' \|^p_{\infty} \ \mu(dM \times dM') < \eta^p,
	\end{equation*}
	where $\| \cdot \|_\infty$ denotes the $\ell^\infty$ norm on $\R^{n \times n}$. Notice that the support of $\mu$ is contained in $K_n(X) \times K_n(Y)$. The pushforward $\nu=(\dgm_k^\FiltFunc \times \dgm_k^\FiltFunc)_\# \mu$ of the coupling $\mu$ is a coupling of the pushforwards $(\dgm_k^\FiltFunc)_\# \mu_n(X) = \Unk{n,k}{\FiltFunc}(X)$ and $(\dgm_k^\FiltFunc)_\# \mu_n(Y) = \Unk{n,k}{\FiltFunc}(Y)$. Thus, a change of variables gives
	\begin{align*}
		\left[\dW{p}^\D(\Unk{n,k}{\FiltFunc}(X), \Unk{n,k}{\FiltFunc}(Y)) \right]^p
		&\leq \iint_{\Dnk{n,k}{\FiltFunc}(X) \times \Dnk{n,k}{\FiltFunc}(Y)} [\dB(D,D')]^p \: \nu(dD \times dD') \\
		&= \iint_{\Kn_n(X) \times \Kn_n(Y)} \left[\dB(\dgm_{k}^\FiltFunc(M), \dgm_{k}^\FiltFunc(M')) \right]^p \mu(dM \times dM').
	\end{align*}
	
	Recall from the proof of Theorem \ref{thm:stability_Dnk} that $\dB(\dgm_{k}^\FiltFunc(M), \dgm_{k}^\FiltFunc(M')) \leq \frac{L(\FiltFunc)}{2} \|M-M'\|_\infty$. Thus, the previous integral is bounded above by
	\begin{multline*}
		\iint_{\Kn_n(X) \times \Kn_n(Y)} \left[\frac{L(\FiltFunc)}{2} \|M-M'\|_\infty \right]^p \mu(dM \times dM')\\
		\begin{aligned}
			&= \left( \dfrac{L(\FiltFunc)}{2} \right)^p \iint_{\Kn_n(X) \times \Kn_n(Y)}  \|M-M'\|_{\infty}^p \ \mu(dM \times dM')
			&< \left( \dfrac{L(\FiltFunc)}{2} \right)^p \eta^p.
		\end{aligned}
	\end{multline*}
	Taking the $p$-th root and letting $\eta \searrow \dW{p}(\mu_n(X), \mu_n(Y))$ gives 
	\begin{equation*}
		\dW{p}^\D(\Unk{n,k}{\FiltFunc}(X), \Unk{n,k}{\FiltFunc}(Y)) \leq \frac{L(\FiltFunc)}{2} \cdot \dW{p}(\mu_n(X), \mu_n(Y)) \leq L(\FiltFunc) \cdot \widehat{d}_{\mathcal{GW},p}(X,Y).
	\end{equation*}
\end{proof}

\subsection{Probabilistic approximation}
\label{sec:appendix_probability}
Assuming that $f_X(\epsilon) > 0$ for all $\epsilon > 0$ (a condition that is satisfied by compact Riemannian manifolds, for instance), let $C_X:\N \times \R_+ \to \R_+$ by $$\displaystyle C_X(n,\epsilon) := \frac{ \exp(-n f_X(\epsilon/4))}{f_X(\epsilon/4)}.$$
\begin{lemma}[Coverage of $X^n$]
    \label{lemma:coverage_product}
    For the space $X^n$ equipped with the $\ell^\infty$ product metric, $f_{X^n}(\epsilon) = f_X^n(\epsilon)$ and, hence, $C_{X^n}(N,\epsilon) = \exp(-N \cdot f_X^n(\epsilon/4))/ f_X^n(\epsilon/4)$.
\end{lemma}
\begin{proof}
    Denote the elements of $X^n$ as $\mathbf{x} = (x_1, \dots, x_n)$. Since $X^n$ is equipped with the $\ell^\infty$ product metric, $B_\epsilon(\mathbf{x}) = B_\epsilon(x_1) \times \cdots \times B_\epsilon(x_n)$. Then $\mu_{X}^{\otimes n} \left(B_\epsilon(\mathbf{x}) \right) = \mu_X\left(B_\epsilon(x_1)\right) \times \cdots \times \mu_X\left(B_\epsilon(x_n)\right)$ and, since each $x_i$ is independent, 
    \begin{align*}
        f_{X^n}(\epsilon)
        &= \inf_{\mathbf{x} \in X^n} \mu_{X^n}(B_\epsilon(\mathbf{x}))
        = \inf_{x_1, \dots, x_n \in X} \mu_X\left(B_\epsilon(x_1)\right) \times \cdots \times \mu_X\left(B_\epsilon(x_n)\right)\\
        &= \prod_{i=1}^n \inf_{x_i \in X} \mu_X\left( B_\epsilon(x_i) \right)
        = f_X^n(\epsilon).
    \end{align*}
     The formula for $C_{X^n}(N,\epsilon)$ follows immediately.
\end{proof}

\CoverageKn*
\begin{proof}
    Let $\pi_j:X^n \to X$ be the $j$-th coordinate projection, and let $\mathbf{x} = (x_1, \dots, x_n)$ and $\mathbf{y} = (y_1, \dots, y_n)$ be elements of $X^n$. By stability of persistent homology, $\dB(\dgm_k(\mathbf{x}), \dgm_k(\mathbf{y})) \leq \|\Psi_X^{(n)}(\mathbf{x}) - \Psi_X^{(n)}(\mathbf{y})\|_\infty$. We claim that both terms are bounded above by $\displaystyle 2 \cdot d_{X^n}( \mathbf{x}, \mathbf{y}) = \max_{j=1,\dots,n}( x_j, y_j)$. Indeed, the triangle inequality gives $d_X(x_i, x_j) \leq d_X(x_i,y_i) + d_X(y_i,y_j) + d_X(y_j,x_j)$, so
    \begin{equation*}
        d_X(x_i, x_j) - d_X(y_i,y_j) \leq d_X(x_i,y_i) + d_X(x_j,y_j).
    \end{equation*}
    The symmetric argument yields $|d_X(x_i, x_j) - d_X(y_i,y_j)| \leq d_X(x_i,y_i) + d_X(x_j,y_j)$. Then
    \begin{align*}
        \|\Psi_X^{(n)}(\mathbf{x}) - \Psi_X^{(n)}(\mathbf{y})\|_\infty
        &= \max_{i,j=1,\dots,n} |d_X(x_i, x_j) - d_X(y_i,y_j)| \\
        &\leq \max_{i,j=1,\dots,n} d_X(x_i,y_i) + d_X(x_j,y_j) \\
        &\leq 2 \max_{j=1,\dots,n} d_X(x_j,y_j) = 2 d_{X^n}(\mathbf{x}, \mathbf{y}).
    \end{align*}
    This is what we wanted.\\
    \indent These relations extend to the Hausdorff distance, that is,
    \begin{align*}
        \dH^\D \left(\{ \dgm_k^\vr(\mathbf{x}_i) \}_{i=1}^N, \Dvr{n,k}(X) \right)
        &\leq \dH^{K_n(X)} \left(\{ \Psi_X^{(n)}(\mathbf{x}_i) \}_{i=1}^N, \Kn_n(X) \right) \\
        &\leq 2\dH^{X^n} \left(\{ \mathbf{x}_i \}_{i=1}^N, X^n\} \right).
    \end{align*}
    Thus, if $\dH^{X^n} \left(\{ \mathbf{x}_i \}_{i=1}^N, X^n\} \right)$ is smaller than $2\epsilon$, then both $\dH^\D \left(\{ \dgm_k^\vr(\mathbf{x}_i) \}_{i=1}^N, \Dvr{n,k}(X) \right)$ and $\dH^{K_n(X)} \left(\{ \Psi_X^{(n)}(\mathbf{x}_i) \}_{i=1}^N, \Kn_n(X) \right)$ are smaller than $\epsilon$. By \cite[Theorem 34]{clust-um} and Lemma \ref{lemma:coverage_product}, this happens at least with probability
    \begin{equation*}
        1-C_{X^n}(N,2\epsilon) = 1-\frac{\exp(-N \cdot f_X^n(\epsilon/2))}{f_X^n(\epsilon/2)}.
    \end{equation*}
    Since the last expression is decreasing in $N$, the probability is bounded below by $1-C_{X^n}(N_0,2\epsilon) \geq p$. This gives the first claim.\\
    \indent Conversely, $\dH^{\Kn_n(X)} \left(\{ \Psi_X^{(n)}(\mathbf{x}_i) \}_{i=1}^N, \Kn_n(X) \right) > \epsilon$ implies $\dH^{X^n} \left(\{ \mathbf{x}_i \}_{i=1}^N, X^n\} \right) > 2\epsilon$. By \cite[Theorem 34]{clust-um}, the probability of the latter is bounded above by $C_{X^n}(N,2\epsilon)$. The formula in Lemma \ref{lemma:coverage_product} implies that $\sum_{N=1}^\infty C_{X^n}(N,2\epsilon) < \infty$, so $\dH^{K_n(X)} \left(\{ \Psi_X^{(n)}(\mathbf{x}_i) \}_{i=1}^N, \Kn_n(X) \right)$ converges to 0 almost surely by the Borel-Cantelli Lemma. The same argument gives $\dH^\D \left(\{ \dgm_k^\vr(\mathbf{x}_i) \}_{i=1}^N, \Dvr{n,k}(X) \right) \to 0$ almost surely.
\end{proof}

\section{Additional results related to the classification task} \label{sec:weighted-results}
In this section we continue the discussion from Section \ref{sec:classification}.

\paragraph{Results via re-weighting distance matrices.}
To circumvent the issue of a matrix $\mathcal{W}_k$ with high $P_e(\mathcal{W}_k)$ dominating $\mathcal{W}_{\max}$, we define the function $\omega{\text -}\max$ for a given $\omega = (\omega_0, \omega_1, \omega_2)$ by $\omega{\text -}\max_k(\mathcal{W}_k) := \max_k(\omega_k \cdot \mathcal{W}_k)$. We then look for an $\omega$ that minimizes the 1-nearest neighbor classification error $P_e(\omega{\text -}\max \mathcal{W}_k)$. We used the MATLAB function \texttt{fminsearch}, which requires an initial guess $\hat\omega$ to find the optimal value. Our first guess was $\hat\omega = \left(\diam\inv(\mathcal{W}_0), \diam\inv(\mathcal{W}_1), \diam\inv(\mathcal{W}_2) \right)$, and the result is in the first row of Table 4. Even though \texttt{fminsearch} did not find an $\omega$ with better classification error than $\hat\omega$, the 12.72 \% error of the latter is an improvement over the 19.28 \% obtained by $\mathcal{W}_{\max}$. If we wish to ignore $\mathcal{W}_0$ altogether, we can replace the first entry of $\hat\omega$ with 0, and the error is reduced to 9.14 \%. Lastly, the best classification error is 7.38 \%, which was obtained by replacing the first entry of $\hat\omega$ with 1. Table \ref{tab:sumner_weighted} contains the best classification error and the optimal $\omega$ found by \texttt{fminsearch} across several choices of $\hat\omega$, and the corresponding heatmaps are shown in Figure \ref{fig:sumner_weighted}.

\paragraph{Related results.}
Two other papers perform classification experiments on databases from \cite{sumner-paper}. \cite{ccsg09a} used the same database as us and tried to classify the shapes using the persistence diagrams of a certain (more sophisticated) variant of the VR-filtration. The error rate reported therein was $4\%$.\\
The dataset used in \cite{memoli-dghlp-long} contains an extra class (\texttt{lion}). The author defined an mm-space $(\mathbb{X}_i, d_i, \nu_i)$ as a farthest point subsampling of $G_i$ with 50 points and endowed $\mathbb{X}_i$ with a Voronoi probability $\nu_i$. This measure is defined by setting $\nu_i(x)$ to be the the proportion of points in $G_i$ that are closer to $x \in \mathbb{X}_i$ than to any other $x' \in \mathbb{X}_i$. The metrics used therein are (in the notation of Definition \ref{def:curvature_measures}) $\dW{1}(\mu_2(\mathbb{X}_i), \mu_2(\mathbb{X}_j))$ and a function called $\mathbf{FLB}_1(\mathbb{X}_i, \mathbb{X}_j)$. The average classification error of the Wasserstein distance between the curvature measures $\mu_2(\mathbb{X}_i)$ and $\mu_2(\mathbb{X}_j)$ over 10,000 choices of the training set was 2.5 \%. The error of $\mathbf{FLB}_1$ over the same number of trials was $14.1\ \%$.

\begin{table}
    \centering
    \begin{tabular}{c|c|c}
        $\hat\omega$ & Optimal $\omega$ & $P_e(\omega {\text -} \max(\mathcal{W}_k))$ \\
        \hline
        (11.9, 370.8, 8078.8) & (11.9, 370.8, 8078.8) & 12.72 \% \\
        (1, 370.8, 8078.8) & (1.1, 366.2, 8108.6) & 7.38 \% \\
        (0, 370.8, 8078.8) & (0, 370.8, 8078.8) & 9.14 \% \\
        (11.9, 1, 8078.8) & (11.9, 1, 8482.7) & 19.89 \% \\
        (11.9, 0, 8078.8) & (11.9, 0, 8482.7) & 19.89 \% \\
        (11.9, 370.8, 1) & (11.9, 370.8, 1) & 12.72 \% \\
        (11.9, 370.8, 0) & (11.9, 370.8, 0) & 12.72 \% 
    \end{tabular}
    \caption{Average classification error $P_e$ of $\omega {\text -} \max_k(\mathcal{W}_k) = \max_k (\omega_k \cdot \mathcal{W}_k)$ over 2000 trials. See the text for details.}
    \label{tab:sumner_weighted}
\end{table}

\begin{figure}
    \centering
    \includegraphics[width=0.95\linewidth]{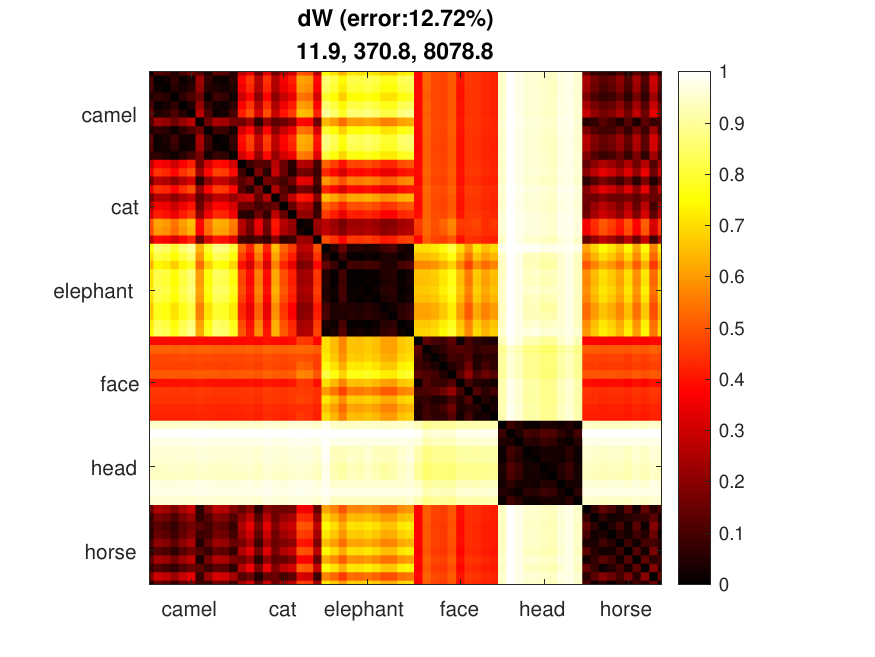}
    \includegraphics[width=\textwidth]{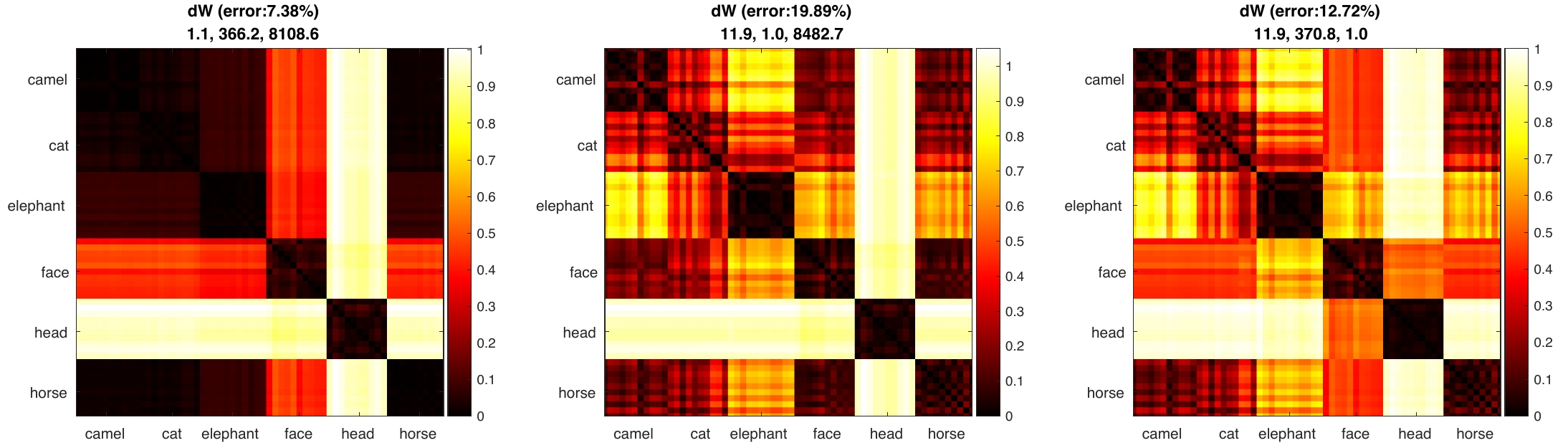}
    \includegraphics[width=\linewidth]{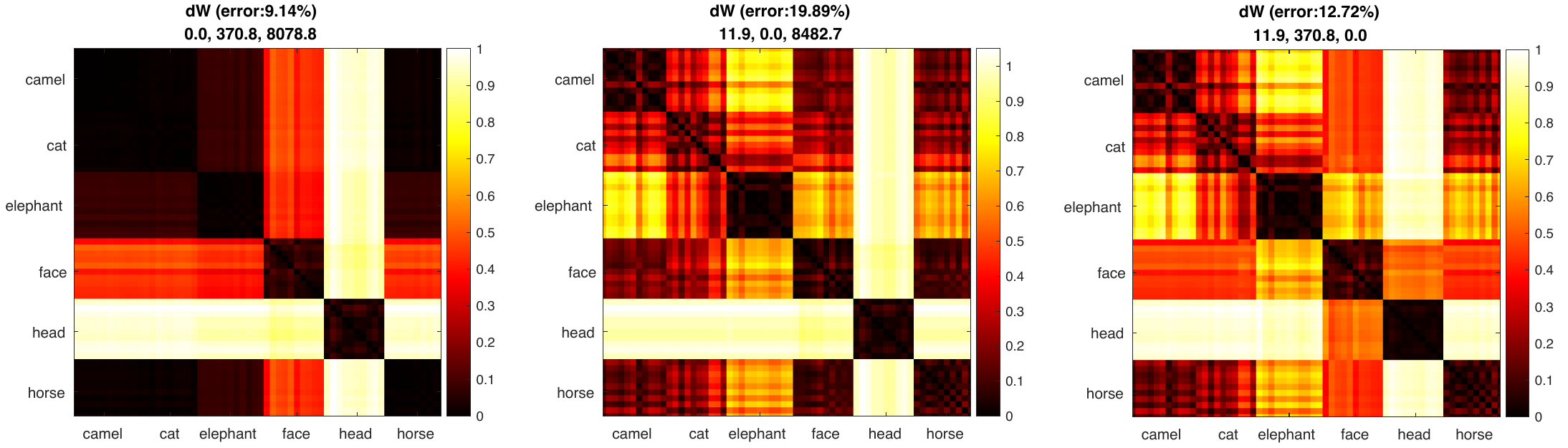}
    
    \caption{The heatmaps of the matrices $\omega{\text -}\max_k(\mathcal{W}_k)$ in Table \ref{tab:sumner_weighted}. The triple of numbers above each heatmap are the optimal $\omega$ found by \texttt{fminsearch}.}
    \label{fig:sumner_weighted}
\end{figure}

\begin{table}
    \centering
    \begin{tabular}{c|c}
        Shape & Codensity \\
        \hline
        \texttt{camel} & $0.0314 \cdot 10^{-3}$ \\
        \texttt{cat} &  $0.0737 \cdot 10^{-3}$ \\
        \texttt{elephant} & $0.0332 \cdot 10^{-3}$ \\
        \texttt{face} & $0.0340 \cdot 10^{-3}$ \\
        \texttt{head} & $0.1162 \cdot 10^{-3}$ \\
        \texttt{horse} & $0.0823 \cdot 10^{-3}$
    \end{tabular}
    \caption{Sampling codensity ($\mathrm{area}/(\#\mathrm{vertices} \cdot \diam)$) of each class in the database. Notice that \texttt{head} has the largest codensity (i.e. lowest density), which might explain why $\mathcal{B}_2$ could separate it from the other classes.}
    \label{tab:codensity}
\end{table}

 \end{document}